\documentclass[11pt,a4paper,usenames,dvipsnames,reqno]{amsart}
\usepackage{amscd}
\usepackage{amsfonts,amssymb}
\usepackage{mathtools}
\usepackage[margin=1in]{geometry}
\usepackage{fancybox}
\usepackage{pdflscape}
\usepackage{rotating}
\usepackage{longtable}

\usepackage[all,arc]{xy}
\usepackage{enumerate}
\usepackage{mathrsfs}
\usepackage{color}   
\usepackage{hyperref}
\hypersetup{
    colorlinks=true, 
    linktoc=all,     
    linkcolor=blue,  
}
\usepackage{amsthm}
\usepackage{amsmath}
\usepackage{graphicx}
\usepackage{wasysym}
\usepackage{pgf,tikz}
\usetikzlibrary{positioning}
\usepackage{ marvosym }
\usepackage{tikz-cd}
\usepackage{ textcomp }
\usepackage{bbm}
\usepackage{ stmaryrd }
\usepackage{enumitem}

\numberwithin{equation}{section}

\newcommand{\cc}{\mathbb C}

\newcommand{\zz}{\mathbb Z}

\newcommand{\qq}{\mathbb Q}
\newcommand{\rr}{\mathbb R}
\newcommand{\A}{\mathbb A}

\newcommand{\GG}{\mathbb G}

\newcommand{\la}{\langle}
\newcommand{\ra}{\rangle}
\newcommand{\lra}{\longrightarrow}

\newcommand{\al}{\alpha}
\newcommand{\be}{\beta}
\newcommand{\ga}{\gamma}
\newcommand{\de}{\delta}
\newcommand{\De}{\Delta}

\newcommand{\ep}{\epsilon}
\newcommand{\lam}{\lambda}

\DeclareMathOperator{\GL}{GL}

\DeclareMathOperator{\Hom}{Hom}

\DeclareMathOperator{\Ind}{Ind}

\DeclareMathOperator{\SL}{SL}

\DeclareMathOperator{\Sp}{Sp}
\DeclareMathOperator{\Sym}{Sym}

\newcommand{\lto}{\longrightarrow}

\newcommand\norm[1]{\left\lVert#1\right\rVert}

\newenvironment{psmatrix}
  {\left(\begin{smallmatrix}}
  {\end{smallmatrix}\right)}

\newcommand{\Gm}{\mathbb{G}_m}

\newcommand{\fg}{\mathfrak g}
\newcommand{\ft}{\mathfrak t}

\newcommand{\fm}{\mathfrak m}
\newcommand{\fn}{\mathfrak n}
\newcommand{\fl}{\mathfrak l}
\newcommand{\fp}{\mathfrak p}

\newcommand{\fsl}{\mathfrak{sl}}

\newcommand{\fgl}{\fg\fl}

\newcommand{\calf}{\mathcal{F}}

\newcommand{\cals}{\mathcal{S}}

\newcommand{\calo}{\mathcal{O}}

\newcommand{\calr}{\mathcal{R}}

\newcommand{\iso}{\overset{\sim}{\longrightarrow}}

\makeatletter
\def\Ddots{\mathinner{\mkern1mu\raise\p@
\vbox{\kern7\p@\hbox{.}}\mkern2mu
\raise4\p@\hbox{.}\mkern2mu\raise7\p@\hbox{.}\mkern1mu}}
\makeatother

\makeatletter
\DeclareRobustCommand\bigop[1]{%
  \mathop{\vphantom{\sum}\mathpalette\bigop@{#1}}\slimits@
}
\newcommand{\bigop@}[2]{%
  \vcenter{%
    \sbox\z@{$#1\sum$}%
    \hbox{\resizebox{\ifx#1\displaystyle.9\fi\dimexpr\ht\z@+\dp\z@}{!}{$\m@th#2$}}%
  }%
}
\makeatother

\newtheorem{Thm}{Theorem}[section]
\newtheorem{Prop}[Thm]{Proposition}

\newtheorem{Lem}[Thm]{Lemma}
\newtheorem{Cor}[Thm]{Corollary}

\newtheorem{Conj}[Thm]{Conjecture}

\newtheorem{Quest}[Thm]{Question}

\newtheorem{Sublem}[Thm]{Sublemma}

\theoremstyle{definition}
\newtheorem{Def}[Thm]{Definition}

\theoremstyle{remark}
\newtheorem{Rem}[Thm]{Remark}

\newcommand{\quash}[1]{}

\title[Harmonic analysis on certain spherical varieties]{Harmonic analysis on certain spherical varieties}

\author{Jayce R. Getz}
\address{Department of Mathematics\\
Duke University\\
Durham, NC 27708}
\email{jgetz@math.duke.edu}

\author{Chun-Hsien Hsu}
\address{Department of Mathematics\\
Duke University\\
Durham, NC 27708}
\email{simonhsu@math.duke.edu}

\author{Spencer Leslie}
\address{Department of Mathematics\\
Duke University\\
Durham, NC, 27708}
\email{lesliew@math.duke.edu}


\thanks{The first author is partially supported by NSF grant DMS 1901883, and received support from  D.~Kazhdan ERC grant AdG 6699655 and Y.~Choie's NRF 2017R1A2B2001807 grant.  The third author was partially supported by an AMS-Simons Travel Award and by NSF grant DMS-1902865.
}

\begin{document}

\begin{abstract}
Braverman and Kazhdan proposed a conjecture, later refined by Ng\^o and broadened to the framework of spherical varieties by Sakellaridis, that asserts that affine spherical varieties admit Schwartz spaces, Fourier transforms, and Poisson summation formulae.  
The first author in joint work with B.~Liu and later the first two authors proved these conjectures for certain spherical varieties $Y$ built out of triples of quadratic spaces.     However, in these works the Fourier transform was only proven to exist.  In the present paper we give, for the first time, an explicit formula for the Fourier transform on $Y.$ We also prove that it is unitary in the nonarchimedean case.   As preparation for this result, we give explicit formulae for Fourier transforms on the affine closures of Braverman--Kazhdan spaces attached to maximal parabolic subgroups of split, simple, simply connected groups.  These Fourier transforms are of independent interest, for example, from the point of view of analytic number theory.   
\end{abstract}

\maketitle
\setcounter{tocdepth}{1}
\tableofcontents

\section{Introduction}

In the seminal paper \cite{BK-lifting}, Braverman and Kazhdan proposed that the Poisson summation formula for a vector space is the first case of a general phenomenon.   Let $X$ be an affine spherical variety over a global field $F$ with smooth locus $X^{\mathrm{sm}} \subset X.$ Building on work in \cite{BKnormalized,Ngo:Hankel,SakellaridisSph},  one now expects that 
 there is a Schwartz space 
$\mathcal{S}(X(\A_F)) \subset C^\infty(X^{\mathrm{sm}}(\A_F))$ and a Fourier transform 
$$
\mathcal{F}_X:\mathcal{S}(X(\A_F)) \lto \mathcal{S}(X(\A_F))
$$
such that 
$$
\sum_{x \in X^{\mathrm{sm}}(F)}f(x)=\sum_{x \in X^{\mathrm{sm}}(F)}\mathcal{F}_X(f)(x),
$$
at least for test functions $f$ satisfying certain assumptions to eliminate ``boundary terms.''  Let us refer to this expectation as the Poisson summation conjecture.  The import of the Poisson summation conjecture is that it implies the analytic properties of Langlands $L$-functions (and hence, by converse theory, Langlands functoriality) in great generality.  

\begin{Rem}
A study of $\gamma$-factors for quite general Langlands $L$-functions using reductive monoids (an important special family of spherical varieties) is contained in \cite{SS}.  In the function field setting for many spherical varieties, including reductive monoids, a geometric interpretation of basic functions in the still conjectural Schwartz space is contained in \cite{BNS,sakellaridis2021intersection}. 
\end{Rem}

The only case that is completely understood is that of a vector space. However, the Poisson summation formula is known under assumptions on the test functions involved provided that $X$ is the affine closure of a Braverman--Kazhdan space, that is, a scheme of the form $P^{\mathrm{der}} \backslash G$ where $G$ is a reductive group and $P \leq G$ is a parabolic subgroup \cite{BKnormalized,Getz:Hsu,Getz:Liu:BK,GurK,Kobayashi:Mano,Jiang:Luo:Zhang}.

Though the original motivation for the Poisson summation conjecture comes from Langlands functoriality, in personal communications to the first author, Kazhdan has emphasized that they should have implications broadly in harmonic analysis.  This is also the theme of the monograph \cite{Kobayashi:Mano}, which develops harmonic analysis on a special family of Braverman--Kazhdan spaces in the archimedean case.  Thus, though our primary motivation for this work was to prove Theorem \ref{Thm:main:intro} below, we have taken the occasion to develop the theory of Fourier transforms for many Braverman--Kazhdan spaces to a point where one can use them in harmonic analysis (or analytic number theory).  The Schwartz space of the affine closure of a  Braverman--Kazhdan space with its associated Fourier transform is in a strict sense a generalization of the Fourier transform on the Schwartz space of a vector space.  Whenever one has employed Fourier transforms to answer questions on vector spaces, one can try to do the same for Braverman--Kazhdan spaces.

\subsection{The Fourier transform for triples of quadratic spaces} In \cite{Getz:Liu:Triple}, the Poisson summation conjecture was proved for the first time for a spherical variety that is not the affine closure of a Braverman--Kazhdan space.  In more detail, let $F$ be a number field, and let $(V_i,Q_i),$ $1 \leq i \leq 3$ be a triple of even dimensional vector spaces $V_i$ over $F$ equipped with nondegenerate quadratic forms $Q_i$.  Let $V:=\prod_{i=1}^3V_i$ and for $F$-algebras $R,$ let
\begin{align}
    Y(R):=\left\{(v_1,v_2,v_3) \in V(R): Q_1(v_1)=Q_2(v_2)=Q_3(v_3)\right\}.
\end{align}
Let $Y^{\mathrm{ani}} \subset Y$ be the open complement of the vanishing locus of $Q_i$ (which is independent of $i$).  The $\mathrm{ani}$ stands for anisotropic.  In \cite{Getz:Liu:Triple} a Poisson summation formula was proved for this scheme.  However it was phrased in terms of functions and a Fourier transform on an auxiliary space; the theory was not intrinsic to $Y.$  In \cite{Getz:Hsu} the first two authors defined the Schwartz space of $Y$ and proved the existence of a Fourier transform 
$$
\mathcal{F}_Y :\mathcal{S}(Y(\A_F)) \lto \mathcal{S}(Y(\A_F))
$$
such that the Poisson summation conjecture holds for suitable functions $f \in \mathcal{S}(Y(\A_F)).$  The Fourier transform $\mathcal{F}_Y$ is a restricted tensor product of 
local transforms
$$
\mathcal{F}_{Y_{F_v}}:\mathcal{S}(Y(F_v)) \lto \mathcal{S}(Y(F_v))
$$
for all places $v.$  Below we will abuse notation and write simply $\mathcal{F}_Y$ for these local transforms.  The proof of the existence of $\mathcal{F}_Y$ in \cite{Getz:Hsu} is indirect, and does not provide any formula for $\mathcal{F}_Y$. In this paper, we prove such a formula.  

Let $F$ be a local field and let 
$
\psi:F \to \cc^\times
$
be a nontrivial additive character.  Moreover, for $a=(a_1,a_2,a_3) \in (F^\times)^3$, let 
 \begin{align}
[a]:=a_1a_2a_3, \quad \mathfrak{r}(a)&:=\frac{(a_1a_2+a_2a_3+a_3a_1)^2}{[a]}.
 \end{align}
\begin{Thm} \label{Thm:main:intro} Let $F$ be a local field of characteristic $0$.  
 Suppose $d_i:=\dim V_i>2$ for all $1\le i\le 3$ and $Y^{\mathrm{sm}}(F) \neq \emptyset$.  There is a constant $c \in \cc^\times$ depending on $\psi,F$, and the $Q_i$ such that for all $f \in \mathcal{S}(V(F))$ and  $\xi \in Y^{\mathrm{ani}}(F),$
\begin{align*}
&\mathcal{F}_Y(f)(\xi)\\
&=c\int_{F^\times}\psi(z^{-1}) \Bigg( \int_{ (F^\times)^3} \overline{\psi}(z^2\mathfrak{r}(a))
 \Bigg(\int_{Y(F)}\psi\Bigg(\left\langle \frac{\xi}{a},y \right\rangle -\frac{Q(\xi)Q(y)}{9z^2[a]}\Bigg)f(y)d\mu(y)\Bigg) \frac{\chi_Q(a)d^\times a}{\{a\}^{d/2-1}}   \Bigg)d^\times z.
\end{align*}
\end{Thm}
\noindent 
Here we use \cite[Lemma 5.3]{Getz:Hsu} to regard $f$ as an element of $\mathcal{S}(Y(F))$ by restriction.  Moreover
\begin{align*}
\frac{\xi}{a}:=\left(\frac{\xi_1}{a_1},\frac{\xi_2}{a_2},\frac{\xi_3}{a_3} \right), \quad Q:=Q_1+Q_2+Q_3, \quad  \frac{\chi_Q(a)d^\times a}{\{a\}^{d/2-1}}:=\prod_{i=1}^3\frac{\chi_{Q_i}(a_i) d^\times a_i}{|a_i|^{d_i/2-1}},
\end{align*}
where the quadratic character $\chi_{Q_i}$ attached to $Q_i$ is defined as in \eqref{chiQi}, $\langle\cdot,\cdot\rangle:=\sum_{i=1}^3 \langle\cdot,\cdot\rangle_i,$ where $\langle\cdot,\cdot\rangle_i$ is the pairing attached to $Q_i,$ 
and $d\mu(y)$ is the measure defined in \S \ref{sec:Y}.  Theorem \ref{Thm:main:intro} is restated and proved as Theorem \ref{Thm:main} below.

The formula in Theorem \ref{Thm:main:intro} will  be useful in applications of the Poisson summation formula on $Y.$  Moreover, it provides a precious example of a Fourier transform for a spherical variety that is not the affine closure of a Braverman--Kazhdan space.  Though intricate, we observe that the formula has a pleasing form.  Na\"ively, one might expect the Fourier transform to take the form
$$
\xi \longmapsto \int_{Y(F)}\psi\left(\langle \xi,y \rangle \right)f(y)d\mu(y),
$$
just like the traditional Fourier transform on a vector space. 
From a less na\"{i}ve perspective, since the Fourier transform is invariant under the product of the orthogonal groups attached to the $V_i$ \cite[Corollary 12.2]{Getz:Hsu}, one might expect an expression in terms of the invariant pairings $\langle \,,\, \rangle_i$ and the invariant polynomial $Q$.  This is indeed the shape of the formula.  It is instructive to compare this with the Fourier transform on the zero locus of a quadratic form given Corollary \ref{Cor:quad} below, generalizing earlier work in \cite{GurK,Kobayashi:Mano}.  

As a first application of Theorem \ref{Thm:main:intro}, we prove the following.
\begin{Thm} \label{thm:intro:FY:unit}
Let $F$ be a nonarchimedean local field of characteristic $0$.  Suppose $\dim V_i>2$ for all $1\le i\le 3$ and $Y^{\mathrm{sm}}(F) \neq \emptyset$. The operator $\mathcal{F}_Y$ extends to an isometry 
$$
\mathcal{F}_Y:L^2(Y(F)) \lto L^2(Y(F)).
$$
For $f_1,f_2 \in L^2(Y(F)),$ we have the Plancherel formula
$$
\int_{Y(F)}\mathcal{F}_Y(f_1)(y)f_2(y)d\mu(y)=\int_{Y(F)}f_1(y)\mathcal{F}_Y(f_2)(y) d\mu(y).
$$
\end{Thm}
\noindent This theorem puts Fourier analysis on $L^2(Y(F))$ on sound footing.  Theorem \ref{thm:intro:FY:unit} is proven as Theorem \ref{thm:FYuni} below. 

\begin{Rem}
The assumption $F$ has characteristic zero is only used to prove the geometric integrality statement in Lemma \ref{lem:finish}. 
Otherwise, the proofs of these two theorems work more generally provided the characteristic is large enough to apply an analogue of Proposition \ref{prop:vandercorputnonarch}.
\end{Rem}

\subsection{The Fourier transform on the affine closures of Braverman--Kazhdan spaces} 
Suppose now $F$ is any local field. For a reductive group $G$ with parabolic subgroup $P,$ let $X_P^{\circ}:=P^{\mathrm{der}} \backslash G$ and let $X_P$ be its affine closure. The space $X_P^\circ$ is known as a Braverman--Kazhdan space.  We prove Theorem \ref{Thm:main:intro} using an explicit formula for the Fourier transform on a certain Braverman--Kazhdan space.  Since the method is general and of independent interest, we prove the formula for split, simple, simply connected groups $G$ and maximal parabolic subgroups $P.$

 Let $M$ be a Levi subgroup of $P$ and let $P^{\mathrm{op}}$ be the unique parabolic subgroup of $G$ with $P \cap P^{\mathrm{op}}=M$.  We define $\mathcal{S}(X_P(F))$ and $\mathcal{S}(X_{P^{\mathrm{op}}}(F))$ in 
\S \ref{ssec:the:Schwartz:space} following previous work in \cite{Getz:Liu:BK,Getz:Hsu} which in turn refines the definition in \cite{BKnormalized}.  We then prove the existence of a Fourier transform
\begin{align} \label{FT0}
    \mathcal{F}_{P|P^{\mathrm{op}}}:\mathcal{S}(X_P(F)) \lto \mathcal{S}(X_{P^{\mathrm{op}}}(F))
\end{align}
(see Theorem \ref{thm:FT}).
This transform is unitary, and induces the same transform that Braverman and Kazhdan defined at the level of $L^2$-functions in \cite{BKnormalized} (see \S \ref{ssec:FT}).  We point out that the construction of the refined Schwartz space $\mathcal{S}(X_P(F))$ and the proof that it is preserved by the Fourier transform is not contained in \cite{BKnormalized}. We explain the relationship between Braverman and Kazhdan's definition of the Schwartz space and ours in \S \ref{ssec:def:rel}.

We observe that $X_P^{\circ}$ and $X_{P^{\mathrm{op}}}^{\circ}$ admit a natural action of $M^{\mathrm{ab}}$ by left multiplication.  Thus, at least formally, it makes sense to integrate functions in $\mathcal{S}(X_{P^{\mathrm{op}}}(F))$ against functions in $M^{\mathrm{ab}}(F)$.  Following Braverman and Kazhdan, we use this to define an operator $\mu_{P}^{\mathrm{aug}}$ on a certain subspace of $C^\infty(X_{P^{\mathrm{op}}}(F))$ (see \eqref{eqn: eta aug}).  It is essentially a sequence of weighted Fourier transforms along the $M^{\mathrm{ab}}(F)$-action.  

\begin{Thm}\label{thm:intro:BK}
We have
$
\calf_{P|P^{\mathrm{op}}}=\mu_P^{\mathrm{aug}} \circ \mathcal{F}_{P|P^{\mathrm{op}}}^{\mathrm{geo}},
$
where
$$
\calf_{P|P^{\mathrm{op}}}^{\mathrm{geo}}(f)(x^\ast)=  \int_{X_P^\circ(F)} f(x)\psi\left(\la x,  x^\ast\ra_{P|P^\mathrm{op}}\right) dx
$$
for $f\in  \mathcal{S}(X_P(F))$ and $x^\ast\in X^\circ_{P^{\mathrm{op}}}(F)$.
Here $\la \cdot,\cdot\ra_{P|P^{\mathrm{op}}}$ is the canonical pairing between $X_P^\circ(F)$ and $X_{P^{\mathrm{op}}}^\circ(F)$ of \eqref{PPop:pairing}, and $dx$ is an appropriately normalized right $G(F)$-invariant Radon measure. 
\end{Thm}
\noindent We use the superscript `$\mathrm{geo}$' to indicate that the geometric part of the Fourier transform is what one might expect of a Fourier transform from na\"ive geometric considerations, and the superscript `$\mathrm{aug}$' to denote the ``augmentation'' of the  normalization that is necessary to obtain the true Fourier transform (for example, to ensure the resulting operator is unitary).
\begin{Rem} Apart from trivial cases where $X_P$ is a vector space, our formula was only known when $G$ is a special orthogonal group on an even dimensional quadratic space and $P$ is the stabilizer of an isotropic line  \cite{Kobayashi:Mano,GurK}.
 In this case, $X_P$ is the zero locus of the quadratic form. The proofs in these two references rely on the interpretation of $L^2(X_P(F))$ as the minimal representation of a larger orthogonal group.  This additional structure on $L^2(X_P(F))$ does not exist in general, so our proof of Theorem \ref{thm:intro:BK} is not a generalization of these proofs.
 \end{Rem}

As mentioned above, to prove Theorem \ref{thm:intro:BK} we extend the refined definition of the Schwartz space given in \cite{Getz:Liu:BK} in the special case where $G=\mathrm{Sp}_{2n}$ and $P$ is the Siegel parabolic to the general case of maximal parabolic subgroups in split, simple, simply connected reductive groups. In the nonarchimedean setting, when $P$ is the Siegel parabolic subgroup of $G=\Sp_{4n},$ the Schwartz space is also investigated in work of Jiang, Luo, and Zhang \cite{Jiang:Luo:Zhang}, although their approach to Schwartz spaces is closer to \cite{BKnormalized} and they do not obtain Theorem \ref{thm:intro:BK} in their setting.  In loc.~cit.~the authors emphasize \cite[Theorem 5.5]{Jiang:Luo:Zhang} as a key technical result.  We obtain the analogous result in general (i.e.~for all maximal parabolic subgroups of simple simply connected groups) in Theorem \ref{thm:FT}.  Our proof is not a generalization of the proof of \cite[Theorem 5.5]{Jiang:Luo:Zhang}.  See  Remark \ref{JLZ:rem} below.
The references \cite{Shahidi:Generalized} and \cite{Li:Schwartz} contain useful information on the Braverman--Kazhdan program, though neither address the analytic issues that must be overcome to prove Theorem \ref{thm:intro:BK}. 

\subsection{Some reductive monoids}
Due to its connection with Langlands functoriality as outlined in \cite{BK-lifting,NgoSums,Ngo:Hankel}, finding explicit formulae for Fourier transforms on reductive monoids has become a focus of research.  It is well-known that Braverman--Kazhdan spaces built using the doubling method construction give rise to reductive monoids \cite{BKnormalized,Li:Schwartz,Shahidi:Generalized,Jiang:Luo:Zhang}.  Theorem \ref{thm:intro:BK} gives an explicit formulae for the 
Fourier transform in these cases.  We  point out three additional reductive monoids to which the results in this paper apply.  Let $\mathfrak{gl}_2$ be the scheme of $2 \times 2$ matrices. For $F$-algebras $R$ set
\begin{align*}
&X_1(R):=\{(A,a) \in \mathfrak{gl}_2(R) \times R: \det A=a^2\},\\
&X_2(R):=\{(A,B) \in \mathfrak{gl}_2^2(R): \det A=\det B\},\\
&X_3(R):=\{(A,B,C) \in \mathfrak{gl}_2^3(R): \det A=\det B=\det C\}.
\end{align*}
Then $X_1$ and $X_2$ are the affine closures of Braverman--Kazhdan spaces for appropriate special orthogonal groups, hence also the affine closures of Braverman--Kazhdan spaces associated to spin groups, and $X_3$ is a special case of the scheme $Y$ above.  We point out that $X_2$ was treated using the circle method in \cite{Getz:RS}, and the Fourier transform in this case is a special case of the Fourier transform computed in \cite{Kobayashi:Mano,GurK}.  The functional equations of the  Langlands $L$-functions giving rise to these reductive monoids are already known. However, even in the relatively simple case of  $\mathcal{S}(X_1(F))$ and $\mathcal{S}(X_3(F)),$ the formula for the Fourier transform given by Theorem \ref{thm:intro:BK} and Theorem \ref{Thm:main:intro} (respectively) is new.

\subsection{Outline of the paper}  In \S \ref{sec:prelim} we state conventions regarding Schwartz spaces, quasi-characters, measures and estimates.    We recall some basic facts   on Braverman--Kazhdan spaces in \S \ref{sec:BK}.  
The definition of the Fourier transform on the Schwartz space of the affine closure of a Braverman--Kazhdan space relies on operators that correspond, under the Mellin transform, to multiplication by $\gamma$-factors.  Only the nonarchimedean case appears in the literature.  Even in this case the domain and range of these operators is never elucidated. This makes it problematic to define the composition of the operator on an explicit space of functions.  We develop a new approach to these operators that works uniformly in the archimedean and nonarchimedean cases in  \S \ref{Section: twisting}. The new approach allows us to explicitly control the domain and range of the operators and to compose them. We expect these ideas will have applications to Fourier transforms beyond those constructed by Braverman and Kazhdan.

In \S \ref{sec:Schwartz:BK} we give a refined definition of the Schwartz space of the affine closure of a Braverman--Kazhdan space  whenever $P$ is a maximal parabolic subgroup of a split, simple, and simply connected $G,$ and prove that the Fourier transform preserves this space.  In the special case where $P$ is the Siegel parabolic of $G=\mathrm{Sp}_{2n},$ this definition is contained in  \cite{Getz:Liu:BK}.
This refinement goes beyond the work in \cite{BKnormalized}, in which 
the Fourier transform is only 
defined via a transform defined on an inexplicit dense subspace of a Hilbert space and then extended by continuity.

In \S \ref{sec:FT:XP} we prove Theorem \ref{thm:intro:BK}, restated as Theorem \ref{Thm: Fourier formula} below.  The proof of Theorem \ref{thm:intro:BK} requires computations of various normalizing factors which are given in \S \ref{App:normalizing}.  These computations also allows us to give an explicit description of $\mu_{P}^{\mathrm{aug}}$.  This is particularly important for readers without extensive background in representation theory who may want to apply our formula.

At this point we begin to shift our attention to the space $Y.$  In the preparatory \S \ref{sec:reg}, we introduce various regularized integrals we will require.  We recall the definition of the Schwartz space of $Y$ in \S \ref{sec:Y} and the indirect characterization of the Fourier transform $\mathcal{F}_Y$ proved in \cite{Getz:Hsu}.  We then prove Theorem \ref{Thm:main:intro}, restated as Theorem \ref{Thm:main}, in \S \ref{sec:FY}.  The proof is satisfying in that we make crucial use of standard tools of Fourier analysis  including the Plancherel formula.  We point out however, that in most cases adapting these tools to our setting is nontrivial.  
The unitarity of $\mathcal{F}_Y$ is proven in \S \ref{sec:unit} (see Theorem \ref{thm:FYuni}).  The heuristic arguments for Theorem \ref{Thm:main} and Theorem \ref{thm:FYuni} are fairly short, but making them rigorous requires careful analysis.  Certain technical estimates are relegated to \S \ref{sec:absolutebd}.    To aid the reader we have appended an index of notation.

\subsection{Acknowledgments} The authors thank the anonymous referee for a remarkably thorough review of the paper, corrections, and for comments that improved the exposition.  They also thank Ary Shaviv for answering questions on Schwartz spaces.  The first author thanks Y.~Choie, D.~Kazhdan and F.~Shahidi for many useful conversations.  Part of this paper was written while the first author was on sabbatical at the Postech Mathematics Institute. He thanks the center, the Postech Math Department and Y.~Choie for their hospitality and excellent working conditions.
He finally thanks H.~Hahn for her constant encouragement and help with editing and the structure of the paper. The second author would like to thank MSRI, P.~Habegger and H.~Pasten for organizing the summer school on Sparsity of Algebraic Points in 2021, during which he learned $o$-minimal geometry.

\section{Preliminaries}
\label{sec:prelim}
\subsection{Schwartz spaces}
\label{ssec:gen:Sch}
In this paper, we will work with various types of Schwartz spaces for quasi-affine schemes over a local field $F.$  If $X$ is a smooth quasi-affine scheme over $F,$ we let
$$
\mathcal{S}(X(F)):=C_c^\infty(X(F))
$$
when $F$ is nonarchimedean.  When $F$ is archimedean, we define $\mathcal{S}(X(F))=\mathcal{S}(\mathrm{Res}_{F/\rr}X(\rr))$ as in \cite[Remark 3.2]{Elazar:Shaviv} (this is based on previous work in \cite{AG:Nash}).  Briefly, one chooses an embedding $\mathrm{Res}_{F/\rr}X(\rr) \to \rr^n$ in the category of real algebraic varieties with closed image  and then defines $\mathcal{S}(X(F))=\mathcal{S}(\rr^n)/I,$ where $I \leq \mathcal{S}(\rr^n)$ is the (closed) ideal of functions that vanish identically on $X(F).$ The embedding $\mathrm{Res}_{F/\rr}X(\rr) \to \rr^n$ always exists in the real algebraic category, even if $X$ is merely quasi-affine (see \cite[\S 2.1]{Elazar:Shaviv} for references).
This recovers the usual definition when $X(F) \cong F^d$ for some $d$.  One endows  $\mathcal{S}(X(F))$ with the quotient topology, which is  Fr\'echet and nuclear. The space $\mathcal{S}(X(F))$ and its topology are independent of the choice of embedding \cite[Lemma 3.6(i)]{Elazar:Shaviv}.
  It is known  \cite[Theorem 3.9]{Elazar:Shaviv} that if $X_2$ is a smooth quasi-affine scheme and $X_1 \subset X_2$ is a closed subscheme, then restriction of functions induces a surjection
$$
\mathcal{S}(X_2(F))  \relbar\joinrel\twoheadrightarrow  \mathcal{S}(X_1(F)).
$$

We will define Schwartz spaces $\mathcal{S}(X(F))$ for several singular affine schemes $X.$  They will always be spaces of functions on $X^{\mathrm{sm}}(F),$ where $X^{\mathrm{sm}} \subset X$ is the smooth locus.  If $F$ is archimedean, the space $\mathcal{S}(X(F))$ will always be a Fr\'echet space.  Unfortunately, we will not always know whether it is nuclear.  

Let $X$ and $Y$ be quasi-affine schemes.  Assuming that Schwartz spaces $\mathcal{S}(X(F))$ and $\mathcal{S}(Y(F))$ have been defined, we define
$$
\mathcal{S}(X(F) \times Y(F)):=\mathcal{S}(X(F)) \otimes \mathcal{S}(Y(F))
$$
in the nonarchimedean case (algebraic tensor product).  In the archimedean case, we define $\mathcal{S}(X(F) \times Y(F))$ to be the completed projective tensor product of $\mathcal{S}(X(F))$ and $\mathcal{S}(Y(F))$.  Unfortunately, we do not know if this product is independent of the choice of realization of $X \times Y$ as a product in general.  Therefore this realization will be part of the data.  We observe that if $X$ and $Y$ are smooth then this definition agrees with the previous definition.  This follows from \cite[Corollary 2.6.3]{AG:deR} and the fact that the Schwartz space of a real algebraic variety and the Schwartz space of its underlying Nash manifold are naturally isomorphic \cite[\S 2.2]{Elazar:Shaviv}.

\subsection{Quasi-characters and the norm}  Let $F$ be a local field.  We 
denote by $|\cdot|$ the number theorist's norm on $F.$  Thus  $|\cdot|$ is the usual Euclidean norm if $F = \rr$, $|z|=z\bar{z}$ if $F=\cc,$ and if $F$ is nonarchimedean with ring of integers $\mathcal{O}_F$ and uniformizer $\varpi$ then $|\varpi^{-1}|$ is the cardinality $q$ of the residue field $\mathcal{O}_F/\varpi.$
For local fields $F$ and quasi-characters $\chi:F^\times \to \cc^\times,$ we let $\mathrm{Re}(\chi) \in \rr$ be the unique real number such that $\chi|\cdot|_F^{-\mathrm{Re}(\chi)}$ is a character (i.e.~is unitary). Consider a function $f$ of quasi-characters.  We say that it is holomorphic (resp.~meromorphic) if for all characters $\chi,$ the function $f(\chi|\cdot|^s)$ is holomorphic (resp.~meromorphic) in $s\in \cc$. 

We also denote the usual norm on $\cc$ by $|\cdot|.$  This creates the possibility of confusion when we have chosen an identification $F=\cc.$  When $F$ is denoted by $\cc,$ we use the standard norm, and when $F$ is denoted simply $F$, we use the number-theorist's norm. Thus, for example, if $X$ is a set and $f:X \to \cc$ is a function,  then $|f(x)|=(f(x)\overline{f(x)})^{1/2}$ for $x \in X.$   This is a standard convention adopted to lighten notation.

\subsection{Measures} \label{ssec:measures}
For local fields $F,$ if $dx$ denotes a Haar measure on $F,$ then $d^\times x:=\frac{\zeta(1)dx}{|x|}$ where $\zeta$ is the usual local zeta function.  We often regard $dx$ as a measure on the open dense subset  $F^\times \subset F$.
We fix once and for all a nontrivial additive character $\psi:F \to \cc^\times$.  The measure $dx$ will always  be normalized so that it is self-dual with respect to the Fourier transform on $\mathcal{S}(F)$ defined by $\psi$.

\subsection{Asymptotic notation} \label{ssec:asymp}
Let $g_1:X \to \rr_{\ge 0}$ and $g_2:X \to \rr_{\ge 0}$ be functions defined on a set $X.$  We write 
\begin{align} \label{impl:constant}
g_1(x) \ll_{?} g_2(x), \quad g_1(x)=O_{?}(g_2(x))
\end{align}
if there is a constant $C_{?}>0$ depending on the set $?$ such that $g_1(x)<C_{?}g_2(x)$ for all $x \in X.$  We drop set symbols when denoting the set, e.g.~we write $C_{a,b}$ instead of $C_{\{a,b\}}.$  
We will also say $g_2$ \textbf{dominates} $g_1$ in order to avoid repeating the phrase ``is bounded by a constant times.'' If $F$ is archimedean and $?$ contains an element of $\mathcal{S}(V(F))$ (or another topological vector space of functions) \eqref{impl:constant} will in addition mean that the implied constant can be chosen continuously as a function of $f$ when the other elements in $?$ are fixed.

\section{Braverman--Kazhdan spaces}
\label{sec:BK}

\subsection{Braverman--Kazhdan spaces} \label{ssec:BKspace}  
Let $G$ be a split connected simple reductive group over a field $F$ and let $P$ be a maximal parabolic subgroup with Levi decomposition $P=MN_P$.  By simple, we mean that $\fg:=\mathrm{Lie} (G)$ is simple.  
Set 
\[
X_P^{\circ}:=P^{\mathrm{der}}\backslash G. 
\]
We refer to $X_P^\circ$ as a \textbf{Braverman--Kazhdan space}; it is also known as a \textbf{pre-flag variety} since it is a $\mathbb{G}_m$-torsor over the generalized flag variety $P \backslash G$.  
This is a right $M^{\mathrm{ab}}\times G$-space, where the action is given on points in an $F$-algebra $R$ by 
\begin{align} \label{geo:act} \begin{split}
X_P^\circ(R) \times M^{\mathrm{ab}}(R) \times G(R) &\lto X_P^\circ(R)\\
    (x,m,g) &\longmapsto m^{-1}xg. \end{split}
\end{align}
We point out that Braverman and Kazhdan work with $G/P^{\mathrm{der}}$ instead.

\subsection{Pl\"{u}cker embeddings } \label{ssec:Pl}

Fix a maximal split torus $T\leq M,$ a Borel subgroup $T\le B\le P,$ and let $\Delta=\Delta_G$ be the corresponding set of simple roots.  Then $B \cap M$ is a Borel subgroup of $M.$
Suppose that $\be\in \De$ is the simple root of $(G,T)$ associated to $P$; that is, we have that $\De_M= \De-\{\be\}$ is the set of simple roots for the based root system of $(M,M\cap B)$. 
Let $\omega_\beta \in X^*(T)_{\qq}:=X^*(T) \otimes_{\mathbb{Z}}\qq$ be the fundamental weight of $T$ determined by the relation
\begin{align*}
\la\omega_\beta,\al^\vee\ra=\de_{\al,\be} \textrm{ for all } \alpha\in \Delta,
\end{align*}
where $\delta_{\alpha,\beta}$ is the Kronecker $\delta$.  It is not necessarily true that $\omega_{\beta} \in X^*(T)$.  We let $m_\beta$ be the least positive rational number such that $m_\beta \omega_\beta \in X^*(T)$ and define
\begin{align} \label{omegaP}
\omega_P:=m_\beta \omega_{\beta}.
\end{align}
 In particular $\omega_P$ is denoted $\omega$ in \cite{Getz:Liu:BK}. We claim that $m_\beta\in \zz$. To see this, note that if  $\Lambda$ is the lattice in $X^*(T)_{\qq}$ spanned by the fundamental weights, one has
 \[
 \lam\in \Lambda \iff \la \lam,\al^\vee\ra\in \zz
 \]
 for all simple roots $\al\in \De.$ Since $X^*(T) \leq \Lambda$
the claim now follows by pairing $\omega_P$ with $\beta^\vee.$

We leave the proof of the following lemma to the reader:
 
\begin{Lem} \label{lem:tori}
If $T$ is a maximal torus of a (connected) reductive  $F$-group $H,$ then $T \cap H^{\mathrm{der}}$ is a maximal torus of $H^{\mathrm{der}}$.  If $T$ is split, then so are $T \cap H^{\mathrm{der}}$ and $T/T \cap H^{\mathrm{der}}.$ \qed
\end{Lem}

\begin{Lem} \label{lem:mab}
The torus $M^{\mathrm{ab}}$ is split and isomorphic to $\GG_m$.
The map $M(F) \to M^{\mathrm{ab}}(F)$ is surjective.
\end{Lem}

\begin{proof}
The first assertion follows from our assumption that  $G$ is a split, simple reductive group and $P$ is maximal and Lemma \ref{lem:tori}.

For the second assertion, consider the maximal split torus $T \leq M$.  The intersection $T \cap M^{\mathrm{der}}$ is a split torus by Lemma \ref{lem:tori}, and the restriction of the map $M \to M^{\mathrm{ab}}$ to $T$ is the quotient map
$$
T \lto T/(T \cap M^{\mathrm{der}}) \tilde{\lto} M^{\mathrm{ab}}.
$$
Since $T \cap M^{\mathrm{der}}$ is a split torus, this map is surjective on $F$-points by Hilbert's theorem 90 and we deduce the lemma.
\end{proof}

\begin{Cor} \label{cor:surj}
The map $G(F)\to X_P^\circ(F)$ is surjective.
\end{Cor}
\begin{proof} Consider the commutative diagram
\[
\begin{tikzcd}
G(F)\ar[r,"q_1"]\ar[rd,"q_3"]&X_P^{\circ}(F)\ar[d,"q_2"]\\
&(P \backslash G)(F),
\end{tikzcd}
\]
where the $q_i$ are the canonical quotient maps. 
The map $q_3$ is surjective \cite[Th\'eor\`eme 4.13]{Borel:Tits}.
For $y \in (P \backslash G)(F),$ choose $g \in G(F)$ such that $q_3(g)=y$.  Set $x=q_1(g)$. Since $M^{\mathrm{ab}}$ is a split torus by Lemma \ref{lem:mab}, 
 $q_2^{-1}(y)$ is a $M^{\mathrm{ab}}(F)$-torsor.  
 In other words,
 $$
 q_2^{-1}(y)=\{tx:t \in M^{\mathrm{ab}}(F)\}=\{q_1(mg):m \in M(F)\}
 $$
 since $M(F) \to M^{\mathrm{ab}}(F)$ is surjective by Lemma \ref{lem:mab}.  Thus $q_2^{-1}(y)$ is in the image of $q_1$ for all $y \in (P \backslash G)(F).$
\end{proof}

 Let $V_P$ be the right representation of $G$ of highest weight $-\omega_P$. We remind the reader that for a right representation, the character of a highest weight vector is anti-dominant, explaining why the highest weight is $-\omega_P$.
Fix a highest weight vector $v_P \in V_P(F)$. 

\begin{Lem}\label{Lem: general Plucker}
The derived subgroup $P^{\mathrm{der}}$ is the stabilizer of $v_P,$ so that the map $ \mathrm{Pl}:=\mathrm{Pl}_{v_P}: X_P^{\circ} \to V_P$ induced by
\begin{align*}
  G(R)&\lra V_P(R)\\
    g&\longmapsto v_P g,
\end{align*}
maps $X_P^{\circ}$ isomorphically onto to the orbit of $v_P$ under $G$.  The map $\omega_P,$ originally a character of $T,$ extends to a character of $M,$ and the induced map
$$
\omega_P:M^{\mathrm{ab}} \lto \GG_m
$$
is an isomorphism.
For $m \in M^{\mathrm{ab}}(R),$ one has 
\begin{align} \label{intertwine}
\mathrm{Pl}(m^{-1}g)=\omega_P(m)\mathrm{Pl}(g).
\end{align}
\end{Lem}

\begin{proof}
It is well-known that $P$ is the stabilizer of the line spanned by $v_{P}$ (this follows from the discussion in \cite[\S 24.4]{Borel}), and thus this line is a one-dimensional representation of $P.$  
We deduce that $-\omega_P$ extends from $T$ to a character of $P,$ and 
$P$ acts via the character $-\omega_P$  on the line and hence the stabilizer of $v_{P}$ contains $P^{\mathrm{der}}$.  

Since $P^{\mathrm{der}}=M^{\mathrm{der}}N_P,$ to prove that $P^{\mathrm{der}}$ is the full stabilizer, it suffices to check that $\omega_P:M^{\mathrm{ab}} \to \GG_m$ is an isomorphism.  Upon choosing an isomorphism $M^{\mathrm{ab}} \cong \GG_m,$ we have that $\omega_P$ is given on points by $x \mapsto x^n$ for some non-zero $n \in \zz$.  Then $\omega_P/n \in X^*(T)$.  By our choice of $\omega_P,$ we deduce $n=\pm 1$ and $\omega_P$ is an isomorphism.  
The equivariance property \eqref{intertwine} of $\mathrm{Pl}$ is now clear. 
\end{proof}

Consider the affine closure
\begin{align*}
    X_P:=\mathrm{Spec}(F[X_P^\circ]) = \overline{X^{\circ}_P}^{\mathrm{aff}}.
\end{align*}  
The scheme $X_P$ is of finite type over $F$ and the natural map $X_P^\circ \to X_P$ is an open immersion \cite[Theorem 1.1.2]{Braverman:Gaitsgory}.
We will actually not require any properties of $X_P$ in this paper, but the fact that it has simple singularities provides good intuition for the Schwartz spaces we define later.  Therefore we recall the following theorem (see \cite[Theorem 1 and 2]{Popov:Vinberg}):
\begin{Thm} \label{Thm:Pl}
The embedding 
$\mathrm{Pl}:X_P^{\circ} \to V_P$ extends to a closed immersion $\mathrm{Pl}:X_P \to V_P$.   The closed subscheme $X_P-X_P^{\circ}$ is a point and it is mapped under $\mathrm{Pl}$ to $0$. \qed
\end{Thm}

Let $P^{\mathrm{op}}$ be the parabolic subgroup opposite to $P$ so that $P\cap P^{\mathrm{op}}=M$.  Let $V_P^\vee$ be the representation of $G$ dual to $V_P$ and let $v^*_{P^{\mathrm{op}}} \in V_P^{\vee}(F)$ be the lowest weight vector of $V_P^{\vee}(F)$ dual to $v_P$.  We then have an embedding $\mathrm{Pl}_{v^*_{P^{\mathrm{op}}}}:X_{P^{\mathrm{op}}}^{\circ} \to V_P^{\vee}$ induced by 
\begin{align*}
G(R) &\lto V_{P}^\vee (R)\\
 g&\longmapsto v^*_{P^{\mathrm{op}}} g.
\end{align*}

Let $\langle \cdot, \cdot \rangle$ be the canonical pairing of $V_P$ and $V_P^{\vee}$. Consider the $G$-equivariant pairing given on $F$-algebras $R$ by  
\begin{align} \label{PPop:pairing}
\la\cdot,\cdot\ra_{P|P^{\mathrm{op}}}:X_{P}^\circ (R)\times X_{P^{\mathrm{op}}}^\circ (R)&\lra R\\
    (x,x^*)&\longmapsto \la \mathrm{Pl}_{v_P}(x),\mathrm{Pl}_{v^*_{P^{\mathrm{op}}}}(x^*)\ra.\nonumber
\end{align}
If we replace $v_P$ by any other highest weight vector $v_P',$ then $v_P'=tv_P$ for some $t\in F^\times$. Thus the dual vector of $v_P'$ is $t^{-1}v^*_{P^{\mathrm{op}}}$.  It follows that $\langle\cdot,\cdot\rangle_{P|P^{\mathrm{op}}}$ is independent of the choice of $v_P$.

\subsection{Relation to induced representations} \label{sec:rel:induced}  We now assume $F$ is a local field.
The space $\cals(X_P^\circ(F)),$ equipped with the $M^{\mathrm{ab}}(F)$-action induced by \eqref{geo:act},  can be thought of as a universal (degenerate) principal series representation.

For a quasi-character $\chi:F^\times \to \cc^\times,$ let
\begin{align}\label{normalizedinduct}
I(\chi):=I_P(\chi):=\Ind_{P(F)}^{G(F)}(\chi \circ \omega_P), \quad \overline{I}(\chi):=\overline{I}_{P^{\mathrm{op}}}(\chi):=\Ind_{P^{\mathrm{op}}(F)}^{G(F)}(\chi \circ \omega_P)
\end{align}
be the normalized inductions in the category of smooth representations.  Let $\delta_P$ be the modular character of $P.$
We define Mellin transforms
\begin{align} \label{Mellin} \begin{split}
\mathcal{S}(X_P^{\circ}(F)) &\lra I(\chi) \\
f &\longmapsto f_{\chi}(\cdot):=f_{\chi,P}(\cdot):=\int_{M^{\mathrm{ab}}(F)}\delta_P^{1/2}(m)\chi(\omega_P(m))f(m^{-1}\cdot)dm, \\
\mathcal{S}(X_{P^{\mathrm{op}}}^\circ (F)) &\lra \overline{I}(\chi) \\
f &\longmapsto f_{\chi}^{\mathrm{op}}(\cdot):=f_{\chi,P^{\mathrm{op}}}^{\mathrm{op}}(\cdot):=\int_{M^{\mathrm{ab}}(F)}\delta_{P^{\mathrm{op}}}^{1/2}(m)\chi(\omega_P(m))f(m^{-1} \cdot)dm.
\end{split}
\end{align}
Here $dm$ is the Haar measure on $M^{\mathrm{ab}}(F)$ obtained from the isomorphism 
$
\omega_P:M^{\mathrm{ab}}(F) \to F^\times
$
and the Haar measure $d^\times x$ on $F^\times$ by our convention in \S \ref{ssec:measures}.
In the notation $\overline{I}_{P^{\mathrm{op}}}(\chi)$ and $f_{\chi,P^{\mathrm{op}}}^{\mathrm{op}},$ the bar and the superscript $\mathrm{op}$ indicate that we are inducing from $\chi\circ\omega_P$ instead of $\chi\circ \omega_{P^{\mathrm{op}}}$. 

We use the same notation for extensions of the Mellin transform to larger subsets of $C^\infty(X_P^{\circ}(F))$ and $C^\infty(X_{P^{\mathrm{op}}}^\circ(F)),$ when in general the integrals defining $f_{\chi},f_{\chi}^{\mathrm{op}}$ only exist for $\mathrm{Re}(\chi)$ in a proper subset of $\rr,$ and in some cases will be extended to larger complex domains by analytic continuation.  

Let $A<B$ be extended real numbers (we allow $A=-\infty$ and $B=\infty$) and let
\begin{align} \label{VAB}
V_{A,B}:=\{s \in \cc:A<\mathrm{Re}(s)<B\}.
\end{align}
For quasi-characters $\chi$ of $F^\times$ and $s \in \cc$,  let $\chi_s := \chi|\cdot|^s.$
Assume $F$ is archimedean.  
We say that a section $f(\chi)^{(s)} \in I(\chi_s)$ is holomorphic (resp.~meromorphic) in $V_{A,B}$
if for all $g \in G(F)$ and (unitary) characters $\chi$ of $F^\times$, the function
\begin{align} \label{holo} \begin{split}
    V_{A,B} &\lto \cc\\
    s &\longmapsto f(\chi)^{(s)}(g) \end{split}
\end{align}
is holomorphic (resp.~meromorphic).  In the nonarchimedean case, we say that a section $f(\chi)^{(s)} \in I(\chi_s)$ is holomorphic (resp.~meromorphic) if for all $g \in G(F)$ and characters $\chi$ of $F^\times$ \eqref{holo}
is in $\cc[q^{-s},q^s]$ (resp. $\cc(q^{-s},q^s)$).

\section{Twisting by abelian $\ga$-factors}\label{Section: twisting}
For the remainder of the paper, $F$ denotes a local field.  As discussed in \S \ref{ssec:FT} below, the definition of the Fourier transform $\calf_{P|P^{\mathrm{op}}}$ involves normalization operators $\lambda_!(\mu_s)$ which correspond, under the Mellin transform, to multiplication by $\gamma(-s,\chi^{\lambda},\psi)$ (see Lemma \ref{lem:gamma}).  Here and below $\gamma(s,\chi,\psi)$ denotes the usual Tate $\gamma$-factor attached to a complex number $s$, a quasi-character $\chi:F^\times \to \mathbb{C}^\times$, and the additive character $\psi.$  The operators $\lambda_!(\mu_s)$ were previously defined in \cite{BKnormalized} and an exposition is given in \cite{Shahidi:Generalized}.
The approach of \cite{BKnormalized} is inconvenient in the sense that each operator is only defined on an inexplicit subspace of $\mathcal{S}(X^{\circ}_P(F))$ that is dense in $L^2(X_P(F))$.  Thus as one composes operators, one loses control of their domain and range.  Moreover, the operators are only defined in the nonarchimedean case in \cite{BKnormalized}.

In this section we set up a general theory of the operators $\lambda_!(\mu_s)$ that is applicable uniformly in the archimedean and nonarchimedean settings.  We also explain how to control their domain and range.  This is quite delicate.  In particular, to construct the   Fourier transform, the normalizing operators $\lambda_{!}(\mu_s)$ have to be composed in a particular order.  This motivates the definition of a \textbf{good ordering} in Definition \ref{defn:good:order} below.  Essentially the situation is as follows: to compose the operators $\lambda_!(\mu_s),$ we require the domain of absolute convergence of certain Tate integrals to overlap.  This is only possible if we arrange the operators in a particular order. 

\begin{Rem} \label{JLZ:rem} This difficulty was also encountered in the nonarchimedean setting in a special case in \cite{Jiang:Luo:Zhang}. They overcame it  by packaging all the normalizing operators together and relating them to transforms coming from prehomogeneous vector spaces.  We do not know if their method can be used to obtain an explicit formula for the Fourier transform, or if it can be applied in the generality considered here.
\end{Rem}

For $\lambda \in \zz$ and $s \in \cc,$ we define a linear map 
\begin{align} \label{eta:op}
  \lambda_!(\mu_s):  \mathcal{S}(X_{P^{\mathrm{op}}}^\circ(F)) &\lto C^\infty(X_{P^{\mathrm{op}}}^\circ(F))
  \end{align}
  given by 
\begin{align}\label{eqn: eta def}
\lambda_!(\mu_s)(f)(x):=\int_{M^{\mathrm{ab}}(F)}\psi(\omega_P(m))|\omega_P(m)|^{s+1}\delta_{P^{\mathrm{op}}}^{\lambda /2}(m)f(m^{-\lambda}x)\frac{dm}{\zeta(1)}.
\end{align}
This was denoted $\lambda_!(\eta^s_{\psi})$ in  \cite{BKnormalized}.  In loc.~cit.~a measure is incorporated into the distribution; this is why our formula looks different. 

 We work with $X_{P^\mathrm{op}}^\circ$ here to be consistent with our notation later on, when these operators are applied after the operator 
 $\mathcal{R}_{P|P^{\mathrm{op}}}$ of \eqref{R}.  Of course in the formula for \eqref{eqn: eta def} we could write everything in terms of $P$ or $P^{\mathrm{op}}$ by taking appropriate inverses.  We have written it in the form above to remind the reader that $f$ is a function on $X_{P^{\mathrm{op}}}^\circ(F),$ but the normalizing factors $\lambda$ and $s$ we will use in our case of primary interest are defined in terms of $P$ (see \S \ref{ssec:BKL} below).  

To extend the domain of definition of $\lambda_!(\mu_s),$ choose $\Phi \in \mathcal{S}(F)$ such that $\Phi(0)=1$ and $\widehat{\Phi} \in C_c^\infty(F)$. Here $\widehat{\Phi}(x):=\int_{F}\Phi(y)\psi(xy)dx$ is the Fourier transform of $\Phi.$  For continuous functions $f:X_{P^{\mathrm{op}}}^\circ(F) \to \cc,$ we define the regularized integral 
\begin{align} \label{lambda:reg}
\lambda_!(\mu_s)^{\mathrm{reg}}(f)(x):=\lim_{|b| \to \infty}\int_{M^{\mathrm{ab}}(F)}\Phi\left(\frac{\omega_P(m)}{b}\right)\psi(\omega_P(m))|\omega_P(m)|^{s+1}\delta_{P^{\mathrm{op}}}^{\lambda /2}(m)f(m^{-\lambda}x)\frac{dm}{\zeta(1)}.
\end{align}
We say this integral is well-defined if
\begin{align} \label{each:b0}
\int_{M^{\mathrm{ab}}(F)}|\Phi|\left(\frac{\omega_P(m)}{b}\right)|\omega_P(m)|^{\mathrm{Re}(s)+1}\delta_{P^{\mathrm{op}}}^{\lambda /2}(m)|f|(m^{-\lambda}x)dm
\end{align}
 is finite for $|b|$ sufficiently large, and the limit  in the definition of $\lambda_!(\mu_s)^{\mathrm{reg}}(f)(x)$ exists  and is independent of $\Phi$.

\begin{Lem} \label{lem:coincidence}
If the integral defining $\lambda_{!}(\mu_s)(f)$ is absolutely convergent, then
$\lambda_!(\mu_s)^{\mathrm{reg}}(f)=\lambda_!(\mu_s)(f)$.  In particular, $\lambda_!(\mu_s)^{\mathrm{reg}}(f)=\lambda_!(\mu_s)(f)$ whenever $f \in \mathcal{S}(X_{P^{\mathrm{op}}}^\circ(F))$.
\qed
\end{Lem}
To avoid more proliferation of notation, we will drop the $\mathrm{reg}$ from notation.  Lemma \ref{lem:coincidence} shows this is harmless, as it implies that the two integrals yield the same result when both are well-defined.  

\begin{Lem} \label{lem:gamma}
Assume $f \in \overline{I}(\chi)$ and that $\mathrm{Re}(s)+1-\lambda\mathrm{Re}(\chi)>0$.  The function 
$$
\lambda_!(\mu_s)(f)(x)
$$ 
is well-defined and equal to $\gamma(-s,\chi^\lambda,\psi)f(x)$.
\end{Lem}

\begin{proof}
Since $\mathrm{Re}(s)+1-\lambda\mathrm{Re}(\chi)>0,$ \eqref{each:b0} is finite for all $b$. By the functional equation of Tate zeta functions, we have
\begin{align*}
&\int_{M^{\mathrm{ab}}(F)}\Phi\left(\frac{\omega_P(m)}{b} \right)\psi(\omega_P(m))|\omega_P(m)|^{s+1}\chi^{-\lambda}(\omega_P(m))f(x)\frac{dm}{\zeta(1) }\\&=\gamma(-s,\chi^{\lambda},\psi)f(x)\int_{F^\times}\left(\int_{F}\Phi\left(\frac{t}{b} \right)\psi(t)\overline{\psi}(yt) dt\right)
|y|^{-s}\chi^{\lambda}(y)\frac{d^\times y}{\zeta(1)}.
\end{align*}
Using our assumption that $\Phi(0)=1,$ we have
\begin{align}
\label{eq:becomeone}
\begin{split}
 &\lim_{|b| \to \infty}  |b| \int_{F^\times}\left(\int_{F}\Phi\left(t \right)\psi(bt(1-y)) dt\right)
|y|^{-s}\chi^{\lambda}(y)\frac{d^\times y}{\zeta(1)}\\
&=\lim_{|b| \to \infty}  |b|\int_{F^\times}
\widehat{\Phi}(b(1-y))
|y|^{-s}\chi^{\lambda}(y)\frac{d^\times y}{\zeta(1)}\\
&=\lim_{|b| \to \infty}  \int_{F}
\widehat{\Phi}(y)
|1-\tfrac{y}{b}|^{-s-1}\chi^{\lambda}(1-\tfrac{y}{b})dy\\
&=1.
\end{split}
\end{align}
Here for small $|b|$ the integral may diverge, but since $\widehat{\Phi} \in C_c^\infty(F)$, the integral converges for $|b|$ sufficiently large.  
\end{proof}

Now consider a graded $\mathbb{G}_m$-representation
$$
L=\bigoplus_{i\in I}L_i
$$
for some finite index set $I$.  We assume that each $L_i$ is $1$-dimensional and that 
 $\GG_m$ acts via a character $\lambda_i$ on $L_i$.     We identify $X^*(\mathbb{G}_{m})$ with $\zz$ by taking the identity character to $1,$ so we can speak of positive or negative  characters.  We assume that each character $\lam_i$ is non-zero and assign to each $L_i$ a real number $s_i$.

We then have linear maps
\begin{align} \label{etaLi}
\lambda_{i!}(\mu_{s_i}):\mathcal{S}(X_{P^{\mathrm{op}}}^\circ(F)) \lto C^\infty(X_{P^{\mathrm{op}}}^\circ(F))
\end{align}
for each $i\in I$.
Following \cite{BKnormalized}, we wish to compose these linear maps to give a single transform 
\[
\mu_L:\mathcal{S}(X_{P^{\mathrm{op}}}^\circ(F)) \lto C^\infty(X_{P^{\mathrm{op}}}^\circ(F))
\]
associated to the $\Gm$-module $L$ and the data $\{(s_i,\lam_i) \in \rr\times \zz :i\in I\}$. 

It is convenient (and perhaps necessary) to extend the work in \cite{BKnormalized} by elucidating the domain and range of these operators.  We proceed as in \cite{Getz:Liu:BK}, which borrows from \cite{Ikeda:poles:triple}.
Let
\begin{align}\label{al}
    a_L(\chi):=\prod_{i\in I}L(-s_i,\chi^{\lambda_i}).
\end{align}
We introduce extended real numbers $A(L),B(L)$ as follows:
\begin{align}\label{AB}
\begin{split}
A(L)&:=\begin{cases}\max\left\{ \frac{s_i}{\lambda_i}:i \in I, \lambda_i>0\right\}&\textrm{if } \lam_i>0\text{ for some $i$},\\
-\infty &\text{otherwise},\end{cases}\\
B(L)&:=\begin{cases}\min\left\{\frac{s_i}{\lambda_i}: i \in I, \lambda_i<0\right\}&\textrm{if } \lam_i<0\text{ for some $i$},\\
\infty &\text{otherwise}.\end{cases} 
\end{split}
\end{align}
  Assume that $A(L)<B(L)$.

\begin{Lem} \label{lem:no:poles}
The function
$a_L(\chi)$ has no poles for $A(L) < \mathrm{Re}(\chi) <B(L)$.  \qed
\end{Lem}
  
  We now define the space
\begin{align} \label{SL}
\mathcal{S}_{L}:=\mathcal{S}_L(X_{P^{\mathrm{op}}}(F)) <C^\infty(X^\circ_{P^{\mathrm{op}}}(F)).
\end{align} 
When $F$ is nonarchimedean, we define $\cals_L$ to be the space of smooth functions $f: X^\circ_{P^{\mathrm{op}}}(F)\to \cc$ that are finite under a maximal compact subgroup of $G(F)$ and satisfy the following additional condition: the integral defining $f_{\chi_s}^{\mathrm{op}}$ is absolutely convergent for  $A(L)<\mathrm{Re}(s)<B(L)$ and 
$$
\frac{f^{\mathrm{op}}_{\chi_s}(x)}{a_L(\chi_s)}
$$
lies in $\cc[q^{-s},q^s]$ for each $x \in X^\circ_{P^{\mathrm{op}}}(F)$ and all (unitary) characters $\chi:F^\times \to \cc^\times$.  

When $F$ is archimedean, we require a bit more notation. Let $\widehat{K}_{\GG_m}$ be the set of characters of the maximal compact subgroup $K_{\GG_m}$ of $F^\times$.  Thus, setting
$$
\mu(z):=\frac{z}{(z\overline{z})^{1/2}},
$$
where we use the positive square root, we have
\begin{align}\label{KGM}
\widehat{K}_{\GG_m}:=\left\{\mu^{\alpha}: \begin{matrix}\alpha \in \{0,1\} \textrm{ if }F \textrm{ is real}\\
\:\alpha \in \zz \textrm{ if }F \textrm{ is complex}\end{matrix}\right\}.
\end{align}
For real numbers $A<B,$ $p \in \cc[s],$ and meromorphic functions $\phi:\cc \to \cc,$ we let
\begin{align}\label{semi:normf}
|\phi|_{A,B,p}:&=\sup_{s \in V_{A,B}}|p(s)\phi(s)|
\end{align}
where $V_{A,B}$ is defined as in \eqref{VAB}.  Consider the Lie algebra
\begin{align}\label{Liemabg}
\mathfrak{m}^{\mathrm{ab}} \oplus \mathfrak{g}:=\mathrm{Lie}(M^{\mathrm{ab}}(F) \times G(F)).
\end{align}
It acts on $C^\infty(X_{P^{\mathrm{op}}}^{\circ}(F))$ via the differential of the action \eqref{geo:act} and hence we obtain an action of $U(\mathfrak{m}^{\mathrm{ab}} \oplus \mathfrak{g}),$ the universal enveloping algebra of  $(\mathfrak{m}^{\mathrm{ab}} \oplus \mathfrak{g})_{\cc}$ (here we view $\mathfrak{m}^{\mathrm{ab}} \oplus \mathfrak{g}$ as a real Lie algebra).

We let $\mathcal{S}_L$ be the space of smooth functions $f:X^\circ_{P^{\mathrm{op}}}(F) \to \cc$ such that for all $\eta \in \widehat{K}_{\GG_m}$ and all $D\in U(\mathfrak{m}^{\mathrm{ab}} \oplus \mathfrak{g}),$ the integral defining
\[
(D.f)_{\eta_s}^{\mathrm{op}}(x)
\]
converges absolutely for all $A(L)<\mathrm{Re}(s)<B(L),$ and admits a meromorphic continuation to the plane such that
\begin{enumerate}
    \item for all $A<B,$
    \item  all polynomials $p \in \cc[s]$ such that $p(s)a_{L}(\eta_s)$ has no poles in $V_{A,B}$ for all $\eta \in \widehat{K}_{\GG_m},$
    \item all compact subsets $\Omega \subset X_{P^{\mathrm{op}}}^\circ(F),$
    \item all $D \in U(\mathfrak{m}^{\mathrm{ab}} \oplus \mathfrak{g}),$
\end{enumerate}   one has that
\begin{align}
|f|_{A,B,p,\Omega,D}:=\sum_{\eta\in \widehat{K}_{\GG_m}} \sup_{x \in \Omega}|(D.f)_{\eta_s}^{\mathrm{op}}(x)|_{A,B,p}<\infty.
\end{align}
This collection of seminorms gives $\mathcal{S}_L$ the structure of a Fr\'echet space by the same argument as  in \cite[Lemma 3.2]{Getz:Hsu}.  

In all cases, this definition allows us to recover analytic properties of $f$ from its Mellin transforms via Mellin inversion. More specifically, let 
$
\widehat{K}_{\GG_m}
$
be a set of representatives for the characters of $F^\times$ modulo the equivalence relation
\[
\chi_1\sim\chi_2\text{  if and only if  } \chi_1=\chi_2|\cdot|^{it}
\]
for some $t \in \rr$.  The set of equivalence classes can be identified with the set of characters of the maximal compact subgroup $K_{\GG_m} <F^\times,$ which explains the notation.  In the archimedean case we always use the representatives given by \eqref{KGM}.
Let $\kappa \in \rr_{>0}$ (depending on $\psi$) be chosen so that
\begin{align*}
\kappa dx
\end{align*}
is the standard Haar measure on $F$.  Here the standard Haar measure is the Lebesgue measure if $F=\rr,$ twice the Lebesgue measure if $F=\cc,$ and satisfies $\kappa dx(\mathcal{O}_F)=|\mathfrak{d}|^{1/2}$ where $\mathfrak{d}$ is a generator for the absolute different of $\mathcal{O}_F$ when $F$ is nonarchimedean.  We then let
\begin{align} \label{IF} \begin{split}
I_F:=\begin{cases}\left[ -\frac{\pi}{\log q} ,\frac{\pi}{\log q}\right] & \textrm{ if }F \textrm{ is nonarchimedean},\\
\rr & \textrm{ if }F \textrm{ is archimedean,}\end{cases} \end{split}
\end{align}
and
\begin{align*} 
c_F:=\begin{cases} \kappa\log q & \textrm{ if }F \textrm{ is nonarchimedean},\\
\frac{\kappa}{2} & \textrm{ if }F=\rr,\\
\frac{\kappa}{2\pi}& \textrm{ if }F=\cc.\end{cases}
\end{align*}
Suppose $A(L)<\sigma<B(L)$.  Then  $a_L(\chi_s)$ has no poles for any character $\chi$ and any $s$ with $\mathrm{Re}(s)=\sigma$.

We fix now a maximal compact subgroup $K <G(F)$ such that the Iwasawa decomposition
\begin{align}\label{Iwasawa choice}
    P(F) K=G(F)
\end{align}
holds. The following is a version of Mellin inversion (see \cite[Lemma 4.3]{Getz:Liu:BK}, \cite[Theorem 4.32]{Folland}, \cite[(2.2)]{Blomer:Brumley}):
\begin{Lem} \label{lem:Mellin}
Let $f \in C^\infty(X_{P^{\mathrm{op}}}^{\circ}(F))$ and assume for all $\eta \in \widehat{K}_{\GG_m}$ the integral defining $f_{\eta_s}^{{\mathrm{op}}}$ is absolutely convergent for $\mathrm{Re}(s)=\sigma.$  Suppose moreover that for all $x \in X_{P^{\mathrm{op}}}^\circ(F)$ one has 
\begin{align*}
    \sum_{\eta \in \widehat{K}_{\GG_m}} \int_{\sigma+iI_F}|f^{\mathrm{op}}_{\eta_s}(x)| ds<\infty.
\end{align*}
Then for all $x \in X_{P^{\mathrm{op}}}^\circ(F)$ one has 
\begin{align} \label{Mellin:inv}
f(x)=\sum_{\eta \in \widehat{K}_{\GG_m}} \int_{\sigma+iI_F}f^{\mathrm{op}}_{\eta_s}(x)\frac{c_Fds}{2\pi i}.
\end{align}
Moreover, $f$ is $K$-finite if and only if the sum over $\eta$ has support in a finite set independent of $x.$

Conversely, suppose that we are given continuous $f(\eta)^{(s)} \in \overline{I}(\eta_s)$ for all $s$ with $\mathrm{Re}(s)=\sigma$ and all $\eta \in \widehat{K}_{\GG_m}$ and that 
$$
\sum_{\eta \in \widehat{K}_{\GG_m}} \int_{\sigma+iI_F}|f(\eta)^{(s)}(x)|ds<\infty
$$
for all $x \in X_{P^{\mathrm{op}}}^\circ(F).$
Assume moreover in the nonarchimedean case that $f(\eta)^{(s+\frac{2 \pi i}{\log q})}=f(\eta)^{(s)}.$  Define
$$
f(x):=\sum_{\eta \in \widehat{K}_{\GG_m}} \int_{\sigma+iI_F}f(\eta)^{(s)}(x)\frac{c_Fds}{2 \pi i}.
$$
If the integral defining $f_{\eta_{s}}^{\mathrm{op}}$ is absolutely convergent for all $\eta \in \widehat{K}_{\GG_m}$ and $s$ with $\mathrm{Re}(s)=\sigma$ then
$f_{\eta_s}^{{\mathrm{op}}}=f(\eta)^{(s)}.$ \qed
\end{Lem}
\noindent The lemma implies in particular that \eqref{Mellin:inv} holds for $f \in \mathcal{S}_L$ and $A(L)<\sigma<B(L)$.

As an immediate consequence of Mellin inversion \eqref{Mellin:inv}, we deduce the following estimate for functions in $\mathcal{S}_L$:
\begin{Lem} \label{lem:Sch:bounds}
Assume $\varepsilon>0$ is chosen so that $A(L)+\varepsilon<B(L)-\varepsilon,$  and let $\Omega \subset X_{P^{\mathrm{op}}}(F)$ be a compact subset. For each $f \in \mathcal{S}_L$ and $(m,x) \in M^{\mathrm{ab}}(F) \times \Omega,$ one has an estimate 
\begin{equation*}
|f(mx)| \ll_{\Omega,f,\varepsilon} \delta_{P^{\mathrm{op}}}^{1/2}(m)\min(|\omega_P(m)|^{A(L)+\varepsilon},|\omega_P(m)|^{B(L)-\varepsilon}).
\end{equation*} 
Here when $A(L)=-\infty$ we interpret $A(L)+\varepsilon$ as any negative real number $A,$ and when $B(L)=\infty$ we interpret $B(L)-\varepsilon$ 
as any positive real number $B.$  In these cases, the implied constant depends on $A$ and $B.$  
\qed
\end{Lem}

We now use this to give a criterion for when the regularized integral is the usual integral:
\begin{Lem} \label{lem:conv}
Assume $\lambda>0$ and let $s \in \cc$. If
\begin{align*}
A(L) <
 \frac{\mathrm{Re}(s)+1}{\lambda}< B(L),
\end{align*}
then the integral defining $\lambda_{!}(\mu_s)(f)$ is absolutely convergent for $f \in \mathcal{S}_L$.
\end{Lem}
\begin{proof}
Substituting the bounds from Lemma \ref{lem:Sch:bounds}, it suffices to observe that
\begin{align*}
\int_{F^\times}|t|^{\mathrm{Re}(s)+1}\min(|t|^{-\lambda(A(L)+\varepsilon)},|t|^{-\lambda(B(L)-\varepsilon)})d^\times t
\end{align*}
is convergent for $\varepsilon>0$ sufficiently small.
Here when $A(L)=-\infty$ or $B(L)=\infty,$ we interpret $A(L)+\varepsilon$ and $B(L)-\varepsilon$ as in Lemma \ref{lem:Sch:bounds}.
\end{proof}

For each $i,$ let 
\begin{align}
    \widetilde{L}_i
\end{align}
    be $L_i^\vee$ (the one-dimensional vector space on which $\GG_m$ acts via $-\lambda_i$) attached with the real number $-1-s_i$.  If $-\infty<A(L),$ choose $L_k$ such that $A(L)=\frac{s_k}{\lambda_k},$ and define
$$
L':=\widetilde{L}_k \oplus \bigoplus_{i \neq k}L_i.
$$
Since we have assumed $A(L)<B(L),$ we have that  
\begin{align} \label{string}
A(L') \leq A(L)<B(L') \leq B(L),
\end{align}
so
\begin{align} \label{overlap}
    (A(L),B(L))\cap (A(L'),B(L'))=(A(L),B(L'))\neq \emptyset.
\end{align}
Using this observation, we prove the following proposition:
\begin{Prop} \label{prop:Mellin}
For $-\infty<A(L)<\mathrm{Re}(\chi)<B(L'),$ there is a  commutative diagram
\begin{center}
    \begin{tikzcd}
    \mathcal{S}_L \arrow[rr,"\lambda_{k!}(\mu_{s_k})"] \arrow[d,"(\cdot)_{\chi}^{\mathrm{op}}"]&& \mathcal{S}_{L'} \arrow[d,"(\cdot)_{\chi}^{\mathrm{op}}"]\\
    \overline{I}(\chi) \arrow[rr] &&\overline{I}(\chi)
    \end{tikzcd}
\end{center}
where the bottom arrow is multiplication by $\gamma(-s_k,\chi^{\lambda_k},\psi)$ and the vertical arrows are $f \mapsto f^{\mathrm{op}}_{\chi}$.   In particular, the regularized integral $\lam_{k!}(\mu_{s_k})$ is well-defined on $\mathcal{S}_L$.
\end{Prop}
\begin{proof} Let $f \in \mathcal{S}_L$ and $x \in X_{P^{\mathrm{op}}}^\circ(F).$
By Lemma \ref{lem:Sch:bounds}, for any $\varepsilon>0$ we have
\begin{align*}
&\int_{M^{\mathrm{ab}}(F)}|\Phi|\left(\frac{\omega_P(m)}{b}\right)|\omega_P(m)|^{s_k+1}\delta_{P^{\mathrm{op}}}^{\lambda_k /2}(m)|f|(m^{-\lambda_k}x)\frac{dm}{\zeta(1)}\\
&\ll_{f,\varepsilon,x} 
\int_{F^\times}|\Phi|\left(\frac{t}{b} \right)|t|^{s_k+1-\lambda_k(A(L)+\varepsilon)}d^\times t,
\end{align*}
which is finite for any $b$ when $\varepsilon$ is sufficiently small. 

 We claim that
\begin{align*} \begin{split}
&\lambda_{k!}(\mu_{s_k})(f)(x)=\lim_{|b| \to \infty}\int_{M^{\mathrm{ab}}(F)}\Phi\left(\frac{\omega_P(m)}{b}\right)\psi(\omega_P(m))|\omega_P(m)|^{s_k+1}\delta_{P^{\mathrm{op}}}^{\lambda_k /2}(m)f(m^{-\lambda_k}x)\frac{dm}{\zeta(1)} \end{split}
\end{align*}
converges and is equal to 
\begin{align} \label{Mell2}
h(x):=\sum_{\eta \in \widehat{K}_{\GG_m}}\int_{\sigma+iI_F}\gamma(\lambda_ks-s_k,\eta^{\lambda_k},\psi)f_{\eta_s}^{\mathrm{op}}(x)\frac{c_Fds}{2\pi i}
\end{align}
for 
\begin{align*} 
    A(L)<\sigma<B(L').
\end{align*}
Before proving the claim, it is convenient to study $h(x).$   By standard properties of the Tate $\ga$-factor, we have
\begin{align}\label{eqn: lands in right space}
\frac{\gamma(\lambda_ks-s_k,\eta^{\lambda_k},\psi)f^{\mathrm{op}}_{\eta_s}}{a_{L'}(\eta_s)}=\frac{g(s,\eta,\psi)f^{\mathrm{op}}_{\eta_s}}{a_{L}(\eta_s)}
\end{align}
where $g(s,\eta,\psi)$ lies in $\cc[q^{-s},q^s]$ in the nonarchimedean case
and is holomorphic and bounded in $V_{A,B}$ for all $-\infty<A<B<\infty$ by a constant independent of $\eta$ when $F$ is archimedean. 
Thus the expression defining $h(x)$ is absolutely convergent for $A(L)<\sigma <B(L')$ since $a_{L'}(\eta_s)$ has no poles in this range (see \eqref{overlap}). Here when $F$ is nonarchimedean, we have used the fact that functions in $\mathcal{S}_L$ are finite under a maximal compact subgroup of $G(F)$ and hence the sum over $\eta$ in \eqref{Mell2} has finite support.  In the archimedean case we have used the fact that $\gamma(\lambda_ks-s_k,\eta^{\lambda_k},\psi)$ is bounded by a polynomial in $s$ independent of $\eta$ for $A(L)<\mathrm{Re}(s)<B(L')$ (see the proof of \cite[Lemma 3.3]{Getz:Liu:BK}). 

Let $A(L)<\sigma<B(L')$ and $x\in X^\circ_{P^{\mathrm{op}}}(F)$. We claim the integral
\begin{align} \label{mell:abs} \begin{split}
    &\int_{M^{\mathrm{ab}}(F)} \delta_{P^{\mathrm{op}}}^{1/2}(m)|\omega_P(m)|^{\sigma}|h|(m^{-1}x) dm\\
    &=\int_{M^{\mathrm{ab}}(F)}\left|\sum_{\eta \in \widehat{K}_{\GG_m}}\int_{iI_F}(\eta_{s})^{-1}(\omega_P(m))\gamma(\lambda_k\sigma+\lambda_ks-s_k,\eta^{\lambda_k},\psi)f_{\eta_{\sigma+s}}^{\mathrm{op}}(x)\frac{c_Fds}{2\pi i}\right|dm \end{split}
\end{align}
is convergent.  This implies in particular that $h_{\chi_s}^{\mathrm{op}}$ is well-defined for $A(L)<\sigma<B(L').$  If $F$ is nonarchimedean, it suffices to fix $\eta\in \widehat{K}_{\GG_m}$ and show 
\begin{align*}
&\int_{M^{\mathrm{ab}}(F)}\left|\int_{iI_F}(\eta_{s})^{-1}(\omega_P(m))\gamma(\lambda_k\sigma+\lambda_ks-s_k,\eta^{\lambda_k},\psi)f_{\eta_{\sigma+s}}^{\mathrm{op}}(x)\frac{c_Fds}{2\pi i}\right|dm<\infty
\end{align*}
By the smoothness of $\eta$, it suffices to show 
\begin{align} \label{ell1:finite}
    \sum_{n\in \zz}\left|\int_{iI_F} q^{ns}\gamma(\lambda_k\sigma+\lambda_ks-s_k,\eta^{\lambda_k},\psi)f_{\eta_{\sigma+s}}^{\mathrm{op}}(x)\frac{c_Fds}{2\pi i}\right|<\infty
\end{align}
This is nothing but the $\ell^1$-norm of the Fourier transform of the smooth function
\begin{align*}
   \rr/\tfrac{2\pi}{\log q}\zz&\lto \cc\\ s&\longmapsto \gamma(\lambda_k\sigma+\lambda_kis-s_k,\eta^{\lambda_k},\psi)f_{\eta_{\sigma+is}}^{\mathrm{op}}(x).
\end{align*}
Hence \eqref{ell1:finite} is valid by a standard integration by parts argument.  In the archimedean case the proof that \eqref{mell:abs} converges is similar. One uses the fact that $f \in \mathcal{S}_L$ and that $\gamma(\lambda_ks-s_k,\eta^{\lambda_k},\psi)$ is bounded by a polynomial in $s$ independent of $\eta$ for $A(L)<\mathrm{Re}(s)<B(L')$ as mentioned above.

We conclude that $h_{\eta_s}^{\mathrm{op}}=\gamma(\lambda_ks-s_k,\eta^{\lambda_k},\psi)f^{\mathrm{op}}_{\eta_s}$ by Mellin inversion, specifically the converse statement in Lemma \ref{lem:Mellin}.  Using \eqref{eqn: lands in right space} (and its analogues with $f$ and $h$ replaced by various derivatives in the archimedean setting) we also deduce that $h \in \mathcal{S}_{L'}.$  Thus 
 we can conclude the commutativity of the diagram upon verifying our claim that $\lambda_{k!}(\mu_{s_k})(f)(x)$ is equal to $h(x).$

Observe that the convergence in \eqref{eq:becomeone} is uniform in $\mathrm{Re}(s),\mathrm{Re}(\chi),\lambda$ in a compact set. Therefore, we can reverse the proof of Lemma \ref{lem:gamma} and deduce that \eqref{Mell2} is equal to the limit as $|b| \to \infty$ of
\begin{align}  \label{after:Tate}
   &\nonumber \sum_{\eta \in \widehat{K}_{\GG_m}}\int_{\sigma+iI_F}\gamma(\lambda_ks-s_k,\eta^{\lambda_k},\psi)\left(\int_{F^\times}\left(\int_{F}\Phi\left(\frac{t}{b} \right)\psi(t)\overline{\psi}(yt) dt\right)|y|^{-s_k}\eta_s(y^{\lambda_k})d^\times y\right)f_{\eta_s}^{\mathrm{op}}(x)\frac{c_Fds}{2\pi i}\\
   &=\sum_{\eta \in \widehat{K}_{\GG_m}}\int_{\sigma+iI_F}
   \left(\int_{M^{\mathrm{ab}}(F)}
   \Phi\left(\frac{\omega_P(m)}{b} \right)\psi(\omega_P(m))|\omega_P(m)|^{s_k+1}\eta_s(\omega_P(m)^{-\lambda_k}) f_{\eta_s}^{\mathrm{op}}(x)\frac{dm}{\zeta(1)}\right)\frac{c_Fds}{2\pi i}.
\end{align}
 Moreover, the expression 
\begin{align*}
   \sum_{\eta \in \widehat{K}_{\GG_m}}\int_{\sigma+iI_F}
   \int_{M^{\mathrm{ab}}(F)}
  \left|\Phi\left(\frac{\omega_P(m)}{b} \right)|\omega_P(m)|^{s_k+1}\eta_s(\omega_P(m)^{-\lambda_k})\right| \left|f_{\eta_s}^{\mathrm{op}}(x)\right|dm ds
\end{align*}
is finite.
Indeed, the inner integral is bounded independently of $\eta$ and $s$ since we have assumed $\sigma<B(L')$  and 
\begin{align*}
   \sum_{\eta \in \widehat{K}_{\GG_m}}\int_{\sigma+iI_F}
 \left|f_{\eta_s}^{\mathrm{op}}(x)\right|ds
\end{align*}
is finite by definition of $\cals_L$ since $a_L(\eta_{s})$ has no poles for $A(L)<\mathrm{Re}(s)<B(L')$. 

Therefore, we can rearrange the order of the integral in \eqref{after:Tate} and arrive at  
\begin{align*}
   &\int_{M^{\mathrm{ab}}(F)}\Phi\left(\frac{\omega_P(m)}{b}\right)\psi(\omega_P(m))|\omega_P(m)|^{s_k+1}\delta_{P^{\mathrm{op}}}^{\lambda_k /2}(m)\left(\sum_{\eta \in \widehat{K}_{\GG_m}}\int_{\sigma+iI_F}f_{\eta_s}^{\mathrm{op}}(m^{-\lambda_k}x)\frac{c_Fds}{2\pi i}\right)\frac{dm}{\zeta(1)}\\
   &=\int_{M^{\mathrm{ab}}(F)}\Phi\left(\frac{\omega_P(m)}{b}\right)\psi(\omega_P(m))|\omega_P(m)|^{s_k+1}\delta_{P^{\mathrm{op}}}^{\lambda_k /2}(m)f(m^{-\lambda_k}x)\frac{dm}{\zeta(1)}.
\end{align*}
Here in the last step we have used Mellin inversion (Lemma \ref{lem:Mellin}), which is valid by definition of $\mathcal{S}_L$ because $A(L)<\sigma<B(L')$.  This completes the proof of our claim that $\lambda_{k!}(\mu_{s_k})(f)(x)$ is equal to $h(x).$
\end{proof}

\begin{Def} \label{defn:good:order} Let $L=\bigoplus_{i\in I}L_i$ and $\{(s_i,\lam_i)\}_{i\in I}$ be as above.  Assume all $\lambda_i>0$.
A \textbf{good ordering} of $\{L_i\}$ is a bijection $I \tilde{\to} \{1,\dots,k\}$ for some $k,$ such that after identifying $I$ with $\{1,\dots,k\}$ via the bijection
one has
\begin{align} \label{increasing}
\frac{s_{i+1}}{\lambda_{i+1}} \geq \frac{s_{i}}{\lambda_{i}} 
\end{align}
for $1 \leq i \leq k-1.$
\end{Def}
\noindent We also refer to a good ordering of $\{L_i\}$ as a good ordering of $\{(s_i,\lambda_i)\}_{i \in I}$.    We henceforth assume that $\lambda_i>0$ for all $i$ and $\{L_i\}$ is equipped with a good ordering (it is easy to see it exists).  In particular we use the good ordering to identify $I$ and $\{1,\dots,k\}.$

For $0 \leq i \leq k,$ we define 
\begin{align*}
L(i):=\left(\bigoplus_{1 \leq j \leq k-i}L_j \right)\oplus\left(\bigoplus_{k-i<j \leq k}\widetilde{L}_j \right).
\end{align*}
Note that $L=L(0),$ and set 
\begin{align*}
\widetilde{L}:=L(k).
\end{align*}
Under assumption \eqref{increasing}, for each $1\leq i<k$ one has
\begin{align*}
A(L(i))=\tfrac{s_{k-i}}{\lambda_{k-i}}<B(L(i+1))=\min_{k-(i+1) < j\le k}\left\{ \tfrac{1+s_{j}}{\lambda_{j}}\right\}\le B(L(i))=\min_{k-i< j\le k}\left\{ \tfrac{1+s_{j}}{\lambda_{j}}\right\}
\end{align*}
and
\[
(A(L),B(L))=\left(\tfrac{s_k}{\lam_k},\infty\right)\quad \text{ and } \quad (A(\widetilde{L}),B(\widetilde{L}))=\left(-\infty,\min_{1\le j\le k}\left\{ \tfrac{1+s_{j}}{\lambda_{j}}\right\}\right).
\]
In particular, for each $0\leq i< k$ we have $A(L(i))<B(L(i)),$ so Proposition \ref{prop:Mellin} implies the map
$$
\lambda_{(k-i)!}(\mu_{s_{k-i}}):\mathcal{S}_{L(i)} \lto \mathcal{S}_{L(i+1)}
$$
is well-defined. Thus we  define
\begin{align}\label{mul}
    \mu_{L}:=\lambda_{1!}(\mu_{s_1}) \circ \dots \circ \lambda_{k!}(\mu_{s_k}):\mathcal{S}_L\longrightarrow \mathcal{S}_{\widetilde{L}}
\end{align}
as an iterated composition.
Define 
\begin{align}\label{mulchi}
\mu_{L}(\chi):=\prod_{i=1}^k\gamma(-s_i,\chi^{\lambda_i},\psi).
\end{align}
\begin{Cor} \label{cor:SL:commute}
 One has a commutative diagram
\begin{center}
    \begin{tikzcd}
    \mathcal{S}_{L}  \arrow[rr,"\mu_{L}"] \arrow[d,"(\cdot)^{\mathrm{op}}_{\chi}"]&& \mathcal{S}_{\widetilde{L}} \arrow[d,"(\cdot)^{\mathrm{op}}_{\chi}"]\\
    \overline{I}(\chi) \arrow[rr]& &\overline{I}(\chi)
    \end{tikzcd}
\end{center}
where the bottom arrow is multiplication by $\mu_L(\chi)$.
\end{Cor}

Some care is needed in interpreting the commutativity of this diagram.  Indeed, for general elements of $\mathcal{S}_L,$ the half planes of absolute convergence of $f_{\chi}^{\mathrm{op}}$ and $\mu_L(f)_{\chi}^{\mathrm{op}}$ may be disjoint.  Thus, the identity
$\mu_L(\chi)f_{\chi}^{\mathrm{op}}=\mu_L(f)_{\chi}^{\mathrm{op}}$ (for $f \in \mathcal{S}_L$) asserted by the corollary must be understood in the sense of meromorphic continuation.  

\begin{proof}
Suppose that $A(L(i))<\mathrm{Re}(\chi)<B(L(i+1))$ and consider the diagram in Proposition \ref{prop:Mellin} in the special case $L=L(i)$.  Using the string of inequalities \eqref{string} we see that both vertical arrows in Proposition \ref{prop:Mellin} are given by absolutely convergent integrals.  The diagram in Proposition \ref{prop:Mellin} continues to commute for arbitrary $\mathrm{Re}(\chi)$ if interpreted in the sense of meromorphic continuation.  In other words, for all $0 \leq i <k$ and arbitrary $\chi,$ we have an identity of meromorphic functions
 $$
 \gamma(-s_{k-i},\chi^{\lambda_{k-i}},\psi)f_{\chi}^{\mathrm{op}}=\lambda_{(k-i)!}(\mu_{s_{k-i}})(f)_{\chi}^{\mathrm{op}}
 $$
for $f \in \mathcal{S}_{L(i)}$.  The corollary follows. 
\end{proof}

\subsection{Braverman and Kazhdan's graded representation}  \label{ssec:BKL} We now recall the graded representation $L$ identified by Braverman and Kazhdan, restricting our attention to the case of a fixed maximal parabolic $P$ containing $M$ and its opposite $P^{\mathrm{op}}$.  We use fraktur letters to denote Lie algebras and $\widehat{\cdot}$ to denote the complex-algebraic dual groups and dual Lie algebras.  We have embeddings of Lie algebras
\begin{align*} 
\widehat{\mathfrak{n}}_{P}\lto \widehat{\fp}^{}\lto \widehat{\fg}.
\end{align*}
 Let $\{e,h,f\}$ be a principal $\fsl_2$-triple in $\widehat{\mathfrak{m}}$; it defines an embedding $\fsl_2 \to \widehat{\mathfrak{m}}$. The adjoint action of $\widehat{\fm}$ on $\widehat{\mathfrak{n}}_{P}$ restricts to yield an action of $\fsl_2$ on $\widehat{\mathfrak{n}}_P,$ and we let $\widehat{\fn}_P^e$ denote the space of highest weight vectors.

 Recall our fixed isomorphism
 \begin{align*} 
\omega_P: M^{\mathrm{ab}} \tilde{\lto} \GG_m.
 \end{align*}
This induces a dual isomorphism 
 \begin{align} \label{dual:isom}
 \widehat{\omega}_P:\GG_m \tilde{\lto} \widehat{M}^{\mathrm{ab}} = Z(\widehat{M}) 
 \end{align}
 where $Z(\widehat{M})$ is the center of $\widehat{M}$.  Thus we obtain a $\GG_m$-action on $\widehat{\fn}_P^e.$ Setting 
 \begin{align} \label{L:def}
     L:=\widehat{\fn}_P^e=\bigoplus_iL_i,
 \end{align}
 we let $\lam_i$ be the $\GG_m$-character and $s_i$ be $\tfrac{1}{2}$ times the $h$-eigenvalue on the line $L_i$. 

\begin{Lem} \label{lem:pos}
For each $L_i$ as above $s_i$ is nonnegative and $\lambda_i$ is positive.
\end{Lem}
\begin{proof}
The $s_i$ are all $\tfrac{1}{2}$ times the $h$-eigenvalue of a highest weight vector of a $\mathfrak{sl}_2$-representation and hence are nonnegative.  The $\lambda_i$ are all positive by Lemma \ref{lem:posi}.  
\end{proof}

We define  
 \begin{align} \label{etaP}
 \mu_{P}:=\mu_{L}:\mathcal{S}_{L} \lto \mathcal{S}_{\widetilde{L}}\quad \textrm{ and } \quad
    \mu_P(\chi):=\mu_L(\chi)  
\end{align}
for the choice of $L$ given in \eqref{L:def}.  Here $\mu_L(\chi)$ is defined as in \eqref{mul}.  

\subsection{Switching to the opposite parabolic}\label{section: switching}
In Corollary \ref{Cor:Pop} we will switch between $P$ and $P^{\mathrm{op}}$ for self-associate parabolic subgroups. This requires care regarding signs. We choose a principal $\fsl_2$-triple $\{e,h,f\}$ as above and consider $ L^{\mathrm{op}}=(\widehat{\mathfrak{n}}_{P^{\mathrm{op}}})^e$. We claim that
  \begin{align} \label{Lop}
 L^{\mathrm{op}}=\bigoplus_{i\in I}L_i^{\mathrm{op}},
 \end{align}
 where $\GG_m$ and $h$ act on $L^{\mathrm{op}}_i$ via $\lambda_i$ and $2s_i,$ respectively.   Indeed, the $\fsl_2(\cc) \times Z(\widehat{M})$-representations $\widehat{\mathfrak{n}}_{P}$ and $\widehat{\mathfrak{n}}_{P^{\mathrm{op}}}$ are dual.   
 Thus as representations of $\mathfrak{sl}_2(\cc)$ they have the same highest weights.  Since the parameters $\lambda$ are defined using \eqref{dual:isom}, we deduce the claim from the observation that
 \[
 \omega_P = \omega_{P^{\mathrm{op}}}^{-1}.
 \] 

\subsection{The Lagrangian Grassmannian} \label{sec:lagrangian}
As an example, let $\mathrm{Sp}_{2n}$ denote the symplectic group on a $2n$-dimensional vector space and let $P \leq \mathrm{Sp}_{2n},$ $M \leq P$ denote the Siegel parabolic and Levi subgroup, respectively.   Specifically, for $\zz$-algebras $R$, set
\begin{align*}\begin{split}
\mathrm{Sp}_{2n}(R):&=\left\{g \in \GL_{2n}(R): g^t\begin{psmatrix} & I_n \\ -I_n & \end{psmatrix}g=\begin{psmatrix} & I_n \\ -I_n & \end{psmatrix}\right\},\\
M(R):&=\{\begin{psmatrix} A & \\ & A^{-t} \end{psmatrix}: A \in\GL_n(R)\},\\
N(R):&=\{\begin{psmatrix} I_n &  Z\\ & I_n \end{psmatrix} : Z \in \mathfrak{gl}_n(R), Z^t=Z\}, \end{split}
\end{align*}
and $P=MN$.  We have
\begin{align*}
\omega_P:M(R) &\lto R^\times\\
\begin{psmatrix} m & \\ & m^{-t} \end{psmatrix} &\longmapsto \det m,
\end{align*}
 $\widehat{\mathfrak{g}}=\mathfrak{so}_{2n+1},$ and $\widehat{\mathfrak{m}}=\mathfrak{gl}_n$. Moreover, as a representation of $\widehat{\fm},$
$$
\widehat{\mathfrak{n}}_{P} \cong  V_{\mathrm{st}}\oplus \wedge^2 V_{\mathrm{st}},
$$
where $V_{\mathrm{st}}$ is the standard representation of $\fgl_n$. 
 We use the standard principal $\fsl_2$-triple in $\mathfrak{gl}_n$.  Concretely it is the image of $\mathfrak{sl}_2$ under  $\mathrm{Sym}^{n-1}$.  The space $\widehat{\fn}_P^e$ is just the direct sum of the highest weight spaces of the
 $\mathfrak{sl}_2$-representation 
 \begin{align*}
 \mathrm{Sym}^{n-1}(\cc^2) \oplus \wedge^2\mathrm{Sym}^{n-1}(\cc^2)
 &\cong\mathrm{Sym}^{n-1}(\cc^2) \oplus \bigoplus_{j=0}^{\lfloor (n-2)/2 \rfloor}\mathrm{Sym}^{2(n-2)-4j}(\cc^2).
 \end{align*}
 Here we have used some well-known plethysms (see Lemma \ref{lem:plethysm} below).
 Then 
 \begin{align*}
     (s_{r},\lambda_{r})=\left(n+2r-2\lfloor n/2 \rfloor -2,2
      \right) \textrm{ for }1 \leq r \leq \lfloor n/2 \rfloor \textrm{ and }(s_{\lfloor n/2 \rfloor+1},\lambda_{\lfloor n/2 \rfloor+1})=\left(\tfrac{n-1}{2},1\right).
 \end{align*}
This is a good ordering.  

We observe that 
\begin{align} \label{agreement}
a_{I_{2n}}(s,\chi)=a_{\widetilde{L}}((\chi_s)^{-1}) \quad  \textrm{ and } \quad
a_{w_0}(s,\chi)=a_L(\chi_s) 
\end{align}
in the notation of \cite[\S 3]{Getz:Liu:BK}.

\section{The Schwartz space of the affine closure of a Braverman--Kazhdan space}
\label{sec:Schwartz:BK}

Throughout this section we assume that our simple group $G$ is simply connected so that we can apply the results of \cite{BKnormalized}.  Braverman and Kazhdan originally defined operators $\mathcal{F}_{P^{\mathrm{op}}|P}$ via a series of integral operators on an inexplicit subspace of $\mathcal{S}(X_{P^{\mathrm{op}}}^{\circ}(F)),$ proved that the operators extended to unitary operators on $L^2(X_{P^{\mathrm{op}}}(F)),$ and then proposed the following definition:
\begin{Def} \label{defn:BK}
The \emph{BK-Schwartz space} $\cals_{BK}(X_{P}(F))$ is defined as the sum
\[
\cals_{BK}(X_{P}(F))= \mathcal{S}(X_{P}^\circ(F)) +\mathcal{F}_{P^{\mathrm{op}} | P}(\mathcal{S}(X_{P^{\mathrm{op}}}^\circ(F))).
\]
\end{Def}  
\noindent Here the sum is taken in $L^2(X_P(F))$.  We point out that the expression $\mathcal{F}_{P^{\mathrm{op}} | P}(\mathcal{S}(X_{P^{\mathrm{op}}}^\circ(F)))$ means that we apply the $L^2$-extension of $\mathcal{F}_{P^{\mathrm{op}}|P}$ to $\mathcal{S}(X_{P^{\mathrm{op}}}^\circ(F))$.  It is far from obvious that the integral operators defining $\mathcal{F}_{P^{\mathrm{op}}|P}$ converge when applied to elements of $\mathcal{S}(X^\circ_{P^{\mathrm{op}}}(F))$.  In fact this is not known in general, see \S \ref{ssec:def:rel} below.

\begin{Rem} Braverman and Kazhdan only state this definition in the nonarchimedean case, but the extension to the archimedean case is natural and was suggested to the first author by Kazhdan.  \end{Rem}

In \cite{Getz:Liu:BK} the first author and Liu refined Braverman and Kazhdan's definition when $G=\mathrm{Sp}_{2n}$ and $P$ is the Siegel parabolic, and gave explicit spaces of functions that are mapped to each other under the Fourier transform.  We do the same for Braverman--Kazhdan spaces attached to simple groups $G$ and maximal parabolic subgroups $P <G$ in this section.  This goes beyond the work of Braverman and Kazhdan in that it allows us to isolate an explicit subspace on which our formulae for the Fourier transforms given in \S \ref{sec:FT:XP} are valid.  

\subsection{Measures redux}   \label{ssec:measures:re}
Thus far we have only made use of Haar measures $dx$ and $d^\times x$ on $F$ and $F^\times$ related as in \S \ref{ssec:measures}.  In order to study the Schwartz space and the Fourier transform, we require right $G(F)$-invariant measures on $X_P^{\circ}(F),$ $X_{P^{\mathrm{op}}}^{\circ}(F),$ and choices of Haar measures on $N_P(F)$ and $N_{P^{\mathrm{op}}}(F)$. 
First, we fix a Haar measure on $M(F)$.  We give $M^{\mathrm{der}}(F)$ the unique measure such that $\omega_P:M/M^{\mathrm{der}}(F) \tilde{\to} F^\times$ is measure preserving. 

Recall that we fixed a split maximal torus $T \leq M$ and a Borel subgroup $T\le B \leq P$ in \S \ref{ssec:Pl}.  Let
\begin{align*}
\Theta:\mathfrak{g} \lto \mathfrak{g}
\end{align*}
be the opposite involution attached to $\mathfrak{t}$ (here we follow the conventions of  \cite[\S 23.h]{Milne:AGbook}).  
Fix a Chevalley basis of the Lie algebra of $G$ with respect to the Lie algebra of $T$. For all roots $\alpha,$ this gives us vectors $X_{\alpha}$ in the root space of $\alpha$ that satisfy $X_{-\alpha}=-\Theta(X_{\alpha})$ \cite[\S 23.h]{Milne:AGbook}, and provides us with isomorphisms
$$
\GG_a \lto N_{\alpha}
$$
where $N_{\alpha} \leq G$ is the root subgroup of $\alpha$.  We use this to endow each $N_{\alpha}(F)$ with the measure $dx$ by transport of structure, which in turn gives rise to Haar measures on $N_P(F)$ and $
N_{P^{\mathrm{op}}}(F)$. This is the same normalization used in \cite{Langlands:EulerPro}.  The motivation for this choice of measures is to make factorization of intertwining operators valid. 

Now we normalize the right $G(F)$-invariant measures on $X_P^\circ (F)$ and $X_{P^{\mathrm{op}}}^\circ (F)$.  By the Bruhat decomposition, one has an injection
\begin{align*}
M^{\mathrm{ab}}(F) \times N_{P^{\mathrm{op}}}(F) &\lto X_P^{\circ}(F)\\
(m,u) &\longmapsto P^{\mathrm{der}}(F)mu
\end{align*}
with Zariski open and dense (hence, of full measure) image.  We can and do normalize the right $G(F)$-invariant nonnegative Radon measure $dx$ on $X_P^{\circ}(F)$ such that 
\begin{align} \label{meas:comp general}
d(mu)=\frac{\delta_{P^{\mathrm{op}}}(m)dmdu}{\zeta(1)}.
\end{align}
  Similarly, we normalize the  right $G(F)$-invariant non-negative Radon measure $dx$ on $X^\circ_{P^{\mathrm{op}}}(F)$ so that 
$$
d(mu)=\frac{\delta_{P}(m)dmdu}{\zeta(1)}
$$
for $(m,u)\in M^{\mathrm{ab}}(F) \times N_{P}(F)$.

\subsection{The Schwartz space} \label{ssec:the:Schwartz:space}

For functions $f \in C^\infty(X_P^{\circ}(F))$ and $x=P^{\mathrm{op},\mathrm{der}}(F)g\in X_{P^{\mathrm{op}}}^{\circ}(F),$ we define the unnormalized intertwining operator
\begin{align} \label{R}
\mathcal{R}_{P|P^{\mathrm{op}}}(f)(x):=\int_{N_{P^{\mathrm{op}}}(F)}f\left(P^{\mathrm{der}}(F)ug\right)du=\int_{N_{P^{\mathrm{op}}}(F)}f(ug)du
\end{align}
whenever this integral is absolutely convergent (or obtained via some regularization procedure).  We refer to $\mathcal{R}_{P|P^{\mathrm{op}}}$ as a \textbf{Radon transform}, as it is a generalization of the classical Radon transform  \cite[\S 2.9]{BKnormalized}. That this agrees with the operator defined by Braverman and Kazhdan is proved in \cite[\S 5]{Shahidi:Generalized}.
 For example, we have maps
$$
\mathcal{R}_{P|P^{\mathrm{op}}}:\mathcal{S}(X_P^{\circ}(F)) \lto C^\infty(X_{P^{\mathrm{op}}}^\circ(F))
$$
and 
$$
\mathcal{R}_{P|P^{\mathrm{op}}}:I(\chi_s) \lto \overline{I}(\chi_s)
$$
for $\mathrm{Re}(s)$ sufficiently large that may be extended meromorphically to $\cc$ \cite[\S 10.1.2, \S 10.1.6]{Wallach:RGII} \cite[Theorem IV.1.1]{Waldspurger}.
For notational convenience, we write
$$
\mathcal{R}_{P|P}:C^\infty(X_P^{\circ}(F)) \lto C^\infty(X_P^{\circ}(F))
$$
for the identity operator.

Let $L$ be the graded $\mathbb{G}_m$-representation associated to $P$ in \S \ref{ssec:BKL} and let $\{(s_i,\lambda_i)\}$ be a good ordering of $L$. For quasi-characters $\chi:F^\times \to \cc^\times,$ we set
\begin{align} \label{ap} 
    a_{P|P}(\chi):=a_{\widetilde{L}}(\chi^{-1})  \quad \textrm{ and } \quad
    a_{P|P^{\mathrm{op}}}(\chi):=a_L(\chi). 
\end{align}
Bearing in mind the discussion in \S \ref{section: switching}, the definition \eqref{ap} implies
\begin{align} \label{a:rel}
    a_{P|P}(\chi)=a_{P^{\mathrm{op}}|P^{\mathrm{op}}}(\chi) \quad \textrm{ and } \quad a_{P|P^{\mathrm{op}}}(\chi)=a_{P^{\mathrm{op}}|P}(\chi).
\end{align}

\begin{Lem} \label{lem:geq0}
The function $a_{P|P}(\chi)$ is holomorphic for $\mathrm{Re}(\chi) \geq 0$.  
\end{Lem}

\begin{proof}
It suffices to show that $s_i+1>0$ and $\lambda_i>0$ for all $L_i$.  This follows from Lemma \ref{lem:pos}.
\end{proof}

Fix now a character $\chi$ and recall that $\chi_s = \chi|\cdot|^s$ for $s\in \cc$. A section $f(\chi)^{(s)}$ of $I(\chi_s)$ is \textbf{good} if it is meromorphic, 
and if for $Q \in \{P,P^{\mathrm{op}}\}$ the sections
\begin{align}
    g \longmapsto \frac{\mathcal{R}_{P|Q}f(\chi)^{(s)}(g)}{a_{P|Q}(\chi_s)}
\end{align}
of $I(\chi_s)$ and $\overline{I}(\chi_s)$ are holomorphic for all $g \in G(F)$.  
\begin{Def} \label{defn:na} Assume $F$ is nonarchimedean.  The Schwartz space $\cals(X_P(F))$ is defined to be the space of right $K$-finite functions $f \in C^\infty(X_P^{\circ}(F))$ such that for each $g \in G(F)$ and character $\chi$ of $F^\times,$ the integral \eqref{Mellin} defining $f_{\chi_s}(g)$ is absolutely convergent for $\mathrm{Re}(s)$ large enough and defines a good section.
\end{Def}

 For $F$ archimedean, recall we have an action of $U(\mathfrak{m}^{\mathrm{ab}} \oplus \mathfrak{g})$ on $C^\infty(X_P^{\circ}(F))$ via the differential of \eqref{geo:act}. 

\begin{Def} \label{defn:arch} 
Assume $F$ is archimedean. The Schwartz space $\mathcal{S}(X_P(F))$ is defined to be the space of functions $f \in C^\infty(X_P^{\circ}(F))$ such that for all  $D \in U(\mathfrak{m}^{\mathrm{ab}} \oplus \mathfrak{g}),$ $g \in  G(F),$ and each character $\chi$ of $F^\times,$ the integral \eqref{Mellin} defining $(D.f)_{\chi_s}(g)$ is absolutely convergent for $\mathrm{Re}(s)$ large enough, defines a good section, and satisfies the following condition:
For all real numbers $A<B,$ $Q \in \{P,P^{\mathrm{op}}\},$ any polynomial $p_{P|Q} \in \cc[s]$ such that $p_{P|Q}(s)a_{P|Q}(\eta_s)$ has no poles for all $(s,\eta) \in V_{A,B} \times \widehat{K}_{\GG_m},$ and compact subsets $\Omega \subset X_P^{\circ}(F)$ one has that 
\begin{align} \label{seminorm}
|f|_{A,B,p_{P|Q},\Omega,D}:=\sum_{\eta \in \widehat{K}_{\GG_m}}\sup_{g \in \Omega}|\mathcal{R}_{P|Q}(D.f)_{\eta_s}(g)|_{A,B,p_{P|Q}}<\infty.
\end{align}
\end{Def}
To understand this definition, it is useful to point out that it is indeed possible to choose $p_{P|Q}$ (independently of $\eta$) that satisfy the given assumptions.  This follows directly from the definition of the $a_{P|Q}(\eta_s)$.  We also observe that the $|\cdot|_{A,B,p_{P|Q},\Omega,D}$ are seminorms and they give $\mathcal{S}(X_P(F))$ the structure of a Fr\'echet space by essentially the same argument proving \cite[Lemma 3.2]{Getz:Hsu}.

\begin{Rem}
Note that $f_{\chi,P}^{\mathrm{op}}=f_{\chi^{-1},P}.$  Using this observation and the discussion in \S \ref{section: switching} we see that $\mathcal{S}(X_P(F)) \leq \mathcal{S}_{\widetilde{L}}(X_{P}(F)).$
\end{Rem}

For any $F,$ the action of $M^{\mathrm{ab}}(F) \times G(F)$ on $X_P^{\circ}(F)$ induces a smooth action on $\mathcal{S}(X_P(F))$ (in either the archimedean or nonarchimedean setting). In the archimedean setting, this action is continuous in the Fr\'echet topology of $\mathcal{S}(X_P(F))$.

 In the special case of the Siegel parabolic subgroup of $G=\Sp_{2n},$ a slight variant of this space was introduced in \cite{Getz:Hsu}. Their definition generalized and slightly modified the $K$-finite Schwartz space $\cals(X(F),K)$ introduced in \cite{Getz:Liu:BK}.   To compare our definition to the one given in \cite{Getz:Hsu}, we use \eqref{agreement} and observe that 
 we have not applied a Weyl element to turn $P^{\mathrm{op}}$ into a standard parabolic.

The elements of the Schwartz space are well-behaved analytically.  They can be bounded in an intuitive manner using the Pl\"ucker embedding. Let
$$
\mathrm{Pl}:X_P \lto V_P
$$
be the Pl\"ucker embedding defined by a choice of highest weight vector $v_P$ as in  Lemma \ref{Lem: general Plucker}. 
Choose a norm $|\cdot|$ on $V_P(F)$ that is invariant under $K$ and let
\begin{align*}
|\cdot|:X_P^{\circ}(F) &\lto \rr_{>0}\\
x &\longmapsto |\mathrm{Pl}(x)|;
\end{align*}
here, $K$ is chosen as in \eqref{Iwasawa choice}. Replacing $v_P$ by $tv_P$ for $t \in F^\times$ multiplies this norm by $|t|$.

Let $r \in \qq_{>0}$ be such that 
$$
|\omega_P|^r=\delta_P.
$$
Note that our assumption that $G$ is simply connected implies that $r\in \zz_{>0}$; indeed, we compute this value in Proposition \ref{Prop: root explanation} below.

\begin{Lem} \label{lem:Sch:bounds2} Assume $\alpha>0$ is sufficiently small. Let $f\in \mathcal{S}(X_{P}(F))$. When $F$ is nonarchimedean, $f(x)$ vanishes for $|x|$ sufficiently large and 
$$
|f(x)| \ll_{\alpha} |x|^{-r/2+\alpha}.
$$
When $F$ is archimedean, for all $N \in \zz_{\ge 0}$ one has 
\begin{equation*}
|f(x)| \leq \nu_{N,\alpha}(f)
|x|^{-r/2+\alpha}\max(1,|x|)^{-N}
\end{equation*}
where $\nu_{N,\alpha}$ is a continuous seminorm on $\mathcal{S}(X_P(F))$.
\end{Lem}
\begin{proof}
Write $x=P^{\mathrm{der}}(F)mk$ with $m \in M^{\mathrm{ab}}(F)$ and $k \in K$.  
By definition of $\mathcal{S}(X_P(F))$ and Mellin inversion \eqref{Mellin:inv}, we have
\begin{align} 
f(x)=\delta_P^{1/2}(m)\sum_{\eta \in \widehat{K}_{\GG_m}}\int_{\sigma+iI_F}\eta_s(\omega_P(m))f_{\eta_s}(k)\frac{c_Fds}{2\pi i}
\end{align}
provided that there are no poles of $a_{P|P}(\eta_s)$ for $\mathrm{Re}(s) \ge\sigma$.  Moreover the sum and integral converge absolutely.  
Therefore to prove the bounds for
$|x| \leq  1$ in the archimedean case and for all $x$ in the nonarchimedean case it suffices to  recall $a_{P|P}(\eta_s)$ has no poles for $\mathrm{Re}(s)\geq 0$ by Lemma \ref{lem:geq0}.  

The support assertion in the nonarchimedean case follows as in \cite[Lemma 5.1]{Getz:Liu:BK}.  The bound for $|x| \gg 1$ in the archimedean case follows as in \cite[Lemma 3.5]{Getz:Hsu}.
\end{proof}

\noindent As $X_P^\circ(F)$ is open and dense in $X_P(F),$ we can and do extend the right $G(F)$-invariant Radon measure on $X_P^\circ(F)$ by zero to $X_P(F)$.

\begin{Cor} \label{cor:L2}
One has $\mathcal{S}(X_P(F)) <L^2(X_P(F)) \cap L^1(X_P(F))$.  
\end{Cor}
\begin{proof}This follows from
 Lemma \ref{lem:Sch:bounds2} and the Iwasawa decomposition.
\end{proof}

\subsection{The Fourier transform} \label{ssec:FT}

Braverman and Kazhdan \cite{BKnormalized} prove that the Fourier transform
\begin{align}\label{FPPOP}
\mathcal{F}_{P|P^{\mathrm{op}}}:=\mu_P\circ \mathcal{R}_{P|P^{\mathrm{op}}}
\end{align}
is well-defined on a subspace of $\mathcal{S}(X_{P}^\circ (F))$ that is dense in $L^2(X_P(F))$ and defines an isometry
\begin{align} \label{isom}
\mathcal{F}_{P|P^{\mathrm{op}}}:L^2(X_P(F)) \lto L^2(X_{P^{\mathrm{op}}}(F)).
\end{align}
  They also prove that 
\begin{align} \label{relations}
\mathcal{F}_{P|P^{\mathrm{op}}}\circ\mathcal{F}_{P^{\mathrm{op}}|P}=\mathrm{Id}.
\end{align}
We use the results of the previous subsections to refine $\mathcal{F}_{P|P^{\mathrm{op}}}$ to an isomorphism between $\mathcal{S}(X_P(F))$ and $\mathcal{S}(X_{P^{\mathrm{op}}}(F))$ in this subsection.  

\begin{Lem} \label{lem:unnorm:commute} One has a commutative diagram
\begin{center}
\begin{tikzcd}
    \mathcal{S}(X_P(F)) \arrow[r,"\mathcal{R}_{P|P^{\mathrm{op}}}"]
    \arrow[d,"(\cdot)_{\chi}"] &\mathcal{S}_L \arrow[d,"(\cdot)^{\mathrm{op}}_{\chi}"]\\
    I(\chi) \arrow[r,"\mathcal{R}_{P|P^{\mathrm{op}}}"]&\overline{I}(\chi)
    \end{tikzcd}
\end{center}
for $\mathrm{Re}(\chi)$ sufficiently large.  
\end{Lem}

\begin{proof} For $g \in G(F)$ consider the integral 
\begin{align} \label{for:Fubini}
\int_{N_{P^{\mathrm{op}}}(F)} \int_{M^{\mathrm{ab}}(F)}\delta_P^{1/2}(m)|\chi|(\omega_P(m)) |f|(m^{-1} ug)dmdu.
\end{align}
The inner integral converges and defines an element of $I(|\chi|)$ for $\mathrm{Re}(\chi)$ sufficiently large by definition of $\mathcal{S}(X_P(F)),$ and the outer integral converges for $\mathrm{Re}(\chi)$ sufficiently large  \cite[Lemma 10.1.2]{Wallach:RGII} \cite[Theorem IV.1.1]{Waldspurger}.  Thus by Fubini's theorem, we have a commutative diagram
\begin{center}
\begin{tikzcd}
    \mathcal{S}(X_P(F)) \arrow[r,"\mathcal{R}_{P|P^{\mathrm{op}}}"]
    \arrow[d,"(\cdot)_{\chi}"] &\mathcal{R}_{P|P^{\mathrm{op}}}(\mathcal{S}(X_P(F))) \arrow[d,"(\cdot)^{\mathrm{op}}_{\chi}"]\\
    I(\chi) \arrow[r,"\mathcal{R}_{P|P^{\mathrm{op}}}"]&\overline{I}(\chi)
    \end{tikzcd}
\end{center}
for $\mathrm{Re}(\chi)$ sufficiently large.  We are left with proving that 
$
\mathcal{R}_{P|P^{\mathrm{op}}}(\mathcal{S}(X_P(F))) \leq \mathcal{S}_L.$
By the definitions of $\mathcal{S}(X_P(F))$ and $\mathcal{S}_L$ it suffices to check that 
\begin{align} \label{are:smooth}
\mathcal{R}_{P|P^{\mathrm{op}}}(\mathcal{S}(X_P(F)) \leq C^\infty(X_P^{\circ}(F)).
\end{align} Let $f \in \mathcal{S}(X_P(F)).$
By Fubini's theorem and the argument above for almost every $m$ with respect to $dm$ we have that  $\int_{N_{P^{\mathrm{op}}}(F)}f(m^{-1} ug)du$ converges.  When $F$ is nonarchimedean we can use the fact that $f$ is $K$-finite and Lemma \ref{lem:Mellin} to deduce that $f$ is fixed by a compact open subgroup of $M^{\mathrm{ab}}(F).$  This implies that $\int_{N_{P^{\mathrm{op}}}(F)}f( ug)du$ converges absolutely.  Since $\mathcal{R}_{P|P^{\mathrm{op}}}$ is a $G(F)$-intertwining map this implies that $\mathcal{R}_{P|P^{\mathrm{op}}}(f)$ is smooth.  Now assume that $F$ is archimedean.  In this case we can view the integral \eqref{for:Fubini}, as $g$ varies, as valued in the Fr\'echet space $C^\infty(G(F))$ (with the usual Fr\'echet topology).  Using the Fubini theorem in this setting \cite[Theorem 8]{Thomas}, we deduce that for almost all $m$ with respect to $dm$, $\int_{N_{P^{\mathrm{op}}}(F)}f(m^{-1} ug)du$ converges absolutely and defines a smooth function of $g.$ For such an $m$, we change variables $u \mapsto mum^{-1}$ and replacing $g$ by $mg$, we deduce that $\mathcal{R}_{P|P^{\mathrm{op}}}(f)$ is smooth.
\end{proof}

To proceed, we recall the subspaces $\mathcal{C}_Q<\mathcal{S}(X_Q^\circ(F))$ for $Q\in \{P,P^{\mathrm{op}}\}$ considered in \cite[Proposition 4.2]{BKnormalized} that are used to proved the unitarity of the operator $\mathcal{F}_{Q|Q^{\mathrm{op}}}$ on $L^2(X_Q(F))$. In the following, we will use the notation in \eqref{normalizedinduct} and \eqref{Mellin} to keep track of the domain of our Mellin transforms.

\begin{Lem} \label{lem:claim}
For each $\chi \in \widehat{K}_{\GG_m}$ we can choose holomorphic functions 
$h_{Q}(\chi_s)$ that lie in $\cc[q^{-s},q^{s}]$ in the nonarchimedean case and are bounded in vertical strips in the archimedean case such that 
\begin{align} \label{is:holo}
h_{Q}(\chi_s)\mathcal{R}_{Q|Q^{\mathrm{op}}}:I_{Q}(\chi_s)\lto \overline{I}_{Q^{\mathrm{op}}}(\chi_s)
\end{align}
is holomorphic when evaluated on a holomorphic section $f(\chi)^{(s)}\in I_Q(\chi_s)$ and an isomorphism for $s$ outside a discrete countable subset of $\cc$.   
\end{Lem}
\begin{proof}
Assume first that $F$ is nonarchimedean.  Then one can use the usual normalizing factors for intertwining operators \cite[\S 2-4]{Arthur:Intertwining} to construct $h_Q(\chi_s)$ satisfying the requirements in the lemma.  If $F$ is archimedean, loc.~cit.~implies the existence of a set $\{a_i,b_i\}_{i=1}^n$ of complex numbers such that
\begin{align}
 \left(\prod_{i=1}^n\frac{1}{\Gamma(a_is+b_i)} \right)   \mathcal{R}_{Q|Q^{\mathrm{op}}}
\end{align}
is holomorphic when evaluated on a holomorphic section $f(\chi)^{(s)}\in I_Q(\chi_s)$ and is an isomorphism for $s$ outside a discrete countable subset of $\cc$. The factors of $\Gamma$ here are archimedean $L$-functions of quasi-characters of $F^\times$ up to irrelevant factors.  To obtain $h(\chi_s)$ we take the product of reciprocals of $\Gamma$-functions and multiply by $e^{s^2}$ to make the result rapidly decreasing in vertical strips.
\end{proof}

 We henceforth assume the $h_{Q}(\chi_s)$ are chosen as in Lemma \ref{lem:claim}.
 Let $\mathcal{S}(X_Q^\circ(F),K)$ be the space of $K$-finite functions in $\mathcal{S}(X_Q^\circ(F))$, and let
\begin{align*}
    \mathcal{C}_Q:=\left\{f \in \mathcal{S}(X_Q^\circ(F),K):\begin{matrix} \textrm{There exists an }f' \in  \mathcal{S}(X_Q^{\circ}(F),K) \textrm{ such that }\\ f_{\chi_s,Q}=h_{Q^{\mathrm{op}}}((\chi_s)^{-1})h_Q(\chi_s)f'_{\chi_s,Q} \\\textrm{ for all characters }\chi:F^\times \to \cc^\times \textrm{ and all }s \in \cc
     \end{matrix}\right\}.
\end{align*}
For a subspace $W\le \mathcal{S}(X_Q^\circ(F))$, let $W_{\chi_s,Q}$ denote its image in $I(\chi_s)$ under the Mellin transform \eqref{Mellin}.  We also use the notation $\mathcal{S}(X_Q(F))_{\chi_s,Q}$ for the image of $\mathcal{S}(X_Q(F))$ in $I(\chi_s)$ under the Mellin transform, which must be understood in the following sense:  For $\mathrm{Re}(s)$ sufficiently large these Mellin transforms are absolutely convergent by definition of the Schwartz space. Again by definition of the Schwartz space the Mellin transforms are defined by meromorphic continuation for $s$ outside a countable subset of $\cc$ independent of $\chi \in \widehat{K}_{\GG_m}.$

\begin{Lem}\label{lem:defBK}
For $f\in \mathcal{C}_Q$ the functions  $\mathcal{R}_{Q|Q^{\mathrm{op}}}(f_{\chi_s,Q})$ and $\mathcal{R}_{Q^{\mathrm{op}}|Q}(\mathcal{R}_{Q|Q^{\mathrm{op}}}(f_{\chi_s,Q}))$ are holomorphic for all $\chi \in \widehat{K}_{\GG_m}.$ One has $\mathcal{C}_Q< \mathcal{S}(X_Q(F))$.
For $s$ outside a countable subset of $\cc$ (independent of $\chi$)  one has 
$$
(\mathcal{C}_{Q})_{\chi_s,Q}=\mathcal{S}(X_Q^\circ(F),K)_{\chi_s,Q},
$$
which is dense in $\mathcal{S}(X_Q(F))_{\chi_s,Q}$ in the usual Fr\'echet topology if $F$ is archimedean and equal to $\mathcal{S}(X_Q(F))_{\chi_s,Q}$ if $F$ is nonarchimedean.  
\end{Lem}

\begin{proof}  The first assertion is immediate from the definition of $\mathcal{C}_Q.$  The inclusion $\mathcal{C}_Q< \mathcal{S}(X_Q(F))$ follows from the fact that $a_{Q|Q^{\mathrm{op}}}(\chi_s)$ and $a_{Q|Q}(\chi_s)$ have no zeros.   As $h_{Q^{\mathrm{op}}}((\chi_s)^{-1})h_Q(\chi_s)$ is nonzero outside a discrete countable set we have $(\mathcal{C}_{Q})_{\chi_s,Q}=\mathcal{S}(X_Q^\circ(F),K)_{\chi_s,Q}$ outside a discrete countable set.  The union of these sets is again countable.  Since $\mathcal{S}(X_Q^\circ(F),K)_{\chi_s,Q}$ is the space of $K$-finite vectors in $I(\chi_s)$ (which is all of $I(\chi_s)$ in the nonarchimedean case) the last assertion of the lemma follows.
\end{proof}

We remark that here the definition of $\mathcal{C}_Q$ depends on the choice of $h_Q(\chi_s)$ and $h_{Q^{\mathrm{op}}}((\chi_s)^{-1})$. Using Corollary \ref{cor:SL:commute}, and a minor variant of the proof of Lemma \ref{lem:claim} above we fix a choice of $h_Q(\chi_s)$ and $h_{Q^{\mathrm{op}}}((\chi_s)^{-1})$ such that $\mathcal{F}_{P|P^{\mathrm{op}}}(\mathcal{C}_P)<\mathcal{S}(X_{P^{\mathrm{op}}}(F))$.

\begin{Thm} \label{thm:FT}
We have a well-defined isomorphism
\begin{align*}
    \mathcal{F}_{P|P^{\mathrm{op}}}:\mathcal{S}(X_P(F)) \lto \mathcal{S}(X_{P^{\mathrm{op}}}(F)),
\end{align*}
that is continuous with respect to the Fr\'echet topologies in the archimedean case.
The diagram
\begin{equation} \label{commutes}
  \begin{tikzcd}
\mathcal{S}(X_P(F))\ar[r,"\calf_{P | P^{\mathrm{op}}}"]\ar[d,"(\cdot)_{\chi}"]&\mathcal{S}(X_{P^{\mathrm{op}}}(F)))\ar[d,"(\cdot)_{\chi}^{\mathrm{op}}"]\\
I(\chi) \ar[r,"\mu_{P}(\chi)\mathcal{R}_{P|P^{\mathrm{op}}}"]&\overline{I}(\chi)
\end{tikzcd}
\end{equation}
commutes.   
\end{Thm}

\noindent 
As in Corollary \ref{cor:SL:commute}, some care is required in interpreting the statement that the diagram commutes.  The Mellin transform $(\cdot)_{\chi}$ converges absolutely for $\mathrm{Re}(\chi)$ large and the Mellin transform $(\cdot)_{\chi}^{\mathrm{op}}$ converges absolutely for $\mathrm{Re}(\chi)$ small.  The factor $\mu_{P}(\chi)$ is meromorphic, and the operator $\mathcal{R}_{P|Q}:I(\chi) \to\overline{I}(\chi),$ originally defined for $\mathrm{Re}(\chi)$ large, extends to an operator sending meromorphic sections to meromorphic sections.  The definition of $\mathcal{S}(X_P(F))$ is designed to control the poles of all of these objects in terms of the functions $a_{P|Q}(\chi)$.    

\begin{proof}
 By Corollary \ref{cor:L2}, $\mathcal{S}(X_Q(F)) <L^2(X_Q(F))$ for $Q\in \{P,P^{\mathrm{op}}\}$. Combining this with \eqref{isom} and \eqref{relations}, we see that to prove $\mathcal{F}_{P|P^{\mathrm{op}}}$ is an isomorphism, it suffices to check that 
\begin{align} \label{well:def}
\calf_{P | P^{\mathrm{op}}}(\mathcal{S}(X_P(F))) \leq \mathcal{S}(X_{P^{\mathrm{op}}}(F)).
\end{align}
On the other hand, 
Corollary \ref{cor:SL:commute} and Lemma \ref{lem:unnorm:commute} imply that $\calf_{P | P^{\mathrm{op}}}(\mathcal{S}(X_P(F))) \leq \mathcal{S}_{\widetilde{L}}$ and that if we replace $\mathcal{S}(X_{P^{\mathrm{op}}}(F))$ by $\mathcal{S}_{\widetilde{L}}$ in \eqref{commutes} we obtain a commutative diagram.  Thus proving \eqref{well:def} implies everything in the theorem besides the continuity assertion.

Since $a_{P|P}(\chi)=a_{\widetilde{L}}(\chi^{-1})$ for all quasi-characters $\chi,$ by \eqref{a:rel} we deduce that for $f\in \mathcal{S}(X_P(F))$
$$
\frac{\mathcal{F}_{P|P^{\mathrm{op}}}(f)_{\chi_s,P^{\mathrm{op}}}(x)}{a_{P^{\mathrm{op}}|P^{\mathrm{op}}}(\chi_s)}=\frac{\mathcal{F}_{P|P^{\mathrm{op}}}(f)^{\mathrm{op}}_{(\chi_s)^{-1},P^{\mathrm{op}}}(x)}{a_{\widetilde{L}}((\chi_s)^{-1})}\in \cc[q^{s},q^{-s}]
$$
in the nonarchimedean case, and 
$$
|\mathcal{F}_{P|P^{\mathrm{op}}}(f)|_{A,B,p_{P^{\mathrm{op}}|P^{\mathrm{op}}},\Omega,D}<\infty
$$
for all data as in Definition \ref{defn:arch} in the archimedean case
since $\mathcal{F}_{P|P^{\mathrm{op}}}(f) \in \mathcal{S}_{\widetilde{L}}$.  
Hence we are left with checking that
\begin{align} \label{na:check}
\frac{\mathcal{R}_{P^{\mathrm{op}}|P}(\mathcal{F}_{P|P^{\mathrm{op}}}(f)_{\chi_s,P^{\mathrm{op}}})}{a_{P^{\mathrm{op}}|P}(\chi_s)} \in \cc[q^s,q^{-s}]
\end{align}
in the nonarchimedean case
and 
\begin{align} \label{arch:check}
|\mathcal{F}_{P|P^{\mathrm{op}}}(f)|_{A,B,p_{P^{\mathrm{op}}|P},\Omega,D}<\infty
\end{align}
in the archimedean case for all data as in Definition \ref{defn:arch}. 

For any $f \in \mathcal{C}_P$ and any $\chi \in \widehat{K}_{\GG_m}$, by  Corollary \ref{cor:SL:commute}, Lemma \ref{lem:unnorm:commute}, \eqref{relations}, and our choice of $\mathcal{C}_P$, we have the identities
\begin{align}\label{PPop}
  \nonumber \frac{\mu_{L}((\chi_s)^{-1})\mathcal{R}_{P^{\mathrm{op}}|P}(\mathcal{R}_{P|P^{\mathrm{op}}}(f_{(\chi_s)^{-1},P}))}{a_{P^{\mathrm{op}}|P}(\chi_s)}
     &=\frac{\mathcal{R}_{P^{\mathrm{op}}|P}(\mathcal{F}_{P|P^{\mathrm{op}}}(f)_{\chi_s,P^{\mathrm{op}}})}{a_{P^{\mathrm{op}}|P}(\chi_s)}\\
     \nonumber&=\frac{\left(\mathcal{R}_{P^{\mathrm{op}}|P}\mathcal{F}_{P|P^{\mathrm{op}}}(f)\right)^{\mathrm{op}}_{\chi_s,P}}{a_{P^{\mathrm{op}}|P}(\chi_s)}\\
     \nonumber&=\frac{\left(\mathcal{F}_{P^{\mathrm{op}}|P}\mathcal{F}_{P|P^{\mathrm{op}}}(f)\right)^{\mathrm{op}}_{\chi_s,P}}{\mu_{L^{\mathrm{op}}}(\chi_s)a_{P^{\mathrm{op}}|P}(\chi_s)}\\
   &=\frac{f_{\chi_s,P}^{\mathrm{op}}}{\mu_{L^{\mathrm{op}}}(\chi_s)a_{L^{\mathrm{op}}}(\chi_s)}.
\end{align}
Since $f_{(\chi_s)^{-1},P}=f_{\chi_s,P}^{\mathrm{op}}$, the first and last quantities in \eqref{PPop} depend only on the image of $f$ under the map to $I((\chi_s)^{-1}).$  Let $\mathcal{S}(X_P(F),K)$ be the space of $K$-finite functions in $\mathcal{S}(X_P(F));$ it is all of $\mathcal{S}(X_P(F))$ when $F$ is nonarchimedean.  By Lemma \ref{lem:defBK}, the equality of the first and last terms in \eqref{PPop} holds for all $f \in \mathcal{S}(X_P(F),K)$ and all $\chi \in \widehat{K}_{\GG_m}$ for $s$ in a dense subset of $\cc.$
  
  Since the first equality in the previous calculation is valid for all $f \in \mathcal{S}(X_P(F),K)$ by Corollary \ref{cor:SL:commute}, we deduce that 
\begin{align} \label{as:mero}
\frac{\mathcal{R}_{P^{\mathrm{op}}|P}(\mathcal{F}_{P|P^{\mathrm{op}}}(f)_{\chi_s,P^{\mathrm{op}}})}{a_{P^{\mathrm{op}}|P}(\chi_s)}=\frac{f_{\chi_s,P}^{\mathrm{op}}}{\mu_{L^{\mathrm{op}}}(\chi_s)a_{L^{\mathrm{op}}}(\chi_s)}
\end{align}
for all $f \in \mathcal{S}(X_P(F),K)$ and $\chi \in \widehat{K}_{\GG_m}$, at least for all $s$ in a dense subset of $\cc.$
 But then \eqref{as:mero} is valid as an identity of meromorphic functions for all $s.$ As discussed in \S \ref{section: switching}, we have $\mu_{L^{\mathrm{op}}}(\chi_s)a_{L^{\mathrm{op}}}(\chi_s)=\mu_{L}(\chi_s)a_L(\chi_s)$.  Moreover, with $L_i$ defined as in \eqref{L:def}
\begin{align*}
\mu_{L_i}(\chi_s)a_{L_i}(\chi_s)=\gamma(-s_i,(\chi_s)^{\lambda_i},\psi)L(-s_i,(\chi_s)^{\lambda_i})
&=\varepsilon(-s_i,(\chi_s)^{\lambda_i},\psi)L(1+s_i,(\chi_s)^{-\lambda_i})\\
&=\varepsilon(-s_i,(\chi_s)^{\lambda_i},\psi)a_{\widetilde{L}_i}(\chi_s)
\end{align*}
Here $\varepsilon(-s_i,(\chi_s)^{\lambda_i},\psi)$ denotes the usual Tate $\varepsilon$-factor.
Therefore,
\begin{align} \label{for:check}
 g(s,\chi,\psi) \frac{\mathcal{R}_{P^{\mathrm{op}}|P}(\mathcal{F}_{P|P^{\mathrm{op}}}(f)_{\chi_s,P^{\mathrm{op}}})}{a_{P^{\mathrm{op}}|P}(\chi_s)}=  \frac{f_{
    (\chi_s)^{-1},P}}{a_{\widetilde{L}}(\chi_s)}= \frac{f_{
    (\chi_s)^{-1},P}}{a_{P|P}((\chi_s)^{-1})}
\end{align}
where $g(s,\chi,\psi)=\prod_{i}\varepsilon(-s_i,(\chi_s)^{\lambda_i},\psi).$

  In the remainder of the proof we use some basic facts on $\varepsilon$-factors that are nicely summarized in \cite[\S 3.2]{Tate}.
Assume $F$ is nonarchimedean.  In this case $g(s,\chi,\psi)$ is equal to $cq^{p(s)}$ for some polynomial $p$ and some $c \in \cc^\times$.  Thus 
\eqref{for:check} and the fact that $f \in \mathcal{S}(X_P(F),K)$ implies \eqref{na:check}.  Now assume that $F$ is archimedean.  Then 
$$
g(s,\chi,\psi)=\prod_{i}\epsilon_{i,\chi} r \frac{(\chi_s)^{\lambda_i}(a)}{|a|^{1+s_i}}
$$
where $a$ and $r$ depend only on $\psi$ (which determines the normalization of the Haar measure) and $\epsilon_{i,\chi}$ is a fourth root of unity. By an analogue of  \cite[Lemma 3.6]{Getz:Hsu}, \eqref{for:check} and the fact that $f \in \mathcal{S}(X_P(F),K)$ implies \eqref{arch:check}, at least in the special case where $D$ is the identity operator.  It also follows for general $D$ once we note that $\mathcal{R}_{P|P^{\mathrm{op}}} \circ R(m,g)=\delta_{P^{\mathrm{op}}}(m)R(m,g) \circ \mathcal{R}_{P|P^{\mathrm{op}}}.$  Here we have used $R$ to denote the right action of $M(F) \times G(F)$ on $C^\infty(X_P^\circ(F))$ and $C^\infty(X_{P^{\mathrm{op}}}^{\circ}(F)).$

To deduce \eqref{arch:check} without the condition of $K$-finiteness, we point out that the same argument proving \cite[Proposition 3.7]{Getz:Hsu} implies that 
$$
\mathcal{F}_{P|P^{\mathrm{op}}}:\mathcal{S}(X_P(F),K) \lto \mathcal{S}(X_{P^{\mathrm{op}}}(F),K)
$$
is continuous in the Fr\'echet topology.  Since $\mathcal{S}(X_P(F),K)$ is dense in $\mathcal{S}(X_P(F))$ \cite[\S §4.4.3.1]{Warner} it extends to a topological isomorphism 
$$
\mathcal{F}_{P|P^{\mathrm{op}}}:\mathcal{S}(X_P(F)) \lto \mathcal{S}(X_{P^{\mathrm{op}}}(F)).
$$
This already implies the first assertion of the theorem, and additionally \eqref{arch:check}.
\end{proof}

As usual, we say that a parabolic subgroup of a reductive group is self-associate if it is conjugate to its opposite.  Assume $P$ is self-associate.  Choose 
\begin{align} \label{w00}
    w_0 \in G(F)
\end{align} 
normalizing $M$ such that $w_0^{-1}Pw_0=P^{\mathrm{op}}.$  Then conjugation by $w_0$ acts as inversion on $M^{\mathrm{ab}}$.  
\begin{Lem} \label{lem:Weyl}
Let $w$ be a representative for the long Weyl element of the Weyl group of $T$ in $G.$  Then one has $w_0 \in M(F)w=wM(F).$
\end{Lem}
\begin{proof}
The normalizer of $P$ in $G$ is $P$ and the normalizer of $M$ in $P$ is $M$. Therefore, $w_0\in P(F)w$ as $w\left(w_0^{-1}Pw_0\right)w^{-1}=P$. As $w$ normalizes $M$, for $w_0$ to normalize $M$, one must have $w_0=mnw$ for some $n \in N(F)$ such that $n^{-1}Mn=M$. This is only possible when $n$ is the identity. 
\end{proof}
\noindent We observe that $w_0Pw_0^{-1}=P^{\mathrm{op}}.$   Indeed, by Lemma \ref{lem:Weyl} it suffices to check this in the special case where $w_0$ is the long Weyl element of $T(F),$ in which case $w_0$ and $w_0^{-1}$ differ by an element in $T(F).$

   In \S \ref{ssec:measures:re} we gave $N_{P}(F)$ and $N
_{P^{\mathrm{op}}}(F)$ measures using a Chevalley basis.  We assume that $w_0$ is chosen so that these measures correspond under 
\begin{align} \label{meas:corr} \begin{split}
    N_{P}(F) &\lto N_{P^{\mathrm{op}}}(F)\\
    n &\longmapsto w_0^{-1}nw_0. \end{split}
\end{align}

\begin{Lem}\label{Lem: switch sides}
One has an isomorphism
\begin{align*}
    \iota_{w_0}:\mathcal{S}(X_{P^{\mathrm{op}}}(F)) &\lto \mathcal{S}(X_P(F))\\
    f &\longmapsto (x \mapsto f(w_0^{-1}x)).
\end{align*}
\end{Lem}
\begin{proof}
We have an isomorphism 
\begin{align*}
    \iota_{w_0}:C^\infty(X_{P^{\mathrm{op}}}^\circ(F)) &\lto C^\infty(X_P^{\circ}(F))\\
    f &\longmapsto (x \mapsto f(w_0^{-1}x)).
\end{align*}
For $f \in \mathcal{S}(X_{P^{\mathrm{op}}}(F))$ and $Q \in \{P,P^{\mathrm{op}}\},$ one has 
\begin{align*}
    \frac{\mathcal{R}_{P|Q}(\iota_{w_0}(f)_{\chi_s,P})}{a_{P|Q}(\chi_s)}=\frac{\mathcal{R}_{P|Q}(\iota_{w_0}(f_{\chi_s,P^{\mathrm{op}}}))}{a_{P|Q}(\chi_s)}=\frac{\iota_{w_{0}} \circ \mathcal{R}_{P^{\mathrm{op}}|Q^{\mathrm{op}}}(f_{\chi_s,P^{\mathrm{op}}})}{a_{P^{\mathrm{op}}|Q^{\mathrm{op}}}(\chi_s)},
\end{align*}
where we have used \eqref{a:rel}. 
Assume $F$ is nonarchimedean.  Then since $f_{\chi_s,P^{\mathrm{op}}}$ is a good section, we deduce that $\iota_{w_0}(f)_{\chi_s,P}$ is a good section.  Thus the lemma follows from the definition of the Schwartz space.  A similar argument proves the lemma in the archimedean case. 
\end{proof}

Thus when $P$ is self-associate, we have an isomorphism
\begin{align}\label{FXP}
    \mathcal{F}_{X_P}:=\mathcal{F}_{X_P,w_0}:=\iota_{w_0} \circ \mathcal{F}_{P|P^{\mathrm{op}}}:\mathcal{S}(X_P(F)) \lto \mathcal{S}(X_P(F)).
\end{align}
By Theorem  \ref{thm:FT} and \eqref{agreement}, we see that the Fourier transform $\mathcal{F}_{X_P}$ agrees with the Fourier transform used in \cite{Getz:Liu:BK,Getz:Liu:Triple,Getz:Hsu} when $X_P$ is as in \S \ref{sec:lagrangian} and $w_0$ is chosen as in loc.~cit.

For use in \S \ref{Section: general formula}, we also consider how $\iota_{w_0}$ interacts with the operators $\lam_!(\mu_s)$. Suppose that $L$ and $L',$ etc., are as in the discussion prior to Proposition \ref{prop:Mellin}. Recall $L^{\mathrm{op}}$ and its associated data $\{(s_i,\lam_i)\}$ from (\ref{Lop}). Arguing as in the proof of  Lemma \ref{Lem: switch sides}, we have an isomorphism
\[
\iota_{w_0}:\cals_L(X_{P^{\mathrm{op}}}(F))\lra \cals_{L^{\mathrm{op}}}(X_P(F)).
\]
\begin{Lem}\label{Lem: switch and twist}
 We have a commutative diagram
\[
\begin{tikzcd}
\cals_L(X_{P^{\mathrm{op}}}(F)) \ar[rr,"{\lam_{k!}}(\mu_{s_k})"]\ar[d,"\iota_{w_0}"]&&\cals_{L'}(X_{P^{\mathrm{op}}}(F))\ar[d,"\iota_{w_0}"]\\
\cals_{L^{\mathrm{op}}}(X_P(F))\ar[rr,"\lam_{k!}(\mu_{s_k})"]&&\cals_{L'^{\mathrm{op}}}(X_P(F)).
\end{tikzcd}
\]
\end{Lem}

\noindent We caution the reader that the bottom row in the diagram is given by the same definition as  \eqref{lambda:reg}, but the roles of $P$ and $P^{\mathrm{op}}$ switched as we are applying the operator $\lambda_{k!}(\mu_{s_k})$ to functions on $X_P(F)$.

\begin{proof}
Let $f \in \mathcal{S}_L(X_{P^{\mathrm{op}}}(F))$.  Then by Proposition \ref{prop:Mellin}, $\lam_{k!}(\mu_{s_k})(f)\in \cals_{L'}(X_{P^{\mathrm{op}}}(F)),$ and for $A(L)<\mathrm{Re}(\chi)<B(L'),$ traversing the top of the diagram and applying a Mellin transform yields
\begin{align*}
    \left(\iota_{w_0}({\lam_{k!}}(\mu_{s_k})(f))\right)_{\chi,P}^{\mathrm{op}}&=\iota_{w_0}\left({\lam_{k!}}(\mu_{s_k})(f)^{\mathrm{op}}_{\chi,P^{\mathrm{op}}}\right)\\
    &=\ga(-s_k,\chi^{\lam_k},\psi)\iota_{w_0}\left(f^{\mathrm{op}}_{\chi,P^{\mathrm{op}}}\right)\\
    &=\ga(-s_k,\chi^{\lam_k},\psi)\left(\iota_{w_0}(f)\right)_{\chi,P}^{\mathrm{op}}.
\end{align*}
Noting that
\[
A(L^{\mathrm{op}}) = A(L) \quad \text{ and } \quad B(L^{\mathrm{op}}) = B(L),
\]
 we may apply Proposition \ref{prop:Mellin} again to see that this equals
\[
\left({\lam_{k!}}(\mu_{s_k})(\iota_{w_0}(f))\right)_{\chi,P}^{\mathrm{op}}.
\] 
This is the result of traversing the bottom of the diagram and applying a Mellin transform.  Thus applying Mellin inversion yields the lemma.
\end{proof}

\subsection{Containment of Schwartz spaces and relation between definitions} \label{ssec:def:rel}

One can construct many functions in  $\mathcal{S}(X_P(F))$ using Mellin inversion.  What is not as clear is the following conjecture:
\begin{Conj} \label{conj:contains}
The Schwartz space $\mathcal{S}(X_P^{\circ}(F))$ of  $X_{P}^{\circ}(F) \subset X_P(F)$ is contained in $\mathcal{S}(X_P(F))$.
\end{Conj}

\begin{Thm} \label{thm:conj}
When $G$ is a symplectic group and $P$ is the stabilizer of a maximal isotropic subspace, Conjecture \ref{conj:contains} is true.
\end{Thm}
\begin{proof}This was proved in the nonarchimedean case in \cite[Proposition 4.7]{Getz:Liu:BK} and in the archimedean case in \cite[Proposition 3.13]{Getz:Hsu}. 
\end{proof}

More generally, one of the authors recently proved the following \cite{Hsu:nonarch}:
\begin{Thm}
When $F$ is nonarchimedean and $G$ is not of type $E$ or $F$, Conjecture \ref{conj:contains} is true.  \qed
\end{Thm}

Conjecture \ref{conj:contains} implies that $\mathcal{S}_{BK}(X_P(F)) \leq \mathcal{S}(X_P(F)),$ where $\mathcal{S}_{BK}(X_P(F))$ is defined as in Definition \ref{defn:BK}.  It would be convenient if this containment is an equality.  We pose this as the following question:
\begin{Quest}
Is it true that $\mathcal{S}(X_P(F))=\mathcal{S}_{BK}(X_P(F))$?
\end{Quest}

We  discuss the relative benefits of the two definitions of the Schwartz space.  Braverman and Kazhdan's definition of $\mathcal{S}_{BK}(X_P(F))$ is beautifully succinct.   However, it is difficult to extract analytic information about elements of the Schwartz space from the definition.  The definition of $\mathcal{S}(X_P(F))$ is more involved, but it seems to be the correct definition.  For example, in the nonarchimedean case  we certainly want the image of $\mathcal{S}(X_P(F))$ under various Mellin transforms to consist exactly of good sections, and we have defined $\mathcal{S}(X_P(F))$ so that this is true.  Moreover, analytic information is relatively straightforward to extract from the definition of $\mathcal{S}(X_P(F))$.  In particular, the following is an immediate consequence of Lemma \ref{lem:Sch:bounds2}:
\begin{Thm}
Assuming Conjecture \ref{conj:contains}, \cite[Conjecture 5.6]{BKnormalized} is valid for maximal parabolic subgroups of simple reductive groups. That is, when $F$ is nonarchimedean the support of any $f\in \cals_{BK}(X_{P}(F))$ is contained in a compact subset of $X_P(F)$.\qed 
\end{Thm}

In the special case where $X_P$ is as in \S \ref{sec:lagrangian}, this was already proven in \cite{Getz:Liu:BK} up to identifying the Fourier transform of \cite{Getz:Liu:BK} with the Fourier transform of Braverman and Kazhdan.  The agreement of the two Fourier transforms is implied by Theorem \ref{thm:FT} above.

\section{A formula for the Fourier transform on $X_P$} \label{sec:FT:XP}
In this section, we combine our analytic results with the geometric pairing between opposite Braverman--Kazhdan spaces to give a formula for the Fourier transform. We then work out several examples explicitly, connecting this result to known formulae in the literature.  Our aim is to be explicit enough that the formula can be applied by readers who are not experts in algebraic group theory.  
\subsection{Preliminary calculations}\label{Section: roots prelim}
We continue to impose the notation from the previous sections; thus $P \leq G$ is a maximal parabolic subgroup in a simple, simply connected, split group $G$.  Recall that $\omega_\beta$ is the fundamental weight attached to $P$ as in \eqref{omegaP}.  Since $G$ is simply connected $\omega_P=\omega_{\beta}$ in the notation of \eqref{omegaP}. As above $V_P$ is the associated highest weight representation. 
By Lemma \ref{Lem: general Plucker}, $\omega_P$ may be extended to a character of $P$ (trivial on $P^{\mathrm{der}}$) and defines an isomorphism
\begin{align*}
    \omega_P=\omega_\beta: M^{\mathrm{ab}}\iso \GG_m.
\end{align*}

Recall the graded representation $L=\widehat{\fn}_P^e$ of \S \ref{ssec:BKL}.  We fix a good ordering 
\[
\{(s_i,\lam_i):1 \leq i \leq k\},
\]
so
\[
\frac{s_{i+1}}{\lam_{i+1}}\geq \frac{s_i}{\lam_i}
\]
for $1 \leq i <k$ and $k=\dim L.$ In particular, we have the highest datum $(s_k,\lam_k)$. 
\begin{Prop}\label{Prop: first term}
Any good ordering of $L=\widehat{\fn}_P^e$ satisfies 
$
\lambda_k=1.
$
\end{Prop}
\begin{proof} 
Our proof is a case-by-case analysis. As this is a computation on the Langlands dual group, we defer the details to \S \ref{App:normalizing}. In fact we compute all of the parameters  $\{(s_i,\lambda_i)\}$ for all simple Cartan types.  
 The results required to observe the current proposition are lemmas \ref{lem:PGL:case} and \ref{lem:classic:case} and the tables at the end of \S \ref{App:normalizing}.  We alert the reader that we work entirely on the Langlands dual side in the appendix. One must use the following well-known computations of Langlands dual groups:
\begin{align*}
\widehat{\mathrm{Sp}}_{2n}=\mathrm{SO}_{2n+1}(\cc), \quad  \widehat{\mathrm{Spin}}_{2n} =\mathrm{PSO}_{2n}(\cc), \quad \widehat{\mathrm{Spin}}_{2n+1}=\mathrm{PSp}_{2n}(\cc)
\end{align*}
together with the fact that the dual group of a simply connected semisimple group is adjoint.
\end{proof}

\begin{Prop}\label{Prop: root explanation}
One has
\[
 \de_P=|\omega_P|^{2s_k+2}.
 \]
\end{Prop}
\begin{Rem}
The proof shows that the proposition is still valid if we weaken the assumption that $G$ is simply connected to the assumption that $\omega_{\beta} \in X^*(T)$.
\end{Rem}
\begin{proof}
Let $\Phi^+$ denote the set of positive roots of our split maximal torus $T\leq M <G$ with respect to the Borel subgroup $B$ and let $\Phi^+_M \subset \Phi^+$ denote the subset of roots of $T$ in $M.$    As above, $\Delta \subset \Phi^+$ denotes the set of simple roots defined by $\Phi^+$.  Then $\De_M=\De-\{\be\}$ is a set of simple roots of $T$ in $M$.

For $t \in T(F),$ we have
$$
\delta_P(t)=\left|t^{\sum_{\ga  \in \Phi^+ -\Phi^+_M}\gamma} \right|.
$$
On the other hand $X^\ast(M)=\zz\omega_P,$ so there is an integer $r>0$ such that 
$$
\sum_{\ga \in \Phi^+-\Phi^+_M} \gamma=r\omega_P.
$$
We are to show that $r=2s_k+2$.

Let $\{e,h,f\} \subset \widehat{\fm}$ be a principal $\fsl_2$-triple.  The copy of $\fsl_2$ it spans acts on $\widehat{\fn}_P$ by the adjoint action. The root systems of $M$ and $\widehat{M}$ are in Langlands duality. 
We use this to identify
 \begin{align} \label{identification}
 \widehat{\mathfrak{t}}=X_\ast(\widehat{T})\otimes_\zz \cc = X^\ast(T)\otimes_\zz \cc.
 \end{align}
Under this identification, $h \in \widehat{\ft}$ may be chosen so that it is sent to the sum of positive coroots  of $\widehat{M}$ \cite[Section 2]{gross1997motive}:
\begin{align*} 
   2 \rho^\vee_{M}:=\sum_{\gamma \in \Phi^+_{\widehat{M}}}{\gamma^\vee} \in X_*(\widehat{T}),
\end{align*}
which corresponds under the second equality of \eqref{identification} to
\begin{align} \label{2rho}
2\rho_M:=\sum_{\gamma \in \Phi^+_M} \gamma= 2\sum_{\al\in{\De_M}} \widetilde{\omega}_\al \in X^*(T)
\end{align}
 where $\widetilde{\omega}_\al$ is the weight of the fundamental representation of $M$ associated to $\al\in \Delta_M$. Thus
\begin{align} \label{eqn: relation with h}
h+r\omega_P=\sum_{\gamma \in \Phi^+}\gamma =2\sum_{\alpha \in \Delta}\omega_{\alpha},
\end{align}
where $\omega_\alpha$ is the fundamental weight of $G$ associated to $\alpha \in \Delta$. Note that in general $\widetilde{\omega}_\al\neq {\omega}_\al$ for $\al\in \De_M$ since $\widetilde{\omega}_\al\in X^\ast(T\cap M^{\mathrm{der}}).$

Consider now the $h$-eigenvalues on the space of highest weight vectors $L=\widehat{\fn}_P^e$. By Proposition \ref{Prop: first term},
$$
L_k=\cc v_k \leq  \widehat{\fn}_P(1)^e
$$
where the $1$ indicates the subspace on which $Z(\widehat{M})$ acts via $1$. As mentioned in \cite[\S 5.2]{Manivel}, the space $\widehat{\fn}_P(1)$ is the irreducible representation of $\widehat{M}$ with lowest weight space corresponding to $\be^\vee,$ the coroot of $\beta$.  

By the definition of a good ordering, the $h$-eigenvalue $2s_k$ is largest among all $h$-eigenvalues occurring in the $\widehat{M}$ representation $\widehat{\fn}_{P}(1)$. It follows that $v_k$ is a highest weight vector for $\widehat{\fn}_P(1)$. Let 
$$
\ga^\vee_0=\beta^\vee+ \sum_{\al \in \Delta_M}c_\al(\ga_0^\vee) \al^\vee
$$
be the weight of $v_k$. We claim that $2s_k = \sum_{\al \in \Delta_M}c_\al(\ga_0^\vee)$.

 Since this is the largest $h$-eigenvalue in $L,$ it follows that the lowest weight space $\widehat{\fn}_{\be^\vee}<\widehat{\fn}_P(1)$ is the lowest-weight space for the irreducible $\fsl_2$-representation containing $v_k,$ and thus has the eigenvalue
\begin{equation}\label{eqn: lowest weight}
\la h,\be^\vee\ra = -2s_k.
\end{equation}
Here 
 $\la\cdot,\cdot\ra$ is the pairing on $X^\ast(T)\otimes X_\ast(T)$.
Therefore, since \eqref{2rho} implies $\la h, \al^\vee\ra = 2$ for all  $\al\in \De_M,$ 
\begin{align*}
    2s_k=\la h,\gamma^\vee_0\ra &= \sum_{\al\in \De_M}c_\al(\ga_0^\vee)\la h, \al^\vee\ra + \la h, \be^\vee\ra\\
            &=2\sum_{\al\in \De_M}c_\al(\ga_0^\vee) -2s_k,
\end{align*}
proving the claim that $2s_k = \sum_{\al \in \Delta_M}c_\al(\ga_0^\vee)$.
 
Since $\omega_P=\omega_\be,$ we see that for any root $\ga^\vee$ occurring in $\widehat{\fn}_P(1),$ $\la\omega_P,\ga^\vee \ra =1$. Evaluating both sides of \eqref{eqn: relation with h} on $\ga_0^\vee$ thus implies 
 \begin{align*}
 2s_k +r= \left\langle 2 \sum_{\alpha \in \Delta} \omega_\alpha,\gamma_0^\vee \right\rangle = 2+2\sum_{\alpha \in \De_M}c_\al(\ga_0^\vee)=2+4s_k.
 \end{align*}We deduce that $r=2s_k+2,$ and the proposition follows.
\end{proof}

\subsection{The general formula}\label{Section: general formula}

For integers $n$, let 
\begin{align} \label{[n]}
[n]:\GG_m \lto \GG_m
\end{align}
be the map $x \mapsto x^n$.  We define
\begin{align}\label{eqn: eta aug}
    \mu_{P}^{\mathrm{aug}}:=\lambda_{1!}(\mu_{s_1}) \circ \dots \circ \lambda_{(k-1)!}(\mu_{s_{k-1}}) \quad\text{and}\quad\mu_P^{\mathrm{geo}}:=[1]_{!}(\mu_{s_k}),
\end{align}
where the $\mathrm{aug}$ stands for ``augmented''
and consider the factorization
\begin{align*}
    \mu_{P}=\mu_{P}^{\mathrm{aug}}\circ \mu_{P}^{\mathrm{geo}}.
\end{align*}
\begin{Rem}  In light of our formula for the Fourier transform below, it would be interesting to illuminate the relationship between the operator $\mu_{P}^{\mathrm{aug}}$ and the singularity of $X_P$ at $0$. 
\end{Rem}
 Set
 \begin{align}\label{geopart}
     \calf_{P|P^{\mathrm{op}}}^{\mathrm{geo}}:=\mu_{P}^{\mathrm{geo}}\circ\calr_{P|P^{\mathrm{op}}}:\cals(X_P(F))\lra \cals_{L(1)}.
 \end{align}
\begin{Thm} \label{Thm: Fourier formula}
For $f\in  \mathcal{S}(X_P(F))$ and $x^\ast\in X_{P^{\mathrm{op}}}^\circ(F),$ we have
$\mathcal{F}_{P|P^{\mathrm{op}}}=\mu_{P}^{\mathrm{aug}} \circ \mathcal{F}_{P|P^{\mathrm{op}}}^{\mathrm{geo}}$ where
\[
\calf_{P|P^{\mathrm{op}}}^{\mathrm{geo}}(f)(x^\ast)=
\int_{X_{P}^\circ(F)}f\left(x\right)\psi\left(\la x,x^\ast\ra_{P|P^{\mathrm{op}}}\right)dx.
\]
Here $\la \cdot,\cdot\ra_{P|P^{\mathrm{op}}}$ is as in \eqref{PPop:pairing} and the measure on $X_P^\circ(F)$ is normalized as in \S \ref{ssec:measures:re}.
\end{Thm}

\begin{proof}
  For $x^\ast\in X_{P^{\mathrm{op}}}^\circ(F),$ we have
\begin{align*}
    \calf_{P|P^{\mathrm{op}}}^{\mathrm{geo}}(f)(x^\ast) &=\frac{1}{\zeta(1)} \int_{M^{\mathrm{ab}}(F)}\psi(\omega_P(m))|\omega_P(m)|^{s_k+1}\de_{P^{\mathrm{op}}}^{1/2}(m)\calr_{P|P^{\mathrm{op}}}(f)(m^{-1} x^\ast)dm\\
                         &=\frac{1}{\zeta(1)} \int_{M^{\mathrm{ab}}(F)}\psi(\omega_P(m))\calr_{P|P^{\mathrm{op}}}(f)(m^{-1} x^\ast)dm.
\end{align*}
Here we have used Proposition \ref{Prop: root explanation}. We note that there is no need to regularize the outer integral: the absolute convergence of $[1]_!(\mu_{s_k})$ on $\mathcal{R}_{P|P^{\mathrm{op}}}(\cals(X_P(F)))$ follows from our use of a good ordering and from lemmas \ref{lem:conv} and \ref{lem:unnorm:commute}. If we write $x^\ast = P^{\mathrm{op},\mathrm{der}}(F)g,$
\begin{align} \label{rewrite}
    \calr_{P|P^{\mathrm{op}}}(f)(m^{-1}\cdot x^\ast)
    &= \displaystyle\int_{N_{P^{\mathrm{op}}}(F)}f\left(um^{-1}g\right)du\nonumber\\
&= \de_{P}(m)\displaystyle\int_{N_{P^{\mathrm{op}}}(F)}f\left(m^{-1}ug\right)du.
\end{align}
We have an injection
\begin{align*} \begin{split}
\Phi_{g}:M^{\mathrm{ab}}(F) \times N_{P^{\mathrm{op}}}(F) &\lto X^\circ_{P}(F)\\
(m,u) &\longmapsto P^{\mathrm{der}}(F)m^{-1}ug \end{split}
\end{align*}
with dense image denoted by $X^\circ_{P,g}(F)$. Moreover, we have
\begin{align} \label{meas}
d(m^{-1}ug)=\frac{\delta_{P^{\mathrm{op}}}(m^{-1})dmdu}{\zeta(1)}=\frac{\delta_P(m)dmdu}{\zeta(1)}
\end{align}
by \eqref{meas:comp general}.

For $x \in X^{\circ}_{P,g}(F),$ let
$$
(m(x),u(x)):=\Phi_{g}^{-1}(x).
$$
By \eqref{rewrite} and \eqref{meas}, we have 
\begin{align} \label{penult}
    \mathcal{F}_{P|P^{\mathrm{op}}}^{\mathrm{geo}}(f)(x^*)=\int_{X_{P,g}^{\circ}(F)}\psi(\omega_P(m(x)))f(x) dx.
\end{align}
Now for $(m,u) \in M^{\mathrm{ab}}(F) \times N_{P^{\mathrm{op}}}(F)$ and $g$ chosen as above (so $P^{\mathrm{op},\mathrm{der}}(F) g=x^*$), we have
$$
\langle  v_P m^{-1}ug ,v^*_{P^{\mathrm{op}}} g \rangle =\langle  v_P m^{-1},v^*_{P^\mathrm{op}}\rangle=\omega_P(m);
$$
here we have used \eqref{intertwine}.  Thus \eqref{penult} is 
\begin{align*}
    \mathcal{F}_{P|P^{\mathrm{op}}}^{\mathrm{geo}}(f)(x^*)=\int_{X_{P,g}^{\circ}(F)}\psi(\langle x ,x^*\rangle_{P|P^{\mathrm{op}}})f(x) dx=\int_{X_{P}^{\circ}(F)}\psi(\langle x ,x^*\rangle_{P|P^{\mathrm{op}}})f(x) dx
\end{align*}
since $X^{\circ}_{P,g}(F)$ is open and of full measure in $X^{\circ}_P(F)$.  
\end{proof}

Assume for the moment that $P$ is self-associate. In this special case, fix a $w_0\in G(F)$ normalizing $M$ such that $w_0^{-1}Pw_0=P^{\mathrm{op}}$ and such that \eqref{meas:corr} is measure preserving.  We saw in (\ref{FXP}) that this allows us to define a Fourier transform 
\begin{align}\label{FXP2}
        \mathcal{F}_{X_P}:=\mathcal{F}_{X_P,w_0}:=\iota_{w_0} \circ \mathcal{F}_{P|P^{\mathrm{op}}}:\mathcal{S}(X_P(F)) \lto \mathcal{S}(X_P(F)).
\end{align}
\begin{Cor} \label{Cor:Pop} Assume that $P= w_0P^{\mathrm{op}}w_0^{-1}$ is self-associate. Then for $f\in \cals(X_P(F)),$ we have $\mathcal{F}_{X_P}(f)=\mu_{P^{\mathrm{op}}}^{\mathrm{aug}}\circ\calf^{\mathrm{geo}}_{X_P}(f)$ where
\begin{align*}
    \calf^{\mathrm{geo}}_{X_P}(f)(x') = 
\int_{X_{P}^\circ(F)}f\left(x\right)\psi\left(\la x,w_0^{-1}x'\ra_{P|P^{\mathrm{op}}}\right)dx
\end{align*}
for $x'\in X_P^\circ(F).$ Here the measure on $X_P^\circ(F)$ is normalized as in \S \ref{ssec:measures:re}.
\end{Cor}
\begin{proof}
By the discussion in \S \ref{section: switching}, we  have
\[
\iota_{w_0}\circ\mu_{P}^{\mathrm{aug}}=\mu_{P^{\mathrm{op}}}^{\mathrm{aug}}\circ\iota_{w_0}
\]
and it is clear that
\[
\calf^{\mathrm{geo}}_{X_P}=\iota_{w_0}\circ\mathcal{F}^{\mathrm{geo}}_{P|P^{\mathrm{op}}}.\qedhere
\]
\end{proof}

\begin{Rem} \label{rem:abs:conv} 
By Corollary \ref{cor:L2}, $\mathcal{S}(X_P(F)) <L^1(X_P(F))$.  Thus the integrals in the definition of $\mathcal{F}^{\mathrm{geo}}_{P|P^{\mathrm{op}}}$ and $\mathcal{F}_{X_P}^{\mathrm{geo}}$ converge absolutely. 
\end{Rem}

\subsection{Examples}\label{examples}
In this subsection, we explicate the objects appearing in Theorem \ref{Thm: Fourier formula} in several cases of interest.  Only the example in \S \ref{Section: the FT} will be used later in the paper.  It is used in \S \ref{sec:Y} to study $\mathcal{F}_Y.$

\subsubsection{Line bundles over Grassmannians}
The maximal parabolic subgroups of $\SL_n$ are stabilizers of planes.  Concretely, fix $1 \leq \ell <n$ and let $P$ be the stabilizer of the $\ell $-plane $\{e_{n-\ell+1},\dots,e_n\}$.  Here we use the standard basis of $F^n,$ viewed as row vectors with a right action of $G$.  
Then $P\backslash G$ is a classical Grassmannian, and $X_P^\circ(F)$ can be viewed as the space of $\ell$-planes $W \subset F^n$ together with an associated non-zero vector in $\wedge^\ell W$.  

For $F$-algebras $R,$ we have
\[
P(R)=\left\{\left(
\begin{smallmatrix} 
m_1 & \\ & m_2 \end{smallmatrix} \right) \left( \begin{smallmatrix} I_{n-\ell} & w\\ & I_\ell \end{smallmatrix}
\right) \in \SL_{n}(R): (m_1,m_2,w) \in \GL_{n-\ell}(R) \times \GL_\ell(R) \times M_{n-\ell,\ell}(R)\right\}.
\] 
Then
\[
P^{\mathrm{op}}(R)=\left\{\left(\begin{smallmatrix}
    m_1 &  \\
     & m_2
\end{smallmatrix}\right)\left(\begin{smallmatrix}
   I_{n-\ell} &  \\
    w^t &  I_\ell
\end{smallmatrix}\right) \in \SL_n(R):(m_1,m_2,w) \in \GL_{n-\ell}(R) \times \GL_\ell(R) \times M_{n-\ell,\ell}(R)\right\}.
\]
 In this setting,
\begin{align*}
\omega_P:M(F) &\lto F^\times\\
\begin{psmatrix}  m_1&\\&m_2 \end{psmatrix} &\longmapsto \det(m_1)=\det(m_2)^{-1}.
\end{align*}
Our representation $V_P$ is just $\wedge^\ell \GG_a^n$.  We realize the dual as the space $\wedge^{n-\ell} \GG_a^n$ equipped with the pairing
\begin{align*}
    \wedge^\ell R^n \times \wedge^{n-\ell}R^n &\lto R\\
    (w_1,w_2) &\longmapsto e_1^\vee \wedge \dots \wedge e_n^\vee(w_2 \wedge w_1)
\end{align*}
We choose the highest weight vector $v_P:=e_{n-\ell+1}\wedge \dots \wedge e_n$ and dual lowest weight vector $v_{P^{\mathrm{op}}}^*:=e_{1} \wedge \dots \wedge e_{n-\ell}$.
With these choices,
\[
\mathrm{Pl}_{v_P}\left(\begin{smallmatrix}
    a & b \\
    c & d
\end{smallmatrix}\right)\longmapsto\wedge^{\ell}\left(\begin{matrix}
 c&d
\end{matrix}\right)
\]
where we are taking the (ordered) wedge product of the row vectors from top to bottom. Similarly,
\[
\mathrm{Pl}_{v^*_{P^{\mathrm{op}}}}\left(\begin{smallmatrix}
    a & b \\
    c & d
\end{smallmatrix}\right)\longmapsto\wedge^{n-\ell}\left(\begin{matrix}
    a&b
\end{matrix}\right)
\]
where the wedge product is taken from top to bottom.

\subsubsection{Orthogonal groups and the transform on the isotropic cone}  Assume the characteristic of $F$ is not $2$.  
Consider the split orthogonal group $\mathrm{SO}_{n}$ for $n>4,$ defined with respect to the matrix
\[
J_n=\left(\begin{smallmatrix}
     & & 1\\
     & \Ddots &\\
     1&&
\end{smallmatrix}\right).
\]
Denote the corresponding pairing by $\langle\cdot,\cdot\rangle,$ and let 
$$
Q_n(v):=\tfrac{1}{2}\langle v,v \rangle.
$$
Let $T$ be the split maximal torus of diagonal matrices and let $B$ be the Borel subgroup of upper triangular matrices of $\mathrm{SO}_{n}$. There is a natural right action of $\mathrm{SO}_{n}$ on $V_n=\GG_a^n$.  We let $P<\mathrm{SO}_{n}$ be the parabolic subgroup fixing the line spanned by $e_n=(0,\dots,0,1)$.  Then $V_P=V_n,$ and we choose the highest weight vector $v_P:=e_n.$

Consider the split spin group $G=\mathrm{Spin}_{n}$ over $\mathrm{SO}_n$ and let $p:G \to \mathrm{SO}_n$ be the double cover.  Then $\widetilde{P}:=p^{-1}(P)$ is a maximal parabolic subgroup of $G$. Moreover, the representation $V_{\widetilde{P}}$ of $G$ descends to the representation $V_n$ of $\mathrm{SO}_n$ via $p$.  It therefore follows from Lemma \ref{Lem: general Plucker} that $p$ induces an isomorphism
\begin{align*}
    p:X_{\widetilde{P}}^{\circ}=\widetilde{P}^{\mathrm{der}}\backslash G \tilde{\lto} P^{\mathrm{der}}\backslash \mathrm{SO}_n.
\end{align*}
Let $\widetilde{M}$ be a Levi subgroup of $\widetilde{P}$ and $M:=p(\widetilde{M})$.  Since $V_{\widetilde{P}}$ descends to $V_n,$  it also follows from Lemma \ref{Lem: general Plucker} that the map $\widetilde{M}^{\mathrm{ab}} \to M^{\mathrm{ab}}$ induced by $p$ is an isomorphism and the diagram
\begin{equation} \label{act:commute}
\begin{tikzcd}
\widetilde{M} \times X_{\widetilde{P}}^\circ \arrow[r] \arrow[d,"p"] &X_{\widetilde{P}}^{\circ} \arrow[d,"p"] \\
M \times X_{P}^\circ \arrow[r] &X_P^{\circ}
\end{tikzcd}
\end{equation}
commutes.  Here the horizontal arrows are the action maps.  Thus we can and do work with $P^{\mathrm{der}} \backslash G$ in place of $\widetilde{P}^{\mathrm{der}} \backslash \widetilde{G}$ below.

The Pl\"ucker embedding 
$$
\mathrm{Pl}_{e_n}:X_P \lto V_n
$$
maps $X_P$ isomorphically onto the affine scheme whose points in an $F$-algebra $R$ are
$$
C(R):=\{v \in V_n(R):Q_n(v)=0\}.
$$
This is the isotropic cone of $Q_n$. 

We define the Schwartz space $\mathcal{S}(C(F))$ to be 
$$
(\mathrm{Pl}_{e_n}^{-1})^*(\mathcal{S}(X_P(F)))<C^\infty(C(F)-\{0\}).
$$
The parabolic $P$ is self-associate.  Thus the Schwartz space comes equipped with a Fourier transform 
$$
\mathcal{F}_C:=(\mathrm{Pl}_{e_n}^{-1})^* \circ \mathcal{F}_{X_P,w_0} \circ\mathrm{Pl}_{e_n}^* :\mathcal{S}(C(F)) \lto \mathcal{S}(C(F)).
$$
Here $w_0$ is chosen as in Lemma \ref{lem:agreement} below.
There is a natural measure on $C(F)$ as we  now explain.  Let $dv_i$ be the standard $1$-form on $\GG_a,$ viewed as the $i$th coordinate of $V_n=\GG_a^n.$
 Recall \cite[\S III.1.2]{Gelfand:Shilov:I} that  to give a measure on $C(F),$ we may choose any $(n-1)$-form $\omega_{Q_n}$ such that
\begin{equation}\label{eqn: measure on cone}
 dv_1\wedge\cdots\wedge dv_{n}=dQ_n\wedge \omega_{Q_n}
\end{equation}
and then consider the measure $|\omega_{Q_n}|$.  If we write
\[
Q_n(v_1,\ldots,v_{n})=\begin{cases}\frac{1}{2}v^2_{r+1}+\sum_{i=1}^{r}v_iv_{n+1-i}& \textrm{ if } n=2r+1,\\\sum_{i=1}^{r}v_iv_{n+1-i}& \textrm{ if } n=2r,
\end{cases}
\]
with respect to the standard basis of $F^n,$ then on $\GG_a^{n-1} \times \GG_m$  we choose $\omega_{Q_n} =\frac{1}{v_{n}}dv_2\wedge \cdots \wedge dv_{n}$.

\begin{Lem} \label{lem:agreement} We can choose $w_0 \in \mathrm{SO}_{n}(F)$ normalizing $M$ such that
$w_0^{-1}Pw_0=P^{\mathrm{op}}$ and such that
for $x,x' \in X_P^\circ(F)$ one has
$$
\langle x,w_0^{-1} x' \rangle_{P|P^{\mathrm{op}}}=\langle \mathrm{Pl}_{e_n}(x),\mathrm{Pl}_{e_n}( x') \rangle.
$$
Moreover,  $\mathrm{Pl}_{e_n}^*(|\omega_{Q_n}|)=cdx$ for some $c \in \rr_{>0}.$ 
\end{Lem}
\begin{proof}
We identify the dual of $V_n$ with $V_n$ itself via the form $\langle\cdot,\cdot\rangle$.  Then the vector dual to $e_n$ is $e_1$.
Let
\begin{align*}
    w_0:=\begin{cases} J_n & \textrm{ if }n \equiv 0 \pmod 4 \textrm{ or }n \equiv 1 \pmod{4},\\
    \begin{psmatrix} & &J_{(n-1)/2}\\&-1 & \\ J_{(n-1)/2} & & \end{psmatrix} & \textrm{ if }n \equiv 3 \pmod{4},\\
    \begin{psmatrix} & &J_{(n-2)/2}\\& I_2 & \\ J_{(n-2)/2} & & \end{psmatrix} &\textrm{ if }n \equiv 2 \pmod{4}.\end{cases}
\end{align*}
  Thus $w_0 \in \mathrm{SO}_n(F)$,  $\mathrm{Pl}_{e_n}(g)=e_ng$ and $\mathrm{Pl}_{e_1}(w_0^{-1}g)=e_ng$. This yields the first assertion.  
For the second assertion,  it suffices to observe that \eqref{eqn: measure on cone} implies that $\omega_{Q_n}$ is $\mathrm{SO}_{n}$-invariant.
\end{proof}

\begin{Cor} \label{Cor:quad}
If the measure $|\omega_{Q_{n}}(v)|$ is normalized so that 
$\mathrm{Pl}_{e_n}^*(|\omega_{Q_n}|)=dx,$ then for $f \in \mathcal{S}(C(F))$ one has
\begin{align}
    \mathcal{F}_C(f)(v')=\int_{F^\times}\psi(t^{-1})|t|^{(n-4)/2}\left(\int_{C(F)-\{0\}} f(v)\psi\left(\langle v,tv' \rangle \right)|\omega_{Q_n}(v)|\right) \frac{d^\times t}{\zeta(1)}
\end{align}
if $n>4$ is even and
\begin{align}
    \mathcal{F}_C(f)(v')=\int_{F^\times}\psi(t^{-1})|t|^{n-3}\left(\int_{C(F)-\{0\}} f(v)\psi\left(\langle v,t^2 v' \rangle \right)|\omega_{Q_n}(v)|\right)\frac{d^\times t}{\zeta(1)}
\end{align}
if $n>3$ is odd.
\end{Cor}

\begin{proof}
This is a consequence of Corollary \ref{Cor:Pop}
and Lemma \ref{lem:agreement} as we now explain.  Using \eqref{act:commute}, we are free to work with the action of $M^{\mathrm{ab}}$ instead of $\widetilde{M}^{\mathrm{ab}}$ in applying Corollary \ref{Cor:Pop}.  

For $(t,g) \in R^\times \times \mathrm{SO}_{n-2}(R)$ write
$$
m(t,g):=\begin{psmatrix} t & & \\ & g & \\ & &t^{-1} \end{psmatrix}: t \in R^\times, g \in \mathrm{SO}_{n-2}(R)
$$
The character $\omega_P$ is given by 
$\omega_P(m(t,g))=t$. 
Note that for $x,x' \in X_P^\circ(F)$ and $\lambda\in\zz,$ 
$$
\langle x,w_0^{-1}m(t,g)^{-\lam} x' \rangle_{P|P^{\mathrm{op}}}=\omega_P(m(t,g))^\lam\langle x,w_0^{-1} x' \rangle_{P|P^{\mathrm{op}}}=t^\lam\langle x,w_0^{-1} x' \rangle_{P|P^{\mathrm{op}}}.
$$
Applying Lemma \ref{lem:agreement} now shows that if $v=\mathrm{Pl}_{e_n}(x)$ and $v'=\mathrm{Pl}_{e_n}(x'),$ then
\[
\langle x,w_0^{-1} m(t,g)^{-\lam}x' \rangle_{P|P^{\mathrm{op}}}=\langle v,t^\lam v' \rangle.
\]
By Lemma \ref{lem:classic:case}, we have $\mu_{\mathrm{P^{\mathrm{op}}}}^{\mathrm{geo}}=[1]_!(\mu_{\frac{n-4}{2}})$ for all $n,$ and  $$\mu_{P^{\mathrm{op}}}^{\mathrm{aug}}=\begin{cases}[1]_!(\mu_{0})& \textrm{ if } \text{$n>4$ is even},\\ [2]_!(\mu_{0})& \textrm{ if } \text{$n>3$ is odd}.\end{cases}$$ 
The regularized operators are equal to the unregularized operators by Lemma \ref{lem:conv}. 
\end{proof}

When $F$ is nonarchimedean with odd or zero characteristic, Corollary \ref{Cor:quad} implies that when $n$ is even $\mathcal{F}_C$ agrees with the operator $\Pi_K(s_1)$ of \cite[Theorem 3.4]{GurK} (after replacing $\psi$ by $\overline{\psi}$).  Gurevich and Kazhdan also treat nonsplit isotropic quadratic forms.  
When $F=\rr$ and $n$ is even, a Fourier transform on $L^2(C(F),|\omega_Q|)$ was investigated in \cite{Kobayashi:Mano} (they also treated arbitrary isotropic quadratic forms in an even number of variables).  It likely agrees with $\mathcal{F}_C$ when the form is split but we will not verify this.

\subsubsection{The Lagrangian Grassmannian}\label{Section: the FT}
Define $\mathrm{Sp}_{2n}$ and $P$ as in 
 \S \ref{sec:lagrangian}.  We let $\mathrm{Sp}_{2n}$ act on $V=\GG_a^{2n}$ on the right.  
The representation $V_P$ may be realized as an irreducible subrepresentation of $\wedge^nV,$ and we choose the highest weight vector to be $v_P:=e_{n+1} \wedge\dots \wedge e_{2n}$. 
Thus
\begin{align}
    \mathrm{Pl}_{v_P}\begin{psmatrix}*\\ a_{n+1}\\ \vdots \\ a_{2n} \end{psmatrix}=a_{n+1}\wedge \dots \wedge a_{2n}
\end{align}
is the (ordered) wedge product of the last $n$ rows. 

There is a perfect pairing
\begin{align}\label{Sppairing}
\langle\cdot,\cdot\rangle:\wedge^n \GG_a^{2n} \times \wedge^n \GG_a^{2n} \lto \wedge^{2n}\GG_a^{2n} \tilde{\lto} \GG_a,
\end{align}
where the first map is canonical and the second is obtained by specifying that $e_1\wedge \dots \wedge e_{2n}$ is sent to $1$.  We use this pairing to identify the dual of $V_P$ with $V_P$.  Thus
$$
\langle x,x^* \rangle_{P|P^{\mathrm{op}}}=\langle \mathrm{Pl}_{v_P}(x),\mathrm{Pl}_{v^*_{P^{\mathrm{op}}}}(x^*) \rangle
$$
where $v^*_{P^{\mathrm{op}}}=(-1)^ne_{1}\wedge \dots \wedge e_{n}$ is the lowest weight vector dual to $v_P$. 

The parabolic subgroup $P$ is self-associate.  More precisely $w_0^{-1}Pw_0=P^{\mathrm{op}}$ for 
\begin{align} \label{w0}
w_0=\begin{psmatrix} & -I_n\\ I_n & \end{psmatrix}.
\end{align}

\begin{Cor} \label{cor:inverse:Fourier}
For $f\in  \mathcal{S}(X_P(F))$ we have that $\mathcal{F}_{X_P}(f)$ is
$$
[2]_!(\mu_{n-2\lfloor n/2 \rfloor})\circ [2]_!(\mu_{n-2\lfloor n/2 \rfloor+2}) \cdots \circ [2]_!(\mu_{n-2}) \circ \int_{X_P^{\circ}(F)} f(x)\psi((-1)^n\langle  \mathrm{Pl}_{v_P}(x),\mathrm{Pl}_{v_P}(\cdot) \rangle)dx.
$$
\end{Cor}
\noindent Here $[2]_!(\mu_s)$ is defined as in \eqref{eqn: eta def} but with $P$ replaced with $P^{\mathrm{op}}$. See Lemma \ref{Lem: switch and twist}.
\begin{proof}
Since $v^*_{P^{\mathrm{op}}} w_0^{-1}=(-1)^nv_P,$ we have
\begin{align*}
\langle x,w_0^{-1} x' \rangle_{P|P^{\mathrm{op}}}=(-1)^n\langle \mathrm{Pl}_{v_P}(x),\mathrm{Pl}_{v_P}(x') \rangle.
\end{align*}
Thus the result follows from Corollary \ref{Cor:Pop} and the computation in \S \ref{sec:lagrangian}.
\end{proof}

\begin{Cor} \label{cor:extend}
When $n=3,$ one has
\begin{align} \label{FX3}
    \mathcal{F}_{X_P}(f)(x')=\int_{F^\times}\psi(t^{-1})|t|^{2}\left(\int_{X_P^{\circ}(F)}f(x) \overline{\psi}(t^{2}\langle \mathrm{Pl}( x ), \mathrm{Pl}(x')\rangle )dx \right)\frac{d^\times t}{\zeta(1)}
\end{align}
for all $f \in \mathcal{S}(X_P(F))$. In particular, the integral over $t$ is absolutely convergent.  
\end{Cor}

\begin{proof}
Only the last claim is not clear from Corollary \ref{cor:inverse:Fourier}.  By Lemma \ref{lem:conv} the regularized operator $[2]_!(\mu_{1})$ is equal to the unregularized operator in the case at hand as
$$
A(L(1))=\tfrac{1}{2}, \quad B(L(1))=2, \quad \tfrac{s_1+1}{\lambda_1}=1.
$$
This implies that the integral over $t$ converges absolutely.
\end{proof}

\section{Regularized integrals} \label{sec:reg}
In the remainder of the paper we apply the results of \S \ref{examples} in a specific case to establish a formula for the Fourier transform on certain affine spherical varieties. For this, it is convenient in several calculations to work with regularized integrals.   We work in the category of affine schemes because this is what we require; the techniques can probably be generalized to analytic spaces or Nash manifolds.

Let $F$ be a local field. For $r\in \zz_{\ge 1},$ let
$$
\langle\cdot,\cdot\rangle:F^r \times F^r \lto F
$$
be a perfect pairing. Let
\begin{align*}
    \widehat{f}(t):=\int_{F^r}\psi(\langle t,x\rangle)f(x)dx
\end{align*}
be the associated Fourier transform with the Haar measure on $F^r$ normalized so that the Fourier inversion formula 
$$
f(x)=\int_{F^r} \overline{\psi}(\langle t,x\rangle )\widehat{f}(t)dt
$$
is valid for $f \in \mathcal{S}(F^r)$. For $f\in L^1_\mathrm{loc}(F^r)$ and $\Phi\in C^\infty_c(F^r)$ with $\Phi(0)=1,$ we define the regularized integral
\begin{align} \label{pvtimes}
\int^*_{F^r}f(x)dx:=\lim_{|\mathcal{B}| \to \infty} \int_{F^r} \Phi\left(\frac{x}{\mathcal{B}}\right)f(x)dx
\end{align}
whenever the limit exists and is independent of $\Phi$. Here $\mathcal{B} \in F^\times$ is embedded diagonally in $F^r.$  If $f \in L^1(F^r),$  we have
$$
    \int^*_{F^r}f(x)dx=\int_{F^r}f(x)dx.
$$

\begin{Lem} \label{lem:Fourier:inversion}
For $f \in  L^1(F^r),$ let
$$
\widehat{\widehat{f}^{\ast}}(y):=\int_{F^r}^* \overline{\psi}(\langle t,y\rangle)\widehat{f}(t) dt.
$$
Then $\widehat{\widehat{f}^{\ast}}(y)=f(y)$ if $y$ is a Lebesgue point of $f$. In particular, $\widehat{\widehat{f}^{\ast}}=f$ a.e., and if $f$ is continuous at $y,$ then $\widehat{\widehat{f}^{\ast}}(y)=f(y)$.
\end{Lem}

\begin{proof}
The first assertion can be proved following the argument  of \cite[Theorem 13.15]{Wheeden:Zygmund}.  For the second assertion, we need to show that almost every point is a Lebesgue point.
This is the Lebesgue differentiation theorem; a version that is general to treat both the archimedean and the nonarchimedean case is given in
\cite[\S 3.4]{Sobolev}.
\end{proof}

Now let $V=\GG_a^d,$ where $d>r$.  Let $p_1,\dots,p_r \in F[x_1,\dots,x_d]$. For each $c=(c_1,\ldots,c_r)\in F^r,$ let
\begin{align}\label{Yc}
    Y_{c}:=\mathrm{Spec}\big(F[x_1,\dots,x_d]/(p_1-c_1,\dots,p_r-c_r)\big).
\end{align}
We assume that $Y_c$ is geometrically integral for all $c$ (hence a variety).  In particular, the smooth locus $Y^{\mathrm{sm}}_c$ is dense in $Y_c$ for all $c$. Note that if $Y_c^{\mathrm{sm}}(F)$ is nonempty, then it is dense in $Y_c(F)$ in the Hausdorff topology by \cite[Remark 3.5.76]{Poonen}.
 
Let $V(F)$ be equipped with the Haar measure $dv=dv_1dv_2\cdots dv_d,$ where the measure $dv_i$ on $\mathbb{G}_a$ is defined as in \S \ref{ssec:measures}. Let $\omega_{Y_c}$ be the differential form on $Y^{\mathrm{sm}}_c$ satisfying 
\begin{align}
    d(p_1-c_1) \wedge \dots \wedge d(p_r-c_r) \wedge \omega_{Y_c}=\omega_V
\end{align}
where $\omega_V$ is a top degree differential form on $V$ with $|\omega_V|= dv,$ and let 
\begin{align} \label{dy}
d\mu_c(y):=|\omega_{Y_c}|
\end{align}
be the corresponding positive Radon measure on $Y_c(F)$ (we extend by zero to obtain a measure on all of $Y_c(F)$ from the given measure on $Y^{\mathrm{sm}}_c(F)$).  The measure $d\mu_c(y)$ does 
not depend on the choice of $\omega_{Y_c},$ but it does depend on the choice of $p_1,\dots,p_r$.
 We let
\begin{align} \label{0}
d\mu:=d\mu_0 \quad \textrm{ and } \quad Y:=Y_0.
\end{align}

A nice reference for the definition of $d\mu_c(y)$ in a more general context is \cite[\S III.1, B2.1]{Gelfand:Shilov:I} in the real and complex cases.  For the nonarchimedean case we refer to \cite[\S 7.6]{Igusa}.

Suppose $Y^{\mathrm{sm}}(F)$ is nonempty. Hence $Y^{\mathrm{sm}}(F)$ is dense in $Y(F)$.  For $f \in \mathcal{S}(Y^{\mathrm{sm}}(F)),$ the integral $\int_{Y(F)} f(y)d\mu(y)$ is well-defined. Let us extend its domain of definition and at the same time develop a useful formula for it.  
Let $p(v):=(p_1(v),\dots,p_r(v))$. Let $\langle \cdot,\cdot \rangle_{\mathrm{st}}:F^r \times F^r \to F$ be the standard pairing.
For $\widetilde{f} \in L^1(V(F)),$ we have 
\begin{align} \label{COV}
    \int_{V(F)}\psi(\langle t, p(v)\rangle_{\mathrm{st}})\widetilde{f}(v) dv =\int_{F^r}\psi(\langle t,c \rangle_{\mathrm{st}}) \left(\int_{Y_c(F)}\widetilde{f}(y)d\mu_c(y)\right)  dc
\end{align}
by the change of variables formula.  On the left, the integral is absolutely convergent by assumption; on the right, the inner integral over $y$ is finite for almost every $c$ and defines a function of $c$ that is in $L^1(F^r)$.  With this in mind, for $\widetilde{f} \in L^1(V(F))$ we define
\begin{align} \label{reg}
    \int^{\mathrm{reg}}_{Y(F)}\widetilde{f}(y)d\mu(y):=\int_{F^r}^{\ast}\left(\int_{V(F)}\psi(\langle t, p(v)\rangle_{\mathrm{st}})\widetilde{f}(v)dv\right)dt
\end{align}
provided it exists.  

\begin{Rem} \label{rem:int}
 If $t \mapsto \int_{V(F)}\psi(\langle t, p(v)\rangle_{\mathrm{st}})\widetilde{f}(v)dv$ is in $L^1(F^r),$ then 
\begin{align*}
    \int^{\mathrm{reg}}_{Y(F)}\widetilde{f}(y)d\mu(y)=\int_{F^r}\left(\int_{V(F)}\psi(\langle t, p(v)\rangle_{\mathrm{st}})\widetilde{f}(v)dv\right)dt.
\end{align*}
\end{Rem}

\begin{Lem} \label{lem:reg:int}
Suppose $p_i$ are homogeneous polynomials all of degree $k$ and any $r\times r$ minor of the Jacobian of $p=(p_1,\dots,p_r)$ is a monomial. Assume further that for any set $S\subset \left\{1,\ldots,d\right\}$ of cardinality $m,$ there is a nonzero $r\times r$ minor that is a monomial in $v_i,$ $i\in S$.

If $kr+m\le d,$ then for any $\widetilde{f} \in \mathcal{S}(V(F))$ one has $f:=\widetilde{f}|_{Y^{\mathrm{sm}}(F)}\in L^1(Y(F),d\mu)$ and
$$
\int^{\mathrm{reg}}_{Y(F)}\widetilde{f}(y)d\mu(y)=\int_{Y(F)}f(y)d\mu(y).
$$
\end{Lem}

\begin{proof}
 Assume until otherwise stated that $f \in L^1(Y(F),d\mu)$. We claim that $0$ is a Lebesgue point of the function
 \begin{align}\label{functioninc}
c \longmapsto \int_{Y_c(F)}\widetilde{f}(y)d\mu_c(y).
\end{align} Hence the identity in the lemma holds by \eqref{COV} and Lemma \ref{lem:Fourier:inversion}. 

For each positive integer $n,$ let
$$ 
W_n\subseteq \{v \in V(F): |v|:=\max_{1\le i\le d}\{|v_i|\}<n\}
$$
be the subset on which some $r \times r$ minor of the Jacobian of $p$ has norm greater than  $n^{-1}$. These are open, relatively compact subsets of $V(F)$.  If $\widetilde{f}$ is supported on $W_n,$ then \eqref{functioninc} 
defines a continuous compactly supported function on $F^r$ and thus $0$ is a Lebesgue point. 

For general $\widetilde{f}$ we wish to show
\begin{align} \label{limsup}
    \limsup_{|t|\to 0} |t|^{-kr}\int_{|c|\le |t|^k} \left|\int_{Y_c(F)}\widetilde{f}(y)d\mu_c(y)-\int_{Y(F)}f(y)d\mu(y)\right|dc=0.
\end{align}
Here 
 \[
 |c| :=\max_{1\leq i\leq r}\{|c_i|\},
 \]
for $c= (c_1,\ldots,c_r)\in F^r$.
 Note that $\cup_n W_n$  and  $\cup_n W_n\cap Y^{\mathrm{sm}}(F)$ are of full measure in $V(F)$ and $Y^{\mathrm{sm}}(F)$ respectively. Let $\left\{\varphi_n\right\}$ be a smooth partition of unity of $\cup_n W_n$ subordinate to $\{W_n\}$. Put $\widetilde{f}_n:=\sum_{ m=1}^n\widetilde{f} \varphi_m$ and $f_n:=\widetilde{f}_n|_{Y^{\mathrm{sm}}(F)}$. Then $|v|^N\widetilde{f}_n\to |v|^N\widetilde{f}$ in $L^1(V(F))$ for any $N\in\zz_{\ge 0}$, and $f_n\to f$ in $L^1(Y(F),d\mu)$. 

Now for each $t \in F^\times,$ we have
\begin{align*}
    &|t|^{-kr}\int_{|c|\le |t|^k} \left|\int_{Y_c(F)}\widetilde{f}(y)d\mu_c(y)-\int_{Y(F)}f(y)d\mu(y)\right|dc\\
    &\le |t|^{-kr}\int_{|c|\le |t|^k} \left|\int_{Y_c(F)}\widetilde{f}(y)d\mu_c(y)-\int_{Y_c(F)}\widetilde{f}_n(y)d\mu_c(y)\right|dc\\
    &+|t|^{-kr}\int_{|c|\le |t|^k} \left|\int_{Y_c(F)}\widetilde{f}_n(y)d\mu_c(y)-\int_{Y(F)}f_n(y)d\mu(y)\right|dc\\
    &+C\int_{Y(F)}|f_n(y)-f(y)|d\mu(y)
\end{align*}
for any $n$ for some positive constant $C$. Since $\widetilde{f}_n$ is supported on $W_n,$ taking $\limsup$ over $|t|\to 0$ on both sides, we have that the limit superior in \eqref{limsup} is bounded by
\begin{align*}
    \limsup_{|t| \to 0} |t|^{-kr}\int_{|c|\le |t|^k} \left|\int_{Y_c(F)}\widetilde{f}(y)d\mu_c(y)-\int_{Y_c(F)}\widetilde{f}_n(y)d\mu_c(y)\right|dc+ C\int_{Y(F)}|f_n(y)-f(y)|d\mu(y).
\end{align*}
Since the second term converges to $0$ as $n\to\infty,$ the change of variables formula  implies that it suffices to show 
\begin{align} \label{to:show:0}
    \limsup_{n\to\infty }\limsup_{|t|\to 0}|t|^{-kr}\int_{|p(v)|\le |t|^k} \left|\widetilde{f}(v)-\widetilde{f}_n(v)\right|dv=0.
\end{align}

 We have
\begin{align} \label{for:reader}
    &\int_{|p(v)|\le |t|^k} \left|\widetilde{f}(v)-\widetilde{f}_n(v)\right|dv=|t|^{d}\int_{|p(v)|\le 1} \left|\widetilde{f}(tv)-\widetilde{f}_n(tv)\right|dv.
\end{align}
By symmetry it suffices to bound the contribution of the domain $|v_1|\ge \cdots \ge |v_d|$ to the integral. By assumption, there exists a subset $J\subset \{1,\ldots,d\}$ of cardinality $r$ such that $\det (\frac{\partial p_i}{\partial v_j})_{1\le i\le r, j\in J}$ is a nonzero monomial in $\{v_1,\ldots,v_{m}\}.$ Let $T:\{1,\ldots,d\}-J\to \{1,\ldots,d-r\} $ be the increasing bijection.
For $v' \in F^{d-r},$ let
\begin{align*}
h_n(v'):=\sup\left\{ \left|\widetilde{f}(v)-\widetilde{f}_n(v)\right|: v \in F^d, v_{T^{-1}(j)}=v'_j \textrm{ for } 1\le j\le d-r\right\}.
\end{align*}
 Choose the smallest $1 \leq \ell \leq d-r$ such that $T^{-1}(\ell)\in \{m,\ldots,m+r\} -J$. Then by changing variables, the contribution of $|v_1| \geq \dots \geq |v_d|$ to \eqref{for:reader} is dominated by 
\begin{align*}
    &|t|^{d}\int  h_n(tv') |v'_{\ell}|^{-r(k-1)}dv',
\end{align*}
where the integral is taken over $|v'_1|\ge \cdots \ge |v'_{d-r}|. $
Changing variables $v'\mapsto t^{-1}v',$ we arrive at
\begin{align*}
    &|t|^{rk}\int h_n(v') |v'_{\ell}|^{-r(k-1)}dv'\ll  |t|^{rk}\int_{F^{\ell}}  \left(\sup_{w' \in  F^{d-r-\ell}} h_n(w,w')\right)|w_{\ell}|^{d-\ell-r-r(k-1)}dw_1\cdots dw_{\ell}.
\end{align*}
As $\ell\le m,$ we have $d-\ell-r-r(k-1)\ge d-m-rk\ge 0$ by assumption, so the latter integral is finite and converges to $0$ as $n\to\infty$.   This implies \eqref{to:show:0}.

We are left with proving that for any $\widetilde{f} \in \mathcal{S}(V(F))$ one has $\widetilde{f}|_{Y^{\mathrm{sm}}(F)} \in L^1(Y(F),d\mu)$.  This follows from an analogue of the argument bounding \eqref{for:reader}.
\end{proof}

\section{The Schwartz space and Fourier transform on $Y$}
\label{sec:Y}

For the remainder of the paper, let $F$ be a local field of characteristic zero. We refer to \cite{Getz:Liu:Triple,Getz:Hsu} for more details on the constructions in this section.
For $1\leq i \leq 3,$ let $V_i=\GG_a^{d_i}$ where $d_i$ is even and let $Q_i$ be a nondegenerate quadratic form on $V_i(F)$.   We assume that $d_i>2$ for each $i$; this plays a role in some convergence arguments below (see the proof of Theorem \ref{Thm:main}). Let $V:=V_1 \times V_2 \times V_3,$ and for an $F$-algebra $R,$ let
\begin{align}\label{Ydef}
    Y(R):=\{(v_1,v_2,v_3) \in V(R): Q_1(v_1)=Q_2(v_2)=Q_3(v_3)\}.
\end{align}
The anisotropic locus
\begin{align}\label{Yani}
Y^{\mathrm{ani}} \subset Y
\end{align}
is the open complement of the vanishing locus of $Q_i$ (which is independent of $i$). We assume $Y^{\mathrm{sm}}(F)$ is nonempty, which implies $Y^{\mathrm{ani}}(F)$ is nonempty and dense in $Y(F)$ in the Hausdorff topology by \cite[Remark 3.5.76]{Poonen}. 

We assume that $G=\mathrm{Sp}_6$ and $P$ is the Siegel parabolic as in \S \ref{sec:lagrangian}, and set $X:=X_P$.  
We identify $\SL_2^3$ with the subgroup of $\mathrm{Sp}_6$ whose points in an $F$-algebra $R$ are given by 
\begin{align*} 
\left\{\left(\begin{smallmatrix} a_1 & & &b_1& & \\ & a_2 & && b_2 & \\ & & a_3 & & & b_3\\ c_1 & & & d_1 & & \\ & c_2 & & & d_2 & \\ & & c_3 & & & d_3\end{smallmatrix}\right) \in \GL_6(R) :a_id_i-b_ic_i=1 \textrm{ for }1 \leq i \leq 3 \right\}.
\end{align*}
Thus we obtain an action of $\SL_2^3$ on $X^{\circ}$.  
Let
\begin{align} \label{gamma0}
\gamma_0:=\left(\begin{smallmatrix} 0 & 0 & 0 & -1 & 0& 0\\ 0 & 1 & 0 & 0 & 0 &0\\
0 & 0 & 1 & 0& 0 & 0\\
1 & 1 & 1 & 0 & 0 & 0\\ 0 & 0 & 0 &-1 & 1 &  0\\
0 & 0 & 0 & -1 & 0 & 1  \end{smallmatrix}\right).
\end{align}
Then 
\begin{align}\label{x0}
x_0:=P^{\mathrm{der}}(F)\gamma_0
\end{align}
is a representative for the unique $\SL_2^3$-orbit in $X^{\circ}$ of maximal dimension.  This follows from the computation of the stabilizers of all orbits given in \cite[Lemma 2.1]{Getz:Liu:Triple}.  By the same lemma, the stabilizer of $x_0$ is the group whose points in an $F$-algebra $R$ is
\begin{align}\label{N0}
N_0(R):=\left\{\left(\begin{psmatrix} 1 & t_1 \\ & 1 \end{psmatrix},\begin{psmatrix} 1 & t_2 \\ & 1 \end{psmatrix}, \begin{psmatrix} 1 & -t_1-t_2 \\ & 1 \end{psmatrix}\right): t_1,t_2 \in R\right\}.
\end{align}
By upper semicontinuity of the dimension of stabilizers \cite[\S 0.2]{GIT} we deduce that the orbit of $x_0$ is the unique open $\SL_2^3$-orbit in $X^\circ.$  

Let $\rho$ be the Weil representation of $\SL_2^3(F)$ on $\mathcal{S}(V(F))$ attached to our additive character $\psi$ and the quadratic forms $Q_i$.  
Let $\mathcal{S}(X(F) \times V(F))$ be the algebraic tensor product $\mathcal{S}(X(F)) \otimes \mathcal{S}(V(F))$ in the nonarchimedean case and the completed projective tensor product in the archimedean case.  

There is a map
$$
I:\mathcal{S}(X(F) \times V(F)) \lto C^\infty(Y^{\mathrm{sm}}(F))
$$
given on pure tensors by 
\begin{align} \label{I:def}
I(f_1 \otimes f_2)=\int_{N_0(F) \backslash \SL_2^3(F)} f_1(x_0g) \rho(g)f_2 d\dot{g}.
\end{align}
The integral is absolutely convergent for all $f \in \mathcal{S}(X(F) \times V(F))$ (see \cite[Propositions 7.1, 8.2, 8.3]{Getz:Liu:Triple}).
By definition, the image of $I$ is the Schwartz space $\mathcal{S}(Y(F))$.  The kernel of $I$ is closed in the archimedean case by \cite[Lemma 5.1]{Getz:Hsu}, and hence in this case we equip $\mathcal{S}(Y(F))$ with the quotient Fr\'echet space structure.  

By  \cite[Theorem 12.1]{Getz:Hsu}, there is a
unique $\cc$-linear isomorphism $\mathcal{F}_Y:\mathcal{S}(Y(F)) \to \mathcal{S}(Y(F))$ such that the diagram
\medskip
\begin{center}
\begin{tikzcd}[column sep=large]
\mathcal{S}(X(F)\times V(F))
\arrow[d, two heads, "I"]
\arrow[r, "\mathcal{F}_X"] &\mathcal{S}(X(F) \times V(F)) \arrow[d,two heads, "I"]\\
\mathcal{S}(Y(F)) \arrow[r,"\mathcal{F}_Y"]
& \mathcal{S}(Y(F))
\end{tikzcd}
\end{center}
commutes; in the archimedean case, $\calf_Y$ is continuous in the Fr\'echet topology on $\mathcal{S}(Y(F))$. In \emph{loc.~cit.},~this is the only description of $\mathcal{F}_Y$ that is given. 
The definition of $I(f_1 \otimes f_2)$ depends on the choice of measure $d\dot{g},$ but from the description of $\mathcal{F}_Y$ given above, it follows that $\mathcal{F}_Y$ does not depend on this choice.

The definition of $\mathcal{F}_Y$ is indirect; the goal of the rest of this paper is to give a direct definition of $\mathcal{F}_Y,$ at least on a subspace of $\mathcal{S}(Y(F))$.  Let
\begin{align}\label{calS}
\mathcal{S}:=\mathrm{Im}(\mathcal{S}(V(F)) \to C^\infty(Y^{\mathrm{sm}}(F))) 
\end{align}
where the implicit map is restriction of functions.  Then $\mathcal{S}=I(\mathcal{S}(x_0\SL_2^3(F)\times V(F))) <\mathcal{S}(Y(F))$ by \cite[Lemma 5.3]{Getz:Hsu}. Moreover, $\mathcal{S}(Y(F)) <L^p(Y(F),d\mu)$ for $p \leq 2$ and the inclusion is continuous in the archimedean case by \cite[Proposition 11.2]{Getz:Hsu}. Here the Radon measure $d\mu$ on $Y(F)$ is defined as in \S \ref{sec:reg} using the polynomials $p_1(v_1,v_2,v_3)=Q_1(v_1)-Q_2(v_2)$ and $p_2(v_1,v_2,v_3)=Q_2(v_2)-Q_3(v_3)$.

\section{A formula for $\mathcal{F}_Y$}
\label{sec:FY}

For $u_i,v_i\in V_i(F)$, let 
\begin{align*}
    \langle u_i, v_i\rangle_i:=\tfrac{1}{2}\left(Q_i(u_i+v_i)-Q_i(u_i)-Q_i(v_i)\right).
\end{align*}
 For $u=(u_1,u_2,u_3),v=(v_1,v_2,v_3) \in V(F),$ we write
$$
Q(u):=Q_1(u_1)+Q_2(u_2)+Q_3(u_3), \quad 
\langle u,v \rangle:=\sum_{i=1}^3 \langle u_i,v_i \rangle_i.
$$
For $a= (a_1,a_2,a_3)\in (F^\times)^3,$ let
\begin{align} \label{chiQi}
    \chi_{Q_i}(a_i):=(
a_i, (-1)^{d_i/2}\det(\langle\cdot,\cdot\rangle_i)), \quad \chi_Q(a):=\prod_{i=1}^3 \chi_{Q_i}(a_i)
\end{align}
where $(\cdot,\cdot)$ is the Hilbert symbol.  Finally, let
\begin{align} \label{q2}
[a]:=a_1a_2a_3, \quad \mathfrak{r}(a):=\frac{(a_1a_2+a_2a_3+a_3a_1)^2}{[a]}.
\end{align}

For the convenience of the reader, we restate Theorem \ref{Thm:main:intro}:

\begin{Thm} \label{Thm:main}  Suppose as above that the characteristic of $F$ is zero, $d_i=\dim V_i>2$ for all $1\le i\le 3$ and $Y^{\mathrm{sm}}(F) \neq \emptyset$. There is a $c \in \cc^\times$ depending on $\psi,F$ and the $Q_i$ such that for all $f \in \mathcal{S}$ (defined as in \eqref{calS}) and $\xi \in Y^{\mathrm{ani}}(F),$ one has
\begin{align*}
\begin{split}
&\mathcal{F}_Y(f)(\xi)\\&=c\int_{F^\times}\psi(z^{-1}) \Bigg( \int_{ (F^\times)^3} \overline{\psi}(z^2\mathfrak{r}(a))
 \Bigg(\int_{Y(F)}\psi\Bigg(\left\langle \frac{\xi}{a},y \right\rangle -\frac{Q(\xi)Q(y)}{9z^2[a]}\Bigg)f(y)d\mu(y)\Bigg) \frac{\chi_Q(a)d^\times a}{\{a\}^{d/2-1}}   \Bigg)d^\times z.
 \end{split}
\end{align*}
\end{Thm}
\noindent 
Here the measure on $Y(F)$ is defined as in \S \ref{sec:Y} and
\begin{align*}
\frac{\xi}{a}:=\left(\frac{\xi_1}{a_1},\frac{\xi_2}{a_2},\frac{\xi_3}{a_3} \right), \quad  \{a\}^{d/2-1}:=\prod_{i=1}^3|a_i|^{d_i/2-1}.
\end{align*} 
Let $\gamma(Q_i):=\gamma(\psi \circ Q_i)$ be the Weil index, defined as in \cite[\S 24]{Weil:certains}.
The constant $c \in \cc^\times$ is $\frac{\gamma(Q)k\kappa}{|3|\zeta(1)},$ where $\gamma(Q)=\gamma(Q_1)\gamma(Q_2) \gamma(Q_3)$ is the product of the Weil indices, $k$ is defined as in \eqref{eq:absoluteconst} and $\kappa$ is defined as in Remark \ref{rem:kappa} below; these are both ratios of suitable invariant measures.   When $F$ is nonarchimedean and the matrix of each $Q_i$ with respect to the standard basis of $F^{d_i}$ lies in  $\GL_{d_i}(\mathcal{O}_F)$ for each $i$, then $\kappa=1$.
\subsection{Comment on norms}
We pause to introduce some notation used throughout the rest of the paper. For $n\in \zz_{>0}$, unless otherwise stated, we equip $F^n$ with the box norm
\begin{align}\label{boxnorm}
    |y| :=\max_{1\leq i\leq n}\left\{|y_i|\right\}, \qquad y=(y_1,\ldots,y_n)\in F^n.
\end{align}
This includes vectors in $V_i$ and $V_1\otimes V_2\otimes V_3$. Let
 $y\in F^n-\{0\}$.
 If $F$ is nonarchimedean, let 
 $$
 \mathrm{ord}(y):=\min_{1\le i\le n}\{\mathrm{ord}(y_i)\},
 $$
 where $\mathrm{ord}(y_i)\in \zz\cup\{\infty\}$ satisfies $|y_i|=q^{-\mathrm{ord}(y_i)}$.
When $F$ is nonarchimedean, we fix a uniformizer $\varpi$ and set
\begin{align} \label{widetilde}
    \widetilde{y}:=\left\{\begin{array}{ll}
        \varpi^{\mathrm{ord}(y)}  & \textrm{if $F$ is nonarchimedean,}  \\
        |y|^{\frac{1}{[F:\rr]}} & \textrm{if $F$ is archimedean. }
    \end{array}\right.    
\end{align}
Thus $|y/\widetilde{y}|=1$.
    
    For any integer $m\in\zz$, we denote
\begin{align*}
    \{a\}^m:=\prod_{i=1}^3 |a_i|^m \quad \textrm{and} \quad \{a\}^{ d/2+m}:=\prod_{i=1}^3 |a_i|^{ d_i/2+m}.
\end{align*}

\subsection{Proof of Theorem \ref{Thm:main}}
To aid the reader, we present the proof modulo various technical analytic results that will be proved later in this section and in \S \ref{sec:absolutebd}.  By Theorem \ref{thm:conj} we have  $\mathcal{S}(X^{\circ}(F)) < \mathcal{S}(X(F))$.  We will use this fact without further mention below.

    Let $f\in \mathcal{S}(V(F))$ and for $0<\delta<1$, let \begin{align} \label{Kdelta}
    K_\delta:=\left\{ \begin{psmatrix}
    1+u& v\\
    w & 1+x
    \end{psmatrix}\in \SL_2^3(F) \bigg| |(u,v,w,x)|< \delta\right\}.
\end{align}
It is an open neighborhood of the identity in $\SL_2^3(F)$.

\begin{Lem}\label{Lem: nascent}Assume $F$ is archimedean.  We can choose a sequence of nascent delta-functions $\varphi_n \in C_c^\infty(x_0\SL_2^3(F))$ indexed by $n \in \zz_{>4}$ such that
\begin{align} \label{limit}
\lim_{n\to \infty} I(\varphi_n \otimes f)=f|_{Y^{\mathrm{sm}}(F)}
\end{align}
with respect to the Fr\'echet topology on $\mathcal{S}(Y(F))$. Moreover, we can assume $\mathrm{supp}(\varphi_n)\subset x_0K_{1/n}$.
\end{Lem}

\begin{proof} 
There is a map
\begin{align*}
    \widetilde{I}:\mathcal{S}(\SL_2^3(F) \times V(F)) &\lto \mathcal{S}(V(F)) \lto \mathcal{S}
\end{align*}
where the left arrow is given on pure tensors by $f_1 \otimes f_2 \mapsto \int_{\SL_2^3(F)}f_1(g)\rho(g)f_2dg$ and the right is given by restriction of functions. The left arrow is continuous by continuity of the Weil representation and surjective by the Dixmier--Malliavin lemma.  The right arrow is surjective with closed kernel and we endow $\mathcal{S}$ with the quotient topology (see \cite[\S 3]{Elazar:Shaviv} for more details about this topology).  Thus $\widetilde{I}$ is surjective and continuous.  We also have a continuous map
\begin{align*}
p:\mathcal{S}(\SL_2^3(F)) \lto \mathcal{S}(x_0\SL_2^3(F))
\end{align*}
given by sending $f_1$ to  $p(f_1)(x_0g):=\int_{N_0(F)}f_1(ug)du.$  We observe that 
\begin{align*}
    \widetilde{I}(f_1 \otimes f_2)=I(p(f_1) \otimes f_2).
\end{align*}
Choose $\widetilde{\varphi}_n \in C_c^\infty(\SL_2^3(F))$ with support in $K_{1/n}$ such that $\lim_{n \to \infty}\widetilde{I}(\widetilde{\varphi}_n \otimes f) =f$ with respect to the usual topology on $\mathcal{S}(V(F)).$  Then $\varphi_n:=p(\widetilde{\varphi}_n)$ satisfy \eqref{limit} with respect to the topology on $\mathcal{S}.$
 It thus suffices to show that the inclusion $\mathcal{S} \hookrightarrow \mathcal{S}(Y(F))$ is continuous.  

We have a commutative diagram
 \begin{center}
    \begin{tikzcd}
 \mathcal{S}(\SL_2^3(F) \times V(F)) \arrow[dr,"\widetilde{I}"] \arrow[r,"p \otimes \mathrm{Id}"] & \mathcal{S}(x_0\SL_2^3(F) \times V(F)) \arrow[r] \arrow[d,"I"] & \mathcal{S}(X(F) \times V(F)) \arrow[d,"I"] \\
    & \mathcal{S} \arrow[r]  & \mathcal{S}(Y(F))
    \end{tikzcd}
\end{center}
The top horizontal arrows are continuous, and the maps $\widetilde{I}$ and $I$ are continuous and surjective.  By the open mapping theorem, $\widetilde{I}$ is open, and we deduce that the bottom horizontal arrow is continuous.
\end{proof}
 Up to scaling by a positive real number, there is a unique right $\SL_2^3(F)$-invariant measure $d\dot{g}$ on the open dense subset $x_0\SL_2^3(F) \subset X^{\circ}(F),$ and a unique (up to scaling) right $\mathrm{Sp}_6(F)$-invariant measure $dx$ on $X^{\circ}(F)$.  On the other hand, $dx$ is clearly right $\SL_2^3(F)$-invariant.  Thus the restriction of $dx$ to the open full measure subset $x_0\SL_2^3(F) \subset X^{\circ}(F)$ (which we continue to denote by $dx$) satisfies
\begin{align}\label{eq:uptoscalar}
Cdx=d\dot{g}
\end{align}
for some $C \in \rr_{>0}$. Therefore, we have an isomorphism
\begin{align*}
L^p(X^{\circ}(F)) &\tilde{\lto} L^p(N_0(F) \backslash \SL_2^3(F))\\
f &\longmapsto \left(g \mapsto f(x_0g)\right)
\end{align*}
for $0<p<\infty$. We normalize the measures so that $C=1$.  Thus these isomorphisms are isometries. 

When $F$ is nonarchimedean, we let
$$
\varphi_n :=\frac{\textbf{1}_{x_0K_{1/n}}}{d\dot{g}(x_0K_{1/n})}.
$$
 Since the Weil representation is smooth, there is an integer $N$ depending on $f$ such that $I(\varphi_n \otimes f)=f|_{Y^{\mathrm{sm}}(F)}$ for $n>N\ge 4$.

In all cases, we have
\begin{align*} 
\lim_{n \to \infty} I(\mathcal{F}_X(\varphi_n) \otimes f)=\lim_{n \to \infty} \mathcal{F}_Y(I(\varphi_n \otimes f)) =\mathcal{F}_Y(f|_{Y^\mathrm{sm}(F)}).
\end{align*}
Indeed, in the archimedean case, this follows from the continuity of $\mathcal{F}_Y,$ and in the nonarchimedean case, it follows from the observation that $I(\varphi_n \otimes f)$ stabilizes as $n\to \infty$. Therefore, we are to compute
\begin{align*} 
\lim_{n \to \infty}I(\mathcal{F}_X(\varphi_n) \otimes f)(\xi)
\end{align*}
for $\xi\in Y^{\mathrm{ani}}(F)$.

Extension by zero induces an injection $\mathcal{S}(x_0\SL_2^3(F)) \to \mathcal{S}(X^{\circ}(F)).$  This is obvious if $F$ is nonarchimedean, and \cite[Theorem 3.23]{Elazar:Shaviv} if $F$ is archimedean.  Moreover we have $\mathcal{S}(X^{\circ}(F)) <\mathcal{S}(X(F))$ by Theorem \ref{thm:conj}.
 For $(\xi,W,g) \in Y^{\mathrm{ani}}(F) \times \mathcal{S}(x_0\SL_2^3(F)) \times \SL_2^3(F)$, 
let 
$$
f_{\xi,W}(x_0g):=W(x_0g)\rho(g)f(\xi).
$$
Note that $f_{\xi,W}\in \mathcal{S}(x_0\SL_2^3(F))<\mathcal{S}(X(F))$. In \S\ref{ssec:B1}, we construct a directed set of functions $\mathcal{V}_{\mathcal{B}_1,\mathcal{B}_2}\in \mathcal{S}(x_0\SL_2^3(F))$ such that 
\begin{align}\label{eq:to1}
    |\mathcal{V}_{\mathcal{B}_1,\mathcal{B}_2}|\ll 1,\quad \textrm{ and } \lim_{|\mathcal{B}_2|\to \infty} \lim_{|\mathcal{B}_1|\to \infty} \mathcal{V}_{\mathcal{B}_1,\mathcal{B}_2}(x_0g)=1
\end{align}
where convergence is pointwise a.e. for $g \in N_0(F) \backslash \SL_2^3(F)$ with respect to $d\dot{g}.$  Thus by the dominated convergence theorem and the Plancherel formula \cite[Proposition 3.9]{Getz:Hsu}, 
\begin{align*}
I(\mathcal{F}_X(\varphi_n) \otimes f)(\xi)&= \lim_{|\mathcal{B}_2|\to \infty} \lim_{|\mathcal{B}_1|\to \infty}\int_{N_0(F) \backslash \SL_2^3(F)}\mathcal{F}_X(\varphi_n)(x_0g)f_{\xi,\mathcal{V}_{\mathcal{B}_1,\mathcal{B}_2}}(x_0g) d\dot{g}\\
&= \lim_{|\mathcal{B}_2|\to \infty} \lim_{|\mathcal{B}_1|\to \infty}\int_{N_0(F)\backslash \SL_2^3(F)} \varphi_n(x_0g)\mathcal{F}_X(f_{\xi,\mathcal{V}_{\mathcal{B}_1,\mathcal{B}_2}})(-x_0g)d\dot{g}.
\end{align*}

Let 
\begin{align} 
    \Omega:=\overline{K_{1/4}} \label{Omega}
\end{align} be the closure of $K_{1/4}$ in $\SL_2^3(F)$. For $u\in V(F)$ and $v=(v_1,v_2)\in F^2$, let
\begin{align}\label{cuv} \begin{split}
    \mathfrak{c}(u,v)&:=v_1Q_1(u_1)+v_2Q_2(u_2)-(v_1+v_2)Q_3(u_3) \quad \textrm{ and } \quad dv=dv_1dv_2.\end{split}
\end{align}
We prove in \S\ref{ssec:B1} and \S \ref{ssec:formal} the following:
\begin{Prop} \label{prop:limit}
Fix $\xi\in Y^{\mathrm{ani}}(F)$. For each $\mathcal{B}_2$, there is a constant $M(\mathcal{B}_2)>0$ such that $|\mathcal{F}_X(f_{\xi,\mathcal{V}_{\mathcal{B}_1,\mathcal{B}_2}})(-x_0g)|<M(\mathcal{B}_2)$
for all $g\in \Omega$ and all $|\mathcal{B}_1|$ sufficiently large. Moreover, there is a constant $M>0$ such that
\begin{align*}
    \left|\lim_{|\mathcal{B}_1|\to\infty} \mathcal{F}_X(f_{\xi,\mathcal{V}_{\mathcal{B}_1,\mathcal{B}_2}})(-x_0g)\right| < M
\end{align*}
for all $\mathcal{B}_2$. Furthermore,
\begin{align}\label{eq:finalform}
\begin{split}
    &\lim_{|\mathcal{B}_2|\to\infty}\lim_{|\mathcal{B}_1|\to\infty} \mathcal{F}_X(f_{\xi,\mathcal{V}_{\mathcal{B}_1,\mathcal{B}_2}})(-x_0g)\\
    &=\frac{\gamma(Q)k}{|3|\zeta(1)}\int_{F^\times} \psi(z^{-1}) \Bigg( \int_{ (F^\times)^3} \overline{\psi}(z^2\mathfrak{r}(a))
      \\& \times 
\int_{F^2}\left(\int_{V(F)}\psi\left(\left\langle \frac{\xi}{a},u \right\rangle -\frac{Q(\xi)Q(u)}{9z^2[a]}+\mathfrak{c}(u,v)\right)\rho(g)f(u)du\right) dv \frac{\chi_Q(a)d^\times a}{\{a\}^{d/2-1}}\Bigg)d^\times z
 \end{split}
\end{align}
for all $g\in \Omega$, and the integral defines a continuous function in $g \in \Omega$.
\end{Prop}

\noindent Here $k>0$ is as in \eqref{eq:absoluteconst}, and the Haar measure $du$ on $V(F)$ is normalized to be self-dual with respect to $\psi$ and the pairing $\langle \cdot,\cdot\rangle$.

Assume Proposition \ref{prop:limit}. Applying the bounded convergence theorem, the proof of Theorem \ref{Thm:main} will be complete once we prove the following:

\begin{Lem} \label{lem:finish} There is a constant $\kappa>0$ such that for $f \in \mathcal{S}(V(F)),$ one has
\begin{align*}
    &\int_{F^2}\left(\int_{V(F)}\psi\left(\left\langle \frac{\xi}{a},u \right\rangle -\frac{Q(\xi)Q(u)}{9z^2[a]}+\mathfrak{c}(u,v)\right)f(u)du\right)dv\\
    &=\kappa\int_{Y(F)}\psi\left(\left\langle \frac{\xi}{a},y \right\rangle -\frac{Q(\xi)Q(y)}{9z^2[a]}\right)f(y)d\mu(y).
\end{align*}
\end{Lem}
\begin{proof}
We use the results of \S \ref{sec:reg}.  We must first check 
that the schemes $Y_{c_1,c_2}$ for $(c_1,c_2) \in F^2$ defined as in \eqref{Yc} are geometrically integral.  We can check this over the algebraic closure. Consider the quadratic form defined by $p_1(u_1,u_2,u_3,z)=Q_1(u_1)-Q_3(u_3)-c_1z^2$ and $p_2(u_1,u_2,u_3,z)=Q_2(u_2)-Q_3(u_3)-c_2z^2$. Then $Y_{c_1,c_2}$ is an open subscheme of the projective variety $Y'\subset \mathbb{P}(V\oplus \GG_a)$ cut out by $p_1$ and $p_2$. Therefore $Y_{c_1,c_2}$ is geometrically integral by \cite[Lemma 1.11]{CTSSD:I}. 

By Lemma \ref{lem:quad:esti}, the integral over $F^2$ in the statement of the lemma is absolutely convergent.  Therefore Remark \ref{rem:int} implies that the top integral in the lemma is equal to 
$$
\kappa \int_{Y(F)}^{\mathrm{reg}} \psi\left(\left\langle \frac{\xi}{a},y \right\rangle -\frac{Q(\xi)Q(y)}{9z^2[a]}\right)f(y)d\mu(y)
$$
defined as in \eqref{reg} for some $\kappa>0$.  By changing variables, we may assume each $Q_i$ is associated to a diagonal matrix.  One checks that the hypotheses of  Lemma \ref{lem:reg:int} are satisfied with $m=1+\max(d_1,d_2,d_3)$ as $d_i>2$ for all $i$, and we use it to deduce the current lemma.
\end{proof}

\begin{Rem} \label{rem:kappa}
The constant $\kappa$ of the lemma is the positive real number such that 
$
    \kappa^{-1}du
$
is the Haar measure on $V(F)$ that is self-dual with respect to $\psi$ and the standard pairing on $V(F)$.
\end{Rem}

    \subsection{Finiteness of the integral} \label{ssec:B2}

As a first step toward the proof of Theorem \ref{Thm:main}, 
in this subsection we prove that the integral \eqref{eq:finalform} is finite. 

Let $f\in \mathcal{S}(V(F))$ and $\xi\in Y^{\mathrm{ani}}(F)$. Recall the notation $\widetilde{\xi}:=(
\widetilde{\xi}_1,\widetilde{\xi}_2,\widetilde{\xi}_3)\in (F^\times)^3$ from \eqref{widetilde}. For $(a,b)\in (F^\times)^3\times F,$ define
 \begin{align}\label{Psi}
        &\Psi_{\xi}(a,b,f):=\int_{F^2}\left(\int_{V(F)}\psi\left(\left\langle \frac{a\xi}{\widetilde{\xi}},u \right\rangle +\frac{b[a]Q(u)}{9[\widetilde{\xi}]}+\mathfrak{c}(u,v)\right)f(u)du\right)
 \chi_Q(a\widetilde{\xi})\{a\}^{d/2-2} dv.
    \end{align}
We deduce from \eqref{use:quad:esti} below that the integral over $F^2$ is absolutely convergent. Moreover, $\Psi_\xi(\cdot ,b,f)\in L^1(F^3,da)$ for all $b\in F$ (see Corollary \ref{cor:absolutebd}). After a change of variables $a \mapsto a^{-1}\widetilde{\xi},$ we have
\begin{align}\label{eq:changeaxi}
\nonumber
    & \int_{ (F^\times)^3} \overline{\psi}\left(z^2\mathfrak{r}(a)\right)  \int_{F^2} \left(\int_{V(F)}\psi\left(\left\langle \frac{\xi}{a},u \right\rangle -\frac{Q(\xi)Q(u)}{9z^2[a]}+\mathfrak{c}(u,v)\right)f(u)du\right)dv \frac{
\chi_Q(a)d^\times a}{\{a\}^{d/2-1}} \\
    &=\zeta(1)^3\left(\prod_{i=1}^3|\xi_i|^{1-d_i/2}\right) \int_{ (F^\times)^3} \overline{\psi}(z^2\mathfrak{r}(a^{-1} \widetilde{\xi})) \Psi_{\xi}\left(a,-\tfrac{Q(\xi)}{z^2},f\right) d a 
\end{align}
for $z\in F^\times$. Here we have replaced $d^\times a$ by $\tfrac{\zeta(1)^3da}{\{a\}^1}$. 

Let
\begin{align} \label{Gxi}
    G(\xi):=\left(\frac{|\xi_1 \otimes \xi_2 \otimes \xi_3|}{|Q(\xi)|}\right)^{1/2}.
\end{align}
This subsection is devoted to proving the following:

\begin{Prop} \label{prop:is:finite} 
Let $(z,\xi)\in F^\times\times Y^{\mathrm{ani}}(F)$.  For $\tfrac{1}{2}>\varepsilon>0$ sufficiently small, we have
\begin{align*}
    & \Bigg| \int_{ (F^\times)^3} \overline{\psi}(z^2\mathfrak{r}(a^{-1} \widetilde{\xi})) \Psi_{\xi}\left(a,-\tfrac{Q(\xi)}{z^2},f\right) d a \Bigg| \\
    &\ll_{\varepsilon,\varepsilon',N,f}\left\{
    \begin{array}{ll}
     (|z|G(\xi))^{\min(d_1,d_2,d_3)-2-2\varepsilon}    & \textrm{ if } |z|G(\xi)\le N ,\\
     (|z|G(\xi))^{-\varepsilon}+\min(|\xi_1|,|\xi_2|,|\xi_3|)^{2\varepsilon}|Q(\xi)|^{-\varepsilon} (|z|G(\xi))^{-\varepsilon'}& \textrm{ if } |z|G(\xi)> N
    \end{array}\right.
\end{align*}
for some $\varepsilon'>0$ depending on $\varepsilon$ and $N\in\zz_{\ge 1}$ depending only on $f,\psi$; if $F$ is archimedean, $N=1$.
\end{Prop}
We deduce the following bound directly from Proposition \ref{prop:is:finite}, \eqref{eq:changeaxi}, and Lemma \ref{lem:finish} .
\begin{Cor}\label{cor:bound}
Let $\xi\in Y^{\mathrm{ani}}(F)$. For $\tfrac{1}{2}>\varepsilon>0$ sufficiently small, we have
\begin{align*}
    &\int_{F^\times} \Bigg| \int_{ (F^\times)^3} \overline{\psi}\left(z^2\mathfrak{r}(a)\right)  \left(\int_{Y(F)}\psi\left(\left\langle \frac{\xi}{a},y \right\rangle -\frac{Q(\xi)Q(y)}{9z^2[a]}\right)f(y)d\mu(y)\right)\frac{
\chi_Q(a)d^\times a}{\{a\}^{d/2-1}} \Bigg|d^\times z\\
& \ll_{\varepsilon,f}\left(\prod_{i=1}^3 |\xi_i|^{1-d_i/2}\right)\left(1+\min(|\xi_1|,|\xi_2|,|\xi_3|)^{2\varepsilon}|Q(\xi)|^{-\varepsilon}\right).
\end{align*}
\qed
\end{Cor}
\noindent By the smoothness of the Weil representation, Corollary \ref{cor:bound} implies the integral in \eqref{eq:finalform} defines a continuous function in $g\in \Omega$.  When checking this point in the archimedean case it is helpful to recall our conventions regarding asymptotic notation explained in \S \ref{ssec:asymp}.

The proof of Proposition \ref{prop:is:finite} will involve several reductions relying on results proved in \S \ref{sec:absolutebd}.  Theorem \ref{thm:absolutebdappendix} implies the following corollary:
\begin{Cor}
\label{cor:absolutebd}
Suppose $d_i=\dim V_i>2$ for all $1\le i\le 3$. Let $f\in\mathcal{S}(V(F))$. Given $\frac{1}{2}>\varepsilon>0$ one has
\begin{align*}
    \int_{(F^\times)^3}\left|\Psi_\xi(a,b,f) \right| d a
\ll_{\varepsilon,f} \min\left(1,\left|\frac{\xi_1\otimes\xi_2\otimes\xi_3}{b}\right|\right)^{\min(d_1,d_2,d_3)/2-1-\varepsilon}
\end{align*}
as a function of $(b,\xi)\in F \times Y^{\mathrm{ani}}(F)$.   Here by convention $\min\left(1,\left|\frac{\xi_1 \otimes \xi_2 \otimes \xi_3}{b}\right|\right)=1$ if $b=0.$ \qed
\end{Cor}
   
By Corollary \ref{cor:absolutebd}, for any $N \in \zz_{\geq 1},$ we have for $|z|G(\xi)\le N$
\begin{align} \label{hana}
      \int_{ (F^\times)^3} \Bigg|\Psi_{\xi} \left(a,-\tfrac{Q(\xi)}{z^2},f\right)  \Bigg|da  \ll_{\varepsilon,f,N}  (|z|G(\xi))^{\min(d_1,d_2,d_3)-2-2\varepsilon}.
\end{align}
Thus to prove Proposition \ref{prop:is:finite}, it suffices to bound the integral 
\begin{align} \label{uno}
    & \Bigg| \int_{ (F^\times)^3} \overline{\psi}(z^2\mathfrak{r}(a^{-1}\widetilde{\xi}))\Psi_{\xi} \left(a,-\tfrac{Q(\xi)}{z^2},f\right)  da \Bigg| 
\end{align}
for $N<|z|G(\xi).$

To proceed further, we need the following corollary of Theorem \ref{thm:byproductappendix}:

 \begin{Cor}\label{cor:byproduct}
 Suppose $|\xi_1\otimes\xi_2\otimes \xi_3|>|b|>0$. If $d_i>2$ for all $i,$ given $\alpha>0,$ there exists $\beta>0$ such that
 \begin{align*}
    & \int_{ |a |\ge \left|\frac{\xi_1\otimes \xi_2\otimes \xi_3}{b}\right|^\alpha}|\Psi_\xi(a,b,f)| da \ll_{\alpha,\beta,f}  \left|\frac{\xi_1\otimes \xi_2\otimes \xi_3}{b}\right|^{-\beta}.
\end{align*}
\qed
\end{Cor}

By Corollary \ref{cor:byproduct}, we have for $N<|z|G(\xi),$
\begin{align} \label{dul}
    &\int_{ |a |\ge (|z|G(\xi))^{2\alpha}}\left|\Psi_\xi\left(a,-\tfrac{Q(\xi)}{z^2},f\right)\right| da  \ll_{\alpha,\beta,f} (|z|G(\xi))^{-2\beta}.
\end{align}
Thus to bound \eqref{uno} (and hence prove Proposition \ref{prop:is:finite}), it suffices to bound
\begin{align} \label{dos}
   \left|\int_{|a|\le (|z|G(\xi))^{2\alpha}}\overline{\psi}(z^2\mathfrak{r}(a^{-1}\widetilde{\xi}))\Psi_{\xi}\left(a,-\tfrac{Q(\xi)}{z^2},f\right)da\right|
\end{align}
for $N<|z|G(\xi)$ and $1/6>\alpha>0$ sufficiently small.

 Over the domain $|a|\le (|z|G(\xi))^{2\alpha},$ we have
\begin{align} \label{diff:bound}
    \left|\frac{[a]Q(
    \xi)Q(u)}{9z^2[\widetilde{\xi}]
    }\right|\le |9|^{-1}(|z|G(\xi))^{6\alpha-2}|Q(u)|.
\end{align}
Assume $F$ is nonarchimedean.
Since $f$ has compact support and $|Q(u)|\ll |u|^2,$ we can choose $N$
sufficiently large (depending on $f,\psi$) such that 
\begin{align*}
    \overline{\psi}\left(\frac{[a]Q(\xi)Q(u)}{9z^2[\widetilde{\xi}]}\right)=1
\end{align*}
provided $u \in \mathrm{supp} (f),$  $|a|\le (|z|G(\xi))^{2\alpha},$ and $N<|z|G(\xi)$. In particular, under these assumptions 
$$
\Psi_{\xi}\left(a,-\tfrac{Q(\xi)}{z^2},f\right)=\Psi_{\xi}(a,0,f).
$$

In the archimedean case, we choose $N=1$. 
Observe that 
$$
\frac{\partial}{\partial b}\Psi_{\xi}(a,b,f)=c\frac{[a]}{9[\widetilde{\xi}]}\Psi_{\xi}(a,b,Qf)
$$
for some $c \in F$ depending on $\psi,$ and $Qf \in \mathcal{S}(V(F)).$  Thus the function
    \begin{align*}
        \widetilde{\Psi}_{\xi}(b,z,f):=\int_{|a|\le (|z|G(\xi))^{2\alpha}}\overline{\psi}(z^2\mathfrak{r}(a^{-1}\widetilde{\xi}))\Psi_{\xi}(a,b,f)da
    \end{align*}
    is differentiable and Lipschitz in $b$ by Corollary \ref{cor:absolutebd} and the Leibniz integral rule. 
    Let $C^{0,1}(F)$ denote the H\"older space of Lipschitz continuous functions on $F$. We bound the Lipschitz constant by bounding the derivative in $b$ using Corollary \ref{cor:absolutebd}.  This yields
    \begin{align*}
        \norm{\widetilde{\Psi}_{\xi}(\cdot,z,f)}_{C^{0,1}(F)}:=\mathrm{sup}_{b_1 \neq b_2 \in F}\frac{\left|\widetilde{\Psi}_{\xi}(b_1,z,f)-\widetilde{\Psi}_{\xi}(b_2,z,f) \right|}{|b_1-b_2|^{\frac{1}{[F:\rr]}}}\ll_f \left(\frac{(|z|G(\xi))^{6\alpha}}{|\xi_1\otimes\xi_2\otimes\xi_3|}\right)^{\frac{1}{[F:\rr]}}.
    \end{align*}
We have 
    \begin{align} \label{set} \begin{split}
        & \left|\widetilde{\Psi}_{\xi}\left(-\tfrac{Q(\xi)}{z^2},z,f\right)-\widetilde{\Psi}_{\xi}(0,z,f)\right|
        \ll_f (|z|G(
        \xi))^{\frac{6\alpha-2}{[F:\rr]}}  \end{split}
    \end{align}
for $1<|z|G(\xi)$.
Thus in either the nonarchimedean or archimedean case to bound \eqref{dos} (and hence prove Proposition \ref{prop:is:finite}), it suffices to bound
\begin{align}\label{tres}
    \left|\int_{|a|\le (|z|G(\xi))^{2\alpha}}\overline{\psi}(z^2\mathfrak{r}(a^{-1}\widetilde{\xi}))\Psi_{\xi}(a,0,f)da\right| 
\end{align}
for $N<|z|G(\xi)$ and $1/6>\alpha>0$ sufficiently small.

    Lastly, we require the following consequence of Theorem \ref{thm:b=0appendix}:
\begin{Cor}\label{cor:b=0}
For $\xi \in Y^{\mathrm{ani}}(F),$ if $d_i>2$ for all $i,$ we have 
\begin{align*}
&|\Psi_\xi(a,0,f)|\ll_f \{a\}^{d/2-2}
    \sum_{\sigma\in C_3}\max\left(1,|a_{\sigma(1)}|\right)^{1-d_{\sigma(1)}/2}\max\left(1,|a|\right)^{1-d_{\sigma(2)}/2-d_{\sigma(3)}/2} \ll 1.  
 \end{align*}
 \qed
\end{Cor}
\noindent Here $S_3$ is the permutation group on $\{1,2,3\}$ and $C_3\leq S_3$ is the order-$3$ subgroup generated by the permutation $(123)$.
By Corollary \ref{cor:b=0} and a direct computation, for $1<G(\xi)|z|$ we have
\begin{align}\label{eq:cutoff0}
&\int_{|a|\le (|z|G(\xi))^{2\alpha},|a_{\sigma(1)}|<(|z|G(\xi))^{-2\alpha}}|\Psi_{\xi}(a,0,f)|d a\ll_f (|z|G(\xi))^{-2\alpha(d_{\sigma(1)}/2-1)}
\end{align} 
for all $\sigma\in C_3$. This combined with the following bound \eqref{tres} complete the proof of Proposition \ref{prop:is:finite}.

\begin{Lem}\label{lem:uniformb2}  For $\alpha>0$ sufficiently small, there exist $\frac{1}{2}>\varepsilon>\varepsilon'>0$ such that
\begin{align*}
   &\left|\int_{(|z|G(\xi))^{-2\alpha}\le |a_i|\le (|z|G(\xi))^{2\alpha}}\overline{\psi}(z^2\mathfrak{r}(a^{-1}\widetilde{\xi}))\Psi_\xi(a,0,f)da\right|
   \ll_{\varepsilon,\varepsilon',f}
   \frac{\min(|\xi_1|,|\xi_2|,|\xi_3|)^{2\varepsilon}}{|Q(\xi)|^{\varepsilon}}(|z|G(\xi))^{2(\varepsilon'-\varepsilon)}.
\end{align*}
\end{Lem}

\subsection{Proof of Lemma \ref{lem:uniformb2}}

    Suppose $F=\rr$. We restrict the integral over the domain where all $a_i>0$ as an analogous argument works for the other connected component of $(\rr^\times)^3$. In particular, $\chi_Q(a\tilde{\xi})=1$. We will apply the van der Corput lemma, and for this purpose we first prove the following:
\begin{Lem}\label{lem:vandercorputarch}
  Let $g$ be a bounded continuous function on $\rr_{>0}^3$. For $c>1,$ we have
\begin{align*}
\left|\int_{c^{-1}\le a_i\le c}g(a)\Psi_{\xi}(a,0,f)da\right|
&\ll_f \sup_{c^{-1}<a_i\le c} \left|\int_{c^{-1}}^{a_1}\int_{c^{-1}}^{a_2}\int_{c^{-1}}^{a_3} g(r)dr\right|.
\end{align*}
\end{Lem}
\begin{proof}

First note that $\Psi_{\xi}(a,0,f)$ is smooth as a function of $a$ by Lemma \ref{lem:quad:esti}. By integration by parts over $a_1,$ we have
\begin{align*}
    &\int_{c^{-1}\le a_i\le c}g(a)\Psi_{\xi}(a,0,f)da\\
    &=\int_{c^{-1}\le a_3\le c}\int_{c^{-1}\le a_2\le c}\left(\int_{c^{-1}}^{c} g(a_1,a_2,a_3)da_1\right) \Psi_{\xi}(c,a_2,a_3,0,f)da_2da_3\\
    &-\int_{c^{-1}\le a_i\le c}\left(\int_{c^{-1}}^{a_1} g(r_1,a_2,a_3)dr_1\right) \partial_{a_1}\Psi_{\xi}(a_1,a_2,a_3,0,f)da_1da_2da_3.
\end{align*}
Applying integration by parts over $a_2$ and $a_3$ to these two integrals, we see the original integral is bounded by 
\begin{align*}
    \sup_{c^{-1}<a_i\le c} \left|\int_{c^{-1}}^{a_i}\int_{c^{-1}}^{a_2}\int_{c^{-1}}^{a_3} g(r)dr\right|
\end{align*}
times the sum of following terms and analogues that can be treated at the same time by symmetry:
    \begin{align*}
        &\left|\Psi_\xi(c,c,c,0,f)\right|, \quad 
    \int_{c^{-1}}^{c}|\partial_{a_{3}}\Psi_\xi(c,c,a_{3},0,f)|da_3, \quad 
        \int_{c^{-1}}^{c}\int_{c^{-1}}^{c}|\partial_{a_{2}}\partial_{a_{3}}\Psi_\xi(c,a_{2},a_{3},0,f)|da_{2}da_{3},\\ &
    \int_{c^{-1}}^{c}\int_{c^{-1}}^{c}\int_{c^{-1}}^{c}|\partial_{a_1}\partial_{a_2}\partial_{a_3}\Psi_\xi(a_1,a_2,a_3,0,f)|da_1da_2da_3.
\end{align*}
These are all bounded by a constant that is continuous in $f$ by Corollary \ref{cor:b=0} and the fact that the following terms 
\begin{align*}
    & \int_{c^{-1}}^c  \{a\}^{d/2-2}
    \sum_{\sigma\in C_3}\max\left(1,a_{\sigma(1)}\right)^{1-d_{\sigma(1)}/2}\max\left(1,|a|\right)^{1-d_{\sigma(2)}/2-d_{\sigma(3)}/2} da_3\Bigg|_{a_1=a_2=c},\\
    & \int_{c^{-1}}^c \int_{c^{-1}}^c  \{a\}^{d/2-2}
    \sum_{\sigma\in C_3}\max\left(1,a_{\sigma(1)}\right)^{1-d_{\sigma(1)}/2}\max\left(1,|a|\right)^{1-d_{\sigma(2)}/2-d_{\sigma(3)}/2} d{a_2}da_3\Bigg|_{a_1=c},\\
    &\int_{c^{-1}}^c \int_{c^{-1}}^c  \int_{c^{-1}}^c  \{a\}^{d/2-2}
    \sum_{\sigma\in C_3}\max\left(1,a_{\sigma(1)}\right)^{1-d_{\sigma(1)}/2}\max\left(1,|a|\right)^{1-d_{\sigma(2)}/2-d_{\sigma(3)}/2} d{a_1}d{a_2}da_3
\end{align*}
are bounded by a constant independent of $c$.

\end{proof}

By Lemma \ref{lem:vandercorputarch}, the integral in the statement of Lemma \ref{lem:uniformb2} is bounded in absolute value by $O_f(1)$ times
\begin{align}\label{eq:intbypart}
    \sup \left|\int_{(|z|G(\xi))^{-2\alpha}}^{r_1}\int_{(|z|G(\xi))^{-2\alpha}}^{r_2}\int_{(|z|G(\xi))^{-2\alpha}}^{r_3} \overline{\psi}(z^2\mathfrak{r}(a^{-1}\widetilde{\xi}))da\right|,
\end{align}
where the supremum is taken over $\{r_i:{(|z|G(\xi))^{-2\alpha}\le r_i\le (|z|G(\xi))^{2\alpha}}\}$. Changing variables, this is equal to the supremum over $\{r_i:{(|z|G(\xi))^{-2\alpha}\le r_i\le (|z|G(\xi))^{2\alpha}}\}$ of 
\begin{align}\label{eq:intbypart2}
\left|\int_{[0,1]^3} \overline{\psi}\left(z^2\mathfrak{r}\left(\frac{\widetilde{\xi}}{
(r-(|z|G(\xi))^{-2\alpha})a+(|z|G(\xi))^{-2\alpha}}\right)\right)da\right|\prod_{i=1}^3 \left(r_i-(|z|G(\xi))^{-2\alpha}\right).
\end{align}
Here $ (r-(|z|G(\xi))^{-2\alpha})a+(|z|G(\xi))^{-2\alpha}$ is shorthand for the vector
\begin{align*}
   \left((r_i-(|z|G(\xi))^{-2\alpha})a_i+(|z|G(\xi))^{-2\alpha}\right) \in (\rr^\times)^3.
\end{align*}

Note that
\begin{align} \label{q2:id}
    \mathfrak{r}(a^{-1}\widetilde{\xi})=\sum_{\sigma\in C_3}\left(\frac{\widetilde{\xi}_{\sigma(2)}\widetilde{\xi}_{\sigma(3)}}{\widetilde{\xi}_{\sigma(1)}}\frac{a_{\sigma(1)}}{a_{\sigma(2)}a_{\sigma(3)}}+\frac{2\widetilde{\xi}_{\sigma(1)}}{a_{\sigma(1)}}\right).
\end{align}
Therefore, 
\begin{align} \label{der:comp} \begin{split}
    &\partial_{a_1}\partial_{a_2}^2\partial_{a_3}^2\mathfrak{r}\left(\frac{\widetilde{\xi}}{
(r-(|z|G(\xi))^{-2\alpha})a+(|z|G(\xi))^{-2\alpha}}\right)\\
    &=4\frac{\widetilde{\xi}_{2}\widetilde{\xi}_{3}}{\widetilde{\xi}_{1}}(r_1-(|z|G(\xi))^{-2\alpha})\prod_{i=2}^3\frac{(r_i-(|z|G(\xi))^{-2\alpha})^2}{\left((r_i-(|z|G(\xi))^{-2\alpha})a_i+(|z|G(\xi))^{-2\alpha}\right)^3}\\
    & \ge 4\frac{\widetilde{\xi}_{2}\widetilde{\xi}_{3}}{\widetilde{\xi}_{1}}(r_1-(|z|G(\xi))^{-2\alpha})\prod_{i=2}^3\frac{(r_i-(|z|G(\xi))^{-2\alpha})^2}{r_i^3} \end{split}
\end{align}
if $a_2,a_3\le 1$. Thus, by the van der Corput Lemma \cite[Theorem 1.4]{CCW}, \eqref{eq:intbypart2} is dominated by
\begin{align} \label{after:vdC}
    |z|^{-2\varepsilon}\left(4\frac{\widetilde{\xi}_{2}\widetilde{\xi}_{3}}{\widetilde{\xi}_{1}}(r_1-(|z|G(\xi))^{-2\alpha})\prod_{i=2}^3\frac{(r_i-(|z|G(\xi))^{-2\alpha})^2}{r_i^3}\right)^{-\varepsilon}\prod_{i=1}^3 \left(r_i-(|z|G(\xi))^{-2\alpha}\right)
\end{align}
for some $1/2>\varepsilon>0$. Therefore,
the supremum over $\{r_i:{(|z|G(\xi))^{-2\alpha}\le r_i\le (|z|G(\xi))^{2\alpha}}\}$ of \eqref{after:vdC} is dominated by
\begin{align}\label{eq:vander}
    |z|^{-2\varepsilon}(|z|G(\xi))^{2\alpha(3+\varepsilon)}|\xi_1|^{\varepsilon}|\xi_2|^{-\varepsilon}|\xi_3|^{-\varepsilon}.
\end{align} 
Thus the same bound holds for \eqref{eq:intbypart}, and we deduce the lemma by symmetry.

 Suppose $F=\cc$. Taking a change of variables in $z,u,v$ and replacing $f$ by another Schwartz function if necessary, we may assume $\psi(t)=\psi_\rr(2\mathrm{Re}(t)):=e^{-4\pi i\mathrm{Re}(t)}$. Write $z=|z|^{1/2}e^{i\phi}$. Using polar coordinates, we are to study
 \begin{align} \label{s} \begin{split}
     &\left|\int_{0\le \theta_i\le 2\pi }\int_{(|z|G(\xi))^{-\alpha}\le s_i\le (|z|G(\xi))^{\alpha}}\overline{\psi}(|z|e^{i2\phi}\mathfrak{r}(s^{-1}e^{-i\theta}\widetilde{\xi}))\Psi_\xi(se^{i\theta},0,f)[s]dsd\theta\right|\\
     &\le 
     \int_{0\le \theta_i\le 2\pi } \left|\int_{(|z|G(\xi))^{-\alpha}\le s_i\le (|z|G(\xi))^{\alpha}}\overline{\psi}(|z|e^{i2\phi}\mathfrak{r}(s^{-1}e^{-i\theta}\widetilde{\xi}))\Psi_\xi(se^{i\theta},0,f)[s]ds\right|d\theta.
     \end{split}
 \end{align}
 Now for a fixed $\theta$ and $\phi,$ we have 
\begin{align*}
    &\overline{\psi}(|z|e^{i2\phi}\mathfrak{r}(s^{-1}e^{-i\theta}\widetilde{\xi}))\\
    &=\overline{\psi}_\rr\left(2 |z|\sum_{\sigma\in C_3}\left(\frac{\widetilde{\xi}_{\sigma(2)}\widetilde{\xi}_{\sigma(3)}}{\widetilde{\xi}_{\sigma(1)}}\frac{s_{\sigma(1)}}{s_{\sigma(2)}s_{\sigma(3)}}\cos(2\phi +\theta_{\sigma(1)}-\theta_{\sigma(2)}-\theta_{\sigma(3)})+\frac{2\widetilde{\xi}_{\sigma(1)}}{s_{\sigma(1)}}\cos(2\phi-\theta_{\sigma(1)})\right)\right).
\end{align*}
We observe that Corollary \ref{cor:b=0} implies that $|\Psi_\xi(se^{i\theta},0,f)[s]| \ll_f 1$.  Hence arguing as in Lemma \ref{lem:vandercorputarch}, we have the absolute value of the inner integral over $s$ in \eqref{s} is bounded by a constant (continuous in $f$) times the supremum over $\{r_i:{(|z|G(\xi))^{-\alpha}\le r_i\le (|z|G(\xi))^{\alpha}}\}$ of 
\begin{align*}
    \left|\int \overline{\psi}_\rr\left(2 |z|\sum_{\sigma\in C_3}\left(\frac{\widetilde{\xi}_{\sigma(2)}\widetilde{\xi}_{\sigma(3)}}{\widetilde{\xi}_{\sigma(1)}}\frac{s_{\sigma(1)}}{s_{\sigma(2)}s_{\sigma(3)}}\cos(2\phi +\theta_{\sigma(1)}-\theta_{\sigma(2)}-\theta_{\sigma(3)})+\frac{2\widetilde{\xi}_{\sigma(1)}}{s_{\sigma(1)}}\cos(2\phi-\theta_{\sigma(1)})\right)\right)ds\right|
\end{align*}
where the integral is taken over $(|z|G(\xi))^{-\alpha}\le s_i\le r_i$. Arguing as in the real case, the supremum is bounded by
\begin{align*}
    |z|^{-\varepsilon}|\cos(2\phi +\theta_{1}-\theta_{2}-\theta_{3})|^{-\varepsilon}(|z|G(\xi))^{\alpha(3+\varepsilon)}|\xi_1|^{\varepsilon/2}|\xi_2|^{-\varepsilon/2}|\xi_3|^{-\varepsilon/2}.
\end{align*}
Using the fact that $|\cos|^{-\varepsilon}(x)$ is locally integrable for $1>\varepsilon,$ the desired bound follows by symmetry.  This completes the proof of Lemma \ref{lem:uniformb2} in the archimedean case.

For the nonarchimedean case, we will apply     
 the following van der Corput Lemma:
\begin{Prop}\label{prop:vandercorputnonarch}
 Let $F$ be a nonarchimedean local field of characteristic $0$. For any nonzero multi-index $\beta\in \zz^n_{\ge 0},$ there exists $\varepsilon>0,$ $N\in \zz_{>0},$ that depend on  $\psi,n,\beta,$ such that for any formal power series $f(x)\in F[[x_1,\ldots,x_n]]$ that converges on $\mathcal{O}_F^n$ and satisfies
 \begin{align*}
     |\partial^\beta f (x)|\ge 1 \textrm{ for all } x\in \mathcal{O}_F^n,
 \end{align*}
 we have 
  \begin{align*}
     \left|\int_{\mathcal{O}_F^n} 
     \psi(yf(x))dx\right|\ll_{\psi,n,\beta,F} \max(1,\norm{f-f(0)})^{N}|y|^{-\varepsilon}
 \end{align*}
 for all $y\in F$. Here $\norm{f}$ is the supremum of the norms of the coefficients of $f$.
\end{Prop}

\begin{proof}
    This is a consequence of  \cite[Proposition 3.3, Proposition 3.6]{Cluckers} and their proofs.
\end{proof}

We now complete the proof of Lemma \ref{lem:uniformb2} in the nonarchimedean case.  
 Choose $n\in \zz_{>0}$ such that $(1+\varpi^n\mathcal{O}_F)^3<\mathrm{ker}\chi_Q$. For $r:=(r_1,r_2,r_3)\in (F^{\times})^3$ and $\ell \in \zz_{>0},$ let 
\begin{align*}
    B_{r,\ell}:=\prod_{i=1}^3\left(r_i+\varpi^{\ell}\mathcal{O}_F\right).
\end{align*}
 Choose $m\ge n>0$ such that $\psi\left(\varpi^m\left\langle \frac{\xi}{\widetilde{\xi}},u\right\rangle\right)=1$ for all $\xi\in Y^{\mathrm{ani}}(F)$ and $u\in\mathrm{supp}(f)$. Let $\ell$ be any integer such that $|\varpi^\ell|\le  |\varpi^{m}|(|z|G(\xi))^{-2\alpha}$. Note that for $|a_i|\ge (|z|G(\xi))^{-2\alpha}$ and $u\in \varpi^{\ell}\mathcal{O}_F^3,$ we have $|u/a|\le |\varpi^m|\le |\varpi^n|,$ and thus
 $$\chi_Q((a+u)\widetilde{\xi})=\chi_Q(a\widetilde{\xi}+u\widetilde{\xi})=\chi_Q(a\widetilde{\xi}).$$
 
    Then we can write
\begin{align*}
    &\int_{(|z|G(\xi))^{-2\alpha}\le |a_i|\le (|z|G(\xi))^{2\alpha}}\overline{\psi}(z^2\mathfrak{r}(a^{-1}\widetilde{\xi}))\Psi_\xi(a,0,f)da\\
    &=\sum_{r} \Psi_\xi(r,0,f)\int_{B_{r,\ell}} \overline{\psi}(z^2\mathfrak{r}(a^{-1}\widetilde{\xi}))da\\
    &=\sum_{r} \Psi_\xi(r,0,f)|\varpi^\ell |^3\int_{\mathcal{O}_F^3} \overline{\psi}(z^2\varpi^{-\ell}\mathfrak{r}(( a+r/\varpi^{\ell})^{-1}\widetilde{\xi}))da
\end{align*}
where $r$ runs though a set of  representatives of  $(|z|G(\xi))^{-2\alpha}\le |r_i|\le (|z|G(\xi))^{2\alpha}$ modulo $\varpi^{\ell}\mathcal{O}_F^3$. 
In particular, $|r_i|>|\varpi^{\ell}|$ for all $i$. Therefore, for $a\in\mathcal{O}_F^3$
\begin{align*}
    \frac{1}{a_i+r_i/\varpi^{\ell}}=\sum_{n=0}^\infty (-1)^n\left(\frac{\varpi^{\ell}}{r_i}\right)^{n+1}  a_i^{n}.
\end{align*}
With notation as in Proposition \ref{prop:vandercorputnonarch}, using \eqref{q2:id} one has    
\begin{align*}
\begin{split}
    &\norm{\mathfrak{r}(( a+r/\varpi^{\ell})^{-1}\widetilde{\xi})-\mathfrak{r}((r/\varpi^{\ell})^{-1}\widetilde{\xi})}\\
    &\leq \max_{\sigma\in S_3}\left( \frac{|\xi_{\sigma(2)}||\xi_{\sigma(3)}|}{|\xi_{\sigma(1)}|}\frac{|r_{\sigma(1)}\varpi^{2\ell}|}{|r_{\sigma(2)}r_{\sigma(3)}^2|}, \frac{|\xi_{\sigma(2)}||\xi_{\sigma(3)}|}{|\xi_{\sigma(1)}|}\frac{|\varpi^{\ell}|^{2}}{|r_{\sigma(2)}r_{\sigma(3)}|}, \frac{|\xi_{\sigma(1)}||\varpi^\ell|^{2}}{|r_{\sigma(1)}|^{2}}\right)\\
    &\leq \left(\max_{\sigma\in C_3} \frac{|\xi_{\sigma(2)}||\xi_{\sigma(3)}|}{|\xi_{\sigma(1)}|}\right) \frac{|r|}{|\varpi^\ell|}\\&\le \left(\max_{\sigma\in C_3} \frac{|\xi_{\sigma(2)}||\xi_{\sigma(3)}|}{|\xi_{\sigma(1)}|}\right) \frac{(|z|G(\xi))^{2\alpha}}{|\varpi^\ell|}.
    \end{split}
\end{align*}
Moreover, for all $a\in \mathcal{O}_F^3,$ we have
\begin{align*}
    \left|\partial_{a_1}\partial_{a_2}^2\partial_{a_3}^2\mathfrak{r}(( a+r/\varpi^{\ell})^{-1}\widetilde{\xi})\right| = |4|\frac{|\xi_2||\xi_3|}{|\xi_1|}\frac{|\varpi^\ell|^6}{
    |r_2r_3|^{3}}\ge |4|\frac{|\xi_2||\xi_3|}{|\xi_1|}(|z|G(\xi))^{-12\alpha}|\varpi^{\ell}|^6.
 \end{align*}
Consequently, by symmetry and Proposition \ref{prop:vandercorputnonarch}, we obtain
\begin{align}\label{eq:nonarchbd}
\begin{split}
    \left|\int_{\mathcal{O}_F^3} \overline{\psi}(z^2\varpi^{-\ell}\mathfrak{r}(( a+r/\varpi^{\ell})^{-1}\widetilde{\xi}))da\right|
    &\ll \left(\max_{\sigma\in C_3} \frac{|\xi_{\sigma(2)}||\xi_{\sigma(3)}|}{|\xi_{\sigma(1)}|}\right)^{-\varepsilon}\frac{(|z|G(\xi))^{14\alpha N+12\alpha\varepsilon}}{|z|^{2\varepsilon}|\varpi^\ell|^{7N+5\varepsilon}}.
\end{split}
\end{align}
Now assume $\ell$ is chosen so that $|\varpi^\ell|\ge q^{-1}|\varpi^{m}|(|z|G(\xi))^{-2\alpha}$. Since $|\Psi_{\xi}(r,0,f)|\ll_f 1$ by Corollary \ref{cor:b=0}, using \eqref{eq:nonarchbd} we obtain
\begin{align*}
    &\left|\sum_{r} \Psi_\xi(r,0,f)|\varpi^\ell|^3\int_{\mathcal{O}_F^3} \overline{\psi}(z^2\varpi^{-\ell}\mathfrak{r}(( a+r/\varpi^{\ell})^{-1}\widetilde{\xi}))da\right|\\
    &\ll_f  |z|^{-2\varepsilon}\left(\max_{\sigma\in C_3} \frac{|\xi_{\sigma(2)}||\xi_{\sigma(3)}|}{|\xi_{\sigma(1)}|}\right)^{-\varepsilon}(|z|G(\xi))^{-6\alpha+14\alpha N+12\alpha\varepsilon-2\alpha(-7N-5\varepsilon)}.
\end{align*}
We thus obtain Lemma \ref{lem:uniformb2} in the nonarchimedean case by choosing $\alpha$ small.\qed

\subsection{Preliminary truncation}\label{ssec:truncation}
Let $f\in\mathcal{S}(V(F))$. For $\xi \in Y^{\mathrm{ani}}(F)$ and $W \in \mathcal{S}(x_0 \SL_2^3(F)),$ recall 
\begin{align}\label{truncatenew}\begin{split}
f_{\xi,W}:x_0\SL_2^3(F) &\lto \cc\\
x_0g &\longmapsto W(x_0g)\rho\left(g \right)f(\xi), 
\end{split}
\end{align} 
and $
f_{\xi,W} \in \mathcal{S}(x_0\SL_2^3(F))<\mathcal{S}(X(F)).$  Our goal in this section is to prove the formula 
\eqref{FT:formula0} below for $\mathcal{F}_X(f_{\xi,W})(-x_0)$.

Applying Corollary \ref{cor:extend}, for $x\in X^\circ(F)$ we have that $\mathcal{F}_X(f_{\xi,W})(x)$ equals
\begin{align} \label{start0}  \begin{split}
\int_{F^\times}\psi(z^{-1})|z|^2 \Bigg(\int_{N_0(F) \backslash \SL_2^3(F)}  \overline{\psi}\left(z^{2} \mathrm{Pl}(x_0g)\wedge \mathrm{Pl}\left(x\right)\right) W(x_0g)\rho\left(g\right)f(\xi)d\dot{g}\Bigg) \frac{d^\times z}{\zeta(1)}. \end{split}
\end{align}
Here we have identified $\wedge^6 \mathbb{G}_a^6\tilde{\lto} \mathbb{G}_a$ via $e_1\wedge\cdots \wedge e_6\mapsto 1$ (see \eqref{Sppairing}).  The measure $d\dot{g}$ is normalized so that it coincides with the measure $dx$ in Corollary \ref{cor:extend} restricted to $x_0\SL_2^3(F)$  (see \eqref{eq:uptoscalar}).

Since we are studying $\mathcal{F}_{X}(f_{\xi,W})$ in a neighborhood of $-x_0$ in $-x_0\SL_2^3(F),$ we need only consider 
$$
x=-x_0g'
$$
for  $g' \in \SL_2^3(F)$ sufficiently close to $1$.  In this case, after a change of variables $g \mapsto gg'$ in \eqref{start0}, we have
\begin{align} \label{start}  \begin{split}
\mathcal{F}_X(f_{\xi,W})(x)&=\mathcal{F}_X((\rho(g')f)_{\xi,R(g')W})(-x_0), \end{split}
\end{align}
where $R(g')W(x_0g):=W(x_0gg')$.  This allows us to focus on computing $\mathcal{F}_X(f_{\xi,W})(-x_0)$ as long as we are able to control  its behavior as a function of $W$ and $f$.  

    Let 
$
    w:=\begin{psmatrix} 0 & 1 \\ -1 & 0 \end{psmatrix}\in \SL_2(F),
$
and denote again by $w$ the image of $w$ under the diagonal embedding $\SL_2(F) \to \SL_2^3(F)$. Let
$$
\Delta:F \lto F^3
$$
be the diagonal embedding. Let $B_2 \leq \SL_2$ be the Borel subgroup of upper triangular matrices and let $N_2 \leq B_2$ be its unipotent radical. By Bruhat decomposition, $x_0B_2^3(F)wN_2^3(F)$ is open in $\SL_2^3(F)$ with full measure. Therefore, there is a constant $k>0$ (independent of $\psi$) such that 
\begin{align} \label{eq:absoluteconst}
k|a_1a_2a_3|^2dtd^\times adb_0 =d\dot{\left(\begin{psmatrix} 1 & \Delta(t) \\ & 1 \end{psmatrix}\begin{psmatrix} a^{-1} & \\ & a\end{psmatrix}w \begin{psmatrix} 1 &   b_0\\ & 1 \end{psmatrix}\right)} \quad \textrm{ for } t\in F, a\in (F^\times)^3, b_0\in F^3.
\end{align}
Therefore, we have
\begin{align}\label{start1}
\begin{split}
    &\mathcal{F}_X(f_{\xi,W})(-x_0)\\
    &=\frac{k}{\zeta(1)}\int_{F^\times} \psi(z^{-1})|z|^2\Bigg(\int_{F \times (F^\times)^3 \times F^3} \overline{\psi}\left(z^{2} \mathrm{Pl}\left(x_0\begin{psmatrix} 1 & \Delta(t) \\ & 1 \end{psmatrix}\begin{psmatrix} a^{-1} & \\ & a\end{psmatrix}w \begin{psmatrix} 1 &   b_0\\ & 1 \end{psmatrix}\right)\wedge\mathrm{Pl}(-x_0)\right)\\& \times W\left(x_0\begin{psmatrix} 1 & \Delta(t) \\ & 1 \end{psmatrix}\begin{psmatrix} a^{-1} & \\ & a\end{psmatrix}w \begin{psmatrix} 1 &   b_0\\ & 1 \end{psmatrix}\right)
\rho\left(\begin{psmatrix} 1 & \Delta(t) \\ & 1 \end{psmatrix}\begin{psmatrix} a^{-1} & \\ & a\end{psmatrix} w\begin{psmatrix} 1&b_0 \\ & 1 \end{psmatrix} \right)f(\xi)\{a\}^2 dt d^\times a db_0 \Bigg)d^\times z.
\end{split}
\end{align}
As we will always restrict $W$ to the open Bruhat cell, we identify $W$ with a smooth function on $F\times (F^\times)^3\times F$ by writing 
\begin{align}\label{tab}
    W(t,a,b_0):= W\left(x_0\begin{psmatrix} 1 & \Delta(t) \\ & 1 \end{psmatrix}\begin{psmatrix} a^{-1} & \\ & a\end{psmatrix}w \begin{psmatrix} 1 &   b_0\\ & 1 \end{psmatrix}\right)
\end{align}
for $(t,a,b_0)\in F \times (F^\times)^3 \times F^3$.

By the formula for the Weil representation (see \cite[\S 3.1]{Getz:Liu:Triple} and the references therein), one has 
\begin{align} \label{Weil} \begin{split}
\rho&\left( \begin{psmatrix} 1 & \Delta(t) \\ & 1 \end{psmatrix}\begin{psmatrix} a^{-1} & \\ & a \end{psmatrix}w \begin{psmatrix} 1 & b_0\\ & 1 \end{psmatrix} \right)f(\xi)\\&=\frac{\gamma(Q)\chi_Q(a)}{\{a\}^{d/2}} \psi(tQ(\xi))
\int_{V(F)}\psi\left(\left\langle \frac{\xi}{a},u \right\rangle+\sum_{i=1}^3b_iQ_i(u)\right)f(u)du, \end{split}
\end{align}
where $b_0=(b_1, b_2, b_3).$ Here the Haar measure $du$ on $V(F)$ is normalized to be self-dual with respect to the pairing $\langle \cdot,\cdot\rangle$ and $\psi$.
On the other hand, we have
\begin{align*}
   & \mathrm{Pl}\left(x_0\begin{psmatrix} 1 & \Delta(t) \\ & 1 \end{psmatrix}\begin{psmatrix} a^{-1} & \\ & a \end{psmatrix}w \begin{psmatrix} 1 & b_0\\ & 1\end{psmatrix} \right)\wedge\mathrm{Pl}(-x_0) \\
   &=-\mathrm{Pl}\left(x_0\begin{psmatrix} 1 & \Delta(t) \\ & 1 \end{psmatrix}\begin{psmatrix} a^{-1} & \\ & a\end{psmatrix}w \begin{psmatrix} 1 & b_0\\ & 1\end{psmatrix} \gamma_0^{-1}\right)\wedge\mathrm{Pl}(I_6) \\
   &=-\mathrm{Pl}\left(x_0\begin{psmatrix} -a\Delta(t) & a^{-1}-\Delta(t)ab_0\\-a & -ab_0 \end{psmatrix} \gamma_0^{-1}\right)\wedge e_4 \wedge e_5 \wedge e_6.
\end{align*}
where $\gamma_0$ is defined as in \eqref{gamma0}.  
  Under the identification $e_1\wedge\cdots \wedge e_6\mapsto 1,$ this is
\begin{align*}
&-\det \begin{psmatrix}t\sum_{i=1}^3 a_ib_i -\sum_{i=1}^3a_i^{-1} & t\left(a_1-a_2\right)
     & t\left(a_1-a_3\right)
     \\
    a_2b_2-a_1b_1 &-a_1-a_2 & -a_1
    \\a_3b_3-a_1b_1& -a_1 &-a_1-a_3\end{psmatrix}=-3ta_1a_2a_3\sum_{i=1}^3b_i+\mathfrak{r}(a),
\end{align*}
where $\mathfrak{r}$ is defined as in \eqref{q2}.
 Combining this with \eqref{start1}, \eqref{tab} and \eqref{Weil}, we obtain
\begin{align*}
\begin{split}
\mathcal{F}_X&(f_{\xi,W})(-x_0)\\
&=\frac{\gamma(Q)k}{\zeta(1)}\int_{F^\times}\psi(z^{-1})|z|^2\Bigg( \int_{F\times (F^\times \times F)^3}  \psi\Bigg(3z^{2}t[a]\sum_{i=1}^3b_i-z^2\mathfrak{r}(a)+tQ(\xi)\Bigg)
      \\& \times 
\left(\int_{V(F)}\psi\left(\left\langle \frac{\xi}{a},u \right\rangle +\sum_{i=1}^3b_iQ_i(u_i)\right)f(u)du\right) W\left(t,a,b_0\right)\frac{\chi_Q(a)d^\times a}{\{a\}^{d/2-2}} db_0 dt\Bigg)d^\times z . \end{split}
\end{align*}

Changing variables $b_0=(b_1,b_2,b_3)\mapsto (b+v_1,b+v_2,b-v_1-v_2),$ this becomes
\begin{align*}
\frac{\gamma(Q)k|3|}{\zeta(1)}&\int_{F^\times}\psi(z^{-1})|z|^2\Bigg( \int_{F^2\times (F^\times)^3\times F^2}  \psi(9z^2t[a]b-z^2\mathfrak{r}(a)+tQ(\xi))
      \\& \times 
\left(\int_{V(F)}\psi\left(\left\langle \frac{\xi}{a},u \right\rangle +bQ(u)+\mathfrak{c}(u,v)\right)f(u)du\right)\\
&\times W\left(t,a,\Delta(b)+v'\right)
dv\frac{\chi_Q(a)d^\times a}{\{a\}^{d/2-2}}    dbdt\Bigg)d^\times z.
\end{align*}
Here we have used the notation in \eqref{cuv} and 
\begin{align} \label{v'}
    v':=(v_1,v_2,-v_1-v_2).
\end{align}
 Taking a change of variables $b\mapsto \frac{b}{9z^2[a]},$ we arrive at
 \begin{align} \label{FT:formula0} \begin{split}
\frac{\gamma(Q)k}{|3|\zeta(1)}&
\int_{F^\times}\psi(z^{-1})\Bigg( \int_{F^2\times (F^\times)^3\times F^2} \psi\left(tQ(\xi)+tb\right)  \overline{\psi}(z^2\mathfrak{r}(a))
      \\& \times
\left(\int_{V(F)}\psi\left(\left\langle \frac{\xi}{a},u \right\rangle +\frac{bQ(u)}{9z^2[a]}+\mathfrak{c}(u,v)\right)f(u)du\right)\\& \times W\left(t, a, \Delta\left(\frac{b}{9z^2[a]}\right)+v'\right)
dv \frac{\chi_Q(a)d^\times a}{\{a\}^{d/2-1}}    dbdt\Bigg) d^\times z. \end{split}
\end{align}

\subsection{Choice of $\mathcal{V}_{\mathcal{B}_1,\mathcal{B}_2}$} \label{ssec:B1}

Let $F_{>1}$ be the set $\varpi^{\zz_{<0}}$ (resp. $\rr_{>1}$) when $F$ is nonarchimedean (resp. archimedean). In this subsection, we specify our choice of $\mathcal{V}_{\mathcal{B}_1,\mathcal{B}_2}$ indexed by 
$$
\{(\mathcal{B}_1,\mathcal{B}_2) \in F_{>1}^2: |3\mathcal{B}_1| >\left|\mathcal{B}_2\right|^2\},
$$
and derive the formula \eqref{eq:beforeinv} for $\mathcal{F}_X(f_{\xi,\mathcal{V}_{\mathcal{B}_1,\mathcal{B}_2}})(-x_0g)$. Then we explain how Proposition \ref{prop:limit} follows from Proposition \ref{prop:tech}, which is stated below  and proved in \S \ref{ssec:formal}.  We point out that in the argument below, we only require $|3\mathcal{B}_1|>|\mathcal{B}_2|^2$ in the nonarchimedean case, but for uniformity we impose it in the archimedean case as well.

Recall from \eqref{Omega} that $\Omega:=\overline{K}_{1/4}$ is the closure of $K_{1/4}$ in $\SL_2^3(F)$. Suppose $F$ is archimedean. Choose functions $H,J\in \mathcal{S}(F)$ satisfying following conditions. 
\begin{enumerate}
    \item $H(u)=H(|u|^{\frac{1}{[F:\rr]}})$ and $J(u)=J(|u|^{\frac{1}{[F:\rr]}})$ for all $u\in F$.
    \item $H(0)=1$ and $\widehat{H}$, the Fourier transform of $H$, is compactly supported.
    \item The function $J$ is nonnegative and bounded by $1$, and
\begin{align*}
    J(u)=\begin{cases}
     1 & \textrm{ if } |u|\le 1,\\
     0 & \textrm{ if } |u| \ge 2.
    \end{cases}
\end{align*}
    \item The function $J$ satisfies the condition in Lemma \ref{lem:fin0} below.
\end{enumerate}
For $(u,\mathcal{B})\in F \times F^\times$, define $H_{\mathcal{B}}(u):=H\left(\frac{u}{\mathcal{B}}\right)$ and  $J_{\mathcal{B}}(u)=J\left(\frac{u}{\mathcal{B}} \right)$. Given $\mathcal{B}_1,\mathcal{B}_2\in F_{>1},$ define $\mathcal{V}_{\mathcal{B}_1,\mathcal{B}_2}\in \mathcal{S}(x_0\SL_2^3(F))$ using the coordinates \eqref{tab} by
\begin{align*}
    \mathcal{V}_{\mathcal{B}_1,\mathcal{B}_2}(t,a,b_0):=H_{\mathcal{B}_1}(t)\prod_{i=1}^3 J_{\log \mathcal{B}_2}(\log |a_i|^{\frac{1}{[F:\rr]}})J_{\mathcal{B}_2}(b_i).
\end{align*}
For $g=\begin{psmatrix}
m &n\\
x& y
\end{psmatrix}\in \Omega$, we have
\begin{align}\label{eq:ta} 
    R(g)\mathcal{V}_{\mathcal{B}_1,\mathcal{B}_2}(t,a,b_0)= H_{\mathcal{B}_1}(t-\tau_g(a,b_0))\prod_{i=1}^3 J_{\log \mathcal{B}_2}\left(\log |a_i(m_i+b_ix_i)|^{\frac{1}{[F:\rr]}}\right)J_{\mathcal{B}_2}\left( \frac{n_i+b_iy_i}{m_i+b_ix_i}\right)
\end{align}
if $m_i+b_ix_i \neq 0$ for all $i$, and $R(g)\mathcal{V}_{\mathcal{B}_1,\mathcal{B}_2}(t,a,b_0)=0$ otherwise.  Here 
\begin{align*}
    \tau_g(a,b_0):=\frac{1}{3}\sum_{i=1}^3\frac{x_i}{a_i^2(m_i+b_ix_i)}.
\end{align*}
The limit
\begin{align} \label{inf:cp}
 &R(g)\mathcal{V}_{\infty, \mathcal{B}_2}(a,b_0):=\lim_{|\mathcal{B}_1|\to\infty} R(g)\mathcal{V}_{\mathcal{B}_1,\mathcal{B}_2}(0,a,b_0)=
\prod_{i=1}^3 J_{\log \mathcal{B}_2}\left(\log \left|a_i(m_i+b_ix_i)\right|^{\frac{1}{[F:\rr]}}\right)J_{\mathcal{B}_2}\left( \frac{n_i+b_iy_i}{m_i+b_ix_i}\right)
\end{align}
converges pointwise.

Suppose $F$ is nonarchimedean. For $\mathcal{B}_1,\mathcal{B}_2\in F_{>1}$ and $g\in \SL_2^3(F)$, let
\begin{align*}
    \mathcal{V}_{\mathcal{B}_1,\mathcal{B}_2}(x_0g):=\begin{cases}
    1 & \textrm{ if } x_0g=
    x_0\begin{psmatrix}
    c^{-1} & c\Delta(u)\\
    & c
    \end{psmatrix}h, \textrm{ where } |u|\le |\mathcal{B}_1|, |\mathcal{B}_2|^{-1}< |c_i|\le |\mathcal{B}_2|, h\in \SL_2^3(\calo_F),\\
    0 & \textrm{ otherwise. }
    \end{cases}
\end{align*}
The conditions on $|u|$ and $|c_i|$ are independent of the choice of decomposition of $x_0g$ if
$|3\mathcal{B}_1|>\left|\mathcal{B}_2\right|^2.$ More explicitly, one can check the following lemma:
\begin{Lem} \label{lem:consist}
Suppose $F$ is nonarchimedean. Let $c,c' \in (F^\times)^3$ and $u,u' \in F.$ If 
$$
x_0\begin{psmatrix} c^{-1} & c\Delta(u)\\ & c\end{psmatrix}h=x_0\begin{psmatrix} c'^{-1} & c' \Delta(u')\\ & c'\end{psmatrix}h'
$$
for some $h,h' \in \mathrm{SL}_2^3(\mathcal{O}_F),$ then $|c_i|=|c'_i|$ for each $i.$  If in addition $
|3\mathcal{B}_1|>\left|\mathcal{B}_2\right|^2$ and $|\mathcal{B}_2|^{-1}  \leq|c_i|=|c'_i| \leq |\mathcal{B}_2|$ for each $i$, then $|u| \leq |\mathcal{B}_1|$ if and only if $|u'| \leq |\mathcal{B}_1|.$ \qed 
\end{Lem}

Because of Lemma \ref{lem:consist}, we henceforth assume that $|3\mathcal{B}_1|> |\mathcal{B}_2|^2.$ Clearly, $\mathcal{V}_{\mathcal{B}_1,\mathcal{B}_2}\in \mathcal{S}(x_0\SL_2^3(F))$ is right $\SL_2^3(\calo_F)$-invariant.
Observe that for $(t,a_i,b_i) \in F\times F^\times \times F$, one has
$$\begin{psmatrix} 1 & t\\ & 1 \end{psmatrix} \begin{psmatrix} a^{-1}_i & \\ & a_i \end{psmatrix} \begin{psmatrix} & 1 \\ -1 & \end{psmatrix} \begin{psmatrix} 1 & b_i \\& 1 \end{psmatrix}=\begin{cases}
\begin{psmatrix}
 a^{-1}_i & a_i t\\
    & a_i
\end{psmatrix}\begin{psmatrix}
 -b_i & -1\\
 1 &
\end{psmatrix}^{-1} & \textrm{ if } b_i\in\calo_F,\\
\begin{psmatrix}
 (a_ib_i)^{-1} & a_ib_it-a^{-1}_i\\
    & a_ib_i
\end{psmatrix}\begin{psmatrix}
 -1 & \\
 b_i^{-1} & -1
\end{psmatrix}^{-1} & \textrm{ if } b_i \not \in\calo_F.
\end{cases}
$$
Combining this with Lemma \ref{lem:consist} and using coordinates as in \eqref{tab}, we have
\begin{align}\label{eq:tnona}
    \mathcal{V}_{\mathcal{B}_1,\mathcal{B}_2}(t,a,b_0)=\mathbf{1}_{\mathcal{B}_1\calo_F}\left(t-\tau(a,b_0)\right)\prod_{i=1}^3 \left(\mathbf{1}_{\mathcal{B}_2\calo_F}-\mathbf{1}_{\mathcal{B}_2^{-1}\calo_F}\right)\left( \varpi^{\mathrm{ord}(a_i,a_ib_i)}\right),
\end{align}
where $\tau: (F^\times)^3\times F^3 \to F$ is given by
\begin{align*}
    \tau(a,b_0):=\frac{1}{3}\sum_{i=1}^3 (a_i^2b_i)^{-1}\mathbf{1}_{F-\calo_F}(b_i).
\end{align*}
Here we view $\mathbf{1}_{F-\calo_F}$ as a function valued in $\{0,1\}\subset F$, and take the convention that $(a_i^2b_i)^{-1}\mathbf{1}_{F-\calo_F}(b_i)=0$ if $b_i=0.$ 
Consequently, 
\begin{align*}
    \mathcal{V}_{\infty,\mathcal{B}_2}(a,b_0):=\lim_{|\mathcal{B}_1|\to\infty} \mathcal{V}_{\mathcal{B}_1,\mathcal{B}_2}(0,a,b_0)=\prod_{i=1}^3 \left(\mathbf{1}_{\mathcal{B}_2\calo_F}-\mathbf{1}_{\mathcal{B}_2^{-1}\calo_F}\right)\left( \varpi^{\mathrm{ord}(a_i,a_ib_i)}\right)
\end{align*}
where the convergence is pointwise.
To unify the notation, we write $H:=\mathbf{1}_{\mathcal{O}_F}$, $H_{\mathcal{B}_1}:=\mathbf{1}_{\mathcal{B}_1\calo_F}$, and for $g\in \Omega$ we write 
\begin{align*}
\tau_g:=\tau \quad \textrm{ and } \quad R(g)\mathcal{V}_{\infty,\mathcal{B}_2}:= \mathcal{V}_{\infty,\mathcal{B}_2}.
\end{align*}
For $F$ archimedean or nonarchimedean, we can rewrite \eqref{eq:ta} and \eqref{eq:tnona} as
\begin{align}\label{eq:infty}
     R(g)\mathcal{V}_{\mathcal{B}_1,\mathcal{B}_2}(t,a,b_0)=H_{\mathcal{B}_1}(t-\tau_g(a,b_0))R(g)\mathcal{V}_{\infty,\mathcal{B}_2}(a,b_0).
\end{align}
As $\displaystyle{\lim_{|\mathcal{B}_2| \to \infty}}R(g)\mathcal{V}_{\infty,\mathcal{B}_2}$ converges pointwise a.e. to $1$ for every $g\in \Omega$ and $H(0)=1$, \eqref{eq:to1} is satisfied. 

Let $g\in \Omega$. By \eqref{start}, \eqref{FT:formula0}, and \eqref{eq:infty}, we obtain
\begin{align*}
    &\mathcal{F}_X\left({f_{\xi,\mathcal{V}_{\mathcal{B}_1,\mathcal{B}_2}}}\right)(-x_0g)\\
    &=\frac{\gamma(Q)k}{|3|\zeta(1)}
\int_{F^\times}\psi(z^{-1})\Bigg( \int_{F^2\times (F^\times)^3\times F^2} \psi\left(tQ(\xi)+tb\right)  \overline{\psi}(z^2\mathfrak{r}(a))
      \\& \times 
\left(\int_{V(F)}\psi\left(\left\langle \frac{\xi}{a},u \right\rangle +\frac{bQ(u)}{9z^2[a]}+\mathfrak{c}(u,v)\right)\rho(g)f(u)du\right)\\& \times 
 H_{\mathcal{B}_1}\left(t-\tau_g\left(a,\Delta\left(\frac{b}{9z^2[a]}\right)+v'\right)\right) R(g)\mathcal{V}_{\infty,\mathcal{B}_2}\left(a,\Delta\left(\frac{b}{9z^2[a]}\right)+v'\right)
dv \frac{\chi_Q(a)d^\times a}{\{a\}^{d/2-1}}    dbdt\Bigg) d^\times z. 
\end{align*}
Taking a change of variables $t\mapsto t+\tau_g\left(a,\Delta\left(\frac{b}{9[a]}\right)+v'\right)$ and $b\mapsto b-Q(\xi)$, we arrive at
\begin{align}\label{eq:beforeinv}
&\frac{\gamma(Q)k}{|3|\zeta(1)}
\int_{F^\times}\psi(z^{-1})\Bigg( \int_{F} \widehat{H}_{\mathcal{B}_1}\left(b\right)\Phi_{\xi,g,\mathcal{B}_2}\left(\frac{b-Q(\xi)}{z^2},z,f\right) db\Bigg) d^\times z,
\end{align}
where $\widehat{H}_{\mathcal{B}_1}$ is the Fourier transform of $H_{\mathcal{B}_1}$, and
\begin{align*}
    \Phi_{\xi,g,\mathcal{B}_2}(b,z,f):=&\int_{(F^\times)^3}   \overline{\psi}(z^2\mathfrak{r}(a))\int_{F^2} 
\left(\int_{V(F)}\psi\left(\left\langle \frac{\xi}{a},u \right\rangle +\frac{bQ(u)}{9[a]}+\mathfrak{c}(u,v)\right)\rho(g)f(u)du\right)\\& \times 
 \psi\left((z^2b+Q(\xi))\tau_g\left(a,\Delta\left(\frac{b}{9[a]}\right)+v'\right)\right)\\
 &\times R(g)\mathcal{V}_{\infty,\mathcal{B}_2}\left(a,\Delta\left(\frac{b}{9[a]}\right)+v'\right)
dv \frac{\chi_Q(a)d^\times a}{\{a\}^{d/2-1}}.
\end{align*}

We prove in \S \ref{ssec:formal} the following:
\begin{Prop}\label{prop:tech}
There exist $1/2>\varepsilon>0$ and $\varepsilon'>0$ such that for all $g\in \Omega$ and $b\in F$ with $|Q(\xi)|/2\le |b-Q(\xi)|\le 2|Q(\xi)|$
\begin{align}\label{claim}
    \left|\Phi_{\xi,g,\mathcal{B}_2}\left(\frac{b-Q(\xi)}{z^2},z,f\right)\right|\ll_{\varepsilon,\varepsilon', f,\xi,\mathcal{B}_2} \min(1,|z|)^{\min(d_1,d_2,d_3)-2-2\varepsilon}\max(1,|z|)^{-\varepsilon'}.
\end{align}
In particular, the bound is independent of $b,g$. Moreover, if $b=0$, the implied constant can be chosen independent of $\mathcal{B}_2.$
\end{Prop} 
\noindent We claim Proposition \ref{prop:tech} implies Proposition \ref{prop:limit}. As $\widehat{H}$ is compactly supported, we can choose $M\in \zz_{>0}$ such that for all $|\mathcal{B}_1|>M$, $|Q(\xi)|/2\le |b-Q(\xi)|\le 2|Q(\xi)|$ for $b$ in the support of $\widehat{H}_{\mathcal{B}_1}$. Then by \eqref{eq:beforeinv} and \eqref{claim}, we have
\begin{align*}
    |\mathcal{F}_X(f_{\xi,\mathcal{B}_1,\mathcal{B}_2})(-x_0g)| \ll_{\varepsilon,\varepsilon', f,\xi,\mathcal{B}_2} \|\widehat{H}_{\mathcal{B}_1}\|_1 \int_{F^\times}\min(1,|z|)^{\min(d_1,d_2,d_3)-2-2\varepsilon}\max(1,|z|)^{-\varepsilon'} d^\times z\ll_{\epsilon,\epsilon'} 1
\end{align*}
for all $g\in \Omega$. Here we use the fact that the $L^1$-norm $\|\widehat{H}_{\mathcal{B}_1}\|_1=\|\widehat{H}\|_1<\infty$ for all $\mathcal{B}_1$. By Fourier inversion, $\widehat{H}_{\mathcal{B}_1}$ converges to the Dirac delta distribution as $|\mathcal{B}_1|\to \infty$, and thus by the dominated convergence theorem we have
\begin{align*}
    \lim_{|\mathcal{B}_1|\to\infty} \mathcal{F}_X(f_{\xi,\mathcal{B}_1,\mathcal{B}_2})(-x_0g)=\frac{\gamma(Q)k}{|3|\zeta(1)}
\int_{F^\times}\psi(z^{-1})\Phi_{\xi,g,\mathcal{B}_2}\left(\frac{-Q(\xi)}{z^2},z,f\right) d^\times z
\end{align*}
Applying the bound in Proposition \ref{prop:tech} for $b=0$ and that $R(g)\mathcal{V}_{\infty,\mathcal{B}_2}$ converges to $1$ a.e. for any $g\in \Omega$, \eqref{eq:finalform} follows from the dominated convergence theorem and Theorem \ref{thm:absolutebdappendix}. This proves Proposition \ref{prop:limit}.

We end this subsection with a lemma on the choice of $J$ in the archimedean case.  For ease of notation, let
$$
D:= \{b \in F:|Q(\xi)|/2 \le|b-Q(\xi)|\le 2|Q(\xi)|\} \times F^2\times F^{\times}\times \rr_{>1}
$$
and for $(g,a,\mathfrak{d})=(g,a,b, v,z,\mathcal{B}_2) \in \Omega \times (F^\times)^3 \times D$, let 
\begin{align}\label{T}
    T_{\xi}(g,a,\mathfrak{d}):=R(g)\mathcal{V}_{\infty,\mathcal{B}_2}\Bigg(\frac{\widetilde{\xi}}{a},\Delta\left(\frac{[a](b-Q(\xi))}{9z^2[\widetilde{\xi}]}\right)+v'\Bigg ).
\end{align}
When $F=\cc$,  write $a_n=r_ne^{\sqrt{-1}\theta_n}$ in polar coordinates. 
\begin{Lem}\label{lem:fin0}
The function $J$ can be chosen so that the following conditions are satisfied. Suppose $F=\rr$ (resp. $F=\cc$). In each variable $a_n$ (resp. $r_n$), the preimage of $0$ of each function
\begin{align}
\begin{split}\label{eq:deriv}
    &\partial_{a_i}T_{\xi}(g,a,\mathfrak{d}),\quad \partial_{a_i}\partial_{a_j}T_{\xi}(g,a,\mathfrak{d}), \quad  \partial_{a_1}\partial_{a_2}\partial_{a_3}T_{\xi}(g,a,\mathfrak{d})\\
&\textrm{ (resp. }
    \partial_{r_i}T_{\xi}(g,re^{\sqrt{-1}\theta},\mathfrak{d}),\quad \partial_{r_i}\partial_{r_j}T_{\xi}(g,re^{\sqrt{-1}\theta},\mathfrak{d}), \quad \partial_{r_1}\partial_{r_2}\partial_{r_3}T_{\xi}(g,re^{\sqrt{-1}\theta},\mathfrak{d}) \textrm{ )}
\end{split}
\end{align}
has finitely many connected components, and the number of connected components is bounded by an absolute constant.
\end{Lem}

\begin{proof}
Assume first $F=\rr$. We will make use of some standard facts on $o$-minimal geometry. A nice reference is \cite{vdD:ominimal}. Consider the $o$-minimal structure $\rr_{\mathrm{exp}}$ \cite{Wilkie:Rexp} generated by the exponential function (and algebraic functions). We can choose $J \in \mathcal{S}(F)$ that is definable and satisfies (1) and (3) above. An explicit construction is given in \cite[\S 13.1]{Tu:M}.

 Since the domain of $T_\xi$ is semialgebraic and $\log$ is definable, $T_\xi$ is definable and smooth in $a$, and so are the derivatives in \eqref{eq:deriv}. Let $Y_i$ be the graph of the function 
 \begin{align*}
      \partial_{a_i}T_{\xi}:\Omega \times (\rr^\times)^3 \times D \lto \rr
 \end{align*}
 Thus $Y_i$ is a definable set.  It admits a decomposition into finitely many definable cells \cite[Chapter 3 (2.11)]{vdD:ominimal}.  For each $(g, a_2, a_3, \mathfrak{d})$, consider the fiber over $(g, a_2, a_3,\mathfrak{d},0)$  of the projection map
 $$
 Y_i \lto \Omega\times (\rr^\times)^2\times D\times \rr.
 $$
  The intersection of the fiber with each cell of $Y_i$ is either empty or connected. This follows from \cite[Exercise 7 of Chapter 3 (2.19)]{vdD:ominimal} and the definition of cells; see the proof of Chapter 3 (2.9) in loc.~ cit. Hence the number of connected components of the fiber is bounded by the number of cells, which is an absolute constant. This proves the assertion for the function $\partial_{a_i}T_\xi(g, a, \mathfrak{d})$ in the variable $a_1.$  The same argument can be applied to the other functions and variables.
  
  For $F=\cc$, $T_\xi$ is definable in the $o$-minimal structure $\rr_{\mathrm{an,exp}}$ \cite{vdDM}, generated by the exponential function and restricted real analytic functions. The rest of the arguments carry over.
\end{proof}

\subsection{Proof of Proposition \ref{prop:tech}} \label{ssec:formal} 

Define for $(a,b,z,g)\in (F^\times)^3\times F\times F^\times \times \Omega$
 \begin{align*}
        \Psi_{\xi,g,\mathcal{B}_2}(a,b,z,f):=&\int_{F^2}\left(\int_{V(F)}\psi\left(\left\langle \frac{a\xi}{\widetilde{\xi}},u \right\rangle +\frac{b[a]Q(u)}{9[\widetilde{\xi}]}+\mathfrak{c}(u,v)\right)\rho(g)f(u)du\right)
 \\
 &\times  \psi\left((z^2b+Q(\xi))\tau_g\left(\frac{\tilde{\xi}}{a},\Delta\left(\frac{[a]b}{9[\widetilde{\xi}]}\right)+v'\right)\right)\\
 &\times R(g)\mathcal{V}_{\infty,\mathcal{B}_2}\left(\frac{\widetilde{\xi}}{a},\Delta\left(\frac{[a]b}{9[\widetilde{\xi}]}\right)+v'\right)
\chi_Q(a\widetilde{\xi})\{a\}^{d/2-2} dv.
\end{align*}
Comparing with $\Psi_\xi(a,b,\rho(g)f)$ defined in \eqref{Psi}, the only difference between the two functions is the introduction of the weight function 
\begin{align} \label{weight:func}
\psi\left((z^2b+Q(\xi))\tau_g\left(\frac{\tilde{\xi}}{a},\Delta\left(\frac{[a]b}{9[\widetilde{\xi}]}\right)+v'\right)\right)R(g)\mathcal{V}_{\infty,\mathcal{B}_2}\left(\frac{\widetilde{\xi}}{a},\Delta\left(\frac{[a]b}{9[\widetilde{\xi}]}\right)+v'\right).
\end{align}
Changing variables $a \to a^{-1}\widetilde{\xi}$ in the definition of $\Phi_{\xi,g,\mathcal{B}_2}(b,z,f)$ we have
\begin{align*}
    \Phi_{\xi,g,\mathcal{B}_2}(b,z,f)= \zeta(1)^3\prod_{i=1}^3|\xi_i|^{1-d_i/2} \int_{(F^\times)^3} \overline{\psi}(z^2\mathfrak{r}(a^{-1}\tilde{\xi}))\Psi_{\xi,g,\mathcal{B}_2}(a,b,z,f) da,
\end{align*}
just as in \eqref{eq:changeaxi}.
Thus the bound \eqref{claim} would be implied by Proposition \ref{prop:is:finite} except we have introduced the weight function \eqref{weight:func} and replaced $-\frac{Q(\xi)}{z^2}$ by $\frac{b-Q(\xi)}{z^2}.$  The remainder of the proof of Proposition \ref{prop:tech} amounts to modifying the proof of Proposition \ref{prop:is:finite} to prove \eqref{claim}.  

Let us begin this process.
We assume for the remainder of the proof that $b\in F$ is such that $|Q(\xi)|/2\le |b-Q(\xi)|\le 2|Q(\xi)|$. Put
\begin{align*}
    G(b,\xi):=\left(\frac{\left|\xi_1\otimes \xi_2 \otimes  \xi_3\right|}{\left|b-Q(\xi)\right|}\right)^{1/2}=C_bG(\xi)
\end{align*}
for some $(1/2)^{1/2}\le C_b\le 2^{1/2}$, where $G(\xi)$ is defined in \eqref{Gxi}. As the absolute value of \eqref{weight:func} is bounded by $1$, the bounds in \eqref{hana}, \eqref{dul} are valid if we replace $\Psi_{\xi}\left(a,\frac{-Q(\xi)}{z^2},\rho(g)f\right)$ and $G(\xi)$ by $\Psi_{\xi,g,\mathcal{B}_2}\left(a,\frac{b-Q(\xi)}{z^2},z,f\right)$ and $G(b,\xi)$, and the implied constants can be taken to be independent of $b,g, \mathcal{B}_2.$  One simply 
replaces Corollaries \ref{cor:absolutebd} and \ref{cor:byproduct} with Theorems
\ref{thm:absolutebdappendix} and \ref{thm:byproductappendix}, respectively.  In fact the only difference between the latter results and the former is that one absolute value sign is outside (resp.~inside) the integral over $F^2$ in the former (resp.~latter). 
Thus to prove \eqref{claim} we are left with bounding the analogue of \eqref{dos}, namely:
\begin{align*}
   \left|\int_{|a|\le (|z|G(b,\xi))^{2\alpha}}\overline{\psi}(z^2\mathfrak{r}(a^{-1}\widetilde{\xi}))\Psi_{\xi,g,\mathcal{B}_2}\left(a,\tfrac{b-Q(\xi)}{z^2},z,f\right)da\right|
\end{align*}
for $N<|z|G(b,\xi)$ and $1/6>\alpha>0$ sufficiently small.  Here $N \in \zz_{>0}$ is a constant to be determined in the nonarchimedean case, and is $1$ in the archimedean case. 
We therefore assume for the remainder of the proof that $|a|\le (|z|G(b,\xi))^{2\alpha}$ and $N<|z|G(b,\xi)$, where $0<\alpha<1/6$ and $N\in \zz_{>0}.$

Let $F$ be nonarchimedean. Define
\begin{align*}
\begin{split}
    \Psi_{\xi,g, \mathcal{B}_2}'\left(a,b,  f\right)&=\int_{F^2}\left(\int_{V(F)}\psi\left(\left\langle \frac{a\xi}{\widetilde{\xi}},u \right\rangle +\mathfrak{c}(u,v)\right)\rho(g)f(u)du\right)
 \\
 &\times  \psi\left(b\tau_g\left(\frac{\tilde{\xi}}{a}, v'\right)\right)R(g)\mathcal{V}_{\infty,\mathcal{B}_2}\left(\frac{\widetilde{\xi}}{a},v'\right)
\chi_Q(a\widetilde{\xi})\{a\}^{d/2-2} dv.
\end{split}
\end{align*}
 As
\begin{align} \label{diff:bound2}
    |a|^2\left|\frac{[a](b-Q(\xi))}{9z^2[\widetilde{\xi}]}\right|\ll (|z|G(b,\xi))^{10\alpha-2},
\end{align}
we can choose $N$ large, independent of $b,g,\mathcal{B}_2$, such that
\begin{align*}
    \psi\left(b\tau_g\left(\frac{\tilde{\xi}}{a},\Delta\left(\frac{[a](b-Q(\xi))}{9z^2[\widetilde{\xi}]}\right)+v'\right)-b\tau_g\left(\frac{\tilde{\xi}}{a}, v'\right)\right)=1.
\end{align*}
Moreover, in view of \eqref{eq:infty}, we have that  $R(g)\mathcal{V}_{\infty,\mathcal{B}_2}\left(\frac{\widetilde{\xi}}{a},b_0\right)$ is invariant under  $b_0\mapsto b_0+u'$ for  $u'\in\varpi\calo_F^3$.  
Consequently,
\begin{align*}
\begin{split}
    \Psi_{\xi,g, \mathcal{B}_2}\left(a,\tfrac{b-Q(\xi)}{z^2},z,  f\right)&= \Psi'_{\xi,g, \mathcal{B}_2}\left(a,b, f\right)
\end{split}
\end{align*}
for $N$ large by \eqref{diff:bound2} and \eqref{diff:bound}.  
The analogue of \eqref{eq:cutoff0} with $\Psi_{\xi}(a,0,f)$ replaced by $\Psi'_{\xi,g, \mathcal{B}_2}\left(a,b, f\right)$ is valid since $|\psi|=1$ and $|R(g)\mathcal{V}_{\infty, \mathcal{B}_2}|\le 1.$  We simply replace  Corollary \ref{cor:b=0} by Theorem \ref{thm:b=0appendix}.  Thus to complete the proof of the \eqref{claim} in the nonarchimedean case, it suffices to show that the analogue of Lemma \ref{lem:uniformb2} holds, namely, for $\alpha>0$ sufficiently small, there exists $\varepsilon'>0$ such that
\begin{align} \label{uniformb2}
   &\left|\int_{(|z|G(b,\xi))^{-2\alpha}\le |a_i|\le (|z|G(b,\xi))^{2\alpha}}\overline{\psi}(z^2\mathfrak{r}(a^{-1}\widetilde{\xi}))\Psi_{\xi,g,\mathcal{B}_2}'(a,b,f)da\right|
   \ll_{\varepsilon', f, \xi} |z|^{-\varepsilon'}.
\end{align}
 To deduce the bound \eqref{uniformb2}, observe that for $|u|\ll_\xi (|z|G(b,\xi))^{-2\alpha}$
\begin{align*}
    \psi\left(b\tau_g\left(\frac{\tilde{\xi}}{a+u}, v'\right)-b\tau_g\left(\frac{\tilde{\xi}}{a}, v'\right)\right)=1.
\end{align*}
Note that $R(g)\mathcal{V}_{\infty,\mathcal{B}_2}\left(\frac{\widetilde{\xi}}{a},b_0\right)$ is also invariant under $a\mapsto a+u$ for $u\in \varpi\prod_{i=1}^3 a_i\calo_F.$
Therefore, for $m \in \zz_{>0}$ sufficiently large, depending on $\xi,f,\chi_Q$, and any $\ell \in \zz_{>0}$ such that $|\varpi^{\ell}|\leq |\varpi^{m}|(|z|G(b,\xi))^{-2\alpha},$ the function
$\Psi'_{\xi,g, \mathcal{B}_2}\left(a,b,  f\right)$ is invariant under $a \mapsto a+\varpi^{\ell}x$ for any $x \in \mathcal{O}_F^3. $
Furthermore, since $|\psi|=1$ and $|R(g)\mathcal{V}_{\infty, \mathcal{B}_2}|\le 1$, we can use Theorem \ref{thm:b=0appendix} in place of Corollary \ref{cor:b=0} in the proof of Lemma \ref{lem:uniformb2} and argue as in that proof to prove \eqref{uniformb2}.  We also observe that all of our bounds are independent of $\mathcal{B}_2,$ even if $b \neq 0,$ so the last assertion of Proposition \ref{prop:tech} follows as well.

Suppose $F$ is archimedean. First consider the analogue of the bound \eqref{set} for $\Psi_{\xi,g,\mathcal{B}_2}\left(a,\frac{b-Q(\xi)}{z^2},z, f\right)$. Write
\begin{align} \label{2:summ} \begin{split}
    &\Psi_{\xi,g,\mathcal{B}_2}\left(a,\frac{b-Q(\xi)}{z^2},z, f\right)\\
    &=\int_{F^2}T_\xi(g,a,\mathfrak{d})\psi\left(b\tau_g\left(\frac{\tilde{\xi}}{a}, \Delta\left(\frac{[a](b-Q(\xi))}{9z^2[\widetilde{\xi}]}\right)+v'\right)\right)\\&\times \left(\int_{V(F)}\Bigg(\psi\left(\left\langle \frac{a\xi}{\widetilde{\xi}},u \right\rangle +\frac{[a](b-Q(\xi))Q(u)}{9z^2[\widetilde{\xi}]}+\mathfrak{c}(u,v)\right)-\psi\left(\left\langle \frac{a\xi}{\widetilde{\xi}},u \right\rangle +\mathfrak{c}(u,v)\right)\Bigg)f(u)du\right)
 \frac{\chi_Q(a\widetilde{\xi})}{\{a\}^{2-d/2}} dv\\
    &+\int_{F^2} T_\xi(g,a,\mathfrak{d})\psi\left(b\tau_g\left(\frac{\tilde{\xi}}{a}, \Delta\left(\frac{[a](b-Q(\xi))}{9z^2[\widetilde{\xi}]}\right)+v'\right)\right)\\
    &\times\left(\int_{V(F)}\psi\left(\left\langle \frac{a\xi}{\widetilde{\xi}},u \right\rangle +\mathfrak{c}(u,v)\right)f(u)du\right)
\frac{ \chi_Q(a\widetilde{\xi})}{\{a\}^{2-d/2}} dv. \end{split}
\end{align}
Here we use the notation \eqref{T}. 
Let 
\begin{align*}
    \widetilde{\Psi}_{\xi,g,\mathcal{B}_2}(b,c,z,f):=&\int_{|a| \leq (|z|G(b,\xi))^{2\alpha}}\overline{\psi}\left( z^2 \mathfrak{r}(a^{-1}\widetilde{\xi})\right)\int_{F^2}T_\xi(g,a,\mathfrak{d})\psi\left(b\tau_g\left(\frac{\tilde{\xi}}{a}, \Delta\left(\frac{[a](b-Q(\xi))}{9z^2[\widetilde{\xi}]}\right)+v'\right)\right)\\& \times \left(\int_{V(F)}\psi\left(\left\langle \frac{a\xi}{\widetilde{\xi}},u \right\rangle +\frac{[a]cQ(u)}{9[\widetilde{\xi}]}+\mathfrak{c}(u,v)\right)f(u) \frac{\chi_Q(a\widetilde{\xi})}{\{a\}^{2-d/2}}du \right)dv d a.
\end{align*}
By the same argument proving \eqref{set}, with Theorem \ref{thm:absolutebdappendix} replacing Corollary \ref{cor:absolutebd}, we deduce the following analogue of \eqref{set}:
\begin{align*}
    \left|\widetilde{\Psi}_{\xi,g,\mathcal{B}_2}(b,\tfrac{b-Q(\xi)}{z^2},z,f)-\widetilde{\Psi}_{\xi,g,\mathcal{B}_2}(b,0,z,f) \right| \ll_f (|z|G(b,\xi))^{\frac{6\alpha-2}{[F:\rr]}} \ll_{\xi} |z|^{\frac{6\alpha-2}{[F:\rr]}}
\end{align*}
Here we have used that $|\psi|=1$ and $|R(g)\mathcal{V}_{\infty,\mathcal{B}_2}|\le 1$. As in the nonarchimedean case, the analogue of \eqref{eq:cutoff0} with $\Psi_{\xi}(a,0,f)$ replaced by the second term of \eqref{2:summ} is valid by using Theorem \ref{thm:b=0appendix} in place of Corollary \ref{cor:b=0}.

Therefore, we are left with showing that the analogue of the bound in Lemma \ref{lem:uniformb2} remains valid for the second term of \eqref{2:summ}.  In other words, we need to show that for all $|z|G(b,\xi)>1$ the quantity
\begin{align} \label{like:unifb2} \begin{split}
&\Bigg|\int_{(|z|G(b,\xi))^{-2\alpha}\le |a_i|\le (|z|G(b,\xi))^{2\alpha}}\overline{\psi}(z^2\mathfrak{r}(a^{-1}\widetilde{\xi}))\int_{F^2} T_{\xi}(g,a,\mathfrak{d})\psi\left(b\tau_g\left(\frac{\tilde{\xi}}{a},\Delta\left(\frac{[a](b-Q(\xi))}{9z^2[\widetilde{\xi}]}\right)+v'\right)\right)\\
        &\times\left(\int_{V(F)}\psi\left(\left\langle \frac{a\xi}{\widetilde{\xi}},u \right\rangle +\mathfrak{c}(u,v)\right)f(u)du\right)
\frac{ \chi_Q(a\widetilde{\xi})}{\{a\}^{2-d/2}} dv da\Bigg| \end{split}
\end{align}
is bounded by a constant, that depends on $\mathcal{B}_2$ if $b\neq 0$, times $|z|^{-\varepsilon'}$ for some $\varepsilon'>0$. We show that the argument proving Lemma \ref{lem:uniformb2} can be adapted to the current setting. Assume $F$ is real. We prove an analogue of Lemma \ref{lem:vandercorputarch}, namely, for any bounded continuous function $h$ on $\rr_{>0}^3$ and $(|z|G(b,\xi))^{2\alpha}\ge c>1$ one has 
\begin{align} \label{analogue:vdc} \begin{split}
    &\Bigg|\int_{c^{-1}\leq a_i \leq c}h(a)\int_{F^2}T_\xi(g,a,\mathfrak{d})\psi\left(b\tau_g\left(\frac{\tilde{\xi}}{a}, \Delta\left(\frac{[a](b-Q(\xi))}{9z^2[\widetilde{\xi}]}\right)+v'\right)\right)\\& \times \left(\int_{V(F)}\Bigg(\psi\left(\left\langle \frac{a\xi}{\widetilde{\xi}},u \right\rangle +\mathfrak{c}(u,v)\right)f(u)du\right) \frac{\chi_Q(a\widetilde{\xi})}{\{a\}^{2-d/2}} dvda\Bigg|\\
    &\ll_{f,\mathcal{B}_2,\xi } |z|^{6\alpha}\sup_{c^{-1}<a_i\le c}\left|\int_{c^{-1}}^{a_1}\int_{c^{-1}}^{a_2}\int_{c^{-1}}^{a_3}h(r)dr \right| \end{split}
\end{align}
Combining with \eqref{eq:vander}, this gives rise to the desired bound of \eqref{like:unifb2} by choosing $\alpha>0$ sufficiently small.

We prove \eqref{analogue:vdc} using integration by parts as in the proof of Lemma \ref{lem:vandercorputarch}.  We have analogues of all of the terms appearing in that proof.  They can all be bounded as before using the fact that $|\psi|=1$ and $|T_{\xi}| \leq 1,$ replacing the use of Corollary \ref{cor:b=0} by Theorem \ref{thm:b=0appendix}.  We also have additional terms that can be bounded similarly as we now explain.    First, using \eqref{inf:cp}, one checks that (for $|a|\le (|z|G(b,\xi))^{2\alpha}$)  in the support of $T_\xi(g,a,\mathfrak{d})$ the quantities 
\begin{align} \label{partial:psi} \begin{split}
    &\left|\partial_{a_i}\psi\left(b\tau_g\left(\frac{\tilde{\xi}}{a}, \Delta\left(\frac{[a](b-Q(\xi))}{9z^2[\widetilde{\xi}]}\right)+v'\right)\right)\right|, \quad  \left|\partial_{a_i}\partial_{a_j}\psi\left(b\tau_g\left(\frac{\tilde{\xi}}{a}, \Delta\left(\frac{[a](b-Q(\xi))}{9z^2[\widetilde{\xi}]}\right)+v'\right)\right)\right|,\\
     & \left|\partial_{a_1}\partial_{a_2}\partial_{a_3}\psi\left(b\tau_g\left(\frac{\tilde{\xi}}{a}, \Delta\left(\frac{[a](b-Q(\xi))}{9z^2[\widetilde{\xi}]}\right)+v'\right)\right)\right| \end{split}
\end{align}
are bounded by $O_{\mathcal{B}_2}(1)$.  The additional terms that we have to bound involve the derivatives in \eqref{partial:psi} and derivatives of $T_{\xi}(g,a,\mathfrak{d}).$  One such term is
\begin{align*}
\begin{split}
     &\int_{c^{-1}}^{c}\int_{F^2}\Bigg| T_{\xi}(g,c,c,a_3,\mathfrak{d})\partial_{a_3}\psi\left(b\tau_g\left(\frac{\widetilde{\xi_1}}{c},\frac{\widetilde{\xi_2}}{c}, \frac{\widetilde{\xi_3}}{a_3}, \Delta\left(\frac{c^2a_3(b-Q(\xi))}{9z^2[\widetilde{\xi}]}\right)+v'\right)\right)\\
    &\times\left(\int_{V(F)}\psi\left(\sum_{j=1}^2\left\langle \frac{c\xi_j}{\widetilde{\xi_j}},u_j \right\rangle_j+\left\langle \frac{a_3\xi_3}{\widetilde{\xi_3}},u_3 \right\rangle_3+\mathfrak{c}(u,v)\right)f(u)du\right)\Bigg|\frac{da_3dv}{c^{4-(d_1+d_2)/2}a_3^{2-d_3/2}},
\end{split}
\end{align*}
This term is nonzero only if $b\neq 0$ and is bounded by $O_{\mathcal{B}_2,f}(c)$ times 
\begin{align*}
    \sup_{c^{-1}\le a_3\le c}\int_{F^2}\left|\int_{V(F)}\psi\left(\sum_{j=1}^2\left\langle \frac{c\xi_j}{\widetilde{\xi_j}},u_j \right\rangle_j+\left\langle \frac{a_3\xi_3}{\widetilde{\xi_3}},u_3 \right\rangle_3+\mathfrak{c}(u,v)\right)f(u)du\right)\Bigg|\frac{dv}{c^{4-(d_1+d_2)/2}a_3^{2-d_3/2}}
\end{align*}
by our bounds on \eqref{partial:psi}, which is $O_{\mathcal{B}_2,f}(c)$ by Theorem \ref{thm:b=0appendix}. This is a sufficient bound for the purposes of proving \eqref{analogue:vdc}, and all the terms not involving any derivative of $T_{\xi}(g,a,\mathfrak{d})$ can be treated in the same manner.

Consider the terms involving derivatives of $T_{\xi}(g,a,\mathfrak{d})$, e.g.
\begin{align}\label{eq:finala3}
\begin{split}
    &\int_{F^2}\int_{c^{-1}}^{c}\Bigg| \partial_{a_3}T_{\xi}(g,c,c,a_3,\mathfrak{d})\\
    &\times\left(\int_{V(F)}\psi\left(\sum_{j=1}^2\left\langle \frac{c\xi_j}{\widetilde{\xi_j}},u_j \right\rangle_j+\left\langle \frac{a_3\xi_3}{\widetilde{\xi_3}},u_3 \right\rangle_3+\mathfrak{c}(u,v)\right)f(u)du\right)\Bigg|\frac{da_3dv}{c^{4-(d_1+d_2)/2}a_3^{2-d_3/2}},
\end{split}
\end{align}
where $1\le c\le (|z|G(b,\xi))^{2\alpha}$. 
Here we have used the fact $|\psi|=1$ and the Fubini--Tonelli theorem to switch the order of the integral over $a_3$ and $v$. Note that as a function of $c^{-1}\le a_3\le c$,  \begin{align*}\left|\int_{V(F)}\psi\left(\sum_{j=1}^2\left\langle \frac{c\xi_j}{\widetilde{\xi_j}},u_j \right\rangle_j+\left\langle \frac{a_3\xi_3}{\widetilde{\xi_3}},u_3 \right\rangle_3+\mathfrak{c}(u,v)\right)f(u)du\right|\frac{1}{c^{4-(d_1+d_2)/2}a_3^{2-d_3/2}}
\end{align*}
is Lipschitz continuous by Lemma \ref{lem:quad:esti}  and hence absolutely continuous, so it is differentiable a.e. and the derivative is integrable and satisfies the fundamental theorem of calculus \cite[Theorem 7.29]{Wheeden:Zygmund}. Applying integration by parts \cite[Theorem 7.32]{Wheeden:Zygmund} to the integral over $a_3$, \eqref{eq:finala3} is bounded by the sum of
\begin{align}\label{eq:finala31}
    &\int_{F^2}\int_{c^{-1}}^{c}\Bigg| \partial_{a_3}T_{\xi}(g,c,c,a_3,\mathfrak{d})\Bigg|da_3\Bigg|\int_{V(F)}\psi\left(\sum_{j=1}^3\left\langle \frac{c\xi_j}{\widetilde{\xi_j}},u_j \right\rangle_j+\mathfrak{c}(u,v)\right)f(u)du\Bigg|\frac{dv}{c^{6-(d_1+d_2+d_3)/2}}
\end{align}
and 
\begin{align}\label{eq:finala32}
\begin{split}
    &\int_{F^2}\int_{c^{-1}}^{c}\int_{c^{-1}}^{r}\Bigg| \partial_{a_3}T_{\xi}(g,c,c,a_3,\mathfrak{d})\Bigg|da_3\\
    &\times\partial_r\Bigg|\int_{V(F)}\psi\left(\sum_{j=1}^2\left\langle \frac{c\xi_j}{\widetilde{\xi_j}},u_j \right\rangle_j+\left\langle \frac{r\xi_3}{\widetilde{\xi_3}},u_3 \right\rangle_3+\mathfrak{c}(u,v)\right)f(u)\frac{du}{r^{2-d_3/2}}\Bigg|\frac{drdv}{c^{4-(d_1+d_2)/2} }.
\end{split}
\end{align}
By Theorem \ref{thm:b=0appendix}, the first term \eqref{eq:finala31} is bounded by $O_f(1)$ times 
\begin{align*}
    \sup_{v_1,v_2\in F}\int_{c^{-1}}^{c}\Bigg| \partial_{a_3}T_{\xi}(g,c,c,a_3,\mathfrak{d})\Bigg|da_3.
\end{align*}
By Lemma \ref{lem:fin0} and the fundamental theorem of calculus, the integral is dominated by $1$ for all $1\le c\le (|z|G(b,\xi))^{2\alpha}$. On the other hand, by the second mean value theorem for Lebesgue integrals \cite[Theorem 1]{secondMVT}, \eqref{eq:finala32} equals
\begin{align*}
    &\int_{F^2} \int_{c^{-1}}^{c}\Bigg| \partial_{a_3}T_{\xi}(g,c,c,a_3,\mathfrak{d})\Bigg|da_3\\
    &\times\int_{e(g,\mathfrak{d})}^{c} \partial_r\Bigg|\int_{V(F)}\psi\left(\sum_{j=1}^2\left\langle \frac{c\xi_j}{\widetilde{\xi_j}},u_j \right\rangle_j+\left\langle \frac{r\xi_3}{\widetilde{\xi_3}},u_3 \right\rangle_3+\mathfrak{c}(u,v)\right)f(u)\frac{du}{r^{2-d_3/2}}\Bigg|\frac{drdv}{c^{4-(d_1+d_2)/2} },
\end{align*}
for some $e(g,\mathfrak{d})\in(c^{-1},c)$. For the same reason as above, this is bounded by $O_f(1)$. The other terms can be treated similarly, yielding the bound \eqref{analogue:vdc}.  

The case $F=\cc$ can be argued similarly using polar coordinates. As mentioned above, this is enough to deduce \eqref{claim}.  Now we observe that the only place in this argument our bounds are not uniform in $\mathcal{B}_2$ is in the estimation of \eqref{partial:psi}.  These terms vanish if $b=0.$  Thus we obtain uniformity of the bound when $b=0$ as claimed.
\qed

\section{The operator $\mathcal{F}_Y$ is unitary}\label{sec:unit}

 In this section, we apply Theorem \ref{Thm:main} to prove the following.

\begin{Thm}\label{thm:FYuni}
    Let $F$ be a nonarchimedean local field of characteristic zero.
     Suppose $\dim V_i>2$ for all $1\le i\le 3$ and $Y^{\mathrm{sm}}(F) \neq \emptyset$. The Fourier transform $\mathcal{F}_Y$ extends to a unitary operator on $L^2(Y(F))$. Moreover, for $f_1,f_2\in L^2(Y(F)),$
    \begin{align}\label{PlancherelY}
        \int_{Y(F)} \mathcal{F}_Y(f_1)(y)f_2(y)d\mu(y) =\int_{Y(F)} f_1(y)\mathcal{F}_Y(f_2)(y)d\mu(y).
    \end{align}
\end{Thm}

Recall the definition of $\mathcal{S}$ from \eqref{calS}.  The following lemma is the only place in the argument where we use that $F$ is nonarchimedean:
\begin{Lem}\label{lem:cont}  Assume $F$ is a nonarchimedean local field of characteristic zero.
Given $f\in \mathcal{S},$  there exists a sequence of functions $f_i\in \mathcal{S}$ such that 
\begin{align*}
        \lim_{i\to \infty}\int_{Y(F)} f_i(y)\mathcal{F}_Y(\overline{f})(y)d\mu(y)&=\int_{Y(F)} \mathcal{F}_Y(f)(y)\mathcal{F}_Y(\overline{f})(y)d\mu(y),\\
       \lim_{i\to \infty}\int_{Y(F)} \mathcal{F}_Y(f_i)(y)\overline{f}(y)d\mu(y)&=\int_{Y(F)} f(y)\overline{f}(y)d\mu(y).
\end{align*}
\end{Lem}

\begin{Rem}
We expect that the proof of this lemma could be adapted to the archimedean case if one develops a theory of Sobolev spaces for $X(F)$ together with an analogue of Morrey's inequality.  
\end{Rem}
\begin{proof}
 We can and do assume $\mathcal{F}_Y(f)= I(\widetilde{f}_{1}\otimes \widetilde{f}_{2})$ for some $\widetilde{f}_1\otimes\widetilde{f}_2 \in \mathcal{S}(X(F)) \otimes \mathcal{S}(V(F))$ (see \eqref{I:def}).   Choose a compact open subgroup $K$ of $\mathrm{Sp}_6(\mathcal{O}_F)$ such that $\widetilde{f}_{1}$ is right $K$-invariant. Choose  $\widetilde{f}_{i1}\in C^\infty_c(x_0\SL_2^3(F))^K$ such that $\widetilde{f}_{i1} \to \widetilde{f}_1$ in $L^2(X(F))$. Put $f_i:=I(\widetilde{f}_{i1}\otimes \widetilde{f_2})$. There is a constant $c>0$ (depending only on $\widetilde{f_2}$ and $K$) such that
$$|\mathcal{F}_Y(f)(y)-f_i(y)|\le c\norm{\widetilde{f}_{1}-\widetilde{f}_{i1}}_2\prod_{j=1}^3 |y_j|^{-d_j/2+2/3}$$
for all $y\in Y^{\mathrm{ani}}(F)$. Indeed, this is implicit in the proof of \cite[Proposition 11.4]{Getz:Hsu}. 
Thus by \cite[Proposition 11.1]{Getz:Hsu} we have 
\begin{align} \label{L2:esti:stuff}
|\mathcal{F}_Y(f)(y)-f_i(y)||\mathcal{F}_Y(\overline{f})(y)| \ll_{f,\beta,K} \norm{\widetilde{f}_1-\widetilde{f}_{i1}}_2 \prod_{j=1}^3 |y_j|^{\beta/3-d_j+4/3}
\end{align}
for $\tfrac{1}{2}>\beta>0$.   Moreover, $\mathcal{F}_Y(\overline{f})$ has support contained in a compact subset of $Y(F)$ \cite[Proposition 7.1]{Getz:Liu:Triple}.
Thus applying \eqref{L2:esti:stuff} and \eqref{eq:boundw/oQy} to be proved below, we obtain 
\begin{align*}
\int_{Y(F)} |\mathcal{F}_Y(f)(y)-f_i(y)||\mathcal{F}_Y(\overline{f})(y)|d \mu(y) \ll_{f,K}\norm{\widetilde{f}_1-\widetilde{f}_{i1}}_2.
\end{align*}
The first equality follows. Since  $\mathcal{F}_X$ is an isometry on $L^2(X(F))^K,$ the second equality follows from the same argument.
\end{proof}

Before giving the proof of Theorem \ref{thm:FYuni}, we prepare some estimates.

 \begin{Lem}\label{Lem:levelsetestimate}
Let $F^d$ be a vector space of dimension $d$ and $\mathcal{Q}$ be a nondegenerate quadratic form on $F^d$. There exists $\alpha>0$  such that for any $0<t<1,$
\begin{align*}
    \int_{|v|\le 1,|\mathcal{Q}(v)|\le t} dv\ll_\alpha t^{\alpha}.
\end{align*}
\end{Lem}

\begin{proof}
We can and do assume the matrix of $\mathcal{Q}$ is the diagonal matrix $\mathrm{diag}(c_1,\ldots,c_d)$ where $c_i\in F^\times$.
We first consider the archimedean case. Suppose $F=\rr$. We may assume some $c_i\ge 1$. Then the lemma is a consequence of \cite[Theorem 1.3]{CCW}. If $F=\cc,$ we may assume each $c_i$ equals $1$. Thus in real coordinates, $|\mathcal{Q}(v)|$ is a homogeneous polynomial of degree $4$ with coefficients in $\zz_{\ge 0}$. The assertion again follows from loc.~cit. 

Suppose $F$ is nonarchimedean. We may also assume $|c_i|\ge 1$ for some $i$. Assume that the conductor of the additive character $\psi$ is $\mathcal{O}_F$. According to \cite[Proposition 3.6]{Cluckers}, we have
    \begin{align*}
        \left|\int_{\mathcal{O}_F^d} \psi(u\mathcal{Q}(v))dv\right|\ll \max(1,|u|)^{-1/2}
    \end{align*}
for $u\in F^\times$.
Consequently for $n>0$ we have
\begin{align*}
    \int_{|v|\le 1,|\mathcal{Q}(v)|\le q^{-n}} dv=\frac{1}{q^{n}dx(\mathcal{O}_F)}\int_{\varpi^{-n}\mathcal{O}_F} \left(\int_{\mathcal{O}_F^d} \psi(u\mathcal{Q}(v))dv\right)du\ll \frac{1}{q^{n}}\int_{|u|\le q^n} |u|^{-1/2}d u \ll \frac{1}{q^{n/2}}.
\end{align*}
\end{proof}
 
\begin{Lem}\label{lem:L1Y}
For $f\in\mathcal{S},$
    \begin{align}\label{eq:boundw/oQy}
        \int_{Y^{\mathrm{ani}}(F)} |f(y)|\prod_{i=1}^3 |y_i|^{-e_i}d\mu(y)<\infty
    \end{align}
and 
    \begin{align}\label{eq:boundQy}
    \int_{|Q(y)|\le |y_1||y_2||y_3|} |f(y)||Q(y)|^{-\varepsilon}\prod_{i=1}^{3} |y_i|^{-e_i}d\mu(y)<\infty
\end{align}
for any $e_i< d_i-4/3$ and $\varepsilon>0$ sufficiently small (depending on $e_i$).  Here the integral in \eqref{eq:boundQy} is over $y \in Y^{\mathrm{ani}}(F)$ satisfying the inequality.
\end{Lem}

\begin{proof}
Since $f\in \mathcal{S},$ we have
    \begin{align}\label{replacementnonarch}
        |f(y)|\ll \mathbf{1}_{\varpi^{-n}V(\mathcal{O}_F)}(y)
    \end{align}
    for some $n$ if $F$ is nonarchimedean, and 
    \begin{align}\label{replacementarch}
        |f(y)|\ll_N \max(1,|y|)^{-N}
    \end{align}
    for any $N\in\zz_{>0}$ if $F$ is archimedean. Therefore \eqref{eq:boundw/oQy} follows from the proof of \cite[Proposition 11.2]{Getz:Hsu}.
    
    In the following, we use the homogeneity property: for $r\in F^\times$
$$
d\mu(ry)=|r|^{d_1+d_2+d_3-4}d\mu(y).
$$
All of the integrals with respect to the measure $d\mu(y)$ in the remainder of the proof will be over subsets of $Y^{\mathrm{ani}}(F)$ satisfying the inequalities given in the subscript. 

Assume that $F$ is nonarchimedean.
    By \eqref{replacementnonarch}, for some $n>0$ we have that \eqref{eq:boundQy} is dominated by
\begin{align*}
    &\int_{|Q(y)|\le |y|^3,|y|\le q^n}|Q(y)|^{-\varepsilon} \prod_{i=1}^{3} |y_i|^{-e_i}d\mu(y)\\
    &\ll \int_{|Q(y)|\le q^{n}|y|^3,|y|\le 1} |Q(y)|^{-\varepsilon}\prod_{i=1}^{3} |y_i|^{-e_i}d\mu(y)\\
    &= \sum_{j=0}^\infty  q^{-j(-2\varepsilon-4+\sum_{i=1}^3 (d_i-e_i))}\int_{|Q(y)|\le q^{n-j},1\le |y|<2} |Q(y)|^{-\varepsilon}\prod_{i=1}^{3} |y_i|^{-e_i}d\mu(y).
\end{align*}
Here we could just write $|y|=1,$ but we have written $1\le |y| <2$ so that we can use the same formula in both archimedean and nonarchimedean cases. Suppose $F$ is archimedean. By \eqref{replacementarch}, we have that \eqref{eq:boundQy} is dominated by\begin{align*}
     &\sum_{j=1}^\infty \int_{|Q(y)|\le |y|^3,2^{-j}\le |y|< 2^{-j+1}}|Q(y)|^{-\varepsilon} \prod_{i=1}^{3} |y_i|^{-e_i} d\mu(y)\\
    & +\sum_{j=0}^\infty \int_{|Q(y)|\le |y|^3,2^{j}\le |y|< 2^{j+1}} |y|^{-N}|Q(y)|^{-\varepsilon}\prod_{i=1}^{3} |y_i|^{-e_i}d\mu(y)\\
    =&\sum_{j=1}^\infty 2^{-j(-2\varepsilon-4+\sum_{i=1}^3 (d_i-e_i))}\int_{|Q(y)|\le 2^{-j}|y|^3,1\le |y|< 2} |Q(y)|^{-\varepsilon}\prod_{i=1}^{3} |y_i|^{-e_i}d\mu(y)\\
    &+\sum_{j=0}^\infty 2^{j(-N-2\varepsilon-4+\sum_{i=1}^3 (d_i-e_i))} \int_{|Q(y)|\le 2^{j}|y|^3,1\le |y|< 2} |y|^{-N}|Q(y)|^{-\varepsilon} \prod_{i=1}^{3} |y_i|^{-e_i} d\mu(y)\\
     \le&\sum_{j=1}^\infty 2^{-j(-2\varepsilon-4+\sum_{i=1}^3 (d_i-e_i))}\int_{|Q(y)|\le 2^{-j+3},1\le |y|< 2} |Q(y)|^{-\varepsilon}\prod_{i=1}^{3} |y_i|^{-e_i}d\mu(y)\\
    &+\sum_{j=0}^\infty 2^{j(-N-2\varepsilon-4+\sum_{i=1}^3 (d_i-e_i))} \int_{|Q(y)|\le 2^{j+3},1\le |y|< 2} |Q(y)|^{-\varepsilon}\prod_{i=1}^{3} |y_i|^{-e_i}d\mu(y).
\end{align*}

Consider
\begin{align}\label{ultimate}
    \int_{|Q(y)|\le c^{j},1\le |y|< 2} |Q(y)|^{-\varepsilon}\prod_{i=1}^{3} |y_i|^{-e_i}d\mu(y),
\end{align}
where \begin{align*}
    c=c(F)=\begin{cases}
    q    &  \textrm{if }F\textrm{ is nonarchimedean,}\\
    2     & \textrm{if }F \textrm{ is archimedean.}
    \end{cases}
\end{align*}
The integral \eqref{ultimate} 
is nondecreasing as $j \to \infty$ and is independent of $j$ for $j$ sufficiently large since $|Q(y)|\ll |y|^2$.  Choose $N>-4+\sum_{i=1}^3 (d_i-e_i)$ in the archimedean case.  The manipulations above show that to prove \eqref{eq:boundQy} it suffices to show that we can choose $\varepsilon>0$ that is sufficiently small in a sense depending on $e_i$ such that \eqref{ultimate} is finite for a $j$ greater than a constant depending only on $Q$. 

    Observe that \eqref{ultimate} is bounded by 
\begin{align}\label{ultimate1}
    \sum_{k=-\infty}^{j-1}  c^{-\varepsilon k}\int_{c^k< |Q(y)|\le c^{k+1},1\le |y|< 2} \prod_{i=1}^{3} |y_i|^{-e_i}d\mu(y).
\end{align}
To proceed, we first obtain a bound on
\begin{align*}
    \int_{c^{k}< |Q(y)|\le c^{k+1},1\le |y|<2}   \prod_{i=1}^{3} |y_i|^{-e_i}d\mu(y).
\end{align*}
By symmetry, it suffices to bound
\begin{align}\label{breaksym}
\int_{\substack{c^k<|3Q_3(y_3)|<c^{k+1}\\ \max(|y_1|,|y_2|)\le |y_3|,1\le |y_3|<2}} \prod_{i=1}^{3} |y_i|^{-e_i}d\mu(y).
\end{align}
Arguing as the proof of the finiteness of the integral in (11.0.7) of \cite[Proposition 11.2]{Getz:Hsu}, \eqref{breaksym} is bounded by a constant depending on $d_i,e_i$ times 
    \begin{align}\label{middlestep}
        \int_{|v_3|<2, |3Q_3(v_3)|\le c^{k+1}} dv_3.
    \end{align}
 By Lemma \ref{Lem:levelsetestimate}, \eqref{middlestep} is  $O(c^{\alpha\min(0,k+1)})$
for some $1>\alpha>0$. Take $\varepsilon<\alpha/2$.
 Then \eqref{ultimate1} is dominated by
\begin{align*}
    &\sum_{k=-\infty}^{j-1}c^{-k\alpha/2}c^{\alpha\min(0,k+1)}<\infty.\qedhere
\end{align*}
\end{proof}

\begin{proof}[Proof of Theorem \ref{thm:FYuni}]
    Assume for the moment that \eqref{PlancherelY} is valid for functions in $\mathcal{S}$. Let $f\in \mathcal{S}$ and choose $f_i$ as in Lemma \ref{lem:cont} for $f$.
 We recall from \cite[Corollary 12.7]{Getz:Hsu} that
        \begin{align*}
            \mathcal{F}_Y\circ \mathcal{F}_Y(f)=f \quad  \textrm{ and } \quad 
         \overline{\mathcal{F}_Y(f)}&=\mathcal{F}_Y(\overline{f}).
        \end{align*}
    Thus we obtain
    \begin{align*}
        \norm{\mathcal{F}_Y(f)}_2^2=&\int_{Y(F)} \mathcal{F}_Y(f)(y)\overline{\mathcal{F}_Y(f)}(y)d\mu(y)=\int_{Y(F)} \mathcal{F}_Y(f)(y)\mathcal{F}_Y(\overline{f})(y)d\mu(y)\\
        =&  \lim_{i\to \infty}\int_{Y(F)} f_i(y)\mathcal{F}_Y(\overline{f})(y)d\mu(y)=\lim_{i\to \infty}\int_{Y(F)} \mathcal{F}_Y(f_i)(y)\overline{f}(y)d\mu(y)\\=&\int_{Y(F)} f(y)\overline{f}(y)d\mu(y)=\norm{f}_2^2.
    \end{align*}
    Since $\mathcal{S}$ contains $C^\infty_c(Y^{\mathrm{sm}}(F)),$ which is dense in $L^2(Y(F)),$ the operator $\mathcal{F}_Y$ extends to a unitary operator on $L^2(Y(F))$ and \eqref{PlancherelY} is valid for all $f_1,f_2\in L^2(Y(F))$.

    It remains to prove the identity \eqref{PlancherelY} for $f_1,f_2\in \mathcal{S}$. Recall that as $Y^{\mathrm{sm}}(F)$ is nonempty, $Y^{\mathrm{ani}}(F)$ is dense in $Y(F)$ in the Hausdorff topology. Using Theorem \ref{Thm:main}, we have
    \begin{align} \label{move:y} \begin{split}
         &\int_{Y(F)}\mathcal{F}_Y(f_1)(y)f_2(y)d\mu(y)\\
        &=c\int_{Y^{\mathrm{ani}}(F)}\int_{F^\times} \psi(z^{-1}) \Bigg( \int_{ (F^\times)^3} \overline{\psi}(z^2
        \mathfrak{r}(a))
    \\& \times 
 \Bigg(\int_{Y(F)}\psi\Bigg(\left\langle \frac{y}{a},\xi \right\rangle -\frac{Q(\xi)Q(y)}{9z^2[a]}\Bigg)f_1(\xi)f_2(y)d\mu(\xi)\Bigg) \frac{\chi_Q(a)d^\times a}{\{a\}^{d/2-1}}   \Bigg) d^\times zd\mu(y). \end{split}
 \end{align}
  By Corollary \ref{cor:bound} and Lemma \ref{lem:L1Y}, the integral
\begin{align*}
\begin{split}
    &\int_{Y^{\mathrm{ani}}(F)}\int_{F^\times }|f_2(y)| \Bigg|\int_{ (F^\times)^3} \overline{\psi}(z^2\mathfrak{r}(a))
 \Bigg(\int_{Y(F)}\psi\Bigg(\left\langle \frac{y}{a},\xi \right\rangle -\frac{Q(\xi)Q(y)}{9z^2[a]}\Bigg)f_1(\xi)d\mu(\xi)\Bigg)\frac{\chi_Q(a)d^\times a}{\{a\}^{d/2-1}}\Bigg| d^\times zd\mu(y)
 \end{split}
\end{align*}
is finite.  By \eqref{eq:changeaxi}, Corollary \ref{cor:absolutebd}, and Lemmas \ref{lem:finish} and \ref{lem:L1Y}, we have 
\begin{align*}
    \int_{Y(F)}\int_{(F^\times)^3}|f_2(y)| \Bigg|\int_{Y(F)}\psi\Bigg(\left\langle \frac{y}{a},\xi \right\rangle -\frac{Q(\xi)Q(y)}{9z^2[a]}\Bigg)f_1(\xi)d\mu(\xi)\Bigg| \frac{d^\times a}{\{a\}^{d/2-1}}d\mu(y)<\infty.
\end{align*}
Finally, by Lemma \ref{lem:L1Y} we have
\begin{align*}
    \int_{Y(F)}\int_{Y(F)}|f_2(y)||f_1(\xi)|d\mu(\xi) d\mu(y)<\infty.
\end{align*}
Thus by the Fubini--Tonelli theorem, we can move the integral over $y \in Y(F)$ in \eqref{move:y} to see that it is equal to
\begin{align*}
          &c\int_{F^\times} \psi(z^{-1}) \Bigg( \int_{ (F^\times)^3} \overline{\psi}\Bigg(z^2\mathfrak{r}(a)\Bigg)
    \\& \times 
\Bigg(\int_{Y(F) \times Y(F)}\psi\Bigg(\left\langle y,\frac{\xi}{a} \right\rangle -\frac{Q(\xi)Q(y)}{9z^2[a]}\Bigg)f_1(\xi)f_2(y)d\mu(\xi)d\mu(y)\Bigg) \frac{\chi_Q(a)d^\times a}{\{a\}^{d/2-1}} \Bigg) d^\times z\\
    \end{align*}
    This is visibly symmetric in $f_1$ and $f_2$ and we deduce the theorem.
\end{proof}

\section{Analytic estimates}\label{sec:absolutebd}

Let $F$ be a local field of characteristic $0$. In this section, we establish the estimates used in \S \ref{sec:FY}. We follow the notation in \S \ref{sec:FY}. The Haar measure $du$ on $V(F)$ is normalized to be self-dual with respect to $\psi$ and the pairing $\langle \cdot,\cdot\rangle$. The following theorem is the main result:
\begin{Thm}\label{thm:absolutebdappendix}
 Suppose $\dim V_i>2$ for all $1\le i\le 3$. Let $f\in\mathcal{S}(V(F))$. Given $\frac{1}{2}>\varepsilon>0,$ one has 
 \begin{align}\label{eq:absolutebd}
 \begin{split}
    & \int_{(F^\times)^3 \times F^2} \Bigg|\int_{V(F)}\psi\Bigg(\left\langle \frac{a\xi}{\widetilde{\xi}},u \right\rangle +\frac{b[a]Q(u)}{9[\widetilde{\xi}]}+\mathfrak{c}(u,v)\Bigg)f(u)du\Bigg|  \frac{dv d^\times a }{\{a\}^{1-d/2}}\\ &\qquad\qquad
    \ll_{\varepsilon,f} \min\left(1,\left|\frac{\xi_1 \otimes \xi_2 \otimes \xi_3}{b}\right|\right)^{\min(d_1,d_2,d_3)/2-1-\varepsilon}
    \end{split}
\end{align}
for $(b,\xi)\in F \times Y^{\mathrm{ani}}(F)$.  Here by convention $\min\left(1,\left|\frac{\xi_1 \otimes \xi_2 \otimes \xi_3}{b}\right|\right)=1$ if $b=0.$
\end{Thm}
We bound the left-hand side as an iterated integral, first establishing a bound on the inner integral in (\ref{use:quad:esti}). We then treat the integral over $F^2\cong N_0(F)$ in \S \ref{ssec:overn0} by analyzing the bound in (\ref{use:quad:esti}) term-by-term. Finally, we bound the integral over $(F^\times)^3$ in \S \ref{ssec:overa}, proving the theorem. 

We begin with a lemma that estimates the integral over $V(F)$.  
\begin{Lem} \label{lem:quad:esti}
Let $f \in \mathcal{S}(V_i(F)),$ $\xi \in V_i(F),$ $b \in F$. 
 There are a pair of $\cc$-linear maps
\begin{align*}
\mathcal{S}(V_i(F)) &\lto \mathcal{S}(F \times V_i(F))\\
f &\longmapsto \Psi_{j,f}
\end{align*}
for $1 \leq j \leq 2,$ continuous in the archimedean case, such that
\begin{align*}
\int_{V_i(F)}\psi\left( bQ_i(u)+\left\langle \xi,u \right\rangle_i \right)f(u)du=\begin{cases}\Psi_{1,f}\left( b,\xi \right) + \Psi_{2,f}\left(\frac{1}{b},\frac{\xi  }{b}\right)|b|^{-d_i/2}\gamma(bQ_i)\overline{\psi}\left(\frac{Q_i(\xi)}{b}\right) &\textrm{ if }b \neq 0,\\
\Psi_{1,f}(0,\xi) &\textrm{ if }b=0.
\end{cases}
\end{align*}

\end{Lem}
\noindent Here $\gamma(bQ_i):=\gamma(\psi \circ bQ_i)$ is the Weil index \cite[Th\'eor\`eme 2]{Weil:certains}.  It satisfies $|\gamma(bQ_i)|=1$.  
\begin{proof}
Let $p_1,p_2 \in C^\infty(F)$ be a partition of unity such that 
\begin{align*}
    p_1(x)=1 \textrm{ for }|x|\leq 1 \quad \textrm{and}\quad 
    p_2(x)=1 \textrm{ for }|x| \geq 2.
\end{align*}
For $(x,y) \in F \times V_i(F),$ let
$$
\Psi_{1,f}(x,y):=p_1(x)\int_{V_i(F)} \psi(xQ_i(u)+\left \langle y,u\right \rangle_i)f(u)du.
$$
It is easy to check that $\Psi_{1,f} \in \mathcal{S}(F \times V_i(F))$. 

By Fourier inversion and \cite[Th\'eor\`eme 2]{Weil:certains}, for $(x,y) \in F^\times \times V_i(F)$ one has
\begin{align*}
&\int_{V_i(F)} \psi\left(xQ_i(u)+\left\langle y,u \right\rangle_i \right)f(u)du\\&=\frac{\gamma(xQ_i)}{|x|^{d_i/2}}
\int_{V_i(F)}\overline{\psi}\left(\frac{Q_i(y-u)}{x}\right)
 \widehat{f}(u)du\\
 &=\frac{\gamma(xQ_i) }{|x|^{d_i/2}}
\int_{V_i(F)}\overline{\psi}\left(\frac{ Q_i(u)-2\langle y,u \rangle_i+Q_i(y)}{x}\right)
 \widehat{f}(u)du.
\end{align*}
Here $\widehat{f}$ denotes the Fourier transform of $f$ with respect to the pairing $\langle\cdot,\cdot\rangle_i.$  Thus we can set 
\[
\Psi_{2,f}(x,y):=p_2(x^{-1})
\int_{V_i(F)}\overline{\psi}\left(-2\langle y,u \rangle_i+xQ_i(u)\right)
 \widehat{f}(u)du
\]
if $x \neq 0$ and $\Psi_{2,f}(0,y):=\int_{V_i(F)}\overline{\psi}\left(-2\langle y,u \rangle_i\right)
 \widehat{f}(u)du=f(-2y).$
\end{proof}

It suffices to prove Theorem \ref{thm:absolutebdappendix}  when $f=f_1 \otimes f_2 \otimes f_3$ is a pure tensor; we henceforth assume this. To simplify matters, we introduce the space of \textbf{rapidly decreasing} functions $\mathcal{A}(F^n)$ on $F^n$ as follows.  If $F$ is nonarchimedean, $\mathcal{A}(F^n):=\mathcal{S}(F^n)$.   If $F$ is archimedean,  let $C(F^n)$ be the space of continuous complex-valued functions on $F^n.$
 For a multi-index $\alpha=(\alpha_1,\ldots,\alpha_n)\in \zz_{\ge 0}^{n}$ and $x=(x_1,\ldots,x_n)\in F^n$, define $x^\alpha:=\prod_{i=1}^n x_i^{\alpha_i} \in F$. Put
\begin{align*}
    \mathcal{A}(F^n):=\left\{ f\in C(F^n): \norm{f}_\alpha:=\sup_{x\in F^{n}} |f(x)|(x^{\alpha}\overline{x}^\alpha)^{1/2} <\infty \textrm{ for all }\alpha\in \zz_{\ge 0}^{n}\right\}.
\end{align*}
Here the bar denotes complex conjugation, which is trivial if $F$ is real.
The seminorms $\norm{\cdot}_{\alpha}$ define a topology on $\mathcal{A}(F^n)$ under which the natural inclusions $\mathcal{S}(F^n)\hookrightarrow \mathcal{A}(F^n)\hookrightarrow L^p(F^n)$ for $0<p\le \infty$ and linear operators
\begin{align*}
    \mathcal{A}(F^n)&\lto \cc\\
        f&\mapsto   \int_{F^n} x^\alpha f(x)dx
\end{align*}
are continuous for any $\alpha\in \zz_{\ge 0}^{n}.$

Consider the map
\begin{align}\label{eqn: two variables}
\begin{split}
\mathcal{A}(V_i(F)) &\lto \mathcal{A}(F)\\
f_i &\longmapsto \left(a\mapsto\sup_{0\neq \xi_i\in V_i(F)} \left|f_i\left(\frac{a\xi_i}{\tilde{\xi}_i}\right)\right|\right).
\end{split}
\end{align}
It is continuous when $F$ is archimedean.
Thus by Lemma \ref{lem:quad:esti}  one has 
\begin{align} \label{use:quad:esti}\begin{split} 
    &\left|\int_{V(F)}\psi\left(\left\langle \frac{a \xi}{\widetilde{\xi}},u \right\rangle +\frac{b[a]Q(u)}{9[\widetilde{\xi}]}+\mathfrak{c}(u,v)\right)f(u)du\right| 
     \le H\left(a,\frac{b[a]}{3[\widetilde{\xi}]},v_1,v_2\right) \end{split}
    \end{align}
    where
    \begin{align} \label{Ha}
    \begin{split}
        H(a,c,v_1,v_2):&=
    \left(\Phi_1\Bigg (v_1+\tfrac{c}{3}
    ,a_1\Bigg )+\Phi'_1\Bigg(\frac{(1,a_1)}{v_1+\tfrac{c}{3}} \Bigg)\Bigg|v_1+\tfrac{c}{3}\Bigg|^{-d_1/2}\right)\\
    &\times\left(\Phi_2\Bigg (
    v_2+\tfrac{c}{3}
    ,a_2\Bigg )+\Phi'_2\Bigg(
    \frac{(1,a_2)}{v_2+\tfrac{c}{3}} \Bigg)\Bigg|v_2+\tfrac{c}{3}\Bigg|^{-d_2/2}\right)\\
    &\times\left(\Phi_3\Bigg (-v_1-v_2+\tfrac{c}{3},a_3\Bigg )+\Phi'_3\Bigg(
    \frac{(1,a_3)}{-v_1-v_2+\tfrac{c}{3}} \Bigg)\Bigg|-v_1-v_2+\tfrac{c}{3}\Bigg|^{-d_3/2}\right)\end{split}
\end{align}
for some $\Phi_i:=\Phi_{1,f_i},\Phi'_i:=\Phi_{2,f_i}\in\mathcal{A}(F^{2})$, which can be chosen continuously in $f_i \in \mathcal{S}(V_i(F))$ in the archimedean case.   

\subsection{Estimation of the integral over $N_0(F)$}\label{ssec:overn0}
We will expand the product in $H(a,c,v_1,v_2)$ and then integrate each of the corresponding terms in \eqref{use:quad:esti}. 

\begin{Lem} \label{lem:invconv}
Suppose $d,d_1,d_2\ge 2$. Let $\Phi,\Phi_1,\Phi_2\in \mathcal{A}(F)$. 
\begin{enumerate}[label=\upshape(\roman*),ref=\theLem (\roman*)]
    \item\label{lem:invconv1}  There exist $\Psi_1\in \mathcal{A}(F), \Psi_2\in \mathcal{A}(F^2)$ such that 
\begin{align*}
    \left|\int_{F} \Phi_1\left(\frac{x}{v}\right)|v|^{-d}\Phi_2(c-v)dv\right| \le \Psi_1\left(\frac{x}{c}\right)|c|^{-d}+\Psi_2(x,c)
\end{align*}
for $(x,c)\in F \times F^\times$ with $|x|\ge 1$.
    \item\label{lem:invconv2} There exists $\Psi\in\mathcal{A}(F^{2})$ such that 
\begin{align*}\begin{split}
    &\left|\int_{F} \Phi_1\left(\frac{x_1}{v}\right)|v|^{-d_1}\Phi_2\left(\frac{x_2}{c-v}\right)|c-v|^{-d_2}dv\right|\\
    &\ll \Psi\left(\frac{(x_1,x_2)}{c}\right)\Bigg(|x_1|^{1-d_1}|c|^{-d_2}+|c|^{-d_1}|x_2|^{1-d_2}\Bigg)+\max(|c|,|x_1|,|x_2|)^{1-d_1-d_2}
    \end{split}
\end{align*}
for $(x_1,x_2,c)\in (F^\times)^3 $.
\item\label{lem:maxconv} One has
\begin{align*}
 \left|\int_{F}\max(|v|,|x|)^{-d}\Phi(c-v)dv\right|\ll \max(|c|,|x|)^{-d}
 \end{align*}
 for $(x,c)\in F^2$ with $|x|\ge 1$.
 \item\label{lem:maxinvconv} There exists $\Psi\in\mathcal{A}(F)$ such that 
\begin{align*}
\begin{split}
&\left|\int_{F}\max(|v|,|x_1|)^{-d_1}\Phi\left(\frac{x_2}{c-v}\right)|c-v|^{-d_2}dv\right|\\
&\ll \Psi\left(\frac{x_2}{c}\right)\Bigg(\max(|x_1|,|c|)^{-d_1}|x_2|^{1-d_2}+|x_1|^{-d_1}|c|^{-d_2}\min(|x_1|,|c|)\Bigg)\\
&+\max(|c|,|x_2|)^{1-d_2}\max(|c|,|x_1|,|x_2|)^{-d_1}
\end{split}
\end{align*}
for $(x_1,x_2,c)\in (F^\times)^3$.
\end{enumerate}
Those implicit constants and rapidly decreasing functions $\Psi,\Psi_1,\Psi_2$ may be chosen continuously as a function of $\Phi,\Phi_1,\Phi_2\in \mathcal{A}(F)$ in the archimedean case.
\end{Lem}

For the proof of the lemma, we consider the following extension of functions: Let $\Psi$ be a continuous function defined on $\{x\in F^n: |x|\ge 1\}$. In the nonarchimedean case, if $\Psi$ is smooth of compact support, let $\Psi_{\mathrm{ext}}\in \mathcal{A}(F^n)$ be the function that extends $\Psi$ by zero. In the archimedean case, for $0\le t\le 1$ and $x\in F^n$ of norm $1$, $\Psi_{\mathrm{ext}}(tx):= t\Psi(x)$. One has
\begin{align*}
    \norm{\Psi_{\mathrm{ext}}}_\alpha =\sup_{x\in F^n, |x|\ge 1} \left|x^\alpha \Psi(x)\right|. 
\end{align*}
If this is finite for all $\alpha \in \zz^n_{\geq 0}$ then we say $\Psi$ is rapidly decreasing.  In this case $\Psi_{\mathrm{ext}} \in \mathcal{A}(F^n).$  In either the nonarchimedean or archimedean case if $\Psi$ is nonnegative then so is $\Psi_{\mathrm{ext}}$. 

\begin{proof}
We can and do assume every function under consideration is nonnegative. 

For (i), take a change of variables  $v\mapsto  cv$. The integral becomes
\begin{align*}
    &|c|^{1-d}\int_{F} \Phi_1\left(\frac{x/c}{v}\right)|v|^{-d}\Phi_2(c-cv)dv.
\end{align*}
The contribution of $|v|\le 1/4$  is bounded by 
\begin{align}\label{eq:firsti}
   |c|^{1-d}\int_{|v|\le 1/4}\Phi_1\left(\frac{x/c}{v}\right)|v|^{-d}dv\sup_{3|c|\ge |w|\ge |c|/4} \Phi_2(w).
\end{align}
For $(x,c)\in F^2$, put
\begin{align*}
    \Psi'(c):= |c|\sup_{3|c|\ge |w|\ge |c|/4} \Phi_2(w), \quad \Psi_2'(x,c):= \sqrt{\sup_{|x|\le |r|}\Psi'(r)\Psi'(c)}.
\end{align*}
Note that if $|\frac{x}{c}|\le 1$ and $|x|\ge 1$, then taking a change of variables $v\mapsto vx/c$ in \eqref{eq:firsti}, we see that it is bounded by
\begin{align*}
    \Psi'(c)|x|^{-d}\int_{F} \Phi_1(v^{-1})|v|^{-d}dv\ll_{\Phi_1} \Psi'_2(x,c).
\end{align*}
For $y\in F$ and $|y|\ge 1$, put
\begin{align*}
    \Psi_1'(y):= \int_{|v|\le 1/4}\Phi_1\left(\frac{y}{v}\right)|v|^{-d}dv.
\end{align*}
Then $(\Psi_1')_{\mathrm{ext}}\in \mathcal{A}(F), \Psi_2'\in \mathcal{A}(F^2).$  Bounding the contribution of $\left|\frac{x}{c}\right| \geq 1$ in terms of $\Psi_1'$ and the contribution of $\left|\frac{x}{c} \right| \leq 1$ in terms of $\Psi_2',$ we see that \eqref{eq:firsti} is bounded by a constant $C_{\Phi_1,\Phi_2}$ times
\begin{align*}
    |c|^{-d}(\Psi_1')_{\mathrm{ext}}\left(\frac{x}{c}\right)+\Psi_2'(x,c).
\end{align*}
The contribution of  $|1-v|\le 1/4$ is bounded by $|c|^{-d}$ times
\begin{align*}
    \Psi''_1\left(\frac{x}{c}\right):=\sup_{3\ge |v|\ge 1/4}\Phi_1\left(\frac{x/c}{v}\right)|v|^{-d} \norm{\Phi_2}_1.
\end{align*}
Note that $\Psi''_1\in \mathcal{A}(F)$. Finally,
define 
\begin{align*}
    \Psi'''_1(y)&:= \sup_{\substack{y=x/c\\ |x|\ge 1, |c|\le 1}}\int_{\min(|v|,|1-v|)>\tfrac{1}{4}} \Phi_1\left(\frac{x/c}{v}\right)|v|^{-d}\Phi_2(c(1-v))dv\quad  \textrm{ for } y\in F, |y|\ge 1,\\
    \Psi''_2(x,c)&:= |c|^{1-d}\int_{\min(|v|,|1-v|)>\tfrac{1}{4}} \Phi_1\left(\frac{x/c}{v}\right)|v|^{-d}\Phi_2(c(1-v))dv\quad  \textrm{ for } (x,c)\in F^2, |c|\ge 1,
\end{align*}
Let $\Psi'''_2(x,c):=(\Psi''_2(x,\cdot))_{\mathrm{ext}}(c).$  Thus $\Psi_2''' \in  \mathcal{A}(F^2).$ Taking $\Psi_1:= (\Psi'_1)_{\mathrm{ext}}+\Psi''_1+(\Psi'''_1)_{\mathrm{ext}}$ and $\Psi_{2}:=\Psi'_2+\Psi'''_2$ proves (i). 

For (ii), the contribution of $|v|\le |c|/2$ is dominated by
\begin{align*}
    &\int_{|v|\le |c|/2}\Phi_1\left(\frac{x_1}{v}\right)|v|^{-d_1}dv\sup_{|w|\le |c|/2} \Phi_2\left(\frac{x_2}{c-w}\right) |c|^{-d_2}\\
    &\le \int_{|v|\ge 2|x_1/c|}\Phi_1(v)|v|^{d_1-2}dv\sup_{|w|\ge |x_2|/4|c|} \Phi_2\left(w\right) |x_1|^{1-d_1}|c|^{-d_2}.
\end{align*}
Here we have used the fact that if $|w| \leq |c|/2$ then $|c-w|\leq 4|c|$. Put
\begin{align*}
    \Psi'(y_1,y_2):= \int_{|v|\ge 2|y_1|}\Phi_1(v)|v|^{d_1-2}dv\sup_{|w|\ge |y_2|/4}\Phi_2\left(w\right).
\end{align*}
It is a function in $\mathcal{A}(F^2)$.
We obtain an analogous bound $\Psi''\left(\tfrac{(x_1,x_2)}{c}\right)|c|^{1-d_1}|x_2|^{1-d_2}$ in the
 $|c-v|\le |c|/2$ range by symmetry. Take $\Psi:=\Psi'+\Psi''$.  Now we bound the integral over the range $|v|>|c|/2,|c-v|>|c|/2,$ giving bounds in terms of $|x_1|,|x_2|,$ and $|c|$ so that we can take the minimum.  
 Clearly, the integral is dominated by  $\norm{\Phi_1}_\infty\norm{\Phi_2}_\infty|c|^{1-d_1-d_2}$.  It is also bounded by
\begin{align*}
    &\norm{\Phi_2}_\infty\int_{|v|>|c|/2,|c-v|>|c|/2} \Phi_1\left( \frac{x_1}{v}\right)|v|^{-d_1}|c-v|^{-d_2}dv.\\
    &= \norm{\Phi_2}_\infty|x_1|^{1-d_1-d_2}\int_{|v|>|y|/2,|y-v|>|y|/2} \Phi_1\left( \frac{1}{v}\right)|v|^{-d_1}|y-v|^{-d_2}dv,
\end{align*}
where $y=c/x_1$. Note that $|v|/|y-v| \ll 1$ in  the domain of integration. Therefore the integral is dominated by
\begin{align*}
    \int_{F} \Phi_1\left( \frac{1}{v}\right)|v|^{-d_1-d_2}dv=\int_{F} \Phi_1\left( v\right)|v|^{d_1+d_2-2}dv.
\end{align*}
Now (ii) follows by symmetry.

For (iii), if $|c|\le 2|x|,$ using $\max(|v|,|x|)\ge |x|,$ we have
\begin{align*}
    \int_{F}\max(|v|,|x|)^{-d}\Phi(c-v)dv\le |x|^{-d}\norm{\Phi}_{1}.
\end{align*}
If $|c|>2|x|,$ as $|x|\ge 1,$ we have
\begin{align*}
    &\int_{F}\max(|v|,|x|)^{-d}\Phi(c-v)dv\\
    &\le \int_{|v|\le |x|}|x|^{-d}\Phi(c-v)dv+\int_{|c-v|\le |c|/2}|v|^{-d}\Phi(c-v)dv+\int_{|v|\ge |x|,|c-v|> |c|/2}|v|^{-d}\Phi(c-v)dv\\
    & \ll |x|^{1-d}\sup_{|v|\le |x|} \Phi(c-v)+|c|^{-d}\norm{\Phi}_{1}+|x|^{1-d}\sup_{|v|\ge |c|/2} \Phi(v)\\
    &\ll_\Phi |c|^{-d}.
\end{align*}
Since $\max(|c|,|x|)\le \max(|c|,2|x|)\le 2\max(|c|,|x|),$ this proves (iii).

For (iv), consider the integral over $|v|\ge |x_1|$. By the proof of (ii), the contribution of $|v|\le |c|/2$ is bounded by 
$$\Psi'\left(\frac{x_2}{c}\right)|x_1|^{1-d_1}|c|^{-d_2},
$$ 
and
the contribution of $|c-v|\le |c|/2$ is bounded by $$\Psi'\left(\frac{x_2}{c}\right)|x_2|^{1-d_2}\max(|x_1|,|c|)^{-d_1}$$ for some $\Psi'\in\mathcal{A}(F)$ that can be chosen continuously in $\Phi$ in the archimedean case. The contribution of $|v|\geq |x_1|,|v|> |c|/2,|c-v|> |c|/2$ is dominated by
$$
\max(|c|,|x_1|,|x_2|)^{1-d_1-d_2}\le \max(|c|,|x_2|)^{1-d_2}\max(|c|,|x_1|,|x_2|)^{-d_1}.
$$
Indeed, this follows from the same argument as the last part of the proof of (ii).  

Thus we are left with bounding
\begin{align}\label{eq:maxinvconvfinal}
    |x_1|^{-d_1}\int_{|v|\le |x_1|}\Phi\left(\frac{x_2}{c-v}\right)|c-v|^{-d_2}dv.
\end{align}
If $2|x_2|\ge |c|\ge 2|x_1|,$ then \eqref{eq:maxinvconvfinal} is bounded by 
    \begin{align*}
        &|x_1|^{1-d_1}|x_2|^{-d_2}\sup_{|v|\le |x_1|}\Phi\left(\frac{x_2}{c-v}\right)\left|\frac{x_2}{c-v}\right|^{d_2}\le |x_1|^{1-d_1}|c|^{-d_2}\Psi''\left(\frac{x_2}{c}\right)
    \end{align*}
    where 
    \begin{align*}
        \Psi''(y):=\sup_{\substack{y=x_2/c\\ 2|x_2|\ge |c|\ge 2|x_1|}}\sup_{|v|\le |x_1|}\Phi\left(\frac{x_2}{c-v}\right)\left|\frac{x_2}{c-v}\right|^{d_2} \textrm{ for } y\in F, |y|\ge 1/2.
    \end{align*}
    We can extend this function to a function in $\mathcal{A}(F)$ by a minor variant of the construction explained before the proof.
     For $2|x_2|\ge 2|x_1|\ge  |c|,$ arguing as the case $2|x_2|\ge |c|\ge 2|x_1|,$ the expression \eqref{eq:maxinvconvfinal} is bounded by
\begin{align*}
    |x_1|^{-d_1}|x_2|^{1-d_2}\Psi'''\left(\frac{x_2}{x_1}\right)\le |x_2|^{1-d_2-d_1}\sup_{|v|\ge 1}\Psi'''(v)|v|^{d_1}
\end{align*}
for some $\Psi''' \in \mathcal{A}(F).$
If $|c|\ge 2\max(|x_2|,|x_1|),$ then \eqref{eq:maxinvconvfinal} is bounded by 
\begin{align*}
    &|x_1|^{1-d_1}\sup_{|v|\le |x_1|}\Phi\left(\frac{x_2}{c-v}\right)|c-v|^{-d_2}\\
    &=|x_1|^{1-d_1}|c|^{-d_2}\sup_{|v|\le |x_1/c|}\Phi\left(\frac{x_2/c}{1-v}\right)|1-v|^{-d_2}\\
    &\le |x_1|^{1-d_1}|c|^{-d_2}\sup_{|v|,|u|\le 1/2}\Phi\left(\frac{u}{1-v}\right)|1-v|^{-d_2}\ll_\Phi |x_1|^{1-d_1}|c|^{-d_2}.
\end{align*}

Suppose $2|x_1|>\max(|c|,2|x_2|)$. We write \eqref{eq:maxinvconvfinal} as
$|x_1|^{-d_1}|x_2|^{1-d_2}$ times
\begin{align*}
     &\int_{|v|\le |x_1/x_2|}\Phi\left(\frac{1}{c/x_2-v}\right)|c/x_2-v|^{-d_2}dv\le \int_{F }\Phi(v)|v|^{d_2-2}dv.
\end{align*}
Altogether, we have proven  (iv).
\end{proof}

Recall the definition of 
$H(a,c,v_1,v_2)$ from \eqref{Ha}.  It depends on functions $\Phi_i,\Phi_i' \in \mathcal{A}(F^2)$.  
In the rest of the section, all implicit constants and rapidly decreasing functions can be and are chosen continuously as a function of $\Phi_i,\Phi_i'$ when $F$ is archimedean.

Assume $c \neq 0$.  By taking a change of variables $v_2 \mapsto v_2-v_1$ and then $v_1\mapsto v_1-c/3$ and $v_2\mapsto v_2-2c/3,$ one has
\begin{align} \label{intHvi} \begin{split}
    \int_{F^2}H(a,c,v_1,v_2) dv
   =&\int_{F}\Bigg(\Phi_3\left(c-v_2,a_3\right)+\Phi_3'\left(\frac{(1,a_3)}{c-v_2}\right)|c-v_2|^{-d_3/2}\Bigg) \\
    &\times
    \int_{F} \Bigg(\Phi_1\left(v_1,a_1\right)+\Phi_1'\left(\frac{(1,a_1)}{v_1}\right)|v_1|^{-d_1/2}\Bigg)\\
    &\times\Bigg(\Phi_2\left(v_2-v_1,a_2\right)+\Phi_2'\left(\frac{(1,a_2)}{v_2-v_1}\right)|v_2-v_1|^{-d_2/2}\Bigg)dv. \end{split}
\end{align}
Here the $d_i$ are all even integers bigger than $2$.  This is necessary so that we can invoke Lemma \ref{lem:invconv} in the argument below.  

This integral corresponds to the integral over $N_0(F),$ hence the title of this subsection.
 By parts (i) and (ii) of Lemma \ref{lem:invconv} and the map \eqref{eqn: two variables},  the integral over $v_1$ in \eqref{intHvi} is dominated by
\begin{align*}
    & \Psi_0\left(v_2,a_1,a_2\right)+\Psi_1\left(\frac{(1,a_1)}{v_2},a_2\right)|v_2|^{-d_1/2}+\Psi_2\left(\frac{(1,a_2)}{v_2},a_1\right)|v_2|^{-d_2/2}\\
    &+\Psi_3\left(\frac{(1,a_1,a_2)}{v_2}\right)\Bigg(\max\left(1,|a_1|\right)^{1-d_1/2}|v_2|^{-d_2/2}+|v_2|^{-d_1/2}\max\left(1,|a_2|\right)^{1-d_2/2}\Bigg)\\
    &+\max\left(1,|v_2|,|a_1|,|a_2|\right)^{1-d_1/2-d_2/2}
\end{align*}
for some $\Psi_i \in \mathcal{A}(F^3)$. 
Note that $$\Psi_0\left(v_2,a_1,a_2\right)\ll \max\left(1,|v_2|,|a_1|,|a_2|\right)^{1-d_1/2-d_2/2}.$$
Therefore, by symmetry in $a_1$ and $a_2,$ to bound \eqref{intHvi} it suffices to study the integral
\begin{align*}
    &\int_{F} \Bigg(\Psi_1\left(\frac{(1,a_1)}{v_2},a_2\right)|v_2|^{-d_1/2}+\Psi_3\left(\frac{(1,a_1,a_2)}{v_2}\right)\max\left(1,|a_1|\right)^{1-d_1/2}|v_2|^{-d_2/2}\\
    &+\max\left(1,|v_2|,|a_1|,|a_2|\right)^{1-d_1/2-d_2/2}\Bigg)\Bigg(\Phi_3\left(c-v_2,a_3\right)+\Phi_3'\left(\frac{(1,a_3)}{c-v_2}\right)|c-v_2|^{-d_3/2}\Bigg)dv_2.
\end{align*}
In the following discussion, 
$$
M_1,M_2,M_1',M_2',M_4,M_5\in \mathcal{A}(F^4), \quad M_6\in\mathcal{A}(F^2),\quad  \textrm{ and }\quad M_3,M_4' \in \mathcal{A}(F)
$$
are suitable rapidly decreasing functions.
 By part (i) of Lemma \ref{lem:invconv}, we have
\begin{align} \label{il}
&\int_{F} \Psi_1\left(\frac{(1,a_1)}{v_2},a_2\right)|v_2|^{-d_1/2}\Phi_3\left(c-v_2,a_3\right)dv_2\le M_1\left(\frac{(1,a_1)}{c},a_2,a_3\right)|c|^{-d_1/2}+M_1'\left(c,a\right),
\end{align}
and 
\begin{align} \label{ee}\begin{split}
& \max\left(1,|a_1|\right)^{1-d_1/2}\int_{F}\Psi_3\left(\frac{(1,a_1,a_2)}{v_2}\right)|v_2|^{-d_2/2}\Phi_3\left(c-v_2,a_3\right)dv_2\\
&\le \max\left(1,|a_1|\right)^{1-d_1/2}M_2\left(\frac{(1,a_1,a_2)}{c},a_3\right)|c|^{-d_2/2}+\max\left(1,|a_1|\right)^{1-d_1/2}M_2'\left(c,a\right). \end{split}
\end{align}
By part (iii) of Lemma \ref{lem:invconv}, 
\begin{align} \label{sam} 
 \int_{F}&\max\left(1,|v_2|,|a_1|,|a_2|\right)^{1-d_1/2-d_2/2}\Phi_3\left(c-v_2,a_3\right)dv_2\le M_3\left(a_3\right)\max\left(1,|c|,|a_1|,|a_2|\right)^{1-d_1/2-d_2/2}.
\end{align}
By part (ii) of Lemma \ref{lem:invconv},
\begin{align} \label{sa} \begin{split}
&\int_{F} \Psi_1 \left(\frac{(1,a_1)}{v_2},a_2\right)|v_2|^{-d_1/2}\Phi_3'\left(\frac{(1,a_3)}{c-v_2}\right)|c-v_2|^{-d_3/2}dv_2,\\
\le& M_4\left(\frac{(1,a_1,a_3)}{c},a_2\right)\Bigg(\max\left(1,|a_1|\right)^{1-d_1/2}|c|^{-d_3/2}+|c|^{-d_1/2}\max\left(1,|a_3|\right)^{1-d_3/2}\Bigg)\\
&+M_4'\left(a_2\right)\max\left(1,|c|,|a_1|,|a_3|\right)^{1-d_1/2-d_3/2},\end{split}
\end{align}
and
\begin{align} \label{oh} \begin{split}
&\max\left(1,|a_1|\right)^{1-d_1/2} \int_{F}\Psi_3\left(\frac{(1,a_1,a_2)}{v_2}\right)|v_2|^{-d_2/2}\Phi_3'\left(\frac{(1,a_3)}{c-v_2}\right)|c-v_2|^{-d_3/2}dv_2\\
\ll& \max\left(1,|a_1|\right)^{1-d_1/2}M_5\left(\frac{(1,a)}{c}\right)\Bigg(\max\left(1,|a_1|,|a_2|\right)^{1-d_2/2}|c|^{-d_3/2}+|c|^{-d_2/2}\max\left(1,|a_3|\right)^{1-d_3/2}\Bigg)\\
&+\max\left(1,|a_1|\right)^{1-d_1/2}\max\left(1,|c|,|a|\right)^{1-d_2/2-d_3/2}.\end{split}
\end{align}
By part (iv) of Lemma \ref{lem:invconv}
\begin{align} \label{yok}\begin{split}
 &\int_{F}\max\left(1,|v_2|,|a_1|,|a_2|\right)^{1-d_1/2-d_2/2}\Phi_3'\left(\frac{(1,a_3)}{c-v_2}\right)|c-v_2|^{-d_3/2}dv_2\\
\ll& M_6\left(\frac{(1,a_3)}{c}\right)\Bigg(\max\left(1,|c|,|a_1|,|a_2| \right)^{1-d_1/2-d_2/2}\max\left(1,|a_3|\right)^{1-d_3/2}\\
&+\max\left(1,|a_1|,|a_2|\right)^{1-d_1/2-d_2/2}|c|^{-d_3/2}\min(\max\left(1,|a_1|,|a_2|\right),|c|)\Bigg)\\
&+\max\left(1,|c|,|a_3|\right)^{1-d_3/2}\max\left(1,|c|,|a|\right)^{1-d_1/2-d_2/2}. \end{split}
\end{align}
Note that both $M_1'\left(c,a\right)$ and $\max\left(1,|a_1|\right)^{1-d_1/2}M_2'\left(c,a\right)$ are dominated by $$\max\left(1,|a_1|\right)^{1-d_1/2}\max\left(1,|c|,|a|\right)^{1-d_2/2-d_3/2}.$$
Therefore, by symmetry there are $M\in\mathcal{A}(F^2),$ $M_{1\sigma}, M_{2\sigma},M_{3\sigma} \in \mathcal{A}(F^4)$  (indexed by $\sigma \in S_3,$ the symmetric group on three letters) such that \eqref{intHvi} is dominated by the sum of the following terms: 

    \begin{align}\label{term1}
        \sum_{\sigma\in S_2}M_{1\sigma}\left(\frac{(1,a_{\sigma(1)})}{c},a_{\sigma(2)},a_3\right)|c|^{-d_{\sigma(1)}/2},
    \end{align}
    
    \begin{align}\begin{split}\label{termspecial}
        &M\left(\frac{(1,a_3)}{c}\right)\Bigg(\max\left(1,|c|,|a_1|,|a_2|\right)^{1-d_1/2-d_2/2}\max\left(1,|a_3|\right)^{1-d_3/2}\\
&+\max\left(1,|a_1|,|a_2|\right)^{1-d_1/2-d_2/2}|c|^{-d_3/2}\min(\max\left(1,|a_1|,|a_2|\right),|c|)\Bigg),
    \end{split}
    \end{align}

    \begin{align}\label{term2}
        \sum_{\sigma\in S_3}M_{2\sigma}\left(\frac{(1,a_{\sigma(1)},a_{\sigma(2)})}{c},a_{\sigma(3)}\right)\max\left(1,|a_{\sigma(1)}|\right)^{1-d_{\sigma(1)}/2}|c|^{-d_{\sigma(2)}/2},
    \end{align}

    \begin{align}\label{term3}
        \sum_{\sigma\in C_3}M_{3\sigma}\left(\frac{(1,a)}{c}\right)\max\left(1,|a_{\sigma(1)}|\right)^{1-d_{\sigma(1)}/2}\max\left(1,|a_{\sigma(2)}|\right)^{1-d_{\sigma(2)}/2}|c|^{-d_{\sigma(3)}/2},
    \end{align}

    \begin{align}\label{term4}
    \sum_{\sigma\in C_3}\max\left(1,|a_{\sigma(1)}|\right)^{1-d_{\sigma(1)}/2}\max\left(1,|c|,|a|\right)^{1-d_{\sigma(2)}/2-d_{\sigma(3)}/2}.
    \end{align}

The following table explains how the terms \eqref{term1}, \eqref{termspecial}, \eqref{term2}, \eqref{term3}, \eqref{term4}  are used to bound the previous contributions:
\begin{center}
\begin{tabular}{c c}
    The contribution of  & is dominated by \\
    \eqref{il}  & \eqref{term1}, \eqref{term4}\\
    \eqref{ee} & \eqref{term2}, \eqref{term4}\\
    \eqref{sam} & \eqref{term4}\\ 
    \eqref{sa} & \eqref{term2}, \eqref{term4}\\
    \eqref{oh} & \eqref{term3}, \eqref{term4} \\
    \eqref{yok} & \eqref{termspecial}, \eqref{term4}.
\end{tabular}
\end{center}
The $M$ and $M_{i\sigma}$ can all be chosen continuously as a function of the $\Phi_i,\Phi_i'$ in the archimedean case.  

Before continuing the proof of Theorem \ref{thm:absolutebdappendix}, we
prove the  following strengthening of the theorem in the special case $b=0$:
\begin{Thm} \label{thm:b=0appendix} For $\xi \in \prod_{i=1}^3 (V_i(F)-\{0\}),$ if $d_i>2$ for all $i,$ we have 
\begin{align*}
\int_{F^2}\left|\int_{V(F)}\psi\left(\left\langle \xi,u \right\rangle +\mathfrak{c}(u,v)\right)f(u)du\right|
 dv
 &\ll_f
    \sum_{\sigma\in C_3}\max\left(1,|\xi_{\sigma(1)}|\right)^{1-d_{\sigma(1)}/2}\max\left(1,|\xi|\right)^{1-d_{\sigma(2)}/2-d_{\sigma(3)}/2}\\
    &\ll \prod_{i=1}^3 |\xi_i|^{2-d_i/2}
 \end{align*}
\end{Thm}

\begin{proof} Given Lemma \ref{lem:quad:esti} and our manipulations above, the integral to be bounded is dominated by the limit as $c \to 0$ of  the sum of \eqref{term1}, \eqref{termspecial}, \eqref{term2}, \eqref{term3}, \eqref{term4} in the case $a_i=\widetilde{\xi}_i$.  After taking the limit as $c \to 0$ the only term that is nonzero is \eqref{term4}.
\end{proof}

\subsection{Estimate of the integral over $(F^\times)^3$}\label{ssec:overa}

In this subsection, we bound  \begin{align}\label{H''}
   \int_{(F^\times)^3}\int_{F^2}H\left(a,\frac{b[a]}{3[\widetilde{\xi}]},v_1,v_2\right)dv\{a\}^{d/2-1}d^\times a
\end{align}
using the bounds on the inner integral obtained in the previous section.  Throughout this subsection we assume $d_i\ge 4$ for $1\le i\le 3$ and we fix  $\tfrac{1}{2}>\varepsilon>0$.  To avoid repetition, we point out once and for all that in the archimedean case all implicit constants and rapidly decreasing functions appearing in bounds in this section can be chosen continuously as a function of whatever rapidly decreasing functions appear in the hypotheses of the bounds.  

We start with a general bound that will be useful later.

\begin{Lem}\label{lem:twoterms} Suppose $e_1,e_2\in\rr_{>0}$.  
Let $\Phi\in\mathcal{A}(F)$.  As a function in $(r_1,r_2)\in (F^\times)^2,$ we have 
\begin{align*}
    &\int_{F^\times }\left|\Phi\left(\frac{r_1}{a}\right)\right|\max\left(1,\left|r_2a\right|\right)^{-e_1}
    \frac{d^\times a}{|a|^{e_2-e_1}} \ll_{\varepsilon,\Phi} \left\{\begin{array}{ll}
      |r_1|^{e_1-e_2}\max(1,|r_1r_2|)^{-e_1}   & \textrm{ if } e_2> e_1,  \\
      |r_2|^{e_2-e_1}\max(1,|r_1r_2|)^{-e_2}   & \textrm{ if } e_2< e_1,\\
      \max(1,|r_1r_2|)^{-e_2}\min(1,|r_1r_2|)^{-\varepsilon}    & \textrm{ if } e_2=e_1.
    \end{array} \right.
\end{align*}
\end{Lem}

\begin{proof}
We assume $\Phi$ is nonnegative.
 Taking a change of variables $a\mapsto a^{-1}r_1$ in the integral to be bounded, we obtain
\begin{align*}
    |r_1|^{e_1-e_2}\int_{F^\times }\Phi(a)\max\left(1,\left|\frac{r_1r_2}{a}\right|\right)^{-e_1}|a|^{e_2-e_1}d^\times a.
\end{align*}
Let $r=r_1r_2$.
 We write the integral above as
\begin{align*}
    &|r_1|^{e_1-e_2}\int_{|a|\le |r| }|r|^{-e_1}\Phi(a)|a|^{e_2}d^\times a+ |r_1|^{e_1-e_2}\int_{|a|\ge |r| }\Phi(a)|a|^{e_2-e_1}d^\times a.
\end{align*}
When $e_2 \geq e_1$ we observe that this is
 dominated by the bound claimed in the lemma.  
 Now suppose $e_2< e_1$. Then after a change of variables $a \mapsto ar$, the integral above is 
 \begin{align*}
  &|r_2|^{e_2-e_1}\int_{|a|\le 1 }\Phi(ar)|a|^{e_2}d^\times a+|r_2|^{e_2-e_1} \int_{|a|\ge 1 }\Phi(ar)|a|^{e_2-e_1}d^\times a
 \end{align*}
which is dominated by the bound asserted in the lemma.
\end{proof}

\begin{Lem} \label{lem:term1bound} For nonnegative $M_1 \in \mathcal{A}(F^4)$ and $c \in F^\times,$ one has
\begin{align*}\begin{split}
    &\int_{(F^\times)^3}M_{1}\left(\frac{(1,a_1)c}{[a]},a_2,a_3\right)\left|\frac{[a]}{c}\right|^{-d_{1}/2}\{a\}^{d/2-1}d^\times a \ll_{\varepsilon}  J\left(c\right)\left|c\right|^{\min(d_1,d_2,d_3)/2-1-\varepsilon}
    \end{split}
\end{align*}
for some $J \in \mathcal{A}(F)$.
\end{Lem}

\begin{proof}    A direct computation shows the integral over $a_1$ is dominated by
\begin{align}\label{term1original}
    |c|^{d_1/2-1}J'\left(\frac{c}{a_2a_3},a_2,a_3\right)|a_2|^{(d_2-d_1)/2}|a_3|^{(d_3-d_1)/2}
\end{align}
for some $J' \in \mathcal{A}(F^3)$.  
By symmetry, we may assume $d_2\le d_3$.  Note that  \eqref{term1original} is dominated by
\begin{align*}
    \left|c\right|^{d_1/2-1}J''\left(\frac{c}{a_2a_3},a_2,a_3\right)|a_2|^{(d_2-d_1)/2-\varepsilon}|a_3|^{(d_3-d_1)/2}
\end{align*}
for some $J'' \in \mathcal{A}(F^2).$
The integral of this function over $a_2$ is dominated by
\begin{align*}
   \left|c\right|^{d_1/2-1}J'''\left(\frac{c}{a_3},a_3\right)|a_3|^{(d_3-d_1)/2}\times \begin{cases}
    \left|\frac{c}{a_3}\right|^{(d_2-d_1)/2-\varepsilon}    & \textrm{ if } d_1\ge d_2, \\
     1    & \textrm{ if } d_1<d_2.
    \end{cases}
\end{align*}
for some $J''' \in \mathcal{A}(F^2).$ The integral over $a_3$ of this expression is dominated by the bound in the lemma.
\end{proof}

\begin{Lem} \label{lem:term2bound}  For nonnegative $M_2 \in \mathcal{A}(F^4)$ and $c \in F^\times,$ we have
\begin{align*}\begin{split}
    &\int_{(F^\times)^3}M_{2}\left(\frac{(1,a_1,a_2)c}{[a]},a_3\right)\max\left(1,|a_{1}|\right)^{1-d_{1}/2}\left|\frac{[a]}{c} \right|^{-d_{2}/2}\{a\}^{d/2-1}d^\times a\\
    &\ll_\varepsilon \min(1,|c|)^{\min(d_1,d_2,d_3)/2-1-\varepsilon}\max(1,|c|)^{-1}.
    \end{split}
\end{align*}
\end{Lem}

\begin{proof} Changing variables $a_2 \mapsto a_1^{-1}a_3^{-1}a_2$ and then computing directly, we see that the integral over $a_2$ is dominated by
\begin{align*}
  \left|c\right|^{d_2/2-1}|a_3|^{(d_3-d_2)/2}\max\left(1,|a_1|\right)^{-d_1/2} M'_2\left(\frac{c}{a_1a_3},a_3\right)|a_1|^{(d_1-d_2)/2}
\end{align*}
for some $M_2' \in \mathcal{A}(F^2)$.
 By Lemma \ref{lem:twoterms}, the integral of this function over $a_1$ is bounded by a constant depending on $\varepsilon$ times
\begin{align*}
 \left|c\right|^{d_2/2-1}|a_3|^{(d_3-d_2)/2}M_2''\left(a_3\right)\times \begin{cases}
    \left|\frac{c}{a_3}\right|^{(d_1-d_2)/2-\varepsilon}\max\left(1,\left|\frac{c}{a_3}\right|\right)^{-d_1/2+\varepsilon} & \textrm{ if } d_1\le d_2, \\  
      \max\left(1,\left|\frac{c}{a_3}\right|\right)^{-d_2/2} & \textrm{ if } d_1> d_2
\end{cases}
\end{align*}
for some $M''_2 \in \mathcal{A}(F)$.
Change variables $a_3\mapsto a_3^{-1}$ and apply Lemma  \ref{lem:twoterms} again to complete the proof.
\end{proof}

\begin{Lem}\label{lem:term3bound}
For nonnegative $M_3\in\mathcal{A}(F^4)$ and $c\in F^\times,$ one has
\begin{align*}
    \begin{split}
    &\int_{(F^\times)^3}M_{3}\left(\frac{(1,a)c}{[a]}\right)\max\left(1,|a_{1}|\right)^{1-d_{1}/2}\max\left(1,|a_{2}|\right)^{1-d_{2}/2}\left|\frac{[a]}{c}\right|^{-d_{3}/2}\{a\}^{d/2-1}d^\times a\\
    &\ll_\varepsilon \min\left(1,|c|\right)^{\min(d_1,d_2,d_3)/2-1-\varepsilon}.
    \end{split}
\end{align*}
\end{Lem}

\begin{proof}
By symmetry, we may assume $d_1\le d_2$. A direct computation shows the integral over $a_3$ is dominated by
\begin{align}\label{direct}
\begin{split}
    &|c|^{d_3/2-1}M_3'\left(\frac{c}{a_1a_2}\right)|a_1|^{(d_1-d_3)/2}|a_2|^{(d_2-d_3)/2}\max\left(1,|a_1|,|a_2|\right)^{-1}\\
    &\times\max\left(1,|a_1|\right)^{1-d_1/2}\max\left(1,|a_2|\right)^{1-d_2/2}
\end{split}
\end{align}
for some $M_3' \in \mathcal{A}(F)$.  Since  $\max\left(1,|a_1|,\left|a_2\right|\right)^{-1}\le \max\left(1,|a_1|\right)^{-1},$ by Lemma \ref{lem:twoterms}, the integral over $a_1$ is dominated by
\begin{align} \label{for:a20}
 &|c|^{d_3/2-1}|a_2|^{d_2/2-d_3/2}\max\left(1,|a_2|\right)^{1-d_2/2}\times \begin{cases}
    \left|\frac{c}{a_2}\right|^{d_1/2-d_3/2-\varepsilon}\max\left(1,\left|\frac{c}{a_2}\right|\right)^{\varepsilon-d_1/2} & \textrm{ if } d_1\le d_3, \\  
     \left|\frac{c}{a_2}\right|^{-\varepsilon}\max\left(1,\left|\frac{c}{a_2}\right|\right)^{\varepsilon-d_3/2} & \textrm{ if } d_1\ge d_3.
\end{cases}
\end{align}
Choose $\sigma\in \{\textrm{Id},(13)\}$ such that $d_{\sigma(1)}\le d_{\sigma(3)}$. Then by \eqref{for:a20} the integral in the lemma is bounded by a constant depending on $\varepsilon$ times 
\begin{align*}
\int_{F^\times}
    &|c|^{d_{\sigma(1)}/2-1-\varepsilon}|a_2|^{d_2/2-d_{\sigma(1)}/2+\varepsilon}\max\left(1,|a_2|\right)^{1-d_2/2}\max\left(1,\left|\frac{c}{a_2} \right| \right)^{\varepsilon-d_{\sigma(1)}/2}d^\times a_2.
\end{align*}
Writing $a$ for $a_2$, for $|c|\geq 1$  this is 
\begin{align*}\begin{split}
&|c|^{-1}    \int_{|a|\le 1} |a|^{d_2/2}d^\times a+|c|^{-1}\int_{1<|a|\leq |c|}  |a|d^\times a+|c|^{d_{\sigma(1)}/2-1-\varepsilon}\int_{|c|<|a|}|a|^{1-d_{\sigma(1)}/2+\varepsilon}d^\times a\ll 1.
\end{split}
\end{align*}
For $|c| <1$, this is 
\begin{align*}
&|c|^{-1}    \int_{|a|\le |c|} |a|^{d_2/2}d^\times a+|c|^{d_{\sigma(1)}/2-1-\varepsilon}\int_{|c|<|a|\leq 1}  |a|^{d_2/2-d_{\sigma(1)}/2+\varepsilon}d^\times a+|c|^{d_{\sigma(1)}/2-1-\varepsilon}\int_{1<|a|}|a|^{1-d_{\sigma(1)}/2+\varepsilon}d^\times a\\
&\ll_\varepsilon |c|^{\min(d_1,d_2,d_3)/2-1-\varepsilon}.
\end{align*}
\end{proof}

As a corollary of the proof we obtain:

\begin{Cor}\label{cor:term3bound}
Suppose $|c|>1$. For nonnegative $M_3\in\mathcal{A}(F^4),$ given $\alpha>0,$ there exists $\beta>0$ such that
\begin{align}\label{goodpart}
    \begin{split}
    &\int_{|a|\ge |c|^{\alpha}} M_{3}\left(\frac{(1,a)c}{[a]}\right)\max\left(1,|a_{1}|\right)^{1-d_{1}/2}\max\left(1,|a_{2}|\right)^{1-d_{2}/2}\left|\frac{[a]}{c}\right|^{-d_{3}/2}\{a\}^{d/2-1}d^\times a\ll_{\alpha,\beta}  |c|^{-\beta},
    \end{split}
\end{align}
\end{Cor}

\begin{proof}We can and do assume $\tfrac{1}{2}>\alpha$.  Consider the contribution of $|a_1| \geq |c|^\alpha$.
In view of \eqref{direct}, the integral over $a_3$ is dominated by
\begin{align*}
    &|c|^{d_3/2-1}M_3'\left(\frac{c}{a_1a_2}\right)|a_1|^{1-d_3/2}|a_2|^{(d_2-d_3)/2}\max\left(|a_1|,|a_2|\right)^{-1}\max\left(1,|a_2|\right)^{1-d_2/2}.
\end{align*}
When  $|a_2|\ge |c|^{1-\alpha},$ this expression is dominated by
\begin{align*}
    \norm{M_3'}_\infty|c|^{d_3/2-1} |a_1|^{1-d_3/2}|a_2|^{-d_3/2}.
\end{align*}
Thus, the contribution of the range $|a_1| \geq |c|^{\alpha},$ $|a_2| \ge |c|^{1-\alpha}$ is dominated by $|c|^{\alpha-1}$. 
To bound the contribution of the range $|a_1| \ge |c|^{\alpha},$ $|a_2| <|c|^{1-\alpha}$ to \eqref{goodpart}, we argue as above and then take a change of variables 
 $a_1\mapsto a_1\frac{c}{a_2}$.  This yields a bound of
\begin{align*}
    |c|^{-1}&\int_{|c|^{1-\alpha}\ge |a_2|, |a_1|\ge |a_2||c|^{\alpha -1}} M_3'(a_1^{-1})|a_1|^{-d_3/2}|a_2|^{d_2/2}\max(1,|a_2|)^{1-d_2/2}d^\times a_1d^\times a_2\\
    &\ll |c|^{-\alpha}\int_{F} M'_3(a_1)|a_1|^{d_3/2-1}da_1.
\end{align*}
The contribution of $|a_2| \ge |c|^{\alpha}$ admits the same bound by symmetry.  

Now consider the contribution of $|a_3|\ge |c|^{\alpha}$. It is dominated by $|c|^{-\alpha}$ times
\begin{align*}
    &|c|^{d_3/2}\int_{(F^\times)^2} M_3'\left(\frac{c}{a_1a_2}\right)\\& \times \max\left(1,|a_{1}|\right)^{1-d_{1}/2}\max\left(1,|a_{2}|\right)^{1-d_{2}/2}|a_1|^{d_1/2-d_3/2-1}|a_2|^{d_2/2-d_3/2-1} d^\times a_1 d^\times a_2\\
    &\ll_{M_3',\varepsilon'} |c|^{\varepsilon'}.
\end{align*}
for some $M_3' \in \mathcal{A}(F)$ and $0<\varepsilon'<\alpha$.  
 Here the last inequality follows from an argument similar to that in the proof of Lemma \ref{lem:term3bound}. 
\end{proof}

\begin{Lem}\label{lem:term4bound}
For $c\in F^\times,$ we have
\begin{align}\label{term4bound}
    \int_{(F^\times)^3} \max(1,|a_1|)^{1-d_1/2}\max\left(1,\left|\frac{[a]}{c}\right|,|a|\right)^{1-d_2/2-d_3/2}\{a\}^{d/2-1}d^\times a\ll_\varepsilon \min(1,|c|)^{\min(d_1,d_2,d_3)/2-1-\varepsilon}.
\end{align}
\end{Lem}

\begin{proof} We start with an easy estimate:
\begin{Sublem}\label{lem:onlymax}
Suppose $e_1,e_2\in\rr_{>0}$. As a function of $r\in F^\times,$
\begin{align*} 
    \int_{F^\times} \max\left(1,|a|\right)^{-e_1}\max\left(|a|,|r|\right)^{-e_2}|a|^{e_1}d^\times a
    &\ll_{\varepsilon}\begin{cases}
   |r|^{\varepsilon-e_2}     &  \textrm{ if } |r|\ge 1,\\
    |r|^{-\varepsilon}   &  \textrm{ if } e_1\ge e_2 \textrm{ and } |r|\le 1,\\
   |r|^{e_1-e_2}&  \textrm{ if } e_1< e_2 \textrm{ and } |r|\le 1.
    \end{cases}
\end{align*}
\end{Sublem}

\begin{proof}
    If $|r|\ge 1,$ then the integral is
    \begin{align*}
        &|r|^{-e_2}\int_{|a|\le 1} |a|^{e_1}d^\times a+|r|^{-e_2}\int_{1\le |a|\le |r|} d^\times a+\int_{|r|\le |a|}|a|^{-e_2}d^\times a\ll_\varepsilon |r|^{\varepsilon-e_2}.
    \end{align*}
    If $|r|\le 1,$ the integral is
    \begin{align*}
        &|r|^{-e_2}\int_{|a|\le |r|} |a|^{e_1}d^\times a+\int_{|r|\le |a|\le 1} |a|^{e_1-e_2}d^\times a+\int_{1\le |a|}|a|^{-e_2}d^\times a.
    \end{align*}
\end{proof}

 Applying Sublemma \ref{lem:onlymax} to the integral over $a_1$ with $|r|=\frac{\max(1,|a_2|,|a_3|)}{ \max\left(1,\frac{|a_2a_3|}{|c|}\right)},$  we see the integral over $a_1$ in \eqref{term4bound} is dominated by $|a_2|^{d_2/2-1}|a_3|^{d_3/2-1}$ times
\begin{align*}
   \begin{cases}
    \max\left(1,\frac{|a_2a_3|}{|c|}\right)^{-\varepsilon}\max(1,|a_2|,|a_3|)^{1+\varepsilon-d_2/2-d_3/2}&  \textrm{ if } \frac{\max(1,|a_2|,|a_3|)}{ \max\left(1,\frac{|a_2a_3|}{|c|}\right)} \geq 1,\\
     \max\left(1,\frac{|a_2a_3|}{|c|}\right)^{1+\varepsilon-d_2/2-d_3/2}\max(1,|a_2|,|a_3|)^{-\varepsilon}     &  \textrm{ if } d_1\ge d_2+d_3,\frac{\max(1,|a_2|,|a_3|)}{ \max\left(1,\frac{|a_2a_3|}{|c|}\right)} \leq 1,\\
     \max\left(1,\frac{|a_2a_3|}{|c|}\right)^{1-d_1/2}\max(1,|a_2|,|a_3|)^{d_1/2-d_2/2-d_3/2}&  \textrm{ if } d_1< d_2+d_3, \frac{\max(1,|a_2|,|a_3|)}{ \max\left(1,\frac{|a_2a_3|}{|c|}\right)} \leq 1.
    \end{cases}
\end{align*}
We now bound the integral of this expression over $a_2,a_3$.

Consider the contribution of the domain $|a_2a_3|\ge |c|,$ $|a_3|\ge \max(1,|a_2|)$.  The contribution of $|c|\ge |a_2|$  is dominated by
\begin{align*}
    &|c|^{\varepsilon}\int_{|c|\ge |a_2|}|a_2|^{d_2/2-1-\varepsilon}\int_{|a_3|\ge \max(1,|a_2|,|a_2|^{-1}|c|)}|a_3|^{-d_2/2}d^\times a_3d^\times a_2\\
    &\ll |c|^{\varepsilon} \int_{|c|\ge |a_2|}|a_2|^{d_2/2-1-\varepsilon}\max(1,|a_2|,|a_2|^{-1}|c|)^{-d_2/2}d^\times a_2\\
    &=|c|^{d_2/2-1} \int_{1\ge |a_2|}|a_2|^{d_2/2-1-\varepsilon}\max(|a_2||c|,|a_2|^{-1})^{-d_2/2}d^\times a_2.
\end{align*}
Here we have taken a change of variables $a_2\mapsto a_2c$. This term is 
\begin{align*}
   \ll_\varepsilon \left\{\begin{array}{ll}
        |c|^{d_2/2-1}& \textrm{ if } |c|\le 1, \\
        |c|^{-1/2+\varepsilon/2}& \textrm{ if } |c|\ge 1.
    \end{array}\right.
\end{align*}

Consider the contribution of $|c|\le |a_2|$.  If $d_1\ge d_2+d_3,$ it is
\begin{align*}
    &|c|^{d_2/2+d_3/2-1-\varepsilon}\int_{|c|\le |a_2|}|a_2|^{\varepsilon-d_3/2}\int_{|a_3|\ge \max(1,|a_2|,|a_2|^{-1}|c|)}|a_3|^{-d_2/2}d^\times a_3d^\times a_2\\
    &\ll |c|^{d_2/2-1}\int_{1\le |a_2|}|a_2|^{\varepsilon-d_3/2}\max(1,|a_2||c|)^{-d_2/2}d^\times a_2\\
    &\ll_\varepsilon \left\{\begin{array}{ll}
        |c|^{d_2/2-1} & \textrm{ if } |c|\le 1,  \\
        |c|^{-1}& \textrm{ if } |c|\ge 1.
    \end{array}\right.
\end{align*}
If $d_1< d_2+d_3,$ a similar argument yields a bound of $|c|^{\min(d_1,d_2)/2-1-\varepsilon}$ if $|c| \leq 1$ and $|c|^{-1}$ if $|c| \geq 1$.

Over the domain $|a_2a_3|\le |c|,$ $|a_3|\ge \max(1,|a_2|),$ the integral is
\begin{align}\label{morecare}
\begin{split}
    &\int_{|a_2|^{-1}|c|\ge |a_3|\ge \max(1,|a_2|) } |a_2|^{d_2/2-1}|a_3|^{\varepsilon-d_2/2} d^\times a_3d^\times a_2\\
    &\ll_\varepsilon \int_{|a_2|\ge 1,|c|\ge |a_2|^{2}} |a_2|^{\varepsilon-1}d^\times a_2 +  \int_{\min(1,|c|)\ge |a_2|} |a_2|^{d_2/2-1} d^\times a_2\\
    &\ll \min(1,|c|)^{d_2/2-1}.
\end{split}
\end{align}

By symmetry in $a_2, a_3,$ we are left with bounding the integral over $|a_2|,|a_3|\le 1$. If $d_1\ge d_2+d_3$ then over this domain, the integral is
\begin{align*}
    \int_{|a_2|\le 1,|a_3|\le 1}& \max\left(1,\frac{|a_2a_3|}{|c|}\right)^{1+\varepsilon-d_2/2-d_3/2}|a_2|^{d_2/2-1}|a_3|^{d_3/2-1}d^\times a_2d^\times a_3\\
    =&|c|^{d_2/2+d_3/2-1-\varepsilon}\int_{|a_2|\le 1,|a_2|^{-1}|c|\le |a_3|\le 1} |a_2|^{\varepsilon-d_3/2}|a_3|^{\varepsilon-d_2/2}d^\times a_2d^\times a_3\\
    &+\int_{|a_2|\le 1,|a_3|\le \min(|a_2|^{-1}|c|, 1)} |a_2|^{d_2/2-1}|a_3|^{d_3/2-1}d^\times a_2d^\times a_3\\
    \ll& |c|^{d_3/2-1}\int_{|c|\le |a_2|\le 1} |a_2|^{d_2/2-d_3/2}d^\times a_2+\int_{|a_2|\le 1} |a_2|^{d_2/2-d_3/2}\min(|a_2|,|c|)^{d_3/2-1}d^\times a_2\\
    \ll_\varepsilon&   \begin{cases}
        |c|^{\min(d_2,d_3)/2-1-\varepsilon} & \textrm{ if } |c|\le 1,  \\
        1 & \textrm{ if } |c|\ge 1.
    \end{cases}
\end{align*}
    The case $d_1<d_2+d_3$ follows from a similar computation.  
\end{proof}

\begin{Cor}\label{cor:term4bound}
Suppose $|c|>1$. Given $\alpha>0,$ there exists $\beta>0$ such that
\begin{align}
    \int_{|a|\ge |c|^{\alpha}} \max(1,|a_1|)^{1-d_1/2}\max\left(1,\left|\frac{[a]}{c}\right|,|a|\right)^{1-d_2/2-d_3/2}\{a\}^{d/2-1}d^\times a\ll_{\alpha,\beta} |c|^{-\beta}.
\end{align}
\end{Cor}

\begin{proof}
We may assume $\alpha<1/2$. In view of the proof of Lemma \ref{lem:term4bound}, it suffices to bound the contributions of the domain $|a_2|^{-1}|c|\ge |a_3|\ge |c|^{\alpha}, |a_2|\le |c|^{\alpha},$ the domain $|a_2|^{-1}|c|\ge |a_3|\ge |a_2|\ge |c|^{\alpha},$ and the domain $|a_1|\ge |c|^{\alpha}\ge \max(|a_2|,|a_3|)$. 
    
    Over the first domain, by \eqref{morecare}, the integral to be bounded is dominated by
    \begin{align*}
        &\int_{|a_2|^{-1}|c|\ge |a_3|\ge |c|^{\alpha}, |a_2|\le |c|^{\alpha} } |a_2|^{d_2/2-1}|a_3|^{\varepsilon-d_2/2} d^\times a_3d^\times a_2\\
        &\ll |c|^{\alpha(\varepsilon-d_2/2)}\int_{|c|^{\alpha}\ge |a_2|} |a_2|^{d_2/2-1}d^\times a_2\\
        &\ll |c|^{\alpha(\varepsilon-1)}.
    \end{align*}
Over the second domain, the integral is dominated by
\begin{align*}
    \int_{|a_3|\ge |c|^\alpha} |a_3|^{\varepsilon-d_2/2}\min(|a_3|,|c|^{1/2})^{d_2/2-1} d^\times a_3\ll |c|^{\alpha(\varepsilon-1)}.
\end{align*}
Over the third domain, the integral is dominated by 
    \begin{align*}
    &\int_{|a_1|\ge |c|^{\alpha}\ge \max(|a_2|,|a_3|)} |a_1|^{1-d_2/2-d_3/2}|a_2|^{d_2/2-1}|a_3|^{d_3/2-1}d^\times a \ll |c|^{-\alpha}.
\end{align*}
\end{proof}

\begin{Lem} \label{lem:specialbound}
Assume $d_3\le d_1,d_2$. For $c\in F^\times$ and nonnegative $M\in\mathcal{A}(F^2),$ we have
\begin{align*}
\begin{split}
    &\int_{(F^\times)^3}M\left(\frac{(1,a_3)c}{[a]}\right)\Bigg(\max\left(1,\left|\frac{[a]}{c}\right|,|a_1|,|a_2|\right)^{1-d_1/2-d_2/2}\max\left(1,|a_3|\right)^{1-d_3/2}\\
&+\max\left(1,|a_1|,|a_2|\right)^{1-d_1/2-d_2/2}\left|\frac{[a]}{c}\right|^{-d_3/2}\min\left(\max\left(1,|a_1|,|a_2|\right),\left|\frac{[a]}{c}\right|\right)\Bigg)\{a\}^{d/2-1}d^\times a\\
&\ll_{\varepsilon,M} \min(1,|c|)^{d_3/2-1-\varepsilon}\max(1,|c|)^{-1/2+\varepsilon/2}.
\end{split}
\end{align*}
\end{Lem}

\begin{proof}
 We first compute
    \begin{align}\label{termspecial1}
    \begin{split}
     &\int_{(F^\times)^3}M\left(\frac{c}{a_1a_2a_3},\frac{c}{a_1a_2}\right)\max\left(1,\frac{|a_1a_2a_3|}{|c|},|a_1|,|a_2|\right)^{1-d_1/2-d_2/2}\max\left(1,|a_3|\right)^{1-d_3/2}\\
     &\times|a_1|^{d_1/2-1}|a_2|^{d_2/2-1}|a_3|^{d_3/2-1}d^\times a_3 d^\times a_2 d^\times a_1\\
     &=|c|^{d_3/2-1}\int_{(F^\times)^3}M\left(a_3^{-1},\frac{c}{a_1a_2}\right)\max\left(1,|a|\right)^{1-d_1/2-d_2/2}\max\left(|a_3|^{-1},\left|\frac{c}{a_1a_2}\right|\right)^{1-d_3/2}\\
     &\times|a_1|^{d_1/2-d_3/2}|a_2|^{d_2/2-d_3/2}d^\times a_3 d^\times a_2 d^\times a_1.
     \end{split}
 \end{align}
 We break down the integral into several domains. 
 \begin{Sublem}
 Over the domain $|c|\le |a_1a_2|, |a_3|\le 1,$ \eqref{termspecial1} is dominated by
 $$\min(1,|c|)^{d_3/2-1-\varepsilon}\max(1,|c|)^{-1/2}.$$
 \end{Sublem}

\begin{proof}
Over this domain, \eqref{termspecial1} is dominated by  \begin{align}\label{firstpart}
     &|c|^{d_3/2-1}\int_{ |c|\le |a_1a_2|} \max(1,|a_1|,|a_2|)^{1-d_1/2-d_2/2}|a_1|^{d_1/2-d_3/2}|a_2|^{d_2/2-d_3/2}d^\times a_2 d^\times a_1.
\end{align}
The contribution of $|a_1|,|a_2|\le 1$ (which is zero unless $|c|  \leq 1$)  is  $O_{\varepsilon}(|c|^{d_3/2-1-\varepsilon})$.  For the rest, by symmetry it suffices to bound the integral over the domain $|a_1|\ge \max(1,|a_2|)$. This contribution is bounded by
\begin{align*}
\begin{split}
     &|c|^{d_3/2-1}\int_{ \max(1,|a_2|,|c/a_2|)\le |a_1|} |a_1|^{1-d_2/2-d_3/2}|a_2|^{d_2/2-d_3/2}d^\times a_2 d^\times a_1\\
     &\ll|c|^{d_3/2-1}\int_{ F^\times} \max(1,|a_2|,|c/a_2|)^{1-d_2/2-d_3/2}|a_2|^{d_2/2-d_3/2}d^\times a_2 \\
     &=|c|^{d_3/2-1}\int_{ F^\times} \max(|a_2|,|a_2|^2,|c|)^{1-d_2/2-d_3/2}|a_2|^{d_2-1}d^\times a_2.
\end{split}
\end{align*}
 The contribution of the domain $|a_2|\le 1$ to the integral is bounded by
 \begin{align*}
     &|c|^{-d_2/2}\int_{|a_2|\le \min(1,|c|)}|a_2|^{d_2-1}d^\times a_2+ |c|^{d_3/2-1}\int_{\min(1,|c|)< |a_2|\le 1}|a_2|^{d_2/2-d_3/2}d^\times a_2\\
     &\ll_{\varepsilon} \max(1,|c|)^{-d_2/2}\min(1,|c|)^{d_3/2-1-\varepsilon}.
 \end{align*}
 Over the domain $|a_2|\ge 1,$ the integral becomes 
 \begin{align*}
     &|c|^{-d_2/2}\int_{1\le |a_2|\le |c|^{1/2}}|a_2|^{d_2-1}d^\times a_2+ |c|^{d_3/2-1}\int_{\max(|c|^{1/2},1)\le |a_2|}|a_2|^{1-d_3}d^\times a_2\\
     &\ll \max(1,|c|)^{-1/2}\min(1,|c|)^{d_3/2-1}.
 \end{align*}
 \end{proof}

 \begin{Sublem}
  Over the domain $|c|\le |a_1a_2|, |a_3|\ge 1,$ $|a_1a_2|\le |a_3||c|,$ \eqref{termspecial1} is dominated by
 $$\min(1,|c|)^{\min(d_1,d_2)/2-1-\varepsilon}\max(1,|c|)^{-1/2+\varepsilon/2}.$$
 \end{Sublem}
 \begin{proof}
 Over this domain, \eqref{termspecial1} is dominated by
\begin{align*}
    &\int_{|c|\le |a_1a_2|\le |a_3||c|}|a|^{1-d_1/2-d_2/2}|a_1|^{d_1/2-1}|a_2|^{d_2/2-1}d^\times a_3 d^\times a_2 d^\times a_1\\
    =&\int_{|c|\le |a_1a_2|\le |a_3||c|,|a_3|=|a|}|a|^{1-d_1/2-d_2/2}|a_1|^{d_1/2-1}|a_2|^{d_2/2-1}d^\times a_3 d^\times a_2 d^\times a_1\\
    &+ \sum_{i=1}^2\int_{|c|\le |a_1a_2|\le |a_3||c|,|a_i|=|a|}|a|^{1-d_1/2-d_2/2}|a_1|^{d_1/2-1}|a_2|^{d_2/2-1}d^\times a_3 d^\times a_2 d^\times a_1.
 \end{align*}
The first term is dominated by
\begin{align*}
    &\int_{|c|\le |a_1a_2|} \max\left(|a_1|,|a_2|,\frac{|a_1a_2|}{|c|}\right)^{1-d_1/2-d_2/2}|a_1|^{d_1/2-1}|a_2|^{d_2/2-1}d^\times a_2 d^\times a_1\\
    =&\sum_{\sigma\in S_2}\int_{|a_{\sigma(1)}|\ge \max(|a_{\sigma(2)}|,|c|/|a_{\sigma(2)}|),|a_{\sigma(2)}|\le |c|}|a_{\sigma(1)}|^{-d_{\sigma(2)}/2}|a_{\sigma(2)}|^{d_{\sigma(2)}/2-1}d^\times a_2 d^\times a_1\\
    &+|c|^{d_1/2+d_2/2-1}\int_{|c|\le |a_1|,|c|\le |a_2|,|c|\le |a_1||a_2|} |a_1|^{-d_2/2}|a_2|^{-d_1/2}d^\times a_2 d^\times a_1\\
    \ll &\sum_{\sigma\in S_2}\int_{|a_{\sigma(2)}|\le |c|} \max(|a_{\sigma(2)}|^2,|c|)^{-d_{\sigma(2)}/2}|a_{\sigma(2)}|^{d_{\sigma(2)}-1} d^\times a_{\sigma(2)}\\
    &+|c|^{d_1/2-1}\int_{|c|\le |a_2|} \max(1,|a_2|)^{-d_2/2}|a_2|^{d_2/2-d_1/2}d^\times a_2 \\
    \ll_\varepsilon& \min(1,|c|)^{\min(d_1,d_2)/2-1-\varepsilon}\max(1,|c|)^{-1/2}.
\end{align*}
For the second term, by symmetry, we may assume $i=1$. Then the integral is 
\begin{align*}
    &\int_{|c|\le |a_1a_2|\le |a_3||c|,|a_2|\le |a_1|,|a_3|\le |a_1|}|a_1|^{-d_2/2}|a_2|^{d_2/2-1}d^\times a_3 d^\times a_2 d^\times a_1\\
    &\leq\int_{|c|\le |a_1a_2|\le |a_3||c|,|a_2|\le |a_1|,|a_3|\le |a_1|}|a_1|^{-d_2/2}\left(\frac{|a_1|}{|a_3|} \right)^\varepsilon|a_2|^{d_2/2-1}d^\times a_3 d^\times a_2 d^\times a_1\\
    &\ll_{\varepsilon}  |c|^{\varepsilon}\int_{|c|\le |a_1a_2|,|a_2|\le |a_1|,|a_2|\le |c|}|a_1|^{-d_2/2}|a_2|^{d_2/2-1-\varepsilon} d^\times a_2 d^\times a_1\\
    &\ll |c|^{\varepsilon} \int_{|a_2|\le |c|} \max(|c|,|a_2|^2)^{-d_2/2}|a_2|^{d_2-1-\varepsilon}d^\times a_2\\
    &\ll \min(1,|c|)^{d_2/2-1}\max(1,|c|)^{-1/2+\varepsilon/2}.
\end{align*} 
\end{proof}

\begin{Sublem}
  Over the domain $|c|\le |a_1a_2|, |a_3|\ge 1,$ $|a_1a_2|\ge |a_3||c|,$ \eqref{termspecial1} is dominated by 
 $$\min(1,|c|)^{d_3/2-1-\varepsilon}\max(1,|c|)^{-1/2}.$$
\end{Sublem}
\begin{proof}
Over this domain, \eqref{termspecial1} is dominated by
\begin{align}\label{ugly1}
    &|c|^{d_3/2-1}\int_{|a_3||c|\le |a_1a_2|,|a_3|\ge 1}|a|^{1-d_1/2-d_2/2}|a_3|^{d_3/2-1}|a_1|^{d_1/2-d_3/2}|a_2|^{d_2/2-d_3/2}d^\times a_3 d^\times a_2 d^\times a_1.
 \end{align}
 We argue as above. Over $|a_3|=|a|,$ \eqref{ugly1} is
 \begin{align*}
     &|c|^{d_3/2-1}\int_{|a_3||c|\le |a_1a_2|,|a_3|\ge 1,|a_3|=|a|}|a_3|^{d_3/2-d_1/2-d_2/2}|a_1|^{d_1/2-d_3/2}|a_2|^{d_2/2-d_3/2}d^\times a_3 d^\times a_2 d^\times a_1\\
     &\ll |c|^{d_3/2-1}\int_{|c|\le \min(|a_1|,|a_2|,|a_1a_2|)} \max(1,|a_1|,|a_2|)^{d_3/2-d_1/2-d_2/2}|a_1|^{d_1/2-d_3/2}|a_2|^{d_2/2-d_3/2} d^\times a_2 d^\times a_1.
 \end{align*}
 If $|a_1|,|a_2|\le 1,$ this is
 \begin{align*}
 |c|^{d_3/2-1}\int_{|c|\le |a_1a_2|,|a_1|\le 1,|a_2|\le 1} |a_1|^{d_1/2-d_3/2}|a_2|^{d_2/2-d_3/2} d^\times a_2 d^\times a_1\ll_\varepsilon |c|^{d_3/2-1-\varepsilon}.
 \end{align*}
 If $|a_1|\ge \max(1,|a_2|),$ this is dominated by
 \begin{align*}
     |c|^{d_3/2-1}\int_{|c|\le |a_2|} \max(1,|a_2|)^{-d_2/2}|a_2|^{d_2/2-d_3/2} d^\times a_2\ll_\varepsilon \min(1,|c|)^{d_3/2-1-\varepsilon}\max(1,|c|)^{-1}.
 \end{align*}
Over $|a_1|=|a|,$ \eqref{ugly1} is
\begin{align*}
    &|c|^{d_3/2-1}\int_{|a_3||c|\le |a_1a_2|,|a_1|\ge |a_3|\ge 1, |a_1|\ge |a_2|}|a_3|^{d_3/2-1}|a_1|^{1-d_2/2-d_3/2}|a_2|^{d_2/2-d_3/2}d^\times a_3 d^\times a_2 d^\times a_1\\
    &\ll |c|^{d_3/2-1}\int_{|c|\le |a_1a_2|, \max(1,|a_2|)\le |a_1|} \min(|a_2|/|c|,1)^{d_3/2-1}|a_1|^{-d_2/2}|a_2|^{d_2/2-d_3/2}d^\times a_2 d^\times a_1\\
    &\ll \int_{F^\times} \min(|a_2|,|c|)^{d_3/2-1}\max(|a_2|,|a_2|^2,|c|)^{-d_2/2}|a_2|^{d_2-d_3/2}d^\times a_2 \\
    &\ll_\varepsilon \min(1,|c|)^{d_3/2-1-\varepsilon}\max(1,|c|)^{-1/2}.
\end{align*}
The rest follows by symmetry.
\end{proof}
 
Now suppose $|a_1a_2|\le |c|$. Then the integral \eqref{termspecial1} over $a_3$ is dominated by a constant depending on $\varepsilon$ times
 \begin{align}\label{specialform}
     |c|^{d_3/2-1}\int_{|a_1a_2|\le |c|}J\left(\frac{c}{a_1a_2}\right)\max\left(1,|a_1|,|a_2|\right)^{\varepsilon+1-d_1/2-d_2/2}|a_1|^{d_1/2-d_3/2}|a_2|^{d_2/2-d_3/2}d^\times a_2 d^\times a_1
 \end{align}
 for some $J \in \mathcal{A}(F)$.   Over the domain $|a_1|,|a_2|\le 1,$ this integral is rapidly decreasing as a function of $c$, and dominated by $|c|^{d_3/2-1-\epsilon}$ for $|c|\le 1$. For the rest, by symmetry it suffices to bound the integral over the domain $|a_1|\ge \max(1,|a_2|),$ which is 
 \begin{align*}
     &|c|^{d_3/2-1}\int_{\max(1,|a_2|)\le |a_1|\le |c/a_2|}J\left(\frac{c}{a_1a_2}\right)|a_1|^{1+\varepsilon-d_2/2-d_3/2}|a_2|^{d_2/2-d_3/2}d^\times a_1 d^\times a_2 \\
     &=|c|^{\varepsilon-d_2/2}\int_{\max(|a_2|,|a_2|^2)|c|^{-1}\le |a_1|\le 1 } J(a_1^{-1})|a_1|^{1+\varepsilon-d_2/2-d_3/2} |a_2|^{d_2-1-\varepsilon} d^\times a_1d^\times a_2\\
     &\ll_J |c|^{\varepsilon-d_2/2} \int_{\max(|a_2|,|a_2|^2)\le |c|} |a_2|^{d_2-1-\varepsilon} d^\times a_2\\
     &\ll \max(1,|c|)^{-1/2+\varepsilon/2}\min(1,|c|)^{d_2/2-1}.
\end{align*}
 This completes the proof of the bound for the first summand in the statement of Lemma \ref{lem:specialbound}.  
 
The second summand in the statement of Lemma \ref{lem:specialbound} is bounded by 
 \begin{align*}
    &   \int_{(F^\times)^3}M\left(\frac{c}{a_1a_2a_3},\frac{c}{a_1a_2}\right)\max\left(1,|a_1|,|a_2|\right)^{\varepsilon+1-d_1/2-d_2/2}\left|\frac{[a]}{c}\right|^{-d_3/2+1-\varepsilon} \{a\}^{d/2-1} d^\times a\\
     &=|c|^{d_3/2-1}\int_{(F^\times)^3}M\left(a_3^{-1},\frac{c}{a_1a_2}\right)\max\left(1,|a_1|,|a_2|\right)^{\varepsilon+1-d_1/2-d_2/2}|a_1|^{d_1/2-d_3/2}|a_2|^{d_2/2-d_3/2}|a_3|^{-\varepsilon}d^\times a\\
     &\ll |c|^{d_3/2-1}\int_{(F^\times)^2}J'\left(\frac{c}{a_1a_2}\right)\max\left(1,|a_1|,|a_2|\right)^{\varepsilon+1-d_1/2-d_2/2}|a_1|^{d_1/2-d_3/2}|a_2|^{d_2/2-d_3/2}d^\times a_2d^\times a_1
 \end{align*}
 for some $J'\in \mathcal{A}(F)$. Break down the integral into $|a_1a_2|\ge |c|$ and $|a_1a_2|\le |c|$, which are \eqref{firstpart} (up to $\varepsilon$) and \eqref{specialform} respectively. In both cases, the integral is dominated by $\max(1,|c|)^{-1/2+\varepsilon/2}\min(1,|c|)^{d_3/2-1-\varepsilon}$.
 \end{proof}
 
 \begin{proof}[Proof of Theorem \ref{thm:absolutebdappendix}]  By symmetry, we may assume $d_1,d_2\ge d_3$; this assumption is used to apply Lemma \ref{lem:specialbound}.
 We showed \eqref{intHvi} was dominated by  \eqref{term1}, \eqref{termspecial}, \eqref{term2}, \eqref{term3}, \eqref{term4} with $c=\frac{b[a]}{3[\widetilde{\xi}]}$ in \S \ref{ssec:overn0}.  To estimate the $(F^\times)^3$ integrals over the terms \eqref{term1}, \eqref{termspecial}, \eqref{term2}, \eqref{term3}, \eqref{term4}, we apply Lemmas \ref{lem:term1bound}, \ref{lem:specialbound}, \ref{lem:term2bound},  \ref{lem:term3bound}, \ref{lem:term4bound} (respectively)  with $c=\frac{3[\widetilde{\xi}]}{b}$. Given our comments on the continuity of our bounds as a function of $f$ at the beginning of \S \ref{ssec:overa}, the theorem follows.
 \end{proof}

 Moreover, we have the following bound.
 
 \begin{Thm}\label{thm:byproductappendix}
 Suppose $|\xi_1\otimes \xi_2\otimes \xi_3|>|b|$. If $d_i>2$ for all $i,$ given $\alpha>0,$ there exists  $\beta>0$ such that 
 \begin{align*}
 \begin{split}
    & \int_{ }\int_{F^2} \Bigg|\int_{V(F)}\psi\Bigg(\left\langle \frac{a\xi}{\widetilde{\xi}},u \right\rangle +\frac{[a]bQ(u)}{9[\widetilde{\xi}]}+\mathfrak{c}(u,v))\Bigg)f(u)du\Bigg|   \frac{dvd^\times a}{\{a\}^{1-d/2}} \ll_{\alpha,\beta,f}  \left|\frac{\xi_1\otimes\xi_2\otimes \xi_3}{b}\right|^{-\beta}
    \end{split}
\end{align*}
where the outer integral is over $|a|\ge \left|\frac{\xi_1\otimes\xi_2\otimes \xi_3}{b}\right|^\alpha$.
\end{Thm}

\begin{proof}
Replacing Lemmas \ref{lem:term3bound} and \ref{lem:term4bound} with Corollaries \ref{cor:term3bound} and \ref{cor:term4bound} in the proof of Theorem \ref{thm:absolutebdappendix} yields the bound.
\end{proof}

\appendix

\section{Computation of normalizing factors} \label{App:normalizing}

In this appendix we compute the normalizing factors  $\{(s_i,\lambda_i)\}$.
The parameters depend only on the action of the dual group $\widehat{M}$ acting on the dual Lie algebra $\widehat{\mathfrak{n}}_P$ 
and our fixed isomorphism
\begin{align} \label{Mab:isom}
\omega_P:M^{\mathrm{ab}} \tilde{\lto}\; \GG_m
\end{align}
 where $\omega_P$ is defined as in \eqref{ssec:Pl}.  
 
 To avoid proliferation of duals, we work directly in the dual picture in this section.  Thus now $G$ denotes an adjoint simple group over $\cc$ with maximal parabolic subgroup $P,$ Levi subgroup $M,$ and we are studying the action of $M$ on $\mathfrak{n}_P,$ the complex Lie algebra of the unipotent radical $N_P$ of $P.$   We define the parameters $(s_i,\lambda_i)$ as in \S \ref{ssec:BKL} but with $\widehat{M},$ $\widehat{\mathfrak{n}}_P$ in that section  replaced by $M$ and $\mathfrak{n}_P,$ respectively.  We let $T$ be a maximal torus in $M,$ $T \leq B \leq P$ a Borel subgroup, and $\Delta$ the corresponding set of simple roots.  We let $\beta$ be the simple root such that $\Delta-\{\beta\}$ is the set of simple roots of $T \cap M$ in $M$ with respect to $B \cap M$.  The dual of \eqref{Mab:isom} is an isomorphism 
\begin{align} \label{betaZM}
\varphi:\GG_m \tilde{\lto} Z(M).
\end{align}
 For any representation $W$ of $M$ and any integer $\lambda,$ we write $W(\lambda)$ for the subspace on which $Z(M)=\GG_m$ acts via $x \mapsto x^\lambda$.
 
 \begin{Lem} \label{lem:posi} If $\lambda \leq 0,$ then $\mathfrak{n}_P(\lambda)=0$.
 \end{Lem}
 
 \begin{proof}
Let $\gamma$ be a positive root of $(G,B,T)$. Note that the root space $(\mathfrak{n}_P)_{\ga}$ is non-zero if and only if writing
$\ga=\sum_{\al\in \De}c_{\al}\al$ we have $c_\beta>0$. It follows from \eqref{ssec:Pl} that 
\begin{align*}
    \langle \ga, \varphi\rangle=c_\be\la\be,\varphi\ra=c_\be m_{\be^\lor}>0.
\end{align*} 
We deduce the lemma.
 \end{proof} 

In each of the cases given below, the isomorphism $\varphi:\GG_m \to Z(M)$  will be the ``obvious one'', so we will not record it.  In fact, there are only two choices of isomorphism $\GG_m \tilde{\lto} Z(M),$ and there is only one of them so that Lemma \ref{lem:posi} is true, so the reader can easily check which isomorphism is $\varphi$.  

In the following computations, we interpret $\mathrm{Sym}^0(\cc^2)$ as the trivial $1$-dimensional representation of $\mathfrak{sl}_2$.

\subsection{Projective general linear groups}
The following is the classical Clebsch--Gordan rule  \cite[Exercise 11.11]{Fulton:Harris}:
\begin{Lem} \label{lem:tensor}
We have an isomorphism of $\mathfrak{sl}_2$-representations
\begin{align*}
\Sym^{n}(\cc^2)\otimes \Sym^{m}(\cc^2)\cong \Sym^{n+m}(\cc^2)\oplus \Sym^{n+m-2}(\cc^2)\oplus\cdots \oplus \Sym^{|n-m|}(\cc^2).
\end{align*}
\qed
\end{Lem}

\begin{Lem} \label{lem:PGL:case}
Let $P \leq \mathrm{PGL}_n$ be the parabolic stabilizing an $\ell$-plane.  Then
 $$
\left\{(\tfrac{|n-2\ell|}{2},1),(\tfrac{|n-2\ell|+2}{2},1),\dots,(\tfrac{n-2}{2},1)\right\}
 $$ is a good ordering for $\mathfrak{n}_P$.
\end{Lem}

\begin{proof}
It is not hard to see  $\mathfrak{m}^{\mathrm{der}} \cong \mathfrak{sl}_{n-\ell} \times \mathfrak{sl}_{\ell}$ and $\fn _P$ is isomorphic as a representation of $\mathfrak{m}$ to  
$$
\mathrm{Hom}(\cc^{n-\ell},\cc^{\ell})
$$
with the natural action.  The induced representation  of a  principal $\mathfrak{sl}_2$-triple is 
\begin{align*}
\Sym^{n-\ell-1}(\cc^2)^\vee \otimes\Sym^{\ell-1}(\cc^2)  &\cong \Sym^{n-\ell-1}(\cc^2) \otimes \Sym^{\ell-1}(\cc^2) \\
& \cong \Sym^{n-2}(\cc^2)\oplus \Sym^{n-4}(\cc^2)\oplus\cdots \oplus \Sym^{|n-2\ell|}(\cc^2)
\end{align*}
by Lemma \ref{lem:tensor}.  The lemma follows.  
\end{proof}

\subsection{The classical groups}
\label{ssec:classical:groups}
Let $V$ be a complex vector space equipped with a nondegenerate $\epsilon$-symmetric form $\langle\cdot,\cdot\rangle,$ that is, 
\[
\langle v,w\rangle= \ep\langle  w,v\rangle 
\] 
for $v,w\in V$.  We assume $\epsilon \in \{1,-1\}$.  
For $\cc$-algebras $R,$ let
$$
G_V(R):=\{g \in \mathrm{SL}_V(R):\langle gv,gw\rangle=\langle v,w\rangle
\}.
$$
We refer to $G_V$ as a classical group. The corresponding Lie algebra is
\[
\fg_V= \{X\in \mathfrak{sl}(V): \la X v,w\ra+\la v,X w\ra = 0\:\text{ for } v,w\in V\}.
\]
Let $PG_V$ be the associated projective group.  Concretely,
$$
PG_{V }\cong \begin{cases} 
\mathrm{PSO}_{\dim V} &\textrm{ if }\epsilon=1,\\
\mathrm{PSp}_{\dim V} &\textrm{ if }\epsilon=-1.\end{cases}
$$
We assume that $PG_V$ is simple and not isomorphic to a projective general linear group.  Thus $\dim V \not \in \{2,4\}$ if $\epsilon=1$ and $\dim V \neq 2$ if $\epsilon=-1$.  We also observe that $\mathrm{PSO}_{2r+1}=\mathrm{SO}_{2r+1}$.

The maximal parabolic subgroups of $PG_V$ are precisely the stabilizers of isotropic subspaces.  For a parabolic $P=MN_P,$ we denote by $W_P$ the corresponding isotropic subspace. 
We let $W_P^\vee$ be the linear dual of $W_P$ with respect to $\langle \cdot, \cdot\rangle$. Then there exists a subspace $V_P< V$ such that the pair $(V_P,\langle\cdot,\cdot\rangle|_{V_P})$ is of the same symmetric type as $V$, and such that there is a direct sum decomposition $V=W_P\oplus V_P\oplus W_P^\vee.$ Note that the pair $(W_P\oplus W_P^\vee,\la\cdot,\cdot\ra|_{W_P\oplus W_P^\vee})$ is also non-degenerate and of the same symmetric type as $V$.
We have
\[
\fm\cong \mathfrak{gl}_{W_P}\oplus \fg_{V_P}.
\]
We refer to $\ell:=\dim W_P$ as the \textbf{linear rank} of $M$ or $\mathfrak{m}$.

The following lemma is well known.  See for instance \cite[Theorems 8.6 and 12.6]{Wolf}.

\begin{Lem}\label{Lem: nil rad rep}
As a representation of $\mathfrak{m},$
\[
\mathfrak{n}_P\cong
\Hom_\cc(V_P,W_P) \oplus\Sym_V(W_P)
\]
where 
\begin{align*}
\Sym_V(W_P)&:=
 \begin{cases}\Sym^2(W_P)& \textrm{ if }\ep=-1,\\\;\mathrm{Alt}^2(W_P)& \textrm{ if } \ep=1.\end{cases}
\end{align*}
\end{Lem}
We have
\begin{align} \label{usu:case}
    \mathfrak{n}_P(1)\cong\Hom_\cc(V_P,W_P) \quad \textrm{ and } \quad \mathfrak{n}_P(2)\cong\Sym_V(W_P)
\end{align}
unless $PG_V$ is $\mathrm{PSO}_{2r}$ or $\mathrm{PSp}_{2r}$ and
$\mathfrak{m} \cong \mathfrak{gl}_r,$ in which case 
\begin{align} \label{spec:case}
\mathfrak{n}_P=\mathfrak{n}_P(1) \cong \mathrm{Sym}_V(W_P).
\end{align}
The following lemma explicates the principal $\fsl_2$-subalgebra of $\fg_{V_P}$.
\begin{Lem} \label{lem:prin:sl2}
As a representation of a principal $\fsl_2$-subalgebra of $\fg_V,$ the standard representation $V$ of $\fg_V$ is isomorphic to $\mathrm{Sym}^{\dim V-1}(\cc^2)$ unless $G\cong \mathrm{PSO}_{2r},$ in which case it is $\mathrm{Sym}^{\dim V-2}(\cc^2) \oplus \cc$.
\end{Lem}
\begin{proof}The $n$-th tensor power of the standard symplectic form on $\cc^2$ is $(-1)^{n}$-symmetric, and the $n$-th symmetric power of the standard representation $\cc^2$ of $\mathfrak{sl}_2$ is a subrepresentation of the $n$th tensor power.  Thus the principal $\fsl_2\to \mathfrak{sl}_V$ may be chosen to factor through the standard representation $\fg_V \to \fsl_V$.  This implies the lemma unless $G\cong \mathrm{PSO}_{2r}$. For this last case see \cite[Section 7]{GrossPrincipal}.
\end{proof}

For the following lemma, see 
\cite[Exercise 11.31 and 11.35]{Fulton:Harris}:

\begin{Lem} \label{lem:plethysm}
For any $n\geq1,$ we have the following equivalences of $\fsl_2$-representations:
\begin{align*}
 \wedge^{2}( \mathrm{Sym}^{n}(\cc^2))\cong \mathrm{Sym}^{2}( \mathrm{Sym}^{n-1}(\cc^2))=\bigoplus_{j=0}^{\lfloor (n-1)/2 \rfloor}\mathrm{Sym}^{2(n-1)-4j}(\cc^2).
\end{align*}
\qed
\end{Lem}

Let
$$
p(V):=\dim V \pmod{2}
$$
be the parity of $\dim V,$ viewed as an element of the set $\{0,1\}$. Note that if $G\cong \mathrm{PSO}_{2r}$, linear ranks are either $r$ or $\le r-2$.
\begin{Lem} \label{lem:classic:case}
Assume $r>1$.  
Assume that either $G$ is  $ \mathrm{PSp}_{2r}$ and that the linear rank $\ell$ of $M$ is not $r$ or $G\cong \mathrm{SO}_{2r+1}$.   For $\ell>1,$ the parameters $\{(s_i,\lambda_i)\}$ are 
\begin{align*}
    &\left\{(\tfrac{|2r+p(V)-3\ell|}{2},1),(\tfrac{|2r+p(V)-3\ell|+2}{2},1),\dots,(\tfrac{2r+p(V)-\ell-2}{2},1)\right\}\\& \bigsqcup \left\{(\ell-1-p(V)-2j,2):0 \leq j \leq \lfloor (\ell-1-p(V))/2 \rfloor \right\}.
\end{align*}
If $\ell=1,$ the parameters are
\begin{align*}
   \left\{(\tfrac{2r-2}{2},1)\right\} \textrm{ if } G \cong \mathrm{SO}_{2r+1}, \quad \textrm{ and } \quad 
   \left\{(\tfrac{2r-3}{2},1),(0,2)\right\} \textrm{ if } G \cong \mathrm{PSp}_{2r}.
\end{align*}
Suppose $G \cong \mathrm{PSO}_{2r}$ with $r\ge 3$ and that $\ell\le r-2$. If $\ell>1,$ then the parameters are 
\begin{align*}
&\left\{(\tfrac{|2r-1-3\ell|}{2},1),(\tfrac{|2r-1-3\ell|+2}{2},1),\dots,(\tfrac{2r-\ell-3}{2},1)\right\}\bigsqcup \left\{(\tfrac{\ell-1}{2},1)\right\}\\ & \bigsqcup \left\{(\ell-2-2j,2):0 \leq j \leq \lfloor (\ell-2)/2 \rfloor \right\}.
\end{align*}
If $\ell=1,$ the parameters are $\{(0,1),(r-2,1)\}$.
If $\ell=r$ and $G$ is isomorphic to  $\mathrm{PSp}_{2r}$ or $\mathrm{PSO}_{2r}$ then the parameters $\{(s_i,\lambda_i)\}$ are 
\begin{align*}
    \left\{(r-1-2j,1):0 \leq j \leq \lfloor (r-1)/2 \rfloor \right\} &\textrm{ if } G \cong \mathrm{PSp}_{2r},\\
     \left\{(r-2-2j,1):0 \leq j \leq \lfloor (r-2)/2 \rfloor \right\} &\textrm{ if }G \cong \mathrm{PSO}_{2r} \textrm{ and }  r\ge 3.
\end{align*}
In all cases, every good ordering has the largest parameter $(s_k,\lambda_k)$ with $\lambda_k=1$.
\end{Lem}

\begin{proof}
We use Lemmas \ref{lem:tensor},  \ref{lem:prin:sl2}, and \ref{lem:plethysm} freely in the following.  
If $G \cong \mathrm{PSp}_{2r}$ or $G \cong \mathrm{SO}_{2r+1},$ then as a representation under a principal $\mathfrak{sl}_2$-triple,
\begin{align*}
\mathrm{Hom}_{\cc}(V_P,W_P)&\cong \Sym^{\ell-1}(\cc^2)\otimes \Sym^{2r+p(V)-2\ell-1}(\cc^2)\\
& \cong 
    \Sym^{2r+p(V) -\ell-2}(\cc^2)\oplus \Sym^{2r+p(V)-\ell-4}(\cc^2)\oplus\cdots \oplus \Sym^{|2r+p(V)-3\ell|}(\cc^2).
\end{align*} 
This space is understood to be zero if $r=\ell$ and $G\cong \mathrm{PSp}_{2r}$. 
If $G\cong \mathrm{PSO}_{2r},$
\begin{align*}
 \mathrm{Hom}_{\cc}(V_P,W_P) &\cong \Sym^{\ell-1}(\cc^2)\otimes \left(\cc \oplus \Sym^{2r-2\ell-2}(\cc^2)\right)\\
&\cong 
\Sym^{\ell-1}(\cc^2)\oplus \mathrm{Sym}^{2r-\ell-3}(\cc^2) \oplus \cdots \oplus \mathrm{Sym}^{|2r-3\ell-1|}(\cc^2).\end{align*}
If $G \cong \mathrm{PSp}_{2r}$, we have
\begin{align*}
  \Sym_V(W_P) \cong \Sym^2(\mathrm{Sym}^{\ell-1}(\cc^2)) \cong \bigoplus_{j=0}^{\lfloor (\ell-1)/2 \rfloor}\mathrm{Sym}^{2(\ell-1)-4j}(\cc^2).
\end{align*}
If $G$ is $\mathrm{PSO}_{2r}$ or $\mathrm{SO}_{2r+1},$  we have 
\begin{align*}
    \Sym_V(W_P)
    \cong \wedge^2(\mathrm{Sym}^{\ell-1}(\cc^2)) \cong \bigoplus_{j=0}^{\lfloor (\ell-2)/2 \rfloor}\mathrm{Sym}^{2(\ell-2)-4j}(\cc^2).
\end{align*}
Here by convention this space is zero if $\ell=1$.  
The lemma now follows from Lemma \ref{Lem: nil rad rep}, \eqref{usu:case} and \eqref{spec:case}.
\end{proof}

\subsection{Exceptional cases} 
For exceptional types, we compute the decomposition of $\mathfrak{n}_P$ using LieART 2.0.2 (a Mathematica application) based on the tables of \cite{MillerSahi}.

 Assume $G$ is adjoint of type $E,$ $F,$ or $G$. Let $P_k=M_kN_k\leq G$ denote the maximal parabolic associated to the $k$-th node of the Dynkin diagram of $G,$ using the Bourbaki numbering. For a given parabolic subgroup, consider the grading 
    \[
    \fn_{P_k} = \bigoplus_{i\geq1}\fn_{P_k}(i),
    \] associated to the action of $Z(M)$.  The columns of the tables correspond to the graded piece we consider.
    
   We list the resulting $\fsl_2$-representations by the highest weight. For example, the representation $\Sym^{n}(\cc^2)$ will be denoted $\mathbf{n}.$ In particular, under the assumption $G$ is adjoint, the data $(s,\lam)$ associated to the representation $\mathbf{n}$ appearing in $\fn_{P_k}(i)$ is  $(\frac{\mathbf{n}}{2},i)$.

\smallskip

\noindent    
\underline{$E_6$:}

\begin{center}
\begin{tabular}{ccccc}
\hline
Node & $i=1$       & $i=2$     & $i=3$ \\ \hline
1    & $\mathbf{10},\mathbf{4}$      &       &   \\
2    & $\mathbf{9},\mathbf{5},\mathbf{3}$   & $\mathbf{0}$     &   \\
3    & $\mathbf{7},\mathbf{5},\mathbf{3},\mathbf{1}$   & $\mathbf{4}$     &   \\
4    & $\mathbf{5},\mathbf{3},\mathbf{3},\mathbf{1},\mathbf{1}$ & $\mathbf{4},\mathbf{2},\mathbf{0}$ & $\mathbf{1}$ \\
5    & $\mathbf{7},\mathbf{5},\mathbf{3},\mathbf{1}$   & $\mathbf{4} $    &   \\
6    & $\mathbf{10},\mathbf{4}   $   &       &   \\ \hline
\end{tabular}
\end{center}

\noindent    
\underline{$E_7$:}

\begin{center}
\begin{tabular}{ccccc}
\hline
Node & $i=1$         & $i=2$       & $i=3$   & $i=4$ \\ \hline
1    & $\mathbf{15},\mathbf{9},\mathbf{5}$      & $\mathbf{0}$       &     &   \\
2    &$ \mathbf{12},\mathbf{8},\mathbf{6},\mathbf{4},\mathbf{0}$  & $\mathbf{6}$       &     &   \\
3    & $\mathbf{9},\mathbf{7},\mathbf{5},\mathbf{3},\mathbf{1} $  &$ \mathbf{8},\mathbf{4},\mathbf{0} $  & $\mathbf{1} $ &   \\
4    & $\mathbf{6},\mathbf{4},\mathbf{4},\mathbf{2},\mathbf{2},\mathbf{0}$ & $\mathbf{6},\mathbf{4},\mathbf{2},\mathbf{2} $&$ \mathbf{4},\mathbf{2}$ &$ \mathbf{2}$ \\
5    &$ \mathbf{8},\mathbf{6},\mathbf{4},\mathbf{4},\mathbf{2},\mathbf{0}$ & $\mathbf{6},\mathbf{4},\mathbf{2}$   & $\mathbf{4}$   &   \\
6    & $\mathbf{11},\mathbf{9},\mathbf{5},\mathbf{3}$    & $\mathbf{8},\mathbf{0}$     &     &   \\
7    & $\mathbf{16},\mathbf{8},\mathbf{0}$   &         &     &   \\ \hline
\end{tabular}
\end{center}

\noindent    
\underline{$E_8$:}

\begin{center}
    \begin{tabular}{ccccccc}
\hline
Node & $i=1$               & $i=2$           & $i=3$       & $i=4$     & $i=5$   & $i=6$ \\ \hline
1    & $\mathbf{21},\mathbf{15},\mathbf{11},\mathbf{9},\mathbf{3}$      & $\mathbf{12},\mathbf{0}$        &         &       &     &   \\
2    & $\mathbf{15},\mathbf{11},\mathbf{9},\mathbf{7},\mathbf{5},\mathbf{3}$     & $\mathbf{12},\mathbf{8},\mathbf{4},\mathbf{0}$    & $\mathbf{7}$       &       &     &   \\
3    & $\mathbf{11},\mathbf{9},\mathbf{7},\mathbf{5},\mathbf{3},\mathbf{1}$      & $\mathbf{12},\mathbf{8},\mathbf{6},\mathbf{4},\mathbf{0}$  & $\mathbf{7},\mathbf{5}$     & $\mathbf{6}$     &     &   \\
4    & $\mathbf{7},\mathbf{5},\mathbf{5},\mathbf{3},\mathbf{3},\mathbf{1}$       & $\mathbf{8},\mathbf{6},\mathbf{4},\mathbf{4},\mathbf{2},\mathbf{0}$ & $\mathbf{7},\mathbf{5},\mathbf{3},\mathbf{1}$ & $\mathbf{6},\mathbf{4},\mathbf{2}$ & $\mathbf{3},\mathbf{1}$ & $\mathbf{4}$ \\
5    & $\mathbf{9},\mathbf{7},\mathbf{5},\mathbf{5},\mathbf{3},\mathbf{3},\mathbf{1}$     & $\mathbf{8},\mathbf{6},\mathbf{4},\mathbf{4},\mathbf{2},\mathbf{0}$ & $\mathbf{7},\mathbf{5},\mathbf{3},\mathbf{1}$ & $\mathbf{6},\mathbf{2}$   & $\mathbf{3}$   &   \\
6    & $\mathbf{12},\mathbf{10},\mathbf{8},\mathbf{6},\mathbf{4},\mathbf{2}$     & $\mathbf{10},\mathbf{8},\mathbf{6},\mathbf{2}$    & $\mathbf{10},\mathbf{4}$    & $\mathbf{2}$     &     &   \\
7    & $\mathbf{17},\mathbf{15},\mathbf{9},\mathbf{7}, \mathbf{1}$ & $\mathbf{16},\mathbf{8}, \mathbf{0}$   & $\mathbf{1}$       &       &     &   \\
8    & $\mathbf{27},\mathbf{17},\mathbf{9}$           &   $\mathbf{0}$           &         &       &     &   \\ \hline
\end{tabular}
\end{center}

\noindent
$\underline{F_4:}$

\begin{center}
    \begin{tabular}{ccccc}
\hline
Node & $i=1$   & $i=2$     & $i=3$ & $i=4$ \\ \hline
1    & $\mathbf{9},\mathbf{3}$ & $\mathbf{0}$     &   &   \\
2    & $\mathbf{5},\mathbf{3},\mathbf{1}$ & $\mathbf{4},\mathbf{0}$   & $\mathbf{1}$ &   \\
3    & $\mathbf{3},\mathbf{1}$   & $\mathbf{4},\mathbf{2},\mathbf{0}$ & $\mathbf{1}$ & $\mathbf{2}$ \\
4    & $\mathbf{6},\mathbf{0}$   & $\mathbf{6}$     &   &  \\\hline
\end{tabular}
\end{center}

\noindent $\underline{G_2:}$
\begin{center}
    \begin{tabular}{cccc}
\hline
Node & $i=1$   & $i=2$     & $i=3$ \\ \hline
1    & $\mathbf{1}$ & $\mathbf{0}$     &   $\mathbf{1}$   \\
2    & $\mathbf{3}$ & $\mathbf{0}$   &  \\\hline
\end{tabular}
\end{center}

\section*{List of symbols}

\begin{center}
\begin{longtable}{l c r}
$X_P^{\circ}$ & $P^{\mathrm{der}}\backslash G$& \S \ref{ssec:BKspace}\\
$\beta$ & simple root in $\Delta$ associated to $P$ & \S \ref{ssec:Pl}\\ 
$\omega_\beta$ & fundamental weight associated to $\beta$ & \S \ref{ssec:Pl}\\
$\omega_P$ & weight in $X^\ast(T)$ attached to $P$ & \eqref{omegaP}\\
$X_P$& affine completion of $X_P^{\circ}$ & \S \ref{ssec:Pl}\\
$v_{P}$ & highest weight vector in $V_P(F)$ & \S \ref{ssec:Pl}\\
$v_{P^{\mathrm{op}}}^\ast$  & lowest weight vector dual to $v_p$ in $V_P^{\vee}(F)$ & \S \ref{ssec:Pl}\\
$\mathrm{Pl}=\mathrm{Pl}_{v_P}$ & Pl\"ucker embedding $\mathrm{Pl}:X_P \lto Gv_P$ & \S \ref{ssec:Pl}\\
$V_P$ & right representation of $G$ of highest weight $-\omega_P$ & \S \ref{ssec:Pl}\\
$\langle \cdot, \cdot\rangle_{P|P^{\mathrm{op}}}$ & pairing on $X^\circ_P\times X^\circ_{P^{\mathrm{op}}}$ & \eqref{PPop:pairing}\\
$(\cdot)_{\chi,P},$ $(\cdot)_{\chi,P^{\mathrm{op}}}^{\mathrm{op}}$ & Mellin transform along $\chi$ & \eqref{Mellin} \\ 
$I_P(\chi),\overline{I}_{P^{\mathrm{op}}}(\chi)$ &  normalized induction & \eqref{normalizedinduct}\\
$V_{A,B}$ & $\{s \in \cc:A<\mathrm{Re}(s)<B\}$ & \eqref{VAB}\\
$|\cdot|_{A,B,p}$ & $\sup_{s \in V_{A,B}}|p(s)\phi(s)|$ & \eqref{semi:normf}\\
$ \widehat{K}_{\mathbb{G}_m}$ & set of (unitary) characters & \eqref{KGM}\\
$ |\cdot|_{A,B,p_{P|Q},\Omega,D}$ & seminorm & \eqref{seminorm}\\
$U(\fm^{\mathrm{ab}}\oplus \fg)$ & universal enveloping algebra of $(\fm^{\mathrm{ab}}\oplus \fg)_\cc$ & \S \ref{Section: twisting}\\
$\lambda_!(\mu_s)$ & normalized operator attached to $(s,\lambda)$ & \eqref{eqn: eta def} \\
$L$ & graded representation $L$ of $\mathbb{G}_m$ with attached data $\{(s_i,\lambda_i)\}$ & \S \ref{Section: twisting}\\
$A(L),B(L)$ & extended real numbers attached to $L$ & \eqref{AB}\\
$a_L(\chi)$ & $\prod_{i\in I}L(-s_i,\chi^{\lambda_i})$ & \eqref{al}\\
$a_{P|P}(\chi)$& $a_{\widetilde{L}}(\chi^{-1})$ with $L=\widehat{\fn}_P^e$ &  \eqref{ap}\\
$a_{P|P^{\mathrm{op}}}(\chi)$ & $a_L(\chi)$ with $L=\widehat{\fn}_P^e$ & \eqref{ap}\\
$\mathcal{S}_L$ &  Fr\'echet space attached to $L$ &  \S \ref{Section: twisting} \\
$\mathcal{S}(X_P(F))$ & Schwartz space on $X_P(F)$ & Def. \ref{defn:arch}\\
$\{e,h,f\}$ & principal $\fsl_2$ triple in $\widehat{\fm}$& \S \ref{ssec:BKL} \\
$\widehat{\fn}_P^e$ & space of highest weight vectors in $\widehat{\fn}_P$ for a principal $\fsl_2$-triple& \eqref{L:def}\\
$\mu_L$ & normalized operator attached to $L$ & \eqref{mul}\\
$\mu_L(\chi)$ & $\prod_{i=1}^k\gamma(-s_i,\chi^{\lambda_i},\psi)$& \eqref{mulchi} \\
$\mu_P$ & $\mu_L$ with $L=\widehat{\fn}_P^e$ & \eqref{etaP}\\
$\mathcal{R}_{P|P}, \mathcal{R}_{P|P^{\mathrm{op}}}$ & Radon transform & \eqref{R}\\ 
$\iota_{w_0} $ & isomorphism $\iota_{w_0}:\mathcal{S}(X_P(F)) \tilde{\lto} \mathcal{S}(X_{P^{\mathrm{op}}}(F))$& \S \ref{ssec:FT}\\
$\mathcal{F}_{P|P^{\mathrm{op}}}$ & $\mu_P\circ \mathcal{R}_{P|P^{\mathrm{op}}}$ & \eqref{FPPOP} \\

$w_0$ & $w_0Pw_0^{-1}=P^{\mathrm{op}}$& \eqref{w00}\\
$\mathcal{F}_{P|P^{\mathrm{op}}}^{\mathrm{geo}}$ & $\mu_{P}^{\mathrm{geo}}\circ\calr_{P|P^{\mathrm{op}}}$ & \eqref{geopart}\\
$\mathcal{F}_{X_P}$ & $\iota_{w_0} \circ \mathcal{F}_{P|P^{\mathrm{op}}}$ &  \eqref{FXP2}\\

$\mu_P^{\mathrm{aug}}$ & $\lambda_{1!}(\mu_{s_1}) \circ \dots \circ \lambda_{(k-1)!}(\mu_{s_{k-1}})$ & \eqref{eqn: eta aug}\\
$\mu_P^{\mathrm{geo}}$ & $[1]_!(\mu_{s_k})$& \eqref{eqn: eta aug} \\

$(V_i,Q_i)$ & quadratic space of even dimension & \S \ref{sec:Y}\\
$(V,Q)$ & $(\prod_{i=1}^3 V_i,\sum_{i=1}^3Q_i)$ & \S \ref{sec:Y}\\

$\chi_Q$ & quadratic character attached to $Q$& \eqref{chiQi}\\
$\gamma(Q)$ & Weil index of $Q$& \S \ref{sec:FY}\\

$X$ & $X_P$ where $G=\mathrm{Sp}_6$ and $P$ is the Siegel parabolic & \S \ref{sec:Y}\\
$x_0$ & representative of the open $\mathrm{SL}_2^3$-orbit in $X$& \eqref{x0}\\

$N_0$ & stabilizer of $x_0$ in $\SL_2^3$ & \eqref{N0}\\

$Y$ & $\{ (y_1,y_2,y_3) \in V:Q_1(y_1)=Q_2(y_2)=Q_3(y_3)\}.$ & \eqref{Ydef} \\
$Y^{\mathrm{ani}}$ & anisotropic vectors in $Y$ & \eqref{Yani}\\
$d\mu(y)$ & measure on $Y$ & \S \ref{sec:Y}\\
$\mathcal{S}(Y(F))$ & Schwartz space of $Y(F)$ &  \S \ref{sec:Y}\\ 
$\mathcal{F}_Y$ & Fourier transform on $\mathcal{S}(Y(F))$ & \S \ref{sec:Y}\\
$\mathcal{S}$ & $\mathrm{Im}(\mathcal{S}(V(F)) \to C^\infty(Y^{\mathrm{sm}}(F)))$& \eqref{calS}\\
$\mathfrak{r}$ & rational function on $(F^\times)^3$& \eqref{q2}\\

$\mathfrak{c}(u,v)$ & $v_1Q_1(u_1)+v_2Q_2(u_2)-(v_1+v_2)Q_3(u_3)$ & \eqref{cuv} \\
$v'$ & $(v_1,v_2,-v_1-v_2)$ & \eqref{v'} 
\end{longtable}
\end{center}


\bibliographystyle{alpha}

\bibliography{bibs_FourierTransform}
\end{document}